\pdfoutput=1

\documentclass[a4paper,11pt,oneside]{book}

\usepackage{syntonly}

\usepackage[american]{babel}

\usepackage[applemac]{inputenc}

\usepackage{dsfont}
\usepackage{libertine}
\usepackage[libertine]{newtxmath}

\usepackage{indentfirst}

\usepackage{emptypage}  

\usepackage{perpage}

\usepackage[a4paper, top=4.2cm, bottom=4.2cm, left=4cm, right=3cm]{geometry}
\usepackage{setspace}

\usepackage{titlesec}
\titleformat{\chapter}[display]
{\normalfont\bfseries\filcenter}
{\LARGE\thechapter}
{1ex}
{\titlerule[2pt]
\vspace{2ex}%
\LARGE\scshape}
[\vspace{1ex}%
{\titlerule[2pt]}]
\titleformat*{\section}{\scshape\Large\bfseries}
\titleformat*{\subsection}{\scshape\large\bfseries}

\usepackage{fancyhdr}
\fancypagestyle{plain}{%
\fancyhf{} 
\fancyfoot[LE,RO]{\thepage}

}

\usepackage{enumitem}

\usepackage{fixltx2e}

\usepackage{color}

\newcommand\red{\textcolor{red}}

\usepackage{multicol}

\usepackage{hyperref}

\usepackage{booktabs}
\usepackage{graphicx}
\usepackage{tabularx}
\usepackage{caption,subcaption}

\usepackage{longtable}
\usepackage[vcentermath]{youngtab}
\usepackage{tikz,tikz-cd}
\usetikzlibrary{matrix, calc, arrows, decorations.markings, positioning}

\usepackage{amsmath,amssymb,amsthm}
\usepackage{stmaryrd,mathtools,mathdots}
\usepackage{eufrak}

\usepackage{accents}

\usepackage{bm,bbm}

\usepackage{braket}  

\numberwithin{equation}{chapter}

\usepackage{subdepth}

\usepackage{chngcntr}
\counterwithout{figure}{chapter}
\counterwithout{table}{chapter}

\usepackage[style=alphabetic, backend=bibtex, citestyle=alphabetic, firstinits=true, maxnames=99, sorting=anyt]{biblatex}
\DeclareFieldFormat[article, inbook, incollection, inproceedings, patent, thesis, unpublished]{title}{#1} 
\renewbibmacro{in:}{}

\newenvironment{myenumerate}{%
\renewcommand{\theenumi}{(\roman{enumi})}%
\renewcommand{\labelenumi}{\theenumi}%
\begin{list}{\labelenumi}
	{%
	\setlength{\itemsep}{0.4em}%
	\setlength{\topsep}{0.5em}%
	\setlength\leftmargin{1.8em}%
	\setlength\labelwidth{2.05em}%
	\setlength{\labelsep}{0.4em}%
	\usecounter{enumi}%
	}%
	}%
{\end{list}
}
\renewenvironment{enumerate}{
\begin{myenumerate}}%
{\end{myenumerate}}

\newcommand\Item[1][]{%
  \ifx\relax#1\relax  \item \else \item[#1] \fi
  \abovedisplayskip=0pt\abovedisplayshortskip=0pt~\vspace*{-\baselineskip}}

\MakePerPage{footnote}

\makeatletter
\newcommand*{\myfnsymbolsingle}[1]{%
  \ensuremath{%
    \ifcase#1
    \or 
      \dagger
    \else 
      \@ctrerr  
    \fi
  }%
}   
\makeatother

\usepackage{alphalph}
\newalphalph{\myfnsymbolmult}[mult]{\myfnsymbolsingle}{}

\theoremstyle{plain}
\newtheorem{theorem}{Theorem}[chapter]
\newtheorem{lemma}[theorem]{Lemma}
\newtheorem{proposition}[theorem]{Proposition}
\newtheorem{corollary}[theorem]{Corollary}

\theoremstyle{definition}
\newtheorem{definition}[theorem]{Definition}
\newtheorem{remark}[theorem]{Remark}
\newtheorem{example}[theorem]{Example}

\DeclarePairedDelimiter\abs{\lvert}{\rvert} 
\makeatletter
\let\oldabs\abs
\def\abs{\@ifstar{\oldabs}{\oldabs*}}

\DeclarePairedDelimiterX{\norm}[1]{\lVert}{\rVert}{#1} 
\let\oldnorm\norm
\def\norm{\@ifstar{\oldnorm}{\oldnorm*}}

\DeclarePairedDelimiterX{\ceil}[1]{\lceil}{\rceil}{#1} 
\let\oldceil\ceil
\def\ceil{\@ifstar{\oldceil}{\oldceil*}}

\DeclarePairedDelimiterX{\floor}[1]{\lfloor}{\rfloor}{#1} 
\let\oldfloor\floor
\def\floor{\@ifstar{\oldfloor}{\oldfloor*}}
\makeatother

\newcommand{\e}{\mathrm{e}} 
\newcommand{\Exp}{\mathrm{Exp}} 
\newcommand{\Geom}{\mathrm{Geom}} 
\DeclareMathOperator{\sgn}{sgn} 
\DeclareMathOperator{\shape}{sh} 
\DeclareMathOperator{\type}{Type} 
\DeclareMathOperator{\gtype}{gType} 
\newcommand{\Sym}{\vmathbb{S}} 
\newcommand{\sym}{\mathrm{sym}}
\newcommand{\diff}{\mathop{}\!\mathrm{d}}
\renewcommand{\P}{\vmathbb{P}} 
\newcommand{\E}{\vmathbb{E}} 
\newcommand{\1}{\mathds{1}} 
\renewcommand{\i}{\mathrm{i}} 
\newcommand{\GL}{\mathrm{GL}} 
\newcommand{\gl}{\mathfrak{gl}} 
\newcommand{\SL}{\mathrm{SL}} 
\newcommand{\SO}{\mathrm{SO}} 
\renewcommand{\O}{\mathrm{O}} 
\newcommand{\so}{\mathfrak{so}} 
\newcommand{\Sp}{\mathrm{Sp}} 
\newcommand{\schur}{\mathrm{s}} 
\renewcommand{\sp}{\mathrm{sp}} 
\DeclareMathOperator{\Pf}{Pf} 
\newcommand{\Ai}{\mathrm{Ai}} 
\newcommand{\cont}{\mathrm{cont}} 

\renewcommand{\emptyset}{\varnothing}
\newcommand{\N}{\vmathbb{N}} 
\newcommand{\Z}{\vmathbb{Z}} 
\newcommand{\R}{\vmathbb{R}} 
\newcommand{\C}{\vmathbb{C}} 
\renewcommand{\H}{\vmathbb{H}} 
\newcommand{\arrays}{\R^{\mathcal{I}}}

\renewcommand{\epsilon}{\varepsilon}
\renewcommand{\rho}{\varrho}
\renewcommand{\phi}{\varphi}

\DeclareMathSymbol{\widehatsym}{\mathord}{largesymbols}{"62}

\renewcommand{\hat}{\widehat}

\renewcommand{\tilde}{\widetilde}

  {\left\lbrace\begin{array}{@{}l@{}}}%
  {\end{array}\right.}

\makeatletter
\newsavebox{\mybox}\newsavebox{\mysim}
\newcommand{\asymptotic}[1]{%
  \savebox{\mybox}{\hbox{\kern3pt$\scriptstyle#1$\kern3pt}}%
  \savebox{\mysim}{\hbox{$\sim$}}%
  \mathbin{\overset{#1}{\kern\z@\resizebox{\wd\mybox}{\ht\mysim}{$\sim$}}}%
}
\makeatother

\newcommand{\HRule}{\rule{\linewidth}{0.2mm}}

\newcommand{\RSK}{\mathrm{RSK}}
\newcommand{\gRSK}{\mathrm{gRSK}}

\usepackage[vcentermath]{youngtab}
\usepackage[balance]{diagrams}
\diagramstyle[labelstyle=\scriptstyle]

\newcommand{\Flat}{}
\newcommand{\hFlat}{\mathrm{half}}
\newcommand{\rFlat}{\mathrm{res}}
\newcommand{\sFlat}{\mathrm{sym}}
\newcommand{\antisym}{\antisymTriangle[0.13]\text{-}\mathrm{sym}}
\newcommand{\doublesym}{\doublesymTriangle[0.13]\text{-}\mathrm{sym}}

\newcommand{\fZ}{Z^{\Flat}} 
\newcommand{\hZ}{Z^{\hFlat}} 
\newcommand{\rZ}{Z^{\rFlat}} 
\newcommand{\sZ}{Z^{\sFlat}} 

\newcommand{\fPi}{\Pi^{\Flat}} 
\newcommand{\hPi}{\Pi^{\hFlat}} 
\newcommand{\rPi}{\Pi^{\rFlat}} 

\newcommand{\fI}{\mathcal{I}^{\Flat}} 
\newcommand{\hI}{\mathcal{I}^{\hFlat}} 
\newcommand{\rI}{\mathcal{I}^{\rFlat}} 

\newcommand{\fG}{\Gamma^{\Flat}} 
\newcommand{\hG}{\Gamma^{\hFlat}} 
\newcommand{\rG}{\Gamma^{\rFlat}} 

\newcommand{\fPhi}{\Phi^{\Flat}} 
\newcommand{\hPhi}{\Phi^{\hFlat}} 
\newcommand{\rPhi}{\Phi^{\rFlat}} 

\newcommand{\fT}{\mathrm{T}^{\Flat}} 
\newcommand{\hT}{\mathrm{T}^{\hFlat}} 
\newcommand{\rT}{\mathrm{T}^{\rFlat}} 

\newcommand{\fTau}{\tau^{\Flat}} 
\newcommand{\hTau}{\tau^{\hFlat}} 
\newcommand{\rTau}{\tau^{\rFlat}} 
\newcommand{\sTau}{\tau^{\sFlat}}

\newcommand{\antisymTau}{\tau^{\antisym}}
\newcommand{\doublesymTau}{\tau^{\doublesym}}

\newcommand{\fc}{c^{\Flat}}
\newcommand{\hc}{c^{\hFlat}}
\newcommand{\rc}{c^{\rFlat}}

\newcommand{\fk}{k^{\Flat}}
\newcommand{\hk}{k^{\hFlat}}
\newcommand{\rk}{k^{\rFlat}}

\newcommand{\fK}{K^{\Flat}} 
\newcommand{\hK}{K^{\hFlat}} 

\newcommand{\fKone}{K^{(1)}}
\newcommand{\fKtwo}{K^{(2)}}
\newcommand{\fKthree}{K^{(3)}}
\newcommand{\fKfour}{K^{(4)}}
\newcommand{\fKonetilde}{\tilde{K}^{(1)}}
\newcommand{\fKtwotilde}{\tilde{K}^{(2)}}
\newcommand{\fKthreetilde}{\tilde{K}^{(3)}}
\newcommand{\fKfourtilde}{\tilde{K}^{(4)}}

\newcommand{\hKone}{K^{\hFlat(1)}}
\newcommand{\hKtwo}{K^{\hFlat(2)}}
\newcommand{\hKthree}{K^{\hFlat(3)}}
\newcommand{\hKonetilde}{\tilde{K}^{\hFlat(1)}}
\newcommand{\hKtwotilde}{\tilde{K}^{\hFlat(2)}}
\newcommand{\hKthreetilde}{\tilde{K}^{\hFlat(3)}}

\newcommand{\fKtilde}{\tilde{K}^{\Flat}} 
\newcommand{\hKtilde}{\tilde{K}^{\hFlat}} 

\newcommand{\fHbar}{\underline{H}^{\Flat}} 
\newcommand{\hHbar}{\underline{H}^{\hFlat}} 

\newcommand{\fH}{H^{\Flat}} 
\newcommand{\hH}{H^{\hFlat}} 

\newcommand\llrighttriangle[1][1]{%
\begin{tikzpicture}[scale=#1]
\draw (0,0) -- (0,1) -- (1,0) -- (0,0);
\end{tikzpicture}
}

\newcommand\antisymTriangle[1][1]{%
\begin{tikzpicture}[scale=#1]
\draw (0,0) -- (1,0) -- (1,1) -- (0,1) -- (0,0) -- cycle;
\draw (0,0) -- (1,1);
\end{tikzpicture}
}
\newcommand\doublesymTriangle[1][1]{%
\begin{tikzpicture}[scale=#1]
\draw (0,0) -- (1,0) -- (1,1) -- (0,1) -- cycle;
\draw (0,0) -- (1,1);
\draw (0,1) -- (1,0);
\end{tikzpicture}
}

\newcommand{\T}[2]{\mathrm{T}^{\triangle_{#1}}_{#2}}
\newcommand{\spT}[2]{\mathrm{T}^{\, \llrighttriangle[0.13]_{#1}}_{#2}}
\newcommand{\GT}[2]{\mathrm{GT}^{\triangle_{#1}}_{#2}}
\newcommand{\spGT}[2]{\mathrm{GT}^{\, \llrighttriangle[0.13]_{#1}}_{#2}}
\newcommand{\En}{\mathcal{E}^{\triangle}}
\newcommand{\spEn}{\mathcal{E}^{\, \llrighttriangle[0.13]}}

\begin{document}

\pagenumbering{roman}

\begin{titlepage}





\begin{center}

 
\includegraphics[width=0.35\textwidth]{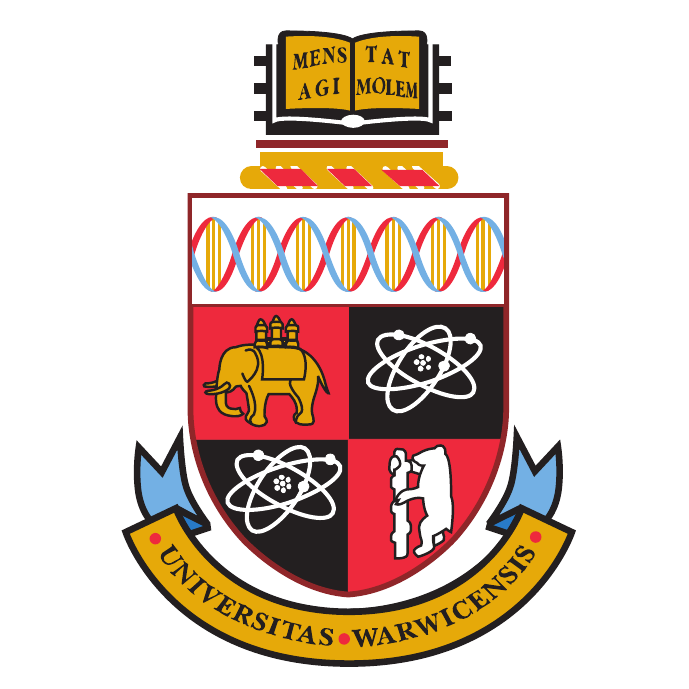} \\[1cm]  

\HRule \\[1cm]


{\scshape \huge \bfseries
Random polymers} \\[0.3cm]
{\scshape \huge \bfseries
via orthogonal Whittaker and} \\[0.3cm]
{\scshape \huge \bfseries
symplectic Schur functions} \\[1cm]

{\LARGE Elia Bisi} \\[0.5cm]

\HRule\\[3.3cm]

{\large Thesis submitted for the degree of \emph{Doctor of Philosophy}}

\vfill


{\large\bfseries University of Warwick \\ Department of Statistics} \\[0.4cm]
{\large July 2018}

\end{center}

\end{titlepage}

\cleardoublepage

\tableofcontents
\markboth{Contents}{Contents}

\renewcommand{\listfigurename}{List of figures}
\cleardoublepage
\markboth{\listfigurename}{\listfigurename}

\listoffigures
\chapter*{Acknowledgments}
\markboth{Acknowledgments}{Acknowledgments}
\addcontentsline{toc}{chapter}{Acknowledgments}
\label{ch:acknowledgments}

I am very grateful to my PhD supervisor Nikos Zygouras for truly caring about my research project and giving me valuable advice.

I also wish to thank Jon Warren and Patrik Ferrari, who carefully examined my thesis, provided useful suggestions, and stimulated interesting discussions.

Finally, I am thankful to Alejandra for always being there.
\chapter*{Declarations}
\markboth{Declarations}{Declarations}
\addcontentsline{toc}{chapter}{Declarations}
\label{ch:declarations}

I declare that I have developed and written this PhD thesis entitled ``Random polymers via orthogonal Whittaker and symplectic Schur functions'' completely by myself, under the supervision of Nikos Zygouras, for the degree of Doctor of Philosophy in Statistics.
I have not used sources or means without declaration in the text.
I also confirm that this thesis has not been submitted for a degree at any other university.

During my PhD I have written the following articles in collaboration with my supervisor Nikos Zygouras:
\begin{itemize}
\item
``Point-to-line polymers and orthogonal Whittaker functions''~\cite{bisiZygouras17a}, to be published in \emph{Transactions of the American Mathematical Society};
\item
``GOE and $\rm{Airy}_{2\to 1}$ marginal distribution via symplectic Schur functions''~\cite{bisiZygouras17b}, to be published in a Springer volume in honor of the 75th birthday of S.~R.~Srinivasa Varadhan.
\end{itemize}
Chapter~\ref{ch:preliminary} deals with preliminary notions that can be found, often presented in a different form, in the literature (references are provided therein); some parts are also based on~\cite{bisiZygouras17a}.
Chapter~\ref{ch:polymer} and section~\ref{sec:expLPP} are an extended version of the content of~\cite{bisiZygouras17a} (even though some of the proofs in section~\ref{sec:expLPP} are obtained via different methods).
Section~\ref{sec:geomLPP} has not been published at the time of writing.
Section~\ref{sec:LPPasymptotics} is based on the content of~\cite{bisiZygouras17b}.
Finally, some of the figures come from one of the two articles cited above.

\chapter*{Abstract}
\markboth{Abstract}{Abstract}
\addcontentsline{toc}{chapter}{Abstract}
\label{ch:abstract}

This thesis deals with some $(1+1)$-dimensional lattice path models from the KPZ universality class: the directed random polymer with inverse-gamma weights (known as \emph{log-gamma polymer}) and its zero temperature degeneration, i.e.\ the last passage percolation model, with geometric or exponential waiting times.
We consider three path geometries: point-to-line, point-to-half-line, and point-to-line with paths restricted to stay in a half-plane.
Through exact formulas, we establish new connections between integrable probabilistic models and the ubiquitous Whittaker and Schur functions.

More in detail, via the use of A.~N.~Kirillov's geometric Robinson-Schensted-Knuth (RSK) correspondence, we compute the Laplace transform of the polymer partition functions in the above geometries in terms of \emph{orthogonal Whittaker functions}.
In the case of the first two geometries we also provide
multiple contour integral formulas.

For the corresponding last passage percolation problems, we obtain new formulas in terms of \emph{symplectic Schur functions}, both directly via $\RSK$ on polygonal arrays and via zero temperature limit from the log-gamma polymer formulas.
As scaling limits of the point-to-line and point-to-half-line models with exponential waiting times, we derive Sasamoto's Fredholm determinant formula for the GOE Tracy-Widom distribution, as well as the one-point marginal distribution of the ${\rm Airy}_{2\to1}$ process.

\cleardoublepage
\pagenumbering{arabic}

\chapter*{Introduction}
\markboth{Introduction}{Introduction}
\addcontentsline{toc}{chapter}{Introduction}
\label{ch:intro}

In statistical physics, surface growth models simulate the behavior of particles such as atoms or molecules that, once ejected onto a surface, attach to each other and form growing islands.
The most celebrated and studied equation that governs the behavior of randomly growing interfaces in physics is the \emph{Kardar-Parisi-Zhang (KPZ) equation}~\cite{kardarParisiZhang86}.
This is a nonlinear stochastic partial differential equation characterized by both \emph{local evolution} and \emph{local randomness}:
\[
\partial_t h =\frac{1}{2}\partial_{xx}^2 h +\frac{1}{2} (\partial_x h)^2+ \dot{W}(t,x) \, ,
\]
where $\dot{W}(t,x)$ is space-time white noise.
The solution $h(x,t)$ is a random function that represents the height of the surface and depends on a spatial coordinate $x$ and a time coordinate $t$.
The KPZ equation is characterized, in one spatial dimension, by the fluctuation exponent $1/3$ and spatial correlation exponent $2/3$, and by certain asymptotic distributions (such as the GOE and GUE Tracy-Widom distributions~\cite{tracyWidom94, tracyWidom96} from random matrix theory) depending on the initial conditions.

Many influential mathematical works have analyzed a few probabilistic models that in various ways represent discretizations of the KPZ equation, leading to remarkable connections with random matrices, combinatorial structures, and representation theoretic objects.
This rich family of models, called \emph{KPZ universality class}, includes:
\begin{itemize}
\item
the problem of the length of the longest increasing subsequence of random permutations~\cite{baikDeiftJohansson99, baikRains01a, baikRains01b}, and the related last passage percolation~\cite{johansson00} and polynuclear growth~\cite{prahoferSpohn00, prahoferSpohn02, ferrari04} models;
\item
the asymmetric simple exclusion process (ASEP)\footnote{This is an interacting particle system on $\Z$ where each site is either occupied by one particle or empty.
Each particle independently, after a mean one exponential time, jumps one site to the right with probability $p$ and one site to the left with probability $1-p$.
The jump is performed only if the target site is vacant, otherwise nothing happens and the particle waits another independent exponential time.
}
with step and step Bernoulli initial conditions, solved via Bethe Ansatz by Tracy and Widom~\cite{tracyWidom09a, tracyWidom09b};
\item
the construction of Macdonald processes by Borodin and Corwin~\cite{borodinCorwin14}, including various interacting particle systems that fall within this scope;
\item
the stochastic six-vertex model~\cite{borodinPetrov16};
\item
the Brownian semi-discrete directed polymer studied by O'Connell~\cite{oConnell12} and its relation to the quantum Toda hamiltonian; 
\item
the log-gamma directed polymer, introduced by Sepp\"al\"ainen~\cite{seppalainen12} and further studied in~\cite{corwinOConnellSeppalainenZygouras14, oConnellSeppalainenZygouras14, borodinCorwinRemenik13, nguyenZygouras17};
\item
the totally asymmetric simple exclusion process (TASEP)\footnote{This particle system is a specialization of ASEP in the case $p=1$, i.e.\ when particles can only jump to the right.} with arbitrary initial conditions, studied via the KPZ fixed point process~\cite{matetskiQuastelRemenik16}.
\end{itemize}

In this thesis we deal with a few discrete models within the KPZ universality class.
In particular, the two main random objects of interest will be the so-called \emph{polymer partition function} and \emph{last passage percolation} (LPP), respectively defined by
\begin{align}
\label{eq:polymerPartitionFn}
Z
&:= \sum_{\pi \in \Pi} \prod_{(i,j)\in\pi} W_{i,j} \, , \\
\label{eq:LPP}
\tau
&:= \max_{\pi \in \Pi} \sum_{(i,j)\in\pi} W_{i,j} \, .
\end{align}
They are both associated to a given path geometry $\Pi$, i.e.\ a finite set of (nearest neighbor) directed paths in the lattice $\Z_{>0}^2$ and to a field of independent random weights $\{W_{i,j}\}$ (assumed to be positive in~\eqref{eq:polymerPartitionFn}) assigned to the sites of the lattice.
Here, by \emph{directed path} we mean a finite sequence $\pi=((i_1,j_1), (i_2, j_2), \dots)$ such that and $(i_{k+1},j_{k+1}) - (i_k,j_k)$ is either $(1,0)$ or $(0,1)$; namely, at each step of the path, one and only one of the two integer coordinates increases by one.
The denomination ``last passage percolation'' is self-explanatory.
Indeed, if each variable $W_{i,j}$ is thought of as a waiting time associated to the site $(i,j)$, the total waiting time associated to a path is the sum of all $W_{i,j}$'s collected along the given path; if then $\Pi$ is the set of all directed paths that go across a region of the lattice, then $\tau$ is the maximal passage time necessary to cross such a region.
On the other hand, explaining where the denomination ``polymer partition function'' for $Z$ arises from will be also helpful to understand its deep connection with the LPP.

The \emph{directed polymer in random environment} (see the lecture notes~\cite{comets17} for a recent review) is a statistical mechanical model defined via the following probability measure $P$ on a set $\Pi$ of directed paths:
\begin{equation}
\label{eq:polymerMeasure}
P(\pi) := \frac{1}{Z^{\beta,\bm{\omega}}} \e^{-\beta H^{\bm{\omega}}(\pi)} \qquad \text{for } \pi\in\Pi \, .
\end{equation}
The objects involved in the polymer measure are:
\begin{itemize}
\item
a parameter $\beta\geq 0$, often thought of as \emph{inverse temperature};
\item
a field $\bm{\omega}$ of random variables $\omega_{i,j}$ assigned to each site $(i,j)$ of the lattice;
\item
a hamiltonian function $H$ that associates to each path $\pi$ the energy
\begin{equation}
\label{eq:polymerHamiltonian}
H^{\bm{\omega}}(\pi)
:= -\sum_{(i,j)\in \pi} \omega_{i,j} \, ;
\end{equation}
\item
a normalization $Z^{\beta,\bm{\omega}}$ (depending on $\beta$ and $\bm{\omega}$), called \emph{partition function}, that makes $P$ a probability measure.
\end{itemize}
According to the intuition, the higher the energy of a path, the lower its probability.
Now, thanks to~\eqref{eq:polymerMeasure} and~\eqref{eq:polymerHamiltonian}, the partition function can be rewritten as
\begin{equation}
\label{eq:polymerPartitionFnBeta}
Z^{\beta,\bm{\omega}} = \sum_{\pi\in\Pi} \e^{\beta \sum_{(i,j)\in\pi} \omega_{i,j}}
= \sum_{\pi\in\Pi} \prod_{(i,j)\in\pi} \e^{\beta \omega_{i,j}}
\end{equation}
and it corresponds to the variable defined in~\eqref{eq:polymerPartitionFn}, if we set $W_{i,j} := \e^{\beta \omega_{i,j}}$ for all $(i,j)$.
It is likewise clear that
\begin{equation}
\label{eq:zeroTempLimit}
\lim_{\beta \to \infty} \frac{1}{\beta} \log Z^{\beta,\bm{\omega}}
= \max_{\pi \in \Pi} \sum_{(i,j)\in\pi} \omega_{i,j}
\end{equation}
is the LPP defined in~\eqref{eq:LPP} for the random environment $\bm{\omega}$.
Since $\beta$ is the inverse temperature, the limit $\beta\to\infty$ above is called \emph{zero temperature limit}.
For this reason, \eqref{eq:LPP} is considered to be the zero temperature version of~\eqref{eq:polymerPartitionFn}.
For a more in-depth and comprehensive analysis of the zero temperature limit, see Appendix~\ref{appendix:zeroTempLimit}.

Let us mention that our two models have direct links to surface growth models: for example, the LPP can be seen to be equivalent to a certain discrete polynuclear growth model~\cite{prahoferSpohn00, prahoferSpohn02, ferrari04}.
In fact, both the polymer partition function and the LPP can be viewed as discretizations of the solution to the KPZ equation with initial conditions depending on the geometry of the lattice paths~\cite{corwin12}.
For instance, the point-to-point geometry corresponds to the so-called \emph{narrow wedge} initial condition for the KPZ equation.

Many breakthroughs in the setting of the KPZ universality class have been possible by analyzing \emph{integrable}  models, where the choice of a particular distribution for the random environment permits to obtain exact formulas, typically in terms of special functions from algebraic combinatorics and representation theory.
For instance, this has been the case for the point-to-point last passage percolation with geometric waiting times~\cite{johansson00}, or the log-gamma polymer~\cite{seppalainen12}.
In this thesis we will focus precisely on the aforementioned models, exploring new path geometries and discovering connections to different special functions.

Let us first discuss the \emph{log-gamma polymer}, i.e.\ simply a polymer model with inverse-gamma weights $W_{i,j}$'s in~\eqref{eq:polymerPartitionFn} (or equivalently log-gamma weights $\omega_{i,j}$'s in~\eqref{eq:polymerPartitionFnBeta}).
We refer to \emph{point-to-point geometry} as the set of all directed paths starting at a fixed point, usually $(1,1)$, and ending at another fixed point, say $(m,n)$.
Using ideas from algebraic combinatorics and in particular a \emph{geometric}\footnote{Originally named \emph{tropical} by Kirillov.} lifting (due to A.~N.~Kirillov~\cite{kirillov01} and further studied in~\cite{noumiYamada04}) of the classical Robinson-Schensted-Knuth correspondence, in~\cite{corwinOConnellSeppalainenZygouras14, oConnellSeppalainenZygouras14} the point-to-point log-gamma polymer partition function was linked to $\GL_n(\R)$-Whittaker functions.
Whittaker functions will be introduced in section~\ref{sec:Whittaker}.
They are ubiquitous in mathematics, as they appear for instance in the theory of automorphic forms, mirror symmetry, quantum integrable systems, and representation theory (see~\cite{lam13} for a review).
An earlier probabilistic connection to $\GL_n(\R)$-Whittaker functions had been found out in~\cite{oConnell12}, for the Brownian semi-discrete polymer.
In~\cite{corwinOConnellSeppalainenZygouras14, oConnellSeppalainenZygouras14}, the Laplace transform of the point-to-point log-gamma polymer partition function was expressed, essentially, as an integral of two $\GL_n(\R)$-Whittaker functions - see~\eqref{eq:pointToPointWhittakerFormula}.
Furthermore, using the Plancherel theory for $\GL_n(\R)$-Whittaker functions and a remarkable Whittaker functions' integral identity due to Bump and Stade (see Theorem~\ref{thm:bumpStade}), the aforementioned Laplace transform was also written as a multiple contour integral of gamma functions.
This formula was subsequently turned into a Fredholm determinant in~\cite{borodinCorwinRemenik13}, thus facilitating the asymptotic analysis (KPZ fluctuation exponent $1/3$ and the GUE Tracy-Widom limiting distribution, in this case).
We will review the point-to-point log-gamma polymer model in the introduction to chapter~\ref{ch:polymer}.

Our contribution in the context of the log-gamma polymer amounts to exploring other path geometries and establishing new connections to different types of Whittaker functions.
We still consider the directed polymer with inverse-gamma weights, but with the endpoint lying free on a line.
In particular, we take into account three different geometries of paths (see also Figure~\ref{fig:directedPath} for a graphical representation): 
\begin{enumerate}
\item
\emph{point-to-line} or \emph{flat} geometry, where paths start from $(1,1)$ and end at any point of the line $\{(m,n)\colon m+n = N+1\}$ for a fixed $N$ (equivalently, all paths of a fixed length $N-1$ are considered);
\item
\emph{point-to-half-line} or \emph{half-flat} geometry, where paths start from $(1,1)$ and end at any point of the half-line $\{(m,n)\colon m+n=N+1, \,\, m\leq n \}$ for a fixed $N$;
\item 
\emph{restricted point-to-line} or \emph{restricted half-flat} geometry, where paths start from $(1,1)$, end at any point of the half-line $\{(m,n)\colon m+n=N+1, \,\, m\leq n \}$ for a fixed $N$, and are restricted to stay in the half-space $\{(i,j)\colon i\leq j\}$.
\end{enumerate}
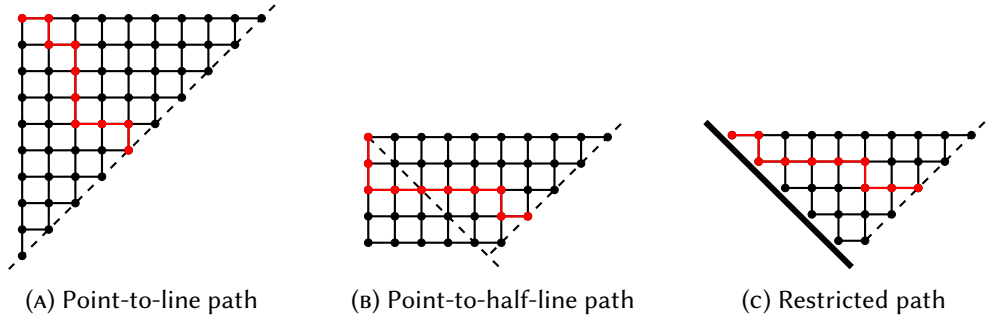
\begin{figure}
\centering
\begin{minipage}[b]{.33\linewidth}
\centering
\begin{tikzpicture}[scale=0.35]
\draw[thick,dashed] (10.5,-.5) -- (.5,-10.5);
\foreach \i in {1,...,10}{
	\draw[thick] (1,-\i) grid (11-\i,-\i);
	\draw[thick] (\i,-1) grid (\i,-11+\i);
	
		\foreach \j in {\i,...,10}{
		\node[draw,circle,inner sep=1pt,fill] at (11-\j,-\i) {};
	}
}

\draw[thick,color=red,-] (1,-1) -- (2,-1) -- (2,-2) -- (3,-2) -- (3,-3) -- (3,-4) -- (3,-5) -- (4,-5) -- (5,-5) -- (5,-6);
\foreach \x in {(1,-1),(2,-1),(2,-2),(3,-2),(3,-3),(3,-4),(3,-5),(4,-5),(5,-5),(5,-6)}{
	\node[draw,circle,inner sep=1pt,fill,red] at \x {};
	}
\end{tikzpicture}
\subcaption{Point-to-line path}
\label{subfig:flatPath}
\end{minipage}%
\begin{minipage}[b]{.33\linewidth}
\centering
\begin{tikzpicture}[scale=0.35]
\draw[thick,dashed] (10.5,-.5) -- (5.5,-5.5); \draw[thick,dashed] (1,-1) -- (6,-6);
\foreach \i in {1,...,5}{
	\draw[thick] (1,-\i) grid (11-\i,-\i);
	
		\foreach \j in {\i,...,10}{
		\node[draw,circle,inner sep=1pt,fill] at (11-\j,-\i) {};
	}
}

\foreach \j in {1,...,5}{
	\draw[thick] (\j,-1) grid (\j,-5);
	\draw[thick] (5+\j,-1) grid (5+\j,-6+\j);
}

\draw[thick,color=red,-] (1,-1) -- (1,-2) -- (1,-3) -- (2,-3) -- (3,-3) -- (4,-3) -- (5,-3) -- (6,-3) -- (6,-4) -- (7,-4);
\foreach \x in {(1,-1),(1,-2),(1,-3),(2,-3),(3,-3),(4,-3),(5,-3),(6,-3),(6,-4),(7,-4)}{
	\node[draw,circle,inner sep=1pt,fill,red] at \x {};
	}
\end{tikzpicture}
\subcaption{Point-to-half-line path}
\label{subfig:hFlatPath}
\end{minipage}%
\begin{minipage}[b]{.33\linewidth}
\centering
\begin{tikzpicture}[scale=0.35]
\draw[line width=0.8mm] (0,-.5) -- (5.5,-6);
\draw[thick,dashed] (10.5,-.5) -- (6,-5);
\foreach \i in {1,...,5}{
	\draw[thick] (\i,-\i) grid (11-\i,-\i);
	
		\foreach \j in {\i,...,5}{
		\node[draw,circle,inner sep=1pt,fill] at (\j,-\i) {};
		\node[draw,circle,inner sep=1pt,fill] at (11-\j,-\i) {};
	}
}

\foreach \j in {1,...,5}{
	\draw[thick] (\j,-1) grid (\j,-\j);
	\draw[thick] (5+\j,-1) grid (5+\j,-6+\j);
}

\draw[thick,color=red,-] (1,-1) -- (2,-1) -- (2,-2) -- (3,-2) -- (4,-2) -- (5,-2) -- (6,-2) -- (6,-3) -- (7,-3) -- (8,-3);
\foreach \x in {(1,-1),(2,-1),(2,-2),(3,-2),(4,-2),(5,-2),(6,-2),(6,-3),(7,-3),(8,-3)}{
	\node[draw,circle,inner sep=1pt,fill,red] at \x {};
	}
\end{tikzpicture}
\subcaption{Restricted path}
\label{subfig:rFlatPath}
\end{minipage}
\caption[Directed lattice paths in the point-to-line, point-to-half-line, and restricted point-to-half-line geometries]{Directed paths in $\Z_{>0}^2$ of length $9$.
The three paths, highlighted in red, correspond to three different geometries, as specified.
The picture is rotated by $90^{\circ}$ clockwise w.r.t.\ the Cartesian coordinate system, to adapt it to the usual matrix/array indexing.}
\label{fig:directedPath}
\end{figure}
We will express the distribution of the log-gamma polymer partition function in the above geometries in terms of Whittaker functions associated to both $\GL_N(\R)$ and the orthogonal group $\SO_{2N+1}(\R)$.
In particular, the Laplace transform of the point-to-line, point-to-half-line, and restricted point-to-line partition functions will be given, respectively, in terms of: an integral of two $\SO_{2N+1}(\R)$-Whittaker functions (formula~\eqref{eq:flatWhittakerFormula}), an integral of one $\SO_{2N+1}(\R)$ and one $\GL_N(\R)$-Whittaker function (formula~\eqref{eq:hFlatWhittakerFormula}), and an integral of one $\SO_{2N+1}(\R)$-Whittaker function (formula~\eqref{eq:rFlatWhittakerFormula}).
It is interesting to note the structure of these formulas in comparison to formula~\eqref{eq:pointToPointWhittakerFormula} for the point-to-point polymer.
Informally, one could say that ``opening'' each part of the endpoint's ``wedge'' to a (diagonally) flat part corresponds to replacing a $\GL_N(\R)$-Whittaker function with an $\SO_{2N+1}(\R)$-Whittaker function.
However, a priori there is no obvious reason why this analogy should take place.
  
Whittaker functions associated to general Lie groups have already appeared in probability to describe the law of the Brownian motion on these Lie groups conditioned on certain exponential functionals~\cite{baudoinOConnell11, chhaibi13}.
In~\cite{nteka18}, orthogonal Whittaker functions also emerged in the description of the Markovian dynamics of systems of interacting particles restricted by a ``soft'' wall.
In our setting, orthogonal Whittaker functions emerge through a combinatorial analysis of the log-gamma polymer via the geometric Robinson-Schensted-Knuth correspondence.
Using an extension of geometric $\RSK$ to polygonal arrays and its properties as in~\cite{nguyenZygouras17}, we determine the joint law of all point-to-point partition functions with endpoint on a line or half-line, also when paths are restricted to a half-plane (see Lemmas~\ref{lemma:flatP2PjointLaw}, \ref{lemma:hFlatP2PjointLaw}, and~\ref{lemma:rFlatP2PjointLaw}).
Subsequently, we derive integral formulas for the Laplace transform of the various point-to-line partition functions after expressing them as a sum of the corresponding point-to-point ones.
Such formulas do not immediately relate to Whittaker functions, but they do so after a certain change of variables and appropriate decompositions of the integrals, thus leading to the aforementioned results.
Even though simple, the alluded change of variables is remarkable in the sense that it precisely couples the structure of orthogonal Whittaker functions with the combinatorial structure of the point-to-line polymers and their Laplace transforms.
As it will become clear in the proofs of Theorems~\ref{thm:flatWhittakerFormula}, \ref{thm:hFlatWhittakerFormula}, and~\ref{thm:rFlatWhittakerFormula}, here the role of the Laplace transform as a functional is crucial.
In fact, we would probably not have been able to see the connection to Whittaker functions, had we aimed to compute different functionals.
On the other hand, the Laplace transform determines the distribution and its use is thus totally justified.

For the point-to-line and point-to-half-line cases, we go one step further by rewriting the Laplace transforms as contour integrals involving gamma functions\footnote{Even though we can also formally do it in the restricted point-to-line case, we miss the estimates that would fully justify such a transformation.}.
This is done via the use of the Plancherel theory for $\GL_N(\R)$-Whittaker functions and certain special integral identities.
The Bump-Stade identity~\eqref{eq:bumpStade} expresses a certain integral of two $\GL_N(\R)$-Whittaker functions in terms of gamma functions.
The Ishii-Stade identity~\eqref{eq:ishiiStade} does a similar job for an integral of one $\SO_{2N+1}(\R)$ and one $\GL_N(\R)$-Whittaker function.
They are important in number theory as they lead to functional equations for automorphic $L$-functions, facilitating the study of their zeros~\cite{bump89, goldfeld06}.
Currently, integral identities for
products of one $\GL_N(\R)$ and one $\SO_{2N}(\R)$-Whittaker functions are not available in the literature; this is the reason why we restrict our presentation to polymers of odd length, even though our combinatorial analysis would also allow us to write analogous formulas for polymers of even length in terms of $\SO_{2N}(\R)$-Whittaker functions.
 
Calabrese and Le Doussal studied in~\cite{calabreseLeDoussal11, leDoussalCalabrese12} the \emph{continuum random polymer} with flat initial conditions.
Via the non-rigorous approach of Bethe ansatz for the Lieb-Liniger model and the \emph{replica trick}, they exhibited that the Laplace transform of the partition function can be written in terms of a Fredholm Pfaffian, from which they obtained the GOE Tracy-Widom asymptotics.
Their method consisted in first deriving a series representation for the half-flat initial condition and then the flat case was deduced from the former via a suitable limit.
More recently, Grange~\cite{grange17} applied the methods of Calabrese-Le Doussal and of Thiery-Le Doussal~\cite{thieryLeDoussal14} to study, again at non-rigorous level, the Laplace transform of the log-gamma polymer with endpoint lying free on a line.

Ortmann-Quastel-Remenik~\cite{ortmannQuastelRemenik16, ortmannQuastelRemenik17} made some steps rigorous in the approach of Calabrese-Le Doussal, working on the ASEP with half-flat and flat initial conditions (see also the earlier work~\cite{lee10}).
Notice that, in the context of particle systems on $\Z$, the flat state is an alternating configuration of particles and holes of type $\cdots 01010101 \cdots$ (where $1$ denotes a particle and $0$ a hole), whereas the half-flat state $\cdots 01010000 \cdots$ alternates particles and holes on one half-line and has holes only on the other half-line.
In the half-flat case~\cite{ortmannQuastelRemenik16}, a series formula was obtained for the $q$-deformed Laplace transform of the ASEP height function.
Formal asymptotics on this formula suggested that the (centered and rescaled) limiting distribution should be given by the one-point marginal of the ${\rm Airy}_{2\to 1} $ process.
Even though such a limiting distribution is expressed in terms of a Fredholm determinant, the Fredholm structure is not apparent before passing to the limit.
In the flat case~\cite{ortmannQuastelRemenik17}, on the track of \cite{calabreseLeDoussal11, leDoussalCalabrese12}, a series formula was obtained for the same $q$-deformed Laplace transform of the height function as a suitable limit of the half-flat case.
This formula does not have an apparent Fredholm structure either.
A Fredholm Pfaffian appears only for a different $q$-deformation of the Laplace transform, which has the drawback of not determining the distribution of the height function.
   
Our approach is orthogonal to the methods used in the above works.
We do not rely on Bethe ansatz computations, but we rather explore the underlying combinatorial structure of the log-gamma polymer.
Moreover, we do not derive the flat case as a limit of the half-flat, but we work instead with their common features; this allows for a more unified and systematic approach and gives access to other geometries too.
We do not pursue in this thesis an asymptotic analysis on the law of the partition functions, as our primary focus has been the analysis of their combinatorial structure and the links to orthogonal Whittaker functions. 
We hope, though, that the methods developed here can provide a rigorous route to the asymptotics of the log-gamma polymer in the flat and half-flat geometries; this is currently under investigation.
Such a hope is also reinforced by the fact that in the zero temperature case (see the discussion that follows) the formulas that emerge from our approach provide alternative derivations of GOE Tracy-Widom and ${\rm Airy}_{2\to 1}$ statistics.

We now pass to the zero temperature setting.
In general, the last passage percolation models are easier to analyze than the corresponding polymer models, thanks to the direct availability of determinantal structures.
For an overview on the LPP and other similar integrable models related to random matrix theory and determinantal structures, we refer the reader to~\cite[ch.~10]{forrester10}.

The first systematic study on the LPP model goes back to Johansson~\cite{johansson00}, who obtained an exact formula for the point-to-point model with i.i.d.\ geometric waiting times via the classical combinatorial Robinson-Schensted-Knuth ($\RSK$) correspondence; he was then able to derive the GUE Tracy-Widom distribution in the scaling limit.
Baik and Rains~\cite{baikRains01a}, among other things, extended Johansson's formula by considering geometric waiting times with a wider range of parameters: this allowed expressing the distribution of the point-to-point LPP in terms of a sum of two Schur functions (see~\eqref{eq:geomPointToPointLPPSchur}).
Schur functions will be introduced in section~\ref{sec:Schur}.
In representation theory, they appear as characters of irreducible representations of classical groups~\cite{fultonHarris91}, and can be defined as certain sums over combinatorial structures such as Gelfand-Tsetlin patterns or Young tableaux.
They possess a very useful determinantal structure, which arises from the Weyl character formulas as well as the Jacobi-Trudi identities.

Our contribution in this setting consists in studying the LPP models corresponding to the same three path geometries considered in positive temperature, with either geometric or exponential waiting times.
In the geometric case, our formulas for the distribution of $\tau$ will involve, besides standard Schur functions (i.e.\ characters of $\GL_N(\C)$), which have appeared so far in the study of LPP models, also symplectic Schur functions (i.e.\ characters of $\Sp_{2N}(\C)$).
In particular, we will find: a sum of two symplectic Schur functions in the point-to-line case, see~\eqref{eq:flatGeomLPP}; a sum of one symplectic and one standard Schur function in the point-to-half-line case, see~\eqref{eq:hFlatGeomLPP}; a sum of one symplectic Schur function in the restricted point-to-line case, see~\eqref{eq:rFlatGeomLPP}.
The proofs will be, in spirit, similar to the analogous ones for the log-gamma polymer: they will be based on a suitable extension (described in subsection~\ref{subsec:tropicalization&RSK}) of the classical $\RSK$ correspondence, instead of its geometric lifting as in positive temperature.

Let us point out here that a formula for the geometric point-to-line LPP was already given in~\cite{baikRains01a}, and involves a sum of \emph{one standard} Schur function parametrized by even partitions - see~\eqref{eq:baikRainsFlatLPP}.
As such, Baik-Rains's formula is essentially different in its structure, but still equivalent, to our point-to-line LPP formula~\eqref{eq:flatGeomLPP} that involves a sum of \emph{two symplectic} Schur functions.
See the discussion at the end of subsection~\ref{subsec:flatGeomLPP} for more details.
Furthermore, Ferrari~\cite{ferrari04} studied the \emph{flat polynuclear growth model}, a continuous version of the LPP model related to the Hammersley process.
Methodologically, both~\cite{baikRains01a} and~\cite{ferrari04} used a symmetrization argument that amounts to considering point-to-point LPP on a square array with symmetric environment about the antidiagonal; this turns out to be essentially equivalent to the point-to-line LPP\footnote{If one thinks about that, it all comes down to the fact that $\max(2a,2b,\dots) = 2\max(a,b.\dots)$.
On the other hand, the analogous statement in positive temperature is not true: the point-to-point partition function for a symmetric environment about the antidiagonal does not equal the point-to-line partition function.
Again, the basic reason for this is quite simple: $a^2 + b^2 + \dots \neq (a+b+\dots)^2$.}.
On the contrary, instead of applying the $\RSK$ to a symmetrized square array, we directly apply it to a triangular array; for this reason, we need to use a generalization of $\RSK$ to generic polygonal arrays.

Let us now discuss the exponential LPP model, which is the most directly related to the TASEP (see the introduction to section~\ref{sec:expLPP} for more details).
The formulas~\eqref{eq:flatExpLPPSchur}, \eqref{eq:hFlatExpLPPSchur}, \eqref{eq:rFlatExpLPPSchur} we obtain for our three path geometries are similar to the geometric case, \emph{mutatis mutandis}: sums need to be replaced by integrals and discrete Schur functions by their continuum analogs.
Continuous Schur functions are continuous limits of rescaled Schur functions, and by Riemann sum approximation turn out to be integrals on continuous Gelfand-Tsetlin patterns (see~\eqref{eq:schurCont} and~\eqref{eq:spSchurCont}).
Notice that, at this stage, we can obtain these formulas in three different ways: (i) via zero temperature limit from the log-gamma polymer formulas, (ii) directly via $\RSK$, and (iii) via exponential limit from the geometric LPP formulas.
We will see all these methods in action.

In principle, one might expect that, in the zero temperature limit, $\SO_{2n+1}(\R)$-Whittaker functions scale to the corresponding (continuous) \emph{orthogonal} Schur functions.
In fact, in Proposition~\ref{prop:soWhittakerRescaling} we will see that orthogonal Whittaker functions scale to (continuous) \emph{symplectic} Schur functions: this is why, for example, the integral of two orthogonal Whittaker functions for the point-to-line log-gamma polymer scales to an integral of two continuous symplectic Schur functions for the point-to-line exponential LPP.
The reason for this lies in the Casselman-Shalika formula~\cite{casselmanShalika80}, which links Whittaker functions on a group $G$ to characters of the irreducible representations of the Langlands dual group of $G$.
Since the dual of the orthogonal group $\SO_{2n+1}(\R)$ is the symplectic group $\Sp_{2n}(\C)$, odd orthogonal Whittaker functions are the analog of symplectic Schur functions.
On the other hand, since the dual of $\GL_n(\R)$ is $\GL_n(\C)$, $\GL_n(\R)$-Whittaker functions are the analog of Schur functions.

We conclude this thesis by an asymptotic analysis of our point-to-line and point-to-half-line exponential LPP models, finding KPZ fluctuations of order $N^{1/3}$ and the expected limiting distributions for these two path geometries.
We thus provide a new route to the GOE Tracy-Widom distribution and to the one-point marginal distribution of the ${\rm Airy}_{2\to 1}$ process respectively, via symplectic Schur formulas.
Using the determinantal form of Schur functions~\eqref{eq:schurContDet} and~\eqref{eq:spSchurContDet} and the so-called Cauchy-Binet identity~\eqref{eq:cauchyBinet}, which expresses the multiple integral of the product between two determinantal functions as the determinant of a single integral, our Schur functions' formulas~\eqref{eq:flatExpLPPSchur} and~\eqref{eq:hFlatExpLPPSchur} can be turned into ratios of determinants (see~\eqref{eq:flatExpLPPdet} and~\eqref{eq:hFlatExpLPPdet}) and then into Fredholm determinants, amenable to asymptotic analysis via steepest descent.

For the point-to-line model we obtain Sasamoto's formula~\cite{sasamoto05} for the GOE Tracy-Widom distribution, see~\eqref{eq:GOE} and~\eqref{eq:GOEkernel}.
This expression is different from the one originally derived by Tracy and Widom, first expressed in terms of Painlev\'e functions~\cite{tracyWidom96} and then in terms of a block-Fredholm Pfaffian~\cite{tracyWidom98, tracyWidom05}. 
Sasamoto's original derivation of~\eqref{eq:GOEkernel} came through the analysis of TASEP with an initial configuration of the form $\cdots01010000\cdots$, where $1$ denotes a particle and $0$ a hole.
The presence of the semi-infinite sequence of holes is technical and the actual focus of the asymptotic analysis in~\cite{sasamoto05} is on the alternating particle-hole regime, which simulates the \emph{flat} initial condition.
The starting point for this derivation was Sch\"utz's determinantal formula~\cite{schutz97} for the occupation probabilites in TASEP, obtained via Bethe ansatz methods.
A proof that Sasamoto's formula actually provides a different expression for the GOE Tracy-Widom distribution was provided in~\cite{ferrariSpohn05}.
Subsequently to~\cite{sasamoto05}, the $\rm{Airy}_{2\to 1}$ process was constructed in~\cite{borodinFerrariSasamoto08} by studying, again via Sch\"utz's and Sasamoto's formulas, the asymptotic distribution of TASEP particles with initial configuration $\cdots01010000\cdots$, but now at the interface between the left half-line $\cdots0101$ of alternating particles-holes and the right half-line $0000\cdots$ of holes.

Asymptotics recovering the GOE Tracy-Widom distribution as a limiting law have been also performed in~\cite{baikRains01b, ferrari04} for last passage percolation and polynuclear growth models, as well as in the more recent nonrigorous work~\cite{leDoussalCalabrese12} for the KPZ equation with flat initial data.
All these already cited works derive Painlev\'e expressions or various forms of block-Fredholm Pfaffian formulas for the GOE Tracy-Widom distribution.
In contrast, our approach leads directly to Sasamoto's Fredholm determinant formula~\eqref{eq:GOE}-\eqref{eq:GOEkernel}, as well as to the ${\rm Airy}_{2\to 1}$ marginal distribution~\eqref{eq:Airy21}-\eqref{eq:Airy21kernel}.
 
{\bf Outline of the thesis.}
In chapter~\ref{ch:preliminary} we introduce the main technical tools that we need (combinatorial and geometric $\RSK$, Schur functions, and Whittaker functions); in this propaedeutic material we will provide most of the proofs to make this thesis as self-contained as possible.
Chapter~\ref{ch:polymer} about the log-gamma polymer models is based on our article~\cite{bisiZygouras17a}: we first present the Whittaker integral formulas and then the contour integrals.
Chapter~\ref{ch:LPP} about the LPP models is structured as follows: section~\ref{sec:geomLPP} on the geometric models does not appear as such in any published work; section~\ref{sec:expLPP} on the exponential models is still mainly based on~\cite{bisiZygouras17a}, even though some proofs are different; section~\ref{sec:LPPasymptotics} on the asymptotic analysis of the exponential models is based on our second article~\cite{bisiZygouras17b}.
In appendix~\ref{appendix:zeroTempLimit} we recap the zero temperature limit in the form we need it; finally, in appendix~\ref{appendix:Whittaker} we discuss Whittaker functions from a number theoretic point of view.

{\bfseries Notation}.
We adopt the standard notations for number sets, such as $\Z$, $\R$ and $\C$.
Other notations for number sets will be either self-explanatory, such as $\Z_{\geq 0}$ for the set of nonnegative integers, or defined where useful.
We denote single numbers in nonbold lower case letters (e.g.\ $x$) and arrays of numbers - such as vectors, matrices, tableaux and partitions - by bold lower case letters (e.g.\ $\bm{x}$).
If such objects are random, we tend to use capitals: again, nonbold for numbers (e.g.\ $X$) and bold for arrays (e.g.\ $\bm{X}$).
We denote by $\Gamma(\cdot)$ the gamma function.
For a real random variables $X$, we call CDF its cumulative distribution function $x \mapsto \P(X\leq x)$.
We also fix once for all the following probability distributions\footnote{Notice that the parametrization of the geometric distribution is not the most standard one.}:

\begin{center}
\begin{tabular}{ccccc}
\toprule
\bfseries Distribution & \bfseries Notation & \bfseries Parameters & \bfseries Support & \bfseries Density/mass \\
\midrule
gamma & ${\rm Gamma}(\alpha,\beta)$ & $\alpha,\beta\in\R_{>0}$ & $x\in\R_{>0}$ & $\frac{\beta^{\alpha}}{\Gamma(\alpha)} x^{\alpha-1} \e^{-\beta x}$ \\
exponential & $\Exp(\lambda)$ & $\lambda\in\R_{>0}$ & $x\in\R_{>0}$ & $\lambda \e^{-\lambda x}$ \\
geometric & $\Geom(p)$ & $p\in(0,1)$ & $k\in\Z_{\geq 0}$ & $(1-p) p^k$ \\
\bottomrule
\end{tabular}
\end{center}

\chapter{Combinatorics and special functions}
\label{ch:preliminary}

In this chapter we present some preliminary notions that will play a crucial role in the thesis: combinatorial correspondences of $\RSK$-type (section~\ref{sec:RSK}), Schur functions (section~\ref{sec:Schur}), and Whittaker functions (section~\ref{sec:Whittaker}).
These tools will be used, in the following chapters, to study the polymer and last passage percolation models in the various point-to-line geometries described in the Introduction.

\section{$\RSK$-type correspondences}
\label{sec:RSK}

We first introduce the classical $\RSK$ correspondence as a combinatorial algorithm.
Next, we describe the so-called geometric $\RSK$ correspondence and its main properties.
Finally, we derive an extended version of the $\RSK$ algorithm as a suitable limit of the geometric $\RSK$.

\subsection{Combinatorial $\RSK$ correspondence}
\label{subsec:RSK}

The classical combinatorial Robinson-Schensted correspondence is a bijection between permutations and pairs of standard Young tableaux of the same shape.
Viewing permutations as special $\{0,1\}$-matrices (``permutation matrices''), it can be generalized to the so-called \emph{Robinson-Schensted-Knuth correspondence} ($\RSK$), which is a bijection between matrices of non-negative integers and pairs of semistandard Young tableaux of the same shape.
We further define the $\RSK$ and describe its main features; for further details and proofs of the results we refer the reader to~\cite{stanley99}.

A \emph{partition} of $n$ of length $l$ is a weakly decreasing sequence $\bm{\lambda} = (\lambda_1, \dots, \lambda_l)$ of $l$ positive integers that sum up to $n$: we write $\bm{\lambda} \vdash n$ and $l(\bm{\lambda}) = l$.
The empty partition $\bm{\lambda} = \emptyset$ is the only one of length $0$.
A graphical representation of a partition is a \emph{Young diagram}, i.e.\ a finite collection of boxes arranged in left-justified rows whose lengths are weakly decreasing, as in Figure~\ref{subfig:YoungDiagram}.

\begin{figure}
\centering
\begin{minipage}[b]{.3\linewidth}
\centering
\[
\yng(4,3,1)
\]
\subcaption{Young diagram corresponding to partition $(4,3,1)$.}
\label{subfig:YoungDiagram}
\end{minipage}\hfill
\begin{minipage}[b]{.3\linewidth}
\centering
\[
\young(1123,244,4)
\]
\subcaption{Semistandard Young tableau of size $8$, shape $(4,3,1)$, and type $(2,2,1,3)$.}
\label{subfig:YoungTableau}
\end{minipage}\hfill
\begin{minipage}[b]{.3\linewidth}
\[
\young(1122,234,44)
\]
\subcaption{Semistandard Young tableau obtained by inserting $2$ into the previous one.}
\label{subfig:YoungTableauInsertion}
\end{minipage}
\caption[Examples of Young diagrams and tableaux]{Examples of Young diagrams and tableaux.}
\label{fig:YoungTableaux}
\end{figure}

A \emph{Young tableau} $\bm{p}$ is obtained by filling the boxes of a Young diagram with positive integers; the \emph{size} of $\bm{p}$ is its total number of boxes; the \emph{shape} of $\bm{p}$, denoted by $\shape(\bm{p})$, is the partition that corresponds to the underlying Young diagram; the \emph{type} of $\bm{p}$ is the sequence $(\type(\bm{p})_1, \type(\bm{p})_2, \dots)$ such that $\type(\bm{p})_i$ is the number of $i$'s in $\bm{p}$ (the sequence is actually finite, in the sense that $\type(\bm{p})_i=0$ for $i$ large enough).
A Young tableau is called \emph{semistandard} if its rows are weakly increasing and its columns are strictly increasing, as in Figure~\ref{subfig:YoungTableau}.

We now define the \emph{insertion} of $k\in\Z_{>0}$ into a semistandard Young tableau $\bm{p}=\{ p_{i,j} \}$.
In the $i$-th row, starting from $i=1$, we search for the smaller $j$ such that $p_{i,j} > k$.
If such a $j$ does not exist, we simply add a box filled with $k$ at the end of the $i$-th row; if such a $j$ does exist, we fill box $(i,j)$ with $k$ and bump the old entry $p_{i,j}$ to the next row $i+1$, where we will try to insert it in the same way.
Clearly, the procedure must stop in a finite number of steps, producing a new semistandard Young tableau whose size is increased by one, as the one of Figure~\ref{subfig:YoungTableauInsertion}.

For any $m \times n$ matrix $\bm{w} = \{w_{i,j}\}$ with entries in $\Z_{\geq 0}$, we consider two words of the same length, which are concatenations of weakly increasing words in the alphabet $\Z_{>0}$:
\begin{enumerate}
\item
the word $v := v_1 \dots v_m$ such that $v_i := 1^{w_{i,1}} \dots n^{w_{i,n}}$;
\item
the word $v' := v'_1 \dots v'_n$ such that $v'_j := 1^{w_{1,j}} \dots m^{w_{m,j}}$.
\end{enumerate}
The \emph{$\RSK$ correspondence} is defined as the map that associates such a matrix $\bm{w}$ to the pair of semistandard Young tableaux $(\bm{p},\bm{q})$, where $\bm{p}$ is obtained by inserting all numbers appearing in word $v$ successively (starting from the empty tableau) and $\bm{q}$ is obtained in the same way using $v'$ instead.
It is immediate by construction that the $j$-th column of $\bm{w}$ sums up to the number of $j$'s in $\bm{p}$, and similarly the $i$-th row of $\bm{w}$ sums up to the number of $i$'s in $\bm{q}$.
Also observe that, transposing matrix $\bm{w}$, the roles of $\bm{p}$ and $\bm{q}$ are interchanged.

\begin{figure}
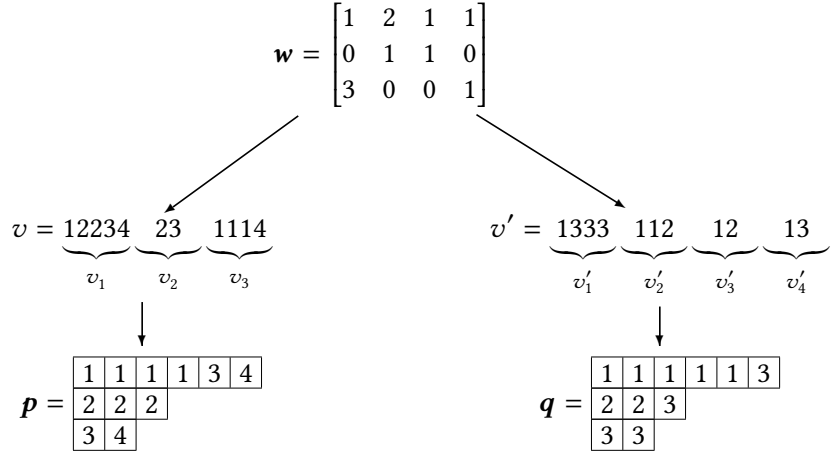

\begin{diagram}
&&
\bm{w} =
\begin{bmatrix}
1 & 2 & 1 & 1 \\
0 & 1 & 1 & 0 \\
3 & 0 & 0 & 1
\end{bmatrix}
&&\\
&\ldTo & &\rdTo &\\
v = \underbrace{12234}_{v_1}\underbrace{23}_{v_2}\underbrace{1114}_{v_3}
&&&
&v' = \underbrace{1333}_{v'_1}\underbrace{112}_{v'_2}\underbrace{12}_{v'_3}\underbrace{13}_{v'_4}
 \\
\dTo &&& &\dTo \\
\bm{p} = \young(111134,222,34)
&&&
&\bm{q} = \young(111113,223,33)
\end{diagram}
\caption[Visualization of the $\RSK$ correspondence]{The $\RSK$ correspondence: matrix $\bm{w}$ is associated to words $v$ and $v'$, which are in turn mapped to the corresponding tableaux $\bm{p}$ and $\bm{q}$.}
\label{fig:RSKexample}
\end{figure}

Looking at the example of $\RSK$ given in figure~\ref{fig:RSKexample}, one notices that $\bm{p}$ and $\bm{q}$ are of the same shape.
This is not fortuitous, because $\bm{q}$ can also be constructed at the same time as $\bm{p}$ as follows: after inserting a number belonging to word $v_i$ into $\bm{p}$, a box is added to $\bm{q}$ in the same position as the box added to $\bm{p}$ in the insertion procedure, and filled with $i$.
It is clear from this alternative construction that $\bm{p}$ and $\bm{q}$, thus called \emph{insertion tableau} and \emph{recording tableau} respectively, are of the same shape.
In fact, any given pair of semistandard Young tableaux with the same shape is the image of a unique $\Z_{\geq 0}$-matrix under $\RSK$.

The reason why we have introduced $\RSK$ lies in its connection with the directed last passage percolation (LPP) model defined in the Introduction: the length of the first row of the two output tableaux corresponds to the LPP from site $(1,1)$ to site $(m,n)$ with waiting time $w_{i,j}$ at site $(i,j)$.
For example, in Figure~\ref{fig:RSKexample}, the LPP time on $\bm{w}$ from $(1,1)$ to $(3,4)$ is $6$, which is also the length of the first row of $\bm{p}$ and $\bm{q}$.
Let us just mention that the lengths of the other rows of $\bm{p}$ and $\bm{q}$ can also be expressed in terms of piecewise linear functions of $\bm{w}$, analogous to LPP and involving non-intersecting  paths: the result is known as Greene's Theorem.

Let us now summarize the properties of $\RSK$ seen so far:
\begin{proposition}[\cite{stanley99}]
\label{prop:RSK}
Let $m,n\geq 1$.
The $\RSK$ correspondence is a bijection between matrices $\bm{w} = \{w_{i,j}\} \in \Z_{\geq 0}^{m\times n}$ and pairs $(\bm{p},\bm{q})$ of semistandard Young tableaux of the same shape such that $\max_{i,j} p_{i,j} \leq n$ and $\max_{i,j} q_{i,j} \leq m$.
The $\RSK$ satisfies the following properties:
\begin{enumerate}
\item
If $\Pi_{m,n}$ is the set of all directed paths from $(1,1)$ to $(m,n)$, then
\[
\shape(\bm{p})_1 = \shape(\bm{q})_1 = \max_{\pi\in \Pi_{m,n}} \sum_{(i,j)\in \pi} w_{i,j} \, .
\]
\item
The type of $\bm{p}$ and $\bm{q}$ is determined by:
\begin{align*}
\type(\bm{p})_j &= \sum_{i=1}^m w_{i,j}
&& \text{for } 1\leq j\leq n \, , \\
\type(\bm{q})_i &= \sum_{j=1}^n w_{i,j}
&& \text{for } 1\leq i\leq m \, .
\end{align*}
\item
If $\bm{w} \xmapsto{\RSK} (\bm{p},\bm{q})$, then $\bm{w}^\top \xmapsto{\RSK} (\bm{q},\bm{p})$.
\end{enumerate}
\end{proposition}

We now introduce a further combinatorial object which is in a bijective correspondence with semistandard Young tableaux.
Let $\bm{p}=\{ p_{i,j} \}$ be a semistandard Young tableau and, for all $i,j\geq 1$, let
\begin{equation}
\label{eq:GTpatternOfYoungTableau}
z_{i,j} := \#\{ \text{entries $\leq i$ in the $j$-th row of $\bm{p}$} \} \, .
\end{equation}
Notice that, for any given $n\geq \max_{i,j}p_{i,j}$, we have that $z_{n,j}=z_{n+1,j}=z_{n+2,j}=\dots$, hence all $z_{i,j}$ with $i>n$ are redundant and may be ignored.
Moreover, by the column strict rule, below the $i$-th row there can only be entries greater than $i$, so that $z_{i,j}=0$ for all $j>i$.
We may thus arrange all significant $z_{i,j}$'s in the triangular array
\begin{equation}
\label{eq:GTpattern}
\bm{z}= \{z_{i,j}\colon 1\leq j\leq i \leq n \} \, .
\end{equation}
An important property of $\bm{z}$ is given by the \emph{interlacing conditions}:
\begin{equation}
\label{eq:interlacing}
z_{i+1,j+1} \leq z_{i,j} \leq z_{i+1,j} \qquad \text{for } 1 \leq j \leq i < n \, ,
\end{equation}
which follow from the column strict rule for $\bm{p}$ as well.
A triangular array $\bm{z}$ of the form~\eqref{eq:GTpattern} satisfying the interlacing conditions~\eqref{eq:interlacing} is called \emph{Gelfand-Tsetlin pattern} of \emph{height} $n$ and is usually pictured as in Figure~\ref{fig:GTpattern} (notice that rows are arranged with second index increasing from right to left).
If, as in the case of~\eqref{eq:GTpatternOfYoungTableau}, its entries are non-negative integers, we call $\bm{z}$ a \emph{$\Z_{\geq 0}$-Gelfand-Tsetlin pattern}.
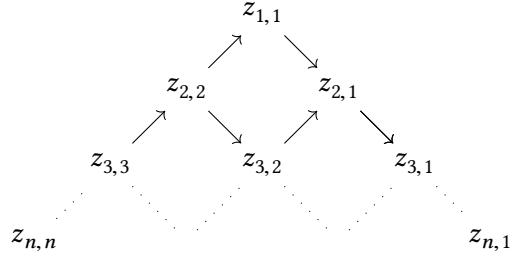
\begin{figure}
\centering
\begin{tikzpicture}[scale=1]

\node (z11) at (1,-1) {$z_{1,1}$};
\node (z21) at (2,-2) {$z_{2,1}$};
\node (z22) at (0,-2) {$z_{2,2}$};
\node (z31) at (3,-3) {$z_{3,1}$};
\node (z32) at (1,-3) {$z_{3,2}$};
\node (z33) at (-1,-3) {$z_{3,3}$};
\node (z41) at (4,-4) {$z_{n,1}$};
\node (z42) at (2,-4) {};
\node (z43) at (0,-4) {};
\node (z44) at (-2,-4) {$z_{n,n}$};

\draw[->] (z22) -- (z11);
\draw[->] (z11) -- (z21);
\draw[->] (z33) -- (z22);
\draw[->] (z22) -- (z32);
\draw[->] (z32) -- (z21);
\draw[->] (z21) -- (z31);
\draw[->] (z21) -- (z31);
\draw[loosely dotted] (z44) -- (z33);
\draw[loosely dotted] (z33) -- (z43);
\draw[loosely dotted] (z43) -- (z32);
\draw[loosely dotted] (z32) -- (z42);
\draw[loosely dotted] (z42) -- (z31);
\draw[loosely dotted] (z31) -- (z41);

\end{tikzpicture}
\caption[Triangular arrays and Gelfand-Tsetlin patterns]{A triangular array of height $n$.
If this is interpreted as a Gelfand-Tsetlin pattern, the arrows illustrate the interlacing conditions: $a \to b$ means $a\leq b$.
The arrows may also refer to the functional~\eqref{eq:glEnergy}: $\mathcal{E}^{\triangle}(\bm{z})$ is the sum of all ratios $a/b$ such that there is an arrow pointing from $a$ to $b$ in the diagram.}
\label{fig:GTpattern}
\end{figure}

By construction, the $i$-th row $(z_{i,1}, \dots, z_{i,i})$ of $\bm{z}$ is the shape of the semistandard Young tableau obtained by removing all boxes filled with numbers greater than $i$ from the original tableau $\bm{p}$; in particular, the bottom row $(z_{n,1}, \dots, z_{n,n})$ is the shape of $\bm{p}$, and is called \emph{shape} of $\bm{z}$ by analogy.
It is likewise clear that the number of $i$'s in $\bm{p}$ equals the difference between the sum of the $i$-th row of $\bm{z}$ and the sum of its $(i-1)$-th row.
Analogously to the type of $\bm{p}$, we then define the \emph{type} of $\bm{z}$ by
\begin{equation}
\label{eq:typeGTpattern}
\type(\bm{z})_i := \sum_{j = 1}^i z_{i,j} - \sum_{j = 1}^{i-1} z_{i-1,j} \qquad \text{for } i=1, \dots, n \, ,
\end{equation}
where the empty sum equals $0$ by convention (as we will always suppose from now on); we thus have that $\type(\bm{z}) = \type(\bm{p})$.

As the correspondence described above between semistandard Young tableaux and integer Gelfand-Tsetlin patterns is easily verified to be bijective, Proposition~\ref{prop:RSK} can be reformulated as follows:
\begin{proposition}
\label{prop:RSKGelfandTsetlin}
Let $m,n\geq 1$.
The $\RSK$ correspondence can be seen as a bijection between matrices $\bm{w} = \{w_{i,j}\} \in \Z_{\geq 0}^{m\times n}$ and pairs $(\bm{z},\bm{z}')$ of $\Z_{\geq 0}$-Gelfand-Tsetlin patterns of height $n$ and $m$ respectively and of the same shape. The $\RSK$ satisfies the following properties:
\begin{enumerate}
\item
\label{prop:RSKGelfandTsetlin_LPP}
If $\Pi_{m,n}$ is the set of all directed paths from $(1,1)$ to $(m,n)$, then
\[
z_{n,1} = z'_{m,1} = \max_{\pi\in \Pi_{m,n}} \sum_{(i,j)\in \pi} w_{i,j} \, .
\]
\item
\label{prop:RSKGelfandTsetlin_type}
The type of $\bm{z}$ and $\bm{z}'$ is determined by:
\begin{align*}
\type(\bm{z})_j &= \sum_{i=1}^m w_{i,j}
&& \text{for } 1\leq j\leq n \, , \\
\type(\bm{z}')_i &= \sum_{j=1}^n w_{i,j}
&& \text{for } 1\leq i\leq m \, .
\end{align*}
\item
\label{prop:RSKGelfandTsetlin_symmetry}
If $\bm{w} \xmapsto{\RSK} (\bm{z},\bm{z}')$, then $\bm{w}^{\top} \xmapsto{\RSK} (\bm{z}',\bm{z})$.
\end{enumerate}
\end{proposition}

For an $m\times n$ input matrix $\bm{w}$, by the column strict rule the common shape of the output tableaux $\bm{p}$ and $\bm{q}$ is at most of length $m\wedge n$, hence in the Gelfand-Tsetlin pattern representation we have $z_{i,j}=0$ and $z'_{i,j}=0$ for all $(i,j)$ such that $j>m\wedge n$ (see for example the zero entry of $\bm{z}$ in Figure~\ref{subfig:GTpatternsEx}).
Ignoring such zero entries, $\bm{z}$ and $\bm{z}'$ can be glued together along their common shape to form a new $m\times n$ matrix $\bm{t}=(t_{i,j})$, as in Figure~\ref{subfig:RSKoutputMatrix}.
For instance, when $m<n$:
\begin{equation}
\label{eq:RSKoutputMatrix}
\bm{t}
=
\begin{tikzpicture}[baseline=(current bounding box.center)]
\matrix (m) [matrix of math nodes,nodes in empty cells,right delimiter={]},left delimiter={[} ]{
z_{m,m} && &z_{n,m}=z'_{m,m} & &z'_{1,1} \\
& & & & & \\
& & & & & \\
& & & & & \\
z_{1,1} && &z_{m,1} & &z_{n,1}=z'_{m,1} \\
} ;
\draw[loosely dotted] (m-1-1)-- (m-5-1);
\draw[loosely dotted] (m-1-1)-- (m-5-4);
\draw[loosely dotted] (m-5-1)-- (m-5-4);
\draw[loosely dotted] (m-1-1)-- (m-1-4);
\draw[loosely dotted] (m-1-4)-- (m-5-6);
\draw[loosely dotted] (m-5-4)-- (m-5-6);
\draw[loosely dotted] (m-1-6)-- (m-5-6);
\draw[loosely dotted] (m-1-4)-- (m-1-6);
\end{tikzpicture}
\, .
\end{equation}
In this representation, the common shape of the two Gelfand-Tsetlin patterns is thus given by the $(n-m)$-th diagonal of $\bm{t}$ read in reverse order, i.e.\ starting from the bottom-right entry of the matrix.

\begin{figure}
\centering
\begin{minipage}[b]{.7\linewidth}
\centering
\[
\bm{z} =
\begin{array}{ccccccc}
        & & &4 \\
        & &3 & &4 \\
        &1 & & 3 & & 5 \\
        0 & &2 & &3 & &6 \\
\end{array}
\qquad
\bm{z}' =
\begin{array}{ccccccc}
        & &5 \\
        &2 & &5 \\
        2 & &3 & &6 \\
\end{array}
\]
\subcaption{Pair of Gelfand-Tsetlin patterns with common shape $(6,3,2)$, corresponding to $\bm{p}$ and $\bm{q}$ respectively.}
\label{subfig:GTpatternsEx}
\end{minipage}\hfill
\begin{minipage}[b]{.27\linewidth}
\centering
\[
\bm{t} =
\begin{bmatrix}
1 &2 &2 &5 \\
3 &3 &3 &5 \\
4 &4 &5 &6
\end{bmatrix}
\]
\subcaption{Matrix constructed by gluing together $\bm{z}$ and $\bm{z}'$.}
\label{subfig:RSKoutputMatrix}
\end{minipage}
\caption[Alternative ways to represent the output of the $\RSK$ correspondence]{Alternative ways to represent the output $(\bm{p},\bm{q})$ of the $\RSK$ correspondence illustrated in Figure~\ref{fig:RSKexample}.}
\label{fig:RSKoutputAlternatives}
\end{figure}

\subsection{Geometric $\RSK$ correspondence}
\label{subsec:gRSK}

As first observed in~\cite{berensteinKirillov01}, the $\RSK$ can be entirely described by operations in the tropical semiring with operations $(\max,+)$.
Via a so-called \emph{geometric lifting}, i.e.\ by replacing these operations with their analogs $(+,\cdot)$ in the usual algebra, A.~N.~Kirillov~\cite{kirillov01} introduced the geometric $\RSK$ ($\gRSK$) correspondence.
In~\cite{oConnellSeppalainenZygouras14}, this was expressed as a bijection between matrices with positive entries via a composition of local birational maps (called local moves).
Here we mainly follow~\cite{nguyenZygouras17}, as the extension to generic polygonal arrays presented therein is needed for the following chapters.
Our description, however, highlights all the possible sequences of local moves that make up $\gRSK$ equivalently, due to their commuting properties.

We define a \emph{polygonal array} to be a numeric array $\bm{t} = \{t_{i,j} \colon (i,j)\in \mathcal{I}\}$ indexed by a finite set $\mathcal{I}\subseteq\Z_{>0}\times\Z_{>0}$ satisfying: if $(i,j)\in\mathcal{I}$, then $(i-1,j)\in\mathcal{I}$ if $i>1$, and $(i,j-1)\in\mathcal{I}$ if $j>1$. 
These conditions simply mean that $\mathcal{I}$ is the index set of a Young diagram, hence $\bm{t}$ can be equivalently defined as a Young diagram filled with numbers.
Given a set of numbers $S$, we denote by $S^{\mathcal{I}}$ the set of all polygonal arrays indexed by $\mathcal{I}$ with entries in $S$.
We say that $(m,n)\in\mathcal{I}$ is a \emph{border index} for $\mathcal{I}$ if $(m+1,n+1)$ does not belong to $\mathcal{I}$, or equivalently if it is the last (i.e.\ rightmost and bottommost) index of its diagonal.
We call $(m,n)$ \emph{outer index} for $\mathcal{I}$ if none of the three sites $(m,n+1),(m+1,n),(m+1,n+1)$ belongs to $\mathcal{I}$, or equivalently if $\mathcal{I}\setminus\{(m,n)\}$ is still the index set of a Young diagram.
Clearly, all outer index is also a border index.
See Figure~\ref{fig:polygonalArray} for a graphical illustration of $\bm{t}$ and its border and outer indices.

\begin{figure}
\centering
\begin{minipage}[b]{.5\linewidth}
\centering
\begin{tikzpicture}[scale=0.8, every node/.style={transform shape}]
\node (t11) at (1,-1) {$t_{1,1}$};
\node (t12) at (2,-1) {$t_{1,2}$};
\node (t13) at (3,-1) {$t_{1,3}$};
\node (t14) at (4,-1) {$t_{1,4}$};
\node[draw,circle,inner sep=1pt] (t15) at (5,-1) {$t_{1,5}$};
\node[draw,circle,inner sep=1pt] (t16) at (6,-1) {$\red{t_{1,6}}$};

\node (t21) at (1,-2) {$t_{2,1}$};
\node (t22) at (2,-2) {$t_{2,2}$};
\node (t23) at (3,-2) {$t_{2,3}$};
\node[draw,circle,inner sep=1pt] (t24) at (4,-2) {$t_{2,4}$};
\node[draw,circle,inner sep=1pt] (t25) at (5,-2) {$\red{t_{2,5}}$};

\node (t31) at (1,-3) {$t_{3,1}$};
\node (t32) at (2,-3) {$t_{3,2}$};
\node[draw,circle,inner sep=1pt] (t33) at (3,-3) {$t_{3,3}$};
\node[draw,circle,inner sep=1pt] (t34) at (4,-3) {$\red{t_{3,4}}$};

\node (t41) at (1,-4) {$t_{4,1}$};
\node[draw,circle,inner sep=1pt] (t42) at (2,-4) {$t_{4,2}$};
\node[draw,circle,inner sep=1pt] (t43) at (3,-4) {$\red{t_{4,3}}$};

\node[draw,circle,inner sep=1pt] (t51) at (1,-5) {$t_{5,1}$};
\node[draw,circle,inner sep=1pt] (t52) at (2,-5) {$\red{t_{5,2}}$};

\node[draw,circle,inner sep=1pt] (t61) at (1,-6) {$\red{t_{6,1}}$};

\draw[->] (0.4,-0.4) -- (t11);

\draw[->] (t11) -- (t12);
\draw[->] (t12) -- (t13);
\draw[->] (t13) -- (t14);
\draw[->] (t14) -- (t15);
\draw[->] (t15) -- (t16);

\draw[->] (t21) -- (t22);
\draw[->] (t22) -- (t23);
\draw[->] (t23) -- (t24);
\draw[->] (t24) -- (t25);

\draw[->] (t31) -- (t32);
\draw[->] (t32) -- (t33);
\draw[->] (t33) -- (t34);

\draw[->] (t41) -- (t42);
\draw[->] (t42) -- (t43);

\draw[->] (t51) -- (t52);

\draw[->] (t11) -- (t21);
\draw[->] (t21) -- (t31);
\draw[->] (t31) -- (t41);
\draw[->] (t41) -- (t51);
\draw[->] (t51) -- (t61);

\draw[->] (t12) -- (t22);
\draw[->] (t22) -- (t32);
\draw[->] (t32) -- (t42);
\draw[->] (t42) -- (t52);

\draw[->] (t13) -- (t23);
\draw[->] (t23) -- (t33);
\draw[->] (t33) -- (t43);

\draw[->] (t14) -- (t24);
\draw[->] (t24) -- (t34);

\draw[->] (t15) -- (t25);
\end{tikzpicture}
\subcaption{Triangular array}
\label{subfig:triangularArray}
\end{minipage}%
\begin{minipage}[b]{.5\linewidth}
\centering
\begin{tikzpicture}[scale=0.8, every node/.style={transform shape}]
\node (t11) at (1,-1) {$t_{1,1}$};
\node (t12) at (2,-1) {$t_{1,2}$};
\node (t13) at (3,-1) {$t_{1,3}$};
\node (t14) at (4,-1) {$t_{1,4}$};
\node (t15) at (5,-1) {$t_{1,5}$};
\node[draw,circle,inner sep=1pt] (t16) at (6,-1) {$t_{1,6}$};

\node (t21) at (1,-2) {$t_{2,1}$};
\node (t22) at (2,-2) {$t_{2,2}$};
\node[draw,circle,inner sep=1pt] (t23) at (3,-2) {$t_{2,3}$};
\node[draw,circle,inner sep=1pt] (t24) at (4,-2) {$t_{2,4}$};
\node[draw,circle,inner sep=1pt] (t25) at (5,-2) {$t_{2,5}$};
\node[draw,circle,inner sep=1pt] (t26) at (6,-2) {$\red{t_{2,6}}$};

\node (t31) at (1,-3) {$t_{3,1}$};
\node (t32) at (2,-3) {$t_{3,2}$};
\node[draw,circle,inner sep=1pt] (t33) at (3,-3) {$t_{3,3}$};

\node (t41) at (1,-4) {$t_{4,1}$};
\node (t42) at (2,-4) {$t_{4,2}$};
\node[draw,circle,inner sep=1pt] (t43) at (3,-4) {$t_{4,3}$};

\node (t51) at (1,-5) {$t_{5,1}$};
\node[draw,circle,inner sep=1pt] (t52) at (2,-5) {$t_{5,2}$};
\node[draw,circle,inner sep=1pt] (t53) at (3,-5) {$\red{t_{5,3}}$};

\node[draw,circle,inner sep=1pt] (t61) at (1,-6) {$t_{6,1}$};
\node[draw,circle,inner sep=1pt] (t62) at (2,-6) {$\red{t_{6,2}}$};

\draw[->] (0.4,-0.4) -- (t11);

\draw[->] (t11) -- (t12);
\draw[->] (t12) -- (t13);
\draw[->] (t13) -- (t14);
\draw[->] (t14) -- (t15);
\draw[->] (t15) -- (t16);

\draw[->] (t21) -- (t22);
\draw[->] (t22) -- (t23);
\draw[->] (t23) -- (t24);
\draw[->] (t24) -- (t25);
\draw[->] (t25) -- (t26);

\draw[->] (t31) -- (t32);
\draw[->] (t32) -- (t33);

\draw[->] (t41) -- (t42);
\draw[->] (t42) -- (t43);

\draw[->] (t51) -- (t52);
\draw[->] (t52) -- (t53);

\draw[->] (t61) -- (t62);

\draw[->] (t11) -- (t21);
\draw[->] (t21) -- (t31);
\draw[->] (t31) -- (t41);
\draw[->] (t41) -- (t51);
\draw[->] (t51) -- (t61);

\draw[->] (t12) -- (t22);
\draw[->] (t22) -- (t32);
\draw[->] (t32) -- (t42);
\draw[->] (t42) -- (t52);
\draw[->] (t52) -- (t62);

\draw[->] (t13) -- (t23);
\draw[->] (t23) -- (t33);
\draw[->] (t33) -- (t43);
\draw[->] (t43) -- (t53);

\draw[->] (t14) -- (t24);

\draw[->] (t15) -- (t25);

\draw[->] (t16) -- (t26);
\end{tikzpicture}
\subcaption{Generic polygonal array}
\label{subfig:polygonalArray}
\end{minipage}
\caption[Young-shaped polygonal arrays]{Examples of polygonal array, where all the border indices are circled and the outer border indices are also highlighted in red.
The arrows refer to formula~\eqref{eq:energy}: $\mathcal{E}(\bm{t})$ is the sum of all $a/b$ such that there is an arrow pointing from $a$ to $b$ in the diagram; the extra arrow pointing to $t_{1,1}$ corresponds to the term $1/t_{1,1}$.
The figures also illustrate the ordering defined in~\eqref{eq:RSKordering}, if one interprets an arrow pointing from $a$ to $b$ as the inequality $a\leq b$ and the arrow pointing to $t_{1,1}$ as the inequality $0\leq t_{1,1}$.}
\label{fig:polygonalArray}
\end{figure}
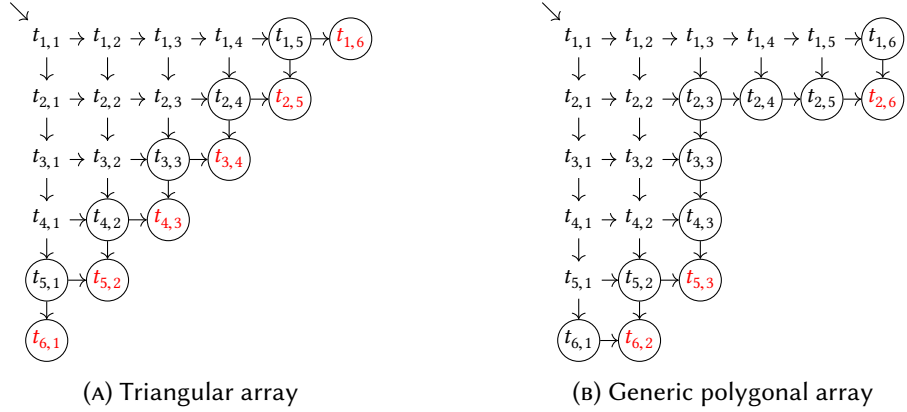

We start by defining the following birational maps acting on $\bm{w}\in\arrays_{>0}$, with the convention that $w_{0,j} = w_{i,0} =0$ but $w_{1,0} + w_{0,1} = 1$:
\begin{itemize}
\item
for all $(i,j)\in\mathcal{I}$, $a_{i,j}$ replaces $w_{i,j}$ with
\begin{equation}
\label{eq:gRSKlocalMoveA}
w_{i,j}(w_{i-1,j}+w_{i,j-1})
\end{equation}
and leaves all other entries of $\bm{w}$ unchanged;
\item
for all \emph{non}-border indices $(i,j)\in\mathcal{I}$, $b_{i,j}$ replaces $w_{i,j}$ with
\begin{equation}
\label{eq:gRSKlocalMoveB}
\frac{1}{w_{i,j}} (w_{i-1,j} + w_{i,j-1}) \left(\frac{1}{w_{i+1,j}} + \frac{1}{w_{i,j+1}} \right)^{-1}
\end{equation}
and leaves all other entries of $\bm{w}$ unchanged.
\end{itemize}
These maps are related to Bender-Knuth transformations~\cite{bender72}, and are called \emph{local moves} because they act on matrices locally, modifying the entry $(i,j)$ only and using its nearest neighbors only.

Next, we define the operation
\begin{equation}
\label{eq:gRSKrho}
\rho_{i,j} :=
\bigcirc_{k\geq 1} b_{i-k,j-k} \circ a_{i,j} \, ,
\end{equation}
where $\bigcirc_{k\geq 1}$ indicates a sequence of compositions in which $b_{i-k,j-k}$ appears if and only if $(i-k,j-k)\in\mathcal{I}$.
It is clear from Figure~\ref{subfig:localMoves} that two local moves indexed by $(i,j)$ and $(i',j')$ commute if they are \emph{not} nearest neighbors, i.e.\ if $\abs{i-i'} + \abs{j-j'} > 1$.
Consequently, the order in which the local moves making up a single $\rho_{i,j}$ are applied does not matter.
Moreover, $\rho_{i,j}$ and $\rho_{i',j'}$ commute whenever the diagonals that $(i,j)$ and $(i',j')$ belong to are neither the same nor consecutive, i.e.\ $\abs{(j-i)-(j'-i')}>1$, as illustrated in Figure~\ref{subfig:diagonalMaps}.

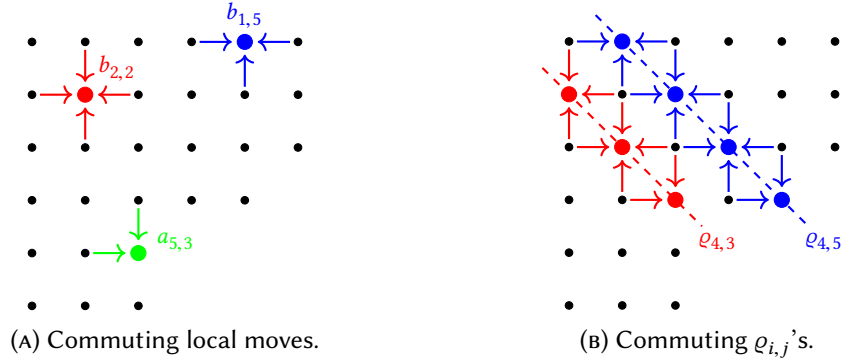
\begin{figure}
\centering
\begin{minipage}[b]{.5\linewidth}
\centering
\begin{tikzpicture}[scale=0.7]
\foreach \x in {(1,-1),(2,-1),(3,-1),(4,-1),(5,-1),(6,-1),(1,-2),(2,-2),(3,-2),(4,-2),(5,-2),(6,-2),(1,-3),(2,-3),(3,-3),(4,-3),(5,-3),(6,-3),(1,-4),(2,-4),(3,-4),(4,-4),(5,-4),(1,-5),(2,-5),(3,-5),(1,-6),(2,-6),(3,-6)}{
	\node[draw,circle,inner sep=1pt,outer sep=2pt,fill] at \x {};
	}

\node[draw,circle,inner sep=2pt,outer sep=2pt,fill,red] (t22) at (2,-2) {};
\node[inner sep=1pt,outer sep=2pt] (t21) at (1,-2) {};
\node[inner sep=1pt,outer sep=2pt] (t12) at (2,-1) {};
\node[inner sep=1pt,outer sep=2pt] (t23) at (3,-2) {};
\node[inner sep=1pt,outer sep=2pt] (t32) at (2,-3) {};
\draw[->, red, thick] (t21) -- (t22);
\draw[->, red, thick] (t12) -- (t22);
\draw[->, red, thick] (t23) -- (t22);
\draw[->, red, thick] (t32) -- (t22);
\node at (2.6,-1.5) {\footnotesize \textcolor{red}{$b_{2,2}$}};

\node[draw,circle,inner sep=2pt,outer sep=2pt,fill,blue] (t15) at (5,-1) {};
\node[inner sep=1pt,outer sep=2pt] (t14) at (4,-1) {};
\node[inner sep=1pt,outer sep=2pt] (t25) at (5,-2) {};
\node[inner sep=1pt,outer sep=2pt] (t16) at (6,-1) {};
\draw[->, blue, thick] (t14) -- (t15);
\draw[->, blue, thick] (t25) -- (t15);
\draw[->, blue, thick] (t16) -- (t15);
\node at (5,-0.5) {\footnotesize \textcolor{blue}{$b_{1,5}$}};

\node[draw,circle,inner sep=2pt,outer sep=2pt,fill,green] (t53) at (3,-5) {};
\node[inner sep=1pt,outer sep=2pt] (t52) at (2,-5) {};
\node[inner sep=1pt,outer sep=2pt] (t43) at (3,-4) {};
\draw[->, green, thick] (t52) -- (t53);
\draw[->, green, thick] (t43) -- (t53);
\node at (3.7,-4.8) {\footnotesize \textcolor{green}{$a_{5,3}$}};

\end{tikzpicture}
\subcaption{Commuting local moves.}
\label{subfig:localMoves}
\end{minipage}%
\begin{minipage}[b]{.5\linewidth}
\centering
\begin{tikzpicture}[scale=0.7]
\foreach \x in {(1,-1),(2,-1),(3,-1),(4,-1),(5,-1),(6,-1),(1,-2),(2,-2),(3,-2),(4,-2),(5,-2),(6,-2),(1,-3),(2,-3),(3,-3),(4,-3),(5,-3),(6,-3),(1,-4),(2,-4),(3,-4),(4,-4),(5,-4),(1,-5),(2,-5),(3,-5),(1,-6),(2,-6),(3,-6)}{
	\node[draw,circle,inner sep=1pt,outer sep=2pt,fill] at \x {};
	}

\node[draw,circle,inner sep=2pt,outer sep=2pt,fill,red] (t43) at (3,-4) {};
\node[draw,circle,inner sep=2pt,outer sep=2pt,fill,red] (t32) at (2,-3) {};
\node[draw,circle,inner sep=2pt,outer sep=2pt,fill,red] (t21) at (1,-2) {};
\node[inner sep=1pt,outer sep=2pt] (t42) at (2,-4) {};
\node[inner sep=1pt,outer sep=2pt] (t31) at (1,-3) {};
\node[inner sep=1pt,outer sep=2pt] (t44) at (4,-4) {};
\node[inner sep=1pt,outer sep=2pt] (t33) at (3,-3) {};
\node[inner sep=1pt,outer sep=2pt] (t22) at (2,-2) {};
\node[inner sep=1pt,outer sep=2pt] (t11) at (1,-1) {};

\draw[->, red, thick] (t42) -- (t43);
\draw[->, red, thick] (t33) -- (t43);
\draw[->, red, thick] (t31) -- (t32);
\draw[->, red, thick] (t22) -- (t32);
\draw[->, red, thick] (t33) -- (t32);
\draw[->, red, thick] (t42) -- (t32);
\draw[->, red, thick] (t31) -- (t21);
\draw[->, red, thick] (t22) -- (t21);
\draw[->, red, thick] (t11) -- (t21);

\draw[thick,dashed,red] (0.5,-1.5) -- (3.5,-4.5);
\node at (3.8,-4.8) {\footnotesize \textcolor{red}{$\rho_{4,3}$}};

\node[draw,circle,inner sep=2pt,outer sep=2pt,fill,blue] (t45) at (5,-4) {};
\node[draw,circle,inner sep=2pt,outer sep=2pt,fill,blue] (t34) at (4,-3) {};
\node[draw,circle,inner sep=2pt,outer sep=2pt,fill,blue] (t23) at (3,-2) {};
\node[draw,circle,inner sep=2pt,outer sep=2pt,fill,blue] (t12) at (2,-1) {};
\node[inner sep=1pt,outer sep=2pt] (t35) at (5,-3) {};
\node[inner sep=1pt,outer sep=2pt] (t24) at (4,-2) {};
\node[inner sep=1pt,outer sep=2pt] (t13) at (3,-1) {};

\draw[->, blue, thick] (t44) -- (t45);
\draw[->, blue, thick] (t35) -- (t45);
\draw[->, blue, thick] (t44) -- (t34);
\draw[->, blue, thick] (t35) -- (t34);
\draw[->, blue, thick] (t33) -- (t34);
\draw[->, blue, thick] (t24) -- (t34);
\draw[->, blue, thick] (t33) -- (t23);
\draw[->, blue, thick] (t24) -- (t23);
\draw[->, blue, thick] (t13) -- (t23);
\draw[->, blue, thick] (t22) -- (t23);
\draw[->, blue, thick] (t11) -- (t12);
\draw[->, blue, thick] (t22) -- (t12);
\draw[->, blue, thick] (t13) -- (t12);

\draw[thick,dashed,blue] (1.5,-0.5) -- (5.5,-4.5);
\node at (5.8,-4.8) {\footnotesize \textcolor{blue}{$\rho_{4,5}$}};

\end{tikzpicture}
\subcaption{Commuting $\rho_{i,j}$'s.}
\label{subfig:diagonalMaps}
\end{minipage}%

\caption[Graphical representation of the local moves that compose the geometric $\RSK$ correspondence]{Graphical representation of how local moves $a_{i,j}$'s and $b_{i,j}$'s and maps $\rho_{i,j}$'s that compose the $\gRSK$ correspondence act on a polygonal array.
The arrows point from a node involved in the definition of a local move to a colored node, which corresponds to the entry that is modified by the local move.
One can see that any two local moves commute if they are indexed by lattice vertices that are \emph{not} nearest neighbors, and any two maps $\rho_{i,j}$'s commute if they are indexed by vertices that do \emph{not} belong to neighboring diagonals.}
\label{fig:localMoves}
\end{figure}

We can now construct, inductively on the size of $\mathcal{I}$, the map $\gRSK \colon \arrays_{>0} \to \arrays_{>0}$.
For the index set of size $0$ we set $\gRSK(\emptyset):=\emptyset$, and for index sets of greater size we define
\begin{equation}
\label{eq:gRSKconstruction}
\gRSK(\bm{w}) := \rho_{m,n}\left(\gRSK\left(\bm{w}^{(m,n)}\right) \sqcup w_{m,n}\right) \, ,
\end{equation}
where $(m,n)\in\mathcal{I}$ is any \emph{outer} index arbitrarily chosen, $\bm{w}^{(m,n)} := \{ w_{i,j}\colon (i,j)\in \mathcal{I}\setminus \{(m,n)\}\}$ is the subarray of $\bm{w}$ that excludes entry $(m,n)$, and $\sqcup$ denotes concatenation.
In words, $\gRSK$ is first applied to the subarray of $\bm{w}$ that excludes entry $(m,n)$, then the output is concatenated with entry $(m,n)$, and finally map $\rho_{m,n}$ is applied to the resulting array; we call the latter two steps \emph{insertion} of $w_{m,n}$.

For example, for arrays of shape $(3,2)$ we have:
\[
\begin{matrix}
w_{1,1} &w_{1,2} &w_{1,3} \\
w_{2,1} &w_{2,2}
\end{matrix}
\qquad\xmapsto{\gRSK}\qquad
\begin{matrix}
\frac{w_{1,2}w_{2,1}}{w_{1,2}+w_{2,1}}
&w_{1,1} w_{1,2}
&w_{1,1}w_{1,2}w_{1,3} \\
w_{1,1}w_{2,1}
&w_{1,1}w_{2,2}(w_{1,2}+w_{2,1})
\end{matrix} \, .
\]

Notice that $\gRSK$ is well defined, because the order in which outer indices are chosen does not matter: distinct outer indices are never on the same diagonal nor on consecutive diagonals, so all $\rho_{i,j}$'s indexed by the outer indices of a given array commute.
 
In order to state the main properties of $\gRSK$, it is convenient to introduce the following notations.
We denote by $\pi_k(\bm{t})$ the product of all elements on the $k$-th diagonal of an array $\bm{t}$:
\begin{equation}
\label{eq:prodDiagonal}
\pi_k(\bm{t}):=\prod_{\substack{(i,j)\in\mathcal{I} , \\ j-i=k}} t_{i,j} \, ,
\end{equation}
with the convention that the empty product equals $1$ (as we will always suppose from now on).
We also set the \emph{energy} of $\bm{t}$ to be
\begin{equation}
\label{eq:energy}
\mathcal{E}(\bm{t})
:= \frac{1}{t_{1,1}}
+ \sum_{(i,j)\in\mathcal{I}} \frac{t_{i-1,j}+t_{i,j-1}}{t_{i,j}}
\, ,
\end{equation}
with the convention that $t_{i,j}:=0$ when $(i,j)\notin \mathcal{I}$.
See Figure~\ref{fig:polygonalArray} for a graphical illustration of the energy of $\bm{t}$.
\begin{proposition}
\label{prop:gRSKproperties}
The $\gRSK$ correspondence is a bijection $\arrays_{>0} \to \arrays_{>0}$ that satisfies the following properties, for any border index $(m,n)\in \mathcal{I}$, $\bm{w}\in\arrays_{>0}$, and $\bm{t} := \gRSK(\bm{w})$:
\begin{enumerate}
\item
\label{prop:gRSKproperties_partitionFn}
If $\Pi_{m,n}$ is the set of all directed paths from $(1,1)$ to $(m,n)$, then
\begin{equation}
\label{eq:gRSKproperties_partitionFn}
t_{m,n} = \sum_{\pi\in\Pi_{m,n}} \prod_{(i,j)\in \pi} w_{i,j} \, .
\end{equation}
\item
\label{prop:gRSKproperties_type}
If $(m,n+1)\notin\mathcal{I}$, then
\begin{equation}
\label{eq:gRSKproperties_type1}
\frac{\pi_{n-m}(\bm{t})}{\pi_{n-m+1}(\bm{t})}
= \prod_{j=1}^{n} w_{m,j}
\, .
\end{equation}
Analogously, if $(m+1,n)\notin\mathcal{I}$, then
\begin{equation}
\label{eq:gRSKproperties_type2}
\frac{\pi_{n-m}(\bm{t})}{\pi_{n-m-1}(\bm{t})}
= \prod_{i=1}^{m} w_{i,n}
\, .
\end{equation}
\item
\label{prop:gRSKproperties_invWeights}
It holds that
\begin{equation}
\label{eq:gRSKproperties_invWeights}
\mathcal{E}(\bm{t})
= \sum_{(i,j)\in\mathcal{I}} \frac{1}{w_{i,j}}
\, .
\end{equation}
\item
\label{prop:gRSKproperties_Jacobian}
The transformation
\begin{equation}
\label{eq:gRSKproperties_Jacobian}
(\log w_{i,j}, \, (i,j)\in\mathcal{I})
\mapsto
(\log t_{i,j}, \, (i,j)\in\mathcal{I})
\end{equation}
has Jacobian equal to $\pm 1$.
\end{enumerate}
\end{proposition}

We will provide a self-contained proof of this proposition, but let us first make a few remarks.
Property~\ref{prop:gRSKproperties_partitionFn} explains how the point-to-point polymer partition functions, defined in the Introduction, can be expressed in terms of the $\gRSK$ correspondence.
In light of this connection, the other properties turn out to be useful in computations related to the log-gamma polymer, as it will become clear in chapter~\ref{ch:polymer}.
We finally observe that property~\ref{prop:gRSKproperties_type} is easily seen to be equivalent to the following: if $(m,n)$ is a border index, then
\begin{equation}
\label{eq:gRSKprodSubarray}
\pi_{n-m}(\bm{t})
= \prod_{i=1}^m \prod_{j=1}^n w_{i,j} \, .
\end{equation}

\begin{proof}
Since local moves are bijective, the $\gRSK$ correspondence is also a bijection.
All the properties can be proven by induction on the size of the arrays.
The induction basis is always trivial, as $\gRSK$ coincides with the identity for arrays of size $1$.
Let us prove the induction steps, fixing an index set $\mathcal{I}$ of size greater than $1$ and assuming that the properties are true for arrays of smaller size.
\begin{enumerate}
\item
Assume first that $(m,n)$ is an outer index.
Then~\eqref{eq:gRSKconstruction} holds and
\[
t_{m,n} = w_{m,n}(t_{m-1,n} + t_{m,n-1})
\]
by the definition of local moves (see~\eqref{eq:gRSKrho} and~\eqref{eq:gRSKlocalMoveA}).
Moreover, $(m-1,n)$ and $(m,n-1)$ are both border indices, hence by induction hypothesis
\[
t_{m-1,n} = \sum_{\pi\in\Pi_{m-1,n}} \prod_{(i,j)\in \pi} w_{i,j}
\qquad\text{and}\qquad
t_{m,n-1} = \sum_{\pi\in\Pi_{m,n-1}} \prod_{(i,j)\in \pi} w_{i,j} \, .
\]
Since the penultimate site visited by any path in $\Pi_{m,n}$ is either $(m-1,n)$ or $(m,n-1)$, \eqref{eq:gRSKproperties_partitionFn} follows immediately.

Let now $(m,n)$ be a generic border index.
Since $(m+1,n+1)\notin\mathcal{I}$, no insertion performed after the insertion of $w_{m,n}$ can possibly modify entry $(m,n)$.
It follows that $t_{m,n}$ equals the value of entry $(m,n)$ after applying $\gRSK$ to the subarray of $\bm{w}$ indexed by $\{1,\dots,m\}\times \{1,\dots,n\}$; for this subarray $(m,n)$ is an outer index, so~\eqref{eq:gRSKproperties_partitionFn} holds true.
\item
The proofs of~\eqref{eq:gRSKproperties_type1} and~\eqref{eq:gRSKproperties_type2} are analogous, so we will only give the first one.

Suppose first that $(m,n)$ is an outer index, denote by $\bm{w}^{(m,n)}$ the subarray of $\bm{w}$ indexed by $\mathcal{I}\setminus\{(m,n)\}$, and let $\bm{t}^{(m,n)} := \gRSK(\bm{w}^{(m,n)})$.
By induction hypothesis we can apply \eqref{eq:gRSKproperties_type1} to $\bm{w}^{(m,n)}$, obtaining
\begin{equation}
\label{eq:gRSKproperties_type1Induction}
\frac{\pi_{n-m-1}(\bm{t}^{(m,n)})}{\pi_{n-m}(\bm{t}^{(m,n)})}
= \prod_{j=1}^{n-1} w_{m,j} \, .
\end{equation}
Using \eqref{eq:gRSKconstruction}, \eqref{eq:gRSKrho}, and the definition of local moves, we find:
\[
\begin{split}
\pi_{n-m}(\bm{t})
&= t_{m,n} \prod_{k\geq 1} t_{m-k,n-k} \\
&= w_{m,n} (t_{m-1,n}^{(m,n)} + t_{m,n-1}^{(m,n)})
\prod_{k\geq 1} \frac{t_{m-k-1,n-k}^{(m,n)} + t_{m-k,n-k-1}^{(m,n)}}{t^{(m,n)}_{m-k,n-k}}
\frac{t_{m-k,n-k+1}^{(m,n)} t_{m-k+1,n-k}^{(m,n)}}{t_{m-k,n-k+1}^{(m,n)} + t_{m-k+1,n-k}^{(m,n)}} \\
&= \frac{w_{m,n}}{\pi_{n-m}(\bm{t}^{(m,n)})} \pi_{n-m+1}(\bm{t}^{(m,n)}) \pi_{n-m-1}(\bm{t}^{(m,n)}) \, ,
\end{split}
\]
under the usual notational conventions.
Combining the latter computation with~\eqref{eq:gRSKproperties_type1Induction}, and using the fact that $\pi_{n-m+1}(\bm{t}^{(m,n)}) = \pi_{n-m+1}(\bm{t})$ (the insertion of $w_{m,n}$ only modifies the $(n-m)$-th diagonal), one easily concludes~\eqref{eq:gRSKproperties_type1}.

Let now $(m,n)$ be a generic a border index.
Since $(m+1,n+1)\notin\mathcal{I}$, no insertion performed after the insertion of $w_{m,n}$ can possibly modify the $(n-m)$-th diagonal.
Since by hypothesis $(m,n+1)\notin\mathcal{I}$ either, the same holds for the $(n-m+1)$-th diagonal.
It follows that $\pi_{n-m}(\bm{t})$ and $\pi_{n-m+1}(\bm{t})$ equal the corresponding values for the image of the subarray indexed by $\{1,\dots,m\}\times \{1,\dots,n\}$; for this subarray $(m,n)$ is an outer index, so~\eqref{eq:gRSKproperties_type1} holds true.
\item
Choosing any outer index $(m,n)$, by induction, it suffices to prove that
\begin{equation}
\label{eq:gRSKproperties_invWeightsInduction}
\mathcal{E}(\bm{t}) - \mathcal{E}(\bm{t}^{(m,n)}) = \frac{1}{w_{m,n}} \, ,
\end{equation}
where $\bm{t}^{(m,n)}$ is the image under $\gRSK$ of the subarray of $\bm{w}$ indexed by $\mathcal{I}\setminus\{(m,n)\}$.
Since the only entries that the insertion of $w_{m,n}$ changes are the ones on the $(n-m)$-th diagonal, it suffices to analyze the addends of $\mathcal{E}(\bm{t})$ that contain them.
Firstly, for each $(i,j)\neq (m,n)$ such that $j-i=n-m$, we rewrite the terms of $\mathcal{E}(\bm{t})$ that contain $t_{i,j}$ in terms of the entries of $\bm{t}^{(m,n)}$:
\[
\frac{t_{i-1,j} + t_{i,j-1}}{t_{i,j}} + t_{i,j} \left(\frac{1}{t_{i+1,j}} + \frac{1}{t_{i,j+1}} \right)
= t^{(m,n)}_{i,j} \left(\frac{1}{t^{(m,n)}_{i+1,j}} + \frac{1}{t^{(m,n)}_{i,j+1}} \right) + \frac{t^{(m,n)}_{i-1,j} + t^{(m,n)}_{i,j-1}}{t^{(m,n)}_{i,j}} \, ,
\]
which results from applying local move $b_{i,j}$ to $\bm{t}^{(m,n)}$ (notice that, in the only case $(i,j)=(1,1)$, $t_{i-1,j} + t_{i,j-1}$ and $t^{(m,n)}_{i-1,j} + t^{(m,n)}_{i,j-1}$ are replaced with $1$).
On the other hand, by applying local move $a_{m,n}$ to $\bm{t}^{(m,n)}$, we find that the terms of $\mathcal{E}(\bm{t})$ that contain $t_{m,n}$ equal the inverse of $w_{m,n}$:
\[
\frac{t_{m-1,n} + t_{m,n-1}}{t_{m,n}}
= \frac{1}{w_{m,n}} \, .
\]
From the latter two equations, it is easy to conclude~\eqref{eq:gRSKproperties_invWeightsInduction}.
\item
This property follows from the fact that $\gRSK$ is defined as a composition of local moves of type $a_{i,j}$ and $b_{i,j}$, and each of them is trivially volume preserving in logarithmic variables.
\qedhere
\end{enumerate}
\end{proof}

Since we will also deal with polymers in symmetric environment, we now wish to specialize the $\gRSK$ correspondence to symmetric arrays.
The particular case of symmetric matrices has been already studied in~\cite{oConnellSeppalainenZygouras14}, but for our purposes it is useful to extend the focus to polygonal arrays of arbitrary shape.
We thus define the \emph{transpose} of an index set $\mathcal{I}$ as the index set $\mathcal{I}^\top := \{(i,j)\in\Z_{>0}\times\Z_{>0} \colon (j,i)\in\mathcal{I}\}$; similarly, we define the \emph{transpose} $\bm{t}^\top$ of a polygonal array $\bm{t}$ by setting $t^\top_{i,j} := t_{j,i}$ for all $(i,j)\in\mathcal{I}^\top$.
An index set $\mathcal{I}$ will be called \emph{symmetric} if $\mathcal{I}=\mathcal{I}^\top$, and a polygonal array $\bm{t}$ indexed by a symmetric $\mathcal{I}$ will be called \emph{symmetric} if $\bm{t}=\bm{t}^\top$.

Properties~\ref{prop:gRSKproperties_partitionFn}-\ref{prop:gRSKproperties_type}-\ref{prop:gRSKproperties_invWeights} in Proposition \ref{prop:gRSKproperties} transfer directly to the case of symmetric arrays.
The volume preserving property is also satisfied: 
\begin{proposition}
\label{prop:symmetricgRSK}
Let $\bm{w}\in\arrays_{>0}$ and $\bm{t} := \gRSK(\bm{w})$.
Then $\gRSK\left(\bm{w}^\top\right) = \bm{t}^\top$.
In particular, if $\bm{w}$ is symmetric, so is $\bm{t}$. 
Moreover, in the symmetric case, the transformation
\begin{equation}
\label{eq:symmetricgRSK}
\{\log w_{i,j} \colon (i,j)\in\mathcal{I}, \, i\leq j\}
\mapsto
\{\log t_{i,j} \colon (i,j)\in\mathcal{I}, \, i\leq j\}
\end{equation}
has Jacobian equal to $\pm 1$.
\end{proposition}
\begin{proof}
The fact that $\gRSK\big(\bm{w}^\top\big)=\gRSK(\bm{w})^\top$ is an easy consequence of the inductive construction~\eqref{eq:gRSKconstruction} of $\gRSK$, since local moves are clearly symmetric, in the sense that $a_{j,i}\big(\bm{w}^\top\big) = a_{i,j}(\bm{w})^\top$, and the same holds for $b_{i,j}$.

We now check the volume preserving property in the case of symmetric $\bm{w}$, proceeding by induction on the size of the array.
For arrays of size $1$ the $\gRSK$ is the identity, so the statement trivially holds.
Let us prove it for a given symmetric index set $\mathcal{I}$ of size greater than $1$, assuming it holds for any array of smaller size.
Suppose first there exists an outer index $(n,n)$ on the diagonal: then $\bm{t} = \rho_{n,n}(\bm{t}^{(n,n)} \sqcup w_{n,n})$, where $\bm{t}^{(n,n)}:= \gRSK(\bm{w}^{(n,n)})$ and $\bm{w}^{(n,n)}$ is the subarray of $\bm{w}$ indexed by $\mathcal{I}\setminus\{(n,n)\}$.
Since $\bm{w}^{(n,n)}$ is symmetric as $\bm{w}$ is, by induction hypothesis we have that
\[
(\log w_{i,j} \colon (i,j)\in\mathcal{I}\setminus\{(n,n)\}, \, i\leq j)
\mapsto
\left( \log t^{(n,n)}_{i,j} \colon (i,j)\in\mathcal{I}\setminus\{(n,n)\}, \, i\leq j \right)
\]
has Jacobian $\pm 1$.
After the insertion of $w_{n,n}$, due to the definition of local moves and the symmetric constraint, the entries on the main diagonal can be written in terms of $w_{n,n}$ and $t^{(n,n)}_{i,j}$ ($i\leq j$), as follows:
\[
t_{n,n}= 2w_{n,n} t^{(n,n)}_{n-1,n} \, ,
\qquad\qquad
t_{i,i} = \frac{t^{(n,n)}_{i-1,i} t^{(n,n)}_{i,i+1}}{t^{(n,n)}_{i,i}} \quad \text{for } 1\leq i <n \, .
\]
The thesis then follows from the fact that the transformations $(x,y)\mapsto (2xy,y)$ and $(x,y,z)\mapsto (yz/x,y,z)$ have Jacobian $\pm 1$ in logarithmic variables.
Suppose now that $\mathcal{I}$ has no outer index on the diagonal: then by symmetry we can find two outer indices $(m,n)$ and $(n,m)$ with $m<n$.
By induction, the volume preserving property holds for the subarray of $\bm{w}$ indexed by $\mathcal{I}\setminus\{(m,n),(n,m)\}$ and its $\gRSK$ image.
To recover the whole $\bm{t}$, one still has to insert $w_{m,n}$ and $w_{n,m}$.
Since $m<n$, the insertion of $w_{m,n}$ only acts on the entries on or above the diagonal and does not modify the volume preserving property, as local moves have Jacobian $\pm 1$ (as we already noticed in the proof of Proposition~\ref{prop:gRSKproperties}-\ref{prop:gRSKproperties_Jacobian}).
On the other hand, the insertion of $w_{n,m}$ (which makes the whole array symmetric again) does not modify the entries on or above the diagonal at all.
This concludes the proof.
\end{proof}

\subsection{Tropicalization and extended $\RSK$}
\label{subsec:tropicalization&RSK}

We now consider an appropriate limit of $\gRSK$ that results in an extension of the combinatorial $\RSK$ correspondence introduced in subsection~\ref{subsec:RSK}.
Such a limit is normally called \emph{tropicalization}, because it transforms the operations $(+,\cdot)$ of the usual algebra into the tropical operations $(\max,+)$, thus producing a piecewise linear map.
In the terminology of statistical mechanics, it corresponds to the \emph{zero temperature limit}, whose meaning in our setting has been explained in the Introduction and will be further discussed in Appendix~\ref{appendix:zeroTempLimit}.

For any $\epsilon>0$, define $g_{\epsilon}\colon \R \to \R_{>0}$ by $g_\epsilon(x) := \e^{x/\epsilon}$, so that $g_{\epsilon}^{-1}\colon \R_{>0} \to \R$ is given by $g_{\epsilon}^{-1}(y) := \epsilon\log(y)$.
For any index set $\mathcal{I}$, we also define $g_{\epsilon}\colon \arrays \to \arrays_{>0}$ by applying $g_{\epsilon}$ to each entry of the array; we similarly define $g_{\epsilon}^{-1}\colon \arrays_{>0} \to \arrays$ by entrywise application of $g_{\epsilon}^{-1}$.
Let us then set
\begin{equation}
\label{eq:RSKtropicalization}
\RSK\colon \arrays \to \arrays \, ,
\qquad\qquad
\RSK :=
\lim_{\epsilon \downarrow 0} g_{\epsilon}^{-1} \circ \gRSK \circ g_{\epsilon} \, .
\end{equation}
We will shortly see in what sense the map defined above is an extension of the combinatorial correspondence defined in subsection~\ref{subsec:RSK}.

We first wish to express this $\RSK$ as a composition of local moves, similarly as for $\gRSK$.
Using the recursive structure~\eqref{eq:gRSKconstruction} of $\gRSK$ and the same notation of subsection~\ref{subsec:gRSK}, we have that
\[
\begin{split}
g_{\epsilon}^{-1} \circ \gRSK \circ g_{\epsilon} (\bm{w})
&= g_{\epsilon}^{-1} \circ \rho_{m,n} \left( \gRSK\left(g_{\epsilon}\left(\bm{w}^{(m,n)}\right)\right) \sqcup g_{\epsilon}(w_{m,n}) \right) \\
&= g_{\epsilon}^{-1} \circ \rho_{m,n} \circ g_{\epsilon} \left(g_{\epsilon}^{-1} \circ \gRSK \circ g_{\epsilon} \left(\bm{w}^{(m,n)}\right) \sqcup w_{m,n} \right) \, ,
\end{split}
\]
for any outer index $(m,n)$.
We then set
\[
\tilde{\rho}_{m,n}
:= \lim_{\epsilon\downarrow 0} g_{\epsilon}^{-1} \circ \rho_{m,n} \circ g_{\epsilon}
= \bigcirc_{k\geq 1} \tilde{b}_{m-k,n-k} \circ \tilde{a}_{m,n} \, ,
\]
where $\tilde{a}_{i,j} := \lim_{\epsilon\downarrow 0} g_{\epsilon}^{-1} \circ a_{i,j} \circ g_{\epsilon}$ and $\tilde{b}_{i,j} := \lim_{\epsilon\downarrow 0} g_{\epsilon}^{-1} \circ b_{i,j} \circ g_{\epsilon}$.
Using the definition of $a_{i,j}$ and $b_{i,j}$ and computing the limits, one can see that $\tilde{a}_{i,j}$ and $\tilde{b}_{i,j}$ are the following piecewise linear maps acting on $\bm{w}\in\arrays$, with the convention that $w_{0,j} = w_{i,0} =-\infty$ but $\max(w_{1,0},w_{0,1}) = 0$:
\begin{itemize}
\item
for all $(i,j)\in\mathcal{I}$, $\tilde{a}_{i,j}$ replaces $w_{i,j}$ with
\begin{equation}
\label{eq:RSKlocalMoveA}
\max(w_{i-1,j}, w_{i,j-1}) + w_{i,j}
\end{equation}
and leaves all other entries of $\bm{w}$ unchanged;
\item
for all \emph{non}-border index $(i,j)\in\mathcal{I}$, $\tilde{b}_{i,j}$ replaces $w_{i,j}$ with
\begin{equation}
\label{eq:RSKlocalMoveB}
\max(w_{i-1,j}, w_{i,j-1}) + \min(w_{i+1,j}, w_{i,j+1}) - w_{i,j}
\end{equation}
and leaves all other entries of $\bm{w}$ unchanged.
\end{itemize}
We have thus derived the inductive structure of $\RSK$ from the corresponding inductive structure of $\gRSK$:
\begin{equation}
\label{eq:RSKconstruction}
\RSK(\bm{w}) := \tilde{\rho}_{m,n}\left(\RSK\left(\bm{w}^{(m,n)}\right) \sqcup w_{m,n}\right) \, ,
\end{equation}
where $(m,n)\in\mathcal{I}$ is any \emph{outer} index.

The action on arrays of shape $(2,2)$, for instance, is given by:
\[
\begin{matrix}
w_{1,1} &w_{1,2} \\
w_{2,1} &w_{2,2}
\end{matrix}
\qquad \xmapsto{\RSK} \qquad
\begin{matrix}
\min(w_{1,2},w_{2,1})
&w_{1,1} + w_{1,2} \\
w_{1,1} + w_{2,1}
&w_{1,1} + w_{2,2} + \max(w_{1,2},w_{2,1})
\end{matrix} \, .
\]
As one can easily check for this example, we are indeed dealing with an extension of the combinatorial $\RSK$ introduced in subsection~\ref{subsec:RSK}, in a twofold sense: nonnegative integer entries are replaced with real entries, and matrices are replaced with more general polygonal arrays.
Namely, the image under this map of a matrix with nonnegative integer entries is the matrix~\eqref{eq:RSKoutputMatrix} obtained by gluing together the two output $\Z_{\geq 0}$-Gelfand-Tsetlin patterns of the classical combinatorial $\RSK$.
The generalization of $\RSK$ to input matrices with real entries was first introduced in~\cite{berensteinKirillov01}, and described - as we did - as a composition of piecewise linear local maps in~\cite{oConnellSeppalainenZygouras14}.

All the properties of $\gRSK$ seen in subsection~\ref{subsec:gRSK} can be specialized to $\RSK$, either by taking the appropriate limit or by using the inductive construction~\eqref{eq:RSKconstruction} of $\RSK$ directly.
In particular, properties~\ref{prop:gRSKproperties_partitionFn}, \ref{prop:gRSKproperties_type} and~\ref{prop:gRSKproperties_invWeights} of Proposition~\ref{prop:gRSKproperties} are all of the form $\phi(\bm{w}) = \psi(\gRSK(\bm{w}))$ for certain rational functions $\phi,\psi\colon \arrays_{>0} \to \R_{>0}$; therefore, the corresponding $\RSK$ properties read as $\tilde{\phi}(\bm{w}) = \tilde{\psi}(\RSK(\bm{w}))$, where $\tilde{\phi} := \lim_{\epsilon\downarrow 0} g_{\epsilon}^{-1} \circ \phi \circ g_{\epsilon}$ and $\tilde{\psi} := \lim_{\epsilon\downarrow 0} g_{\epsilon}^{-1} \circ \psi \circ g_{\epsilon}$.
They can also be practically obtained by formally replacing $(+,\cdot)$ with the tropical operations $(\max,+)$.
On the other hand, the bijective, symmetric and volume preserving properties of $\RSK$ that we will state can be easily deduced from the definition of the piecewise linear local moves.
For convenience, let us denote by $\sigma_k(\bm{t})$ the sum of all elements on the $k$-th diagonal of an array $\bm{t}$:
\begin{equation}
\label{eq:sumDiagonal}
\sigma_k(\bm{t}):=\sum_{\substack{(i,j)\in\mathcal{I} , \\ j-i=k}} t_{i,j} \, .
\end{equation}

\begin{proposition}
\label{prop:RSKproperties}
The $\RSK$ correspondence is a bijection $\arrays \to \arrays$ that satisfies the following properties, for any border index $(m,n)\in \mathcal{I}$, $\bm{w}\in\arrays$, and $\bm{t} := \RSK(\bm{w})$:
\begin{enumerate}
\item
\label{prop:RSKproperties_LPP}
If $\Pi_{m,n}$ is the set of all directed paths from $(1,1)$ to $(m,n)$, then
\begin{equation}
\label{eq:RSKproperties_LPP}
t_{m,n} = \max_{\pi\in\Pi_{m,n}} \sum_{(i,j)\in \pi} w_{i,j} \, .
\end{equation}
\item
\label{prop:RSKproperties_type}
If $(m,n+1)\notin\mathcal{I}$, then
\begin{equation}
\label{eq:RSKproperties_type1}
\sigma_{n-m}(\bm{t}) - \sigma_{n-m+1}(\bm{t})
= \sum_{j=1}^{n} w_{m,j}
\, .
\end{equation}
Analogously, if $(m+1,n)\notin\mathcal{I}$, then
\begin{equation}
\label{eq:RSKproperties_type2}
\sigma_{n-m}(\bm{t}) - \sigma_{n-m-1}(\bm{t})
= \sum_{i=1}^{m} w_{i,n}
\, .
\end{equation}
\item
\label{prop:RSKproperties_interlacing}
It holds that
\begin{equation}
\label{eq:RSKproperties_interlacing}
\min_{(i,j)\in\mathcal{I}} \{ t_{1,1} \, , \,\, t_{i,j} - t_{i-1,j} \, , \,\, t_{i,j} - t_{i,j-1} \}
= \min_{(i,j)\in\mathcal{I}} w_{i,j}
\, ,
\end{equation}
with the convention that $t_{i,j}:=-\infty$ when $(i,j)\notin \mathcal{I}$.
In particular, $\RSK$ can be restricted to a bijection between arrays indexed by $\mathcal{I}$ with non-negative real entries, such that the output array satisfies the following ordering: for all $(i,j)\in\mathcal{I}$
\begin{equation}
\label{eq:RSKordering}
t_{i-1,j} \leq t_{i,j} \quad\text{if } i>1
\quad\qquad \text{and} \quad\qquad
t_{i,j-1} \leq t_{i,j} \quad\text{if } j>1 \, .
\end{equation}
\item
\label{prop:RSKproperties_Jacobian}
The transformation $\bm{w} \mapsto \bm{t}$ has Jacobian equal to $\pm 1$ almost everywhere.
\end{enumerate}
\end{proposition}

It is clear that properties~\ref{prop:RSKproperties_LPP} and~\ref{prop:RSKproperties_type} of Proposition~\ref{prop:RSKproperties} are a generalization of~\ref{prop:RSKGelfandTsetlin_LPP} and~\ref{prop:RSKGelfandTsetlin_type} of Proposition~\ref{prop:RSKGelfandTsetlin} respectively, whereas property~\ref{prop:RSKproperties_interlacing} corresponds to the non-negativity and interlacing conditions that the two output Gelfand-Tsetlin patterns satisfy in the combinatorial $\RSK$.
The ordering~\eqref{eq:RSKordering} can also be visualized in Figure~\ref{fig:polygonalArray}.

\begin{proposition}
\label{prop:symmetricRSK}
Let $\bm{w}\in\arrays$ and $\bm{t} := \RSK(\bm{w})$.
Then $\RSK\left(\bm{w}^\top\right) = \bm{t}^\top$.
In particular, if $\bm{w}$ is symmetric, so is $\bm{t}$. 
Moreover, in the symmetric case, the transformation
\begin{equation}
\label{eq:symmetricRSK}
\{w_{i,j} \colon (i,j)\in\mathcal{I}, \, i\leq j\}
\mapsto
\{t_{i,j} \colon (i,j)\in\mathcal{I}, \, i\leq j\}
\end{equation}
has Jacobian equal to $\pm 1$ almost everywhere.
\end{proposition}

The symmetric property stated in the latter proposition generalizes property~\ref{prop:RSKGelfandTsetlin_symmetry} of Proposition~\ref{prop:RSKGelfandTsetlin}.

\section{Schur functions}
\label{sec:Schur}

Schur functions, named after Issai Schur, appear in various branches of mathematics, especially in Young tableaux combinatorics and in the representation theory of classical Lie groups; see~\cite{sundaram90} for a survey about the connection between these two viewpoints.
In this section, we will introduce the two types of Schur functions that are relevant to our purposes: ``standard'' Schur functions (related to general linear groups) and symplectic Schur functions (related to symplectic groups).

\subsection{Standard Schur functions}
\label{subsec:standardSchur}

Schur functions (that we call ``standard'' in order to distinguish them from their symplectic analog) are symmetric polynomials indexed by integer partitions and form a basis for the algebra of symmetric functions.
In combinatorics they are viewed as generating functions of semistandard Young tableaux, or Gelfand-Tsetlin patterns equivalently, and are connected to the $\RSK$ correspondence (we will see this connection, from a probabilistic standpoint, in the introduction to chapter~\ref{ch:LPP}).
They also occur in the representation theory of symmetric groups, general linear groups and unitary groups, as characters of finite-dimensional irreducible representations.
Finally, Schur functions find applications in algebraic geometry, especially in the Schurbert calculus and the theory of Grassmann varieties.
For a detailed treatment of Schur functions we refer the reader to~\cite{stanley99}.

We will give the combinatorial definition of Schur functions in terms of Gelfand-Tsetlin patterns.
Recall from subsection~\ref{subsec:RSK} that a Gelfand-Tsetlin pattern of height $n$ is a triangular array of the form~\eqref{eq:GTpattern} satisfying the interlacing conditions~\eqref{eq:interlacing} (see Figure~\ref{fig:GTpattern}).
Given $\bm{x}\in \R^n$ and a set $S\subseteq \R$, let us denote by $\GT{n}{S}(\bm{x})$ the set of all Gelfand-Tsetlin patterns of height $n$ with entries in $S$ and shape (i.e.\ bottom row) equal to $\bm{x}$.
Notice that $\GT{n}{S}(\bm{x})$ is empty whenever the chain $x_1\geq \dots\geq x_n$ fails to hold.
Recall the definition~\eqref{eq:typeGTpattern} of type for a Gelfand-Tsetlin pattern.
\begin{definition}
\label{def:schur}
The \emph{(standard) Schur function} in $n$ variables $\bm{a}=(a_1,\dots,a_n) \in\C^n$ indexed by an integer partition $\bm{\lambda}$ of length at most $n$ is the polynomial
\begin{equation}
\label{eq:schur}
\schur_{\bm{\lambda}}(\bm{a})
:=
\sum_{\bm{z}\in \GT{n}{\Z}\left(\bm{\lambda}\right)}
\prod_{k=1}^{n} a_k^{\type(\bm{z})_k}  \, .
\end{equation}
\end{definition}

As characters of the irreducible representations of $\GL_n(\C)$, Schur functions can be written as a ratio of determinants by the Weyl character formula~\cite{fultonHarris91}:
\begin{equation}
\label{eq:schurDet}
\schur_{\bm{\lambda}}(\bm{a})
= \frac{\det\left(a_i^{\lambda_j+n-j}\right)_{1\leq i,j\leq n}}{\det\left(a_i^{n-j}\right)_{1\leq i,j\leq n}}
= \frac{\det\left(a_i^{\lambda_j+n-j}\right)_{1\leq i,j\leq n}}{\prod_{1\leq i<j\leq n} (a_i - a_j)} \, ,
\end{equation}
being the denominator a Vandermonde determinant.
It is clear from the above formula that Schur functions are symmetric, i.e.\ invariant under the action of $\Sym_n$ (which is the Weyl group associated to $\GL_n(\C)$) on its arguments.

It can be easily deduced from either the definition or the determinantal formula that
\[
\schur_{\bm{\lambda}+k}(\bm{a})
= \left(\prod_{i=1}^n a_i \right)^k \schur_{\bm{\lambda}}(\bm{a}) \, ,
\]
where $\bm{\lambda}+k$ stands for $(\lambda_1 +k,\dots,\lambda_n +k)$.
This remark permits extending the definition of Schur functions and its associated determinantal formula to any $\bm{\lambda}\in\Z^n$ such that $\lambda_1 \geq \cdots \geq \lambda_n$.
Notice however that, if we do not require $\lambda_n \geq 0$, the functions thus obtained are in general not polynomial but Laurent polynomials.

In chapter~\ref{ch:LPP} we will come across a natural ``continuous version'' of Schur functions, which is essentially obtained from~\eqref{eq:schur} by replacing the sum on integer Gelfand-Tsetlin patterns with an integral on real Gelfand-Tsetlin patterns.
\begin{definition}
\label{def:schurCont}
We define the \emph{continuous (standard) Schur function} indexed by a parameter $\bm{\alpha}=(\alpha_1,\dots,\alpha_n)\in\C^n$ to be
\begin{equation}
\label{eq:schurCont}
\schur^{\cont}_{\bm{\alpha}}(\bm{x})
:= \int_{\GT{n}{\R}(\bm{x})}
\prod_{k=1}^n
\e^{ \alpha_k \type(\bm{z})_k}
\prod_{\substack{1\leq i<n \\ 1\leq j\leq i}}
\!\!
\diff z_{i,j}
\end{equation}
for all $\bm{x}=(x_1,\dots,x_n)\in\R^n$ such that $x_1 > \dots > x_n$.
\end{definition}
We point out that, if some of the inequalities in $x_1 > \dots > x_n$ are replaced with equalities, then some of the entries of any $\bm{z}\in \GT{n}{\R}(\bm{x})$ are ``blocked'' by the interlacing conditions (e.g., if $x_{n-1}=x_n$, then $z_{n-1,n-1}$ must also equal them).
Therefore, in this case $\GT{n}{\R}(\bm{x})$ has zero Lebesgue measure and the integral in~\eqref{eq:schurCont} vanishes.

We also remark that, comparing to~\eqref{eq:schur}, the roles of the parameter and the argument are exchanged in~\eqref{eq:schurCont}: we have adopted this notation in the continuous setting by analogy with the usual notation of Whittaker functions, whose scaling limits are just continuous Schur functions (see section~\ref{sec:Whittaker} and in particular Proposition~\ref{prop:glWhittakerRescaling}).

By Riemann sum approximation, continuous Schur functions are indeed the continuous limit of Schur functions:
\[
\begin{split}
\schur^{\cont}_{\bm{\alpha}}(\bm{x})
&= \lim_{\delta \downarrow 0}
\delta^{n(n-1)/2}
\!\!\!\!
\sum_{\bm{z} \in \GT{n}{\Z}(\bm{x}/\delta)} \prod_{k=1}^n \e^{\delta \alpha_k \type(\bm{z})_k} \\
&= \lim_{\delta \downarrow 0}
\delta^{n(n-1)/2}
\schur_{\floor{\bm{x}/\delta}}\left(\e^{\delta \alpha_1},\dots, \e^{\delta \alpha_n} \right) \, .
\end{split}
\]
By plugging formula~\eqref{eq:schurDet} for Schur functions into the latter expression and computing the limit, one can see that continuous Schur functions also have a determinantal form:
\begin{equation}
\label{eq:schurContDet}
\schur^{\cont}_{\bm{\alpha}}(\bm{x})
= \frac{\det\left(\e^{\alpha_j x_i}\right)_{1\leq i,j\leq n}}{\prod_{1\leq i<j\leq n} (\alpha_i - \alpha_j)} \, .
\end{equation}

\begin{remark}
\label{rem:schurEqualParameters}
When $\alpha_i=\alpha_j$ for some $i,j$, the determinantal formulas~\eqref{eq:schurDet} and~\eqref{eq:schurContDet} take on the indeterminate form $0/0$, but are still valid in the limit as $\alpha_i-\alpha_j \to 0$.
For example, when all $\alpha_i$'s are equal to a given $\alpha$, \eqref{eq:schurContDet} becomes
\[
\schur^{\cont}_{(\alpha,\dots,\alpha)}(\bm{x})
= \frac{\det\left(x_i^{n-j} \e^{\alpha x_i} \right)_{1\leq i,j\leq n}}
{\prod_{k=1}^n (k-1)!} \, .
\]
This formula is an immediate consequence of the following fact: if the functions $f_1,\dots,f_n$ are differentiable $n-1$ times at $\alpha$, then
\begin{equation}
\label{eq:wronskian}
\frac{\det(f_i(\alpha_j))_{1\leq i,j\leq n}}{\prod_{1\leq i<j\leq n} (\alpha_j - \alpha_i)}
\longrightarrow
\frac{W(f_1,\dots,f_n)(\alpha)}
{\prod_{k=1}^n (k-1)!} \qquad \text{as } \alpha_1,\dots,\alpha_n \to \alpha \, ,
\end{equation}
where $W(f_1,\dots,f_n)(\alpha) := \det\big(f_i^{(j-1)}(\alpha)\big)_{1\leq i,j\leq n}$ is the Wronskian of $f_1,\dots,f_n$ at $\alpha$.
\end{remark}

Schur functions satisfy the celebrated \emph{Cauchy identity}:
\begin{equation}
\label{eq:cauchyIdentity}
\sum_{\bm{\lambda}}
\schur_{\bm{\lambda}}(p_1, \dots, p_n) \,
\schur_{\bm{\lambda}}(q_1, \dots, q_n)
= \prod_{i,j=1}^n \frac{1}{1- p_i q_j} \, ,
\end{equation}
where the sum is over all integer partitions $\bm{\lambda}$ of length at most $n$.
If we view each term $(1-p_i q_j)^{-1}$ as the corresponding geometric series $\sum_{k=0}^{\infty} (p_i q_j)^k$, the Cauchy identity is true in the sense of formal power series, but has analytical meaning only when $\abs{p_i q_j} <1$ for all $i,j$.
The classical combinatorial proof of~\eqref{eq:cauchyIdentity} (see for example~\cite[Th.~7.12.1]{stanley99}) relies on the $\RSK$ correspondence; at the beginning of chapter~\ref{ch:LPP} we will incidentally see a version of such a proof that also involves a probabilistic interpretation.

The analogous Cauchy identity for continuous Schur functions reads as
\begin{equation}
\label{eq:cauchyIdentityCont}
\int_{\{x_1 > \cdots > x_n >0\}} \schur^{\cont}_{-\bm{\alpha}}(\bm{x}) 
\, \schur^{\cont}_{-\bm{\beta}}(\bm{x})
\prod_{i=1}^n \diff x_i
= \prod_{i,j=1}^n \frac{1}{\alpha_i + \beta_j} \, ,
\end{equation}
where the parameters $\bm{\alpha}$ and $\bm{\beta}$ satisfy $\Re(\alpha_i + \beta_j) >0$ for all $i,j$.
It can be easily deduced from~\eqref{eq:cauchyIdentity} by setting $p_i := \e^{-\delta \alpha_i}$ and $q_j := \e^{-\delta \beta_j}$, letting $\delta \downarrow 0$, and using Riemann sum approximations.

\subsection{Symplectic Schur functions}
\label{subsec:symplecticSchur}

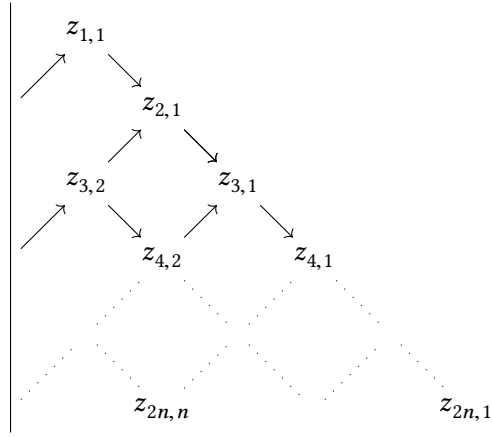
\begin{figure}
\centering
\begin{tikzpicture}[scale=1]

\node (z11) at (1,-1) {$z_{1,1}$};
\node (z21) at (2,-2) {$z_{2,1}$};
\node (z22) at (0,-2) {};
\node (z31) at (3,-3) {$z_{3,1}$};
\node (z32) at (1,-3) {$z_{3,2}$};
\node (z41) at (4,-4) {$z_{4,1}$};
\node (z42) at (2,-4) {$z_{4,2}$};
\node (z43) at (0,-4) {};
\node (z51) at (5,-5) {};
\node (z52) at (3,-5) {};
\node (z53) at (1,-5) {};
\node (z61) at (6,-6) {$z_{2n,1}$};
\node (z62) at (4,-6) {};
\node (z63) at (2,-6) {$z_{2n,n}$};
\node (z64) at (0,-6) {};

\draw[->] (z22) -- (z11);
\draw[->] (z11) -- (z21);
\draw[->] (z32) -- (z21);
\draw[->] (z21) -- (z31);
\draw[->] (z21) -- (z31);
\draw[->] (z43) -- (z32);
\draw[->] (z32) -- (z42);
\draw[->] (z42) -- (z31);
\draw[->] (z31) -- (z41);
\draw[loosely dotted] (z53) -- (z42);
\draw[loosely dotted] (z42) -- (z52);
\draw[loosely dotted] (z52) -- (z41);
\draw[loosely dotted] (z41) -- (z51);
\draw[loosely dotted] (z64) -- (z53);
\draw[loosely dotted] (z53) -- (z63);
\draw[loosely dotted] (z63) -- (z52);
\draw[loosely dotted] (z52) -- (z62);
\draw[loosely dotted] (z62) -- (z51);
\draw[loosely dotted] (z51) -- (z61);

\draw (0,-0.6) -- (0,-6.3);

\end{tikzpicture}
\caption[``Half-triangular'' arrays and symplectic Gelfand-Tsetlin patterns]{A ``half-triangular'' array of height $2n$.
If this is interpreted as a symplectic Gelfand-Tsetlin pattern, the arrows illustrate the interlacing condition: $a \to b$ means $a\leq b$, with the convention that $a$ equals $0$ if it lies on the vertical ``wall''.
The arrows may also refer to the functional~\eqref{eq:soEnergy}: $\spEn(\bm{z})$ is the sum of all ratios $a/b$ such that there is an arrow pointing from $a$ to $b$ in the diagram; in this case, $a$ equals $1$ by convention if it lies on the vertical wall, so that the (only) inhomogeneous addends in $\spEn(\bm{z})$ are $1/z_{2k-1,k}$ for $1\leq k\leq n$.}
\label{fig:spGTpattern}
\end{figure}

Symplectic Schur functions are Laurent polynomials indexed by partitions, invariant under both permutation of the variables and multiplicative inversion of any of them.
In combinatorics they are defined as generating functions of symplectic tableaux, or symplectic Gelfand-Tsetlin patterns equivalently, and are connected to the Berele's insertion algorithm~\cite{berele86}.
A \emph{symplectic Gelfand-Tsetlin pattern} of height $2n$ is a ``half-triangular'' array $\bm{z}=\{z_{i,j}\colon 1\leq i\leq 2n, \, 1\leq j\leq \ceil{i/2}\}$ satisfying the interlacing conditions:
\begin{equation}
\label{eq:interlacingSp}
z_{i+1,j+1}\leq z_{i,j}\leq z_{i+1,j}
\qquad\quad \text{for } 1\leq i < 2n \, , \,\, 1\leq j\leq \ceil{i/2} \, ,
\end{equation}
with the convention that $z_{i,j}:=0$ when $j>\ceil{i/2}$ (so that all its entries are non-negative); the \emph{shape} of $\bm{z}$ is its bottom row $(z_{2n,1},\dots,z_{2n,n})$.
See Figure~\ref{fig:spGTpattern} for a graphical representation.
Given $\bm{x}\in\R^n$ and a set $S\subseteq \R$, let us denote by $\spGT{2n}{S}(\bm{x})$ the set of all symplectic Gelfand-Tsetlin patterns of height $2n$ with entries in $S$ and shape $\bm{x}$; notice that $\spGT{2n}{S}(\bm{x})$ is empty whenever $x_1\geq \dots\geq x_n \geq 0$ fails to hold.
Analogously to~\eqref{eq:typeGTpattern}, we define the \emph{type} of a symplectic Gelfand-Tsetlin pattern $\bm{z}$ to be the vector $\type(\bm{z})\in \R_{\geq 0}^{2n}$ such that
\begin{equation}
\label{eq:typeSpGTpattern}
\type(\bm{z})_i := \sum_{j = 1}^{\ceil{\frac{i}{2}}} z_{i,j} - \sum_{j = 1}^{\ceil{\frac{i-1}{2}}} z_{i-1,j} \qquad \text{for } i=1, \dots, 2n \, .
\end{equation}
\begin{definition}
\label{def:spSchur}
The \emph{symplectic Schur function} in $n$ variables $\bm{a}=(a_1,\dots,a_n)\in\C^n$ indexed by an integer partition $\bm{\lambda}$ of length at most $n$ is the Laurent polynomial
\begin{equation}
\label{eq:spSchur}
\sp_{\bm{\lambda}}(\bm{a})
:=
\sum_{\bm{z}\in \spGT{2n}{\Z}\left(\bm{\lambda}\right)}
\prod_{k=1}^{n} a_k^{\type(\bm{z})_{2k-1} - \type(\bm{z})_{2k}} \, .
\end{equation}
\end{definition}
As characters of the irreducible representations of the symplectic group $\Sp_{2n}(\C)$, symplectic Schur functions can be written as a ratio of determinants (the denominator having a closed form) by the Weyl character formula~\cite{fultonHarris91}:
\begin{equation}
\label{eq:spSchurDet}
\begin{split}
\sp_{\bm{\lambda}}(\bm{a})
&= \frac{\det\left(a_i^{\lambda_j+n-j+1} - a_i^{-(\lambda_j+n-j+1)}\right)_{1\leq i,j\leq n}}
{\det\left(a_i^{n-j+1} - a_i^{-(n-j+1)}\right)_{1\leq i,j\leq n}} \\
&= \frac{\det\left(a_i^{\lambda_j+n-j+1} - a_i^{-(\lambda_j+n-j+1)}\right)_{1\leq i,j\leq n}}{\prod_{1\leq i<j\leq n}
(a_i - a_j)(a_i a_j - 1)
\prod_{k=1}^n (a_k^2-1)a_k^{-n}} \, .
\end{split}
\end{equation}
It is clear from the above formula that Schur functions are invariant under transformations of the type $a_i \mapsto a_i^{-1}$ and permutations of the $a_i$'s, i.e.\ under the action of $(\Z/2\Z)^n \rtimes \Sym_n$ (which is the Weyl group associated to $\Sp_{2n}(\C)$) on its arguments.

In chapter~\ref{ch:LPP} we will also use the ``continuous version'' of symplectic Schur functions, which is an integral on real symplectic Gelfand-Tsetlin patterns as specified below.
\begin{definition}
\label{def:spSchurCont}
We define the \emph{continuous symplectic Schur function} indexed by a parameter $\bm{\alpha}=(\alpha_1,\dots,\alpha_n) \in\C^n$ to be
\begin{equation}
\label{eq:spSchurCont}
\sp^{\cont}_{\bm{\alpha}}(\bm{x})
:= \int_{\spGT{2n}{\R}(\bm{x})}
\prod_{k=1}^n
\e^{ \alpha_k \left(\type(\bm{z})_{2k-1} - \type(\bm{z})_{2k}\right)}
\prod_{\substack{1\leq i<2n \\ 1\leq j\leq \ceil{i/2}}}
\!\!
\diff z_{i,j}
\end{equation}
for all $\bm{x}=(x_1,\dots,x_n)\in\R^n$ such that $x_1 > \dots > x_n > 0$.
\end{definition}
The same remarks hold as for standard Schur functions: if some of the inequalities in $x_1 > \dots > x_n > 0$ are replaced with equalities, the integral in~\eqref{eq:spSchurCont} vanishes; the roles of the variables and the parameters are exchanged in the discrete and continuous case.

By Riemann sum approximation, continuous symplectic Schur functions are indeed the continuous limit of their discrete analog:
\[
\begin{split}
\sp^{\cont}_{\bm{\alpha}}(\bm{x})
&= \lim_{\delta \downarrow 0}
\delta^{n^2}
\!\!\!
\sum_{\bm{z} \in \spGT{2n}{\Z}(\bm{x}/\delta)}
\prod_{k=1}^n \e^{ \delta\alpha_k \left(\type(\bm{z})_{2k-1} - \type(\bm{z})_{2k}\right)} \\
&= \lim_{\delta \downarrow 0}
\delta^{n^2}
\sp_{\floor{\bm{x}/\delta}}\left(\e^{\delta \alpha_1},\dots, \e^{\delta \alpha_n} \right) \, .
\end{split}
\]
By plugging formula~\eqref{eq:spSchurDet} into the above limit, one can see that continuous symplectic Schur functions also have a determinantal form:
\begin{equation}
\label{eq:spSchurContDet}
\sp^{\cont}_{\bm{\alpha}}(\bm{x})
= \frac{\det\left(\e^{\alpha_j x_i} - \e^{-\alpha_j x_i}\right)_{1\leq i,j\leq n}}{\prod_{1\leq i<j\leq n}(\alpha_i-\alpha_j)
(\alpha_i+\alpha_j) \prod_{k=1}^n (2\alpha_k)} \, .
\end{equation}

\begin{remark}
\label{rem:spSchurEqualParameters}
When $\alpha_i=\alpha_j$ for some $i,j$, the determinantal formulas~\eqref{eq:spSchurDet} and~\eqref{eq:spSchurContDet} are still valid in the limit as $\alpha_i-\alpha_j \to 0$.
For example one can use~\eqref{eq:wronskian} to prove that, when all $\alpha_i$'s are equal to a given $\alpha$, \eqref{eq:spSchurContDet} reads as
\[
\sp^{\rm cont}_{(\alpha,\dots,\alpha)}(\bm{x})
= \frac{\det\big(x_i^{n-j} (\e^{\alpha x_i} -(-1)^{n-j} \e^{-\alpha x_i}) \big)_{1\leq i,j\leq n}}
{(2\alpha)^{n(n+1)/2} \prod_{k=1}^n (k-1)!} \, .
\]
\end{remark}

There is a \emph{symplectic Cauchy identity} that pairs a symplectic Schur function with a standard one~\cite{sundaram90}:
\begin{equation}
\label{eq:cauchyIdentitySp}
\sum_{\bm{\lambda}}
\sp_{\bm{\lambda}}(p_1, \dots, p_n) \,
\schur_{\bm{\lambda}}(q_1, \dots, q_n)
= \frac{\prod_{1\leq i<j \leq n} (1-q_i q_j)}{\prod_{1\leq i,j \leq n} (1-q_i p_j)(1-q_i p_j^{-1})} \, ,
\end{equation}
where the sum is over all integer partitions $\bm{\lambda}$ of length at most $n$.
Again, this identity is true in the sense of formal power series, but has analytical meaning only when $\abs{q_i p_j^{\pm 1}} <1$ for all $i,j$ (this in particular implies that $\abs{q_i} < 1$ for all $i$).
Sundaram gave a bijective proof of~\eqref{eq:cauchyIdentity} based on an application of the Berele's insertion algorithm~\cite{berele86}.

The analog of~\eqref{eq:cauchyIdentitySp} for continuous symplectic Schur functions reads as
\begin{equation}
\label{eq:cauchyIdentitySpCont}
\int_{\{x_1 > \cdots > x_n >0\}}
\sp^{\cont}_{\bm{\alpha}}(\bm{x}) 
\, \schur^{\cont}_{-\bm{\beta}}(\bm{x})
\prod_{i=1}^n \diff x_i
= \frac{\prod_{1\leq i<j \leq n} (\beta_i + \beta_j)}{\prod_{1\leq i,j \leq n} (\beta_i + \alpha_j)(\beta_i - \alpha_j)} \, ,
\end{equation}
where the parameters $\bm{\alpha}$ and $\bm{\beta}$ satisfy $\Re(\beta_i \pm \alpha_j) >0$ for all $i,j$.
It can be easily deduced from~\eqref{eq:cauchyIdentitySp} by setting $p_i := \e^{\delta \alpha_i}$ and $q_j := \e^{-\delta \beta_j}$, letting $\delta \downarrow 0$, and using Riemann sum approximations.

\section{Whittaker functions}\label{sec:Whittaker}

Whittaker functions appear in many different mathematical contexts, and have been accordingly defined and used in various ways.
They were first introduced by Edmund T.~ Whittaker~\cite{whittaker03} as solutions to the so-called Whittaker differential equation.
In modern analytic number theory, they play a central role in the area of automorphic forms and $L$-functions.
In this context, they come associated to a real reductive group, which corresponds to $\GL_2(\R)$ for the original Whittaker functions.
It was Jacquet~\cite{jacquet67} who first constructed, via a certain integral representation, Whittaker functions associated to higher rank groups.

In a representation theoretic setting, Konstant~\cite{kostant78} showed how Whittaker functions, viewed as solutions to a certain integrable system called quantum Toda lattice and associated with a real reductive Lie algebra, are connected to certain representations of that algebra.
Givental~\cite{givental97} constructed solutions to the same quantum integrable system via methods of quantum cohomology and mirror symmetry: more precisely, he constructed $\GL_n(\R)$-Whittaker functions as integrals over a mirror family.
This approach was extended further for general classical groups by Gerasimov-Lebedev-Oblezin \cite{gerasimovLebedevOblezin08, gerasimovLebedevOblezin12}. 

A comprehensive summary of various realizations of Whittaker functions can be found in the survey paper~\cite{lam13} and in the PhD thesis of Chhaibi~\cite{chhaibi13}: the first focuses on algebraic, geometric and combinatorial aspects, while the second one gives a probabilistic insight as well.

In this section we focus on Whittaker functions associated to the Lie algebras $\gl_n$ and $\so_{2n+1}$.
In particular, we introduce their integral representations on triangular patterns due to Givental~\cite{givental97} and Gerasimov-Lebedev-Oblezin~\cite{gerasimovLebedevOblezin08, gerasimovLebedevOblezin12}, and deal with their most relevant aspects for our purposes.
We will indifferently talk about Whittaker functions associated to a given Lie group, e.g.\ $\GL_n(\R)$, or to the corresponding Lie algebra, e.g.\ $\gl_n$.
However, we will tend to use more the former notation in the number theoretic context of Appendix~\ref{appendix:Whittaker}, where we provide further information on the role of Whittaker functions in number theory.

\subsection{$\gl_n$-Whittaker functions}
\label{subsec:glWhittaker}

Following~\cite{givental97, gerasimovLebedevOblezin08, gerasimovLebedevOblezin12}, we introduce $\gl_n$-Whittaker functions as integrals on triangular arrays.
Consider a triangular array
\begin{equation}
\label{eq:triangularArray}
\bm{z} = \{ z_{i,j}\colon 1\leq j\leq i\leq n \}
\end{equation}
of height $n\geq 1$ with positive entries.
Recall that if the entries are interlaced in the sense of~\eqref{eq:interlacing}, such an array is known as a Gelfand-Tsetlin pattern.
Even though in this context we work with triangular arrays that are \emph{not} required to satisfy the interlacing conditions, we still define a potential on them that encourages interlacement:
\begin{equation}
\label{eq:glEnergy}
\En(\bm{z})
:= \sum_{i=1}^{n-1} \sum_{j=1}^i
\left( \frac{z_{i+1,j+1}}{z_{i,j}} +\frac{z_{i,j}}{z_{i+1,j}} \right)
\, .
\end{equation}
Whenever an interlacing condition is \emph{not} satisfied, one of the ratios appearing above becomes larger than one, increasing the overall potential of the array.
Figure~\ref{fig:GTpattern} provides a visual representation of~\eqref{eq:glEnergy}.

We call $i$-th \emph{row} of $\bm{z}$ the vector $(z_{i,1},\dots,z_{i,i})$ of all entries with first index equal to $i$.
Replacing the tropical operations $(\max,+)$ with the usual ones $(+,\cdot)$ in~\eqref{eq:typeGTpattern}, we define the \emph{geometric type} of $\bm{z}$ to be the vector $\gtype(\bm{z}) \in \R_{>0}^{n}$ whose $i$-th component is the ratio between the product of the $i$-th row of $\bm{z}$ and the product of its $(i-1)$-th row:
\begin{equation}
\label{eq:glType}
\gtype(\bm{z})_i
:= \frac{\prod_{j = 1}^{i} z_{i,j}}{\prod_{j = 1}^{i-1} z_{i-1,j}} \qquad \text{for } i=1, \dots, n
\, .
\end{equation}
Denoting by $\T{n}{S}(\bm{x})$ the set of all triangular arrays of height $n$ with entries in $S\subseteq\R$ and bottom row equal to a given vector $\bm{x}\in \R^n$, we now define the $\gl_n$-Whittaker functions via the following integral representation:
\begin{definition}
\label{def:glWhittakerFn} 
The \emph{$\gl_{n}$-Whittaker function} with parameter $\bm{\alpha}=(\alpha_1,\dots,\alpha_n)\in\C^n$ and argument $\bm{x}=(x_1,\dots,x_n) \in\R_{>0}^n$ is given by
\begin{equation}
\label{eq:glWhittakerFn}
\Psi^{\gl_{n}}_{\bm{\alpha}}(\bm{x})
:= \int_{\T{n}{\R_{>0}}(\bm{x})}
\prod_{k=1}^{n} \gtype(\bm{z})_k^{\alpha_k}
\exp\left\{-\En(\bm{z})\right\}
\prod_{\substack{1\leq i<n \\ 1\leq j\leq i}} \frac{\diff z_{i,j}}{z_{i,j}} \, .
\end{equation}
\end{definition}
For example, $\Psi^{\gl_1}_{\alpha}(x) = x^{\alpha}$ (no integration is involved), and
\begin{equation}
\label{eq:gl_2WhittakerFn}
\Psi^{\gl_2}_{(\alpha_1,\alpha_2)}(x_1,x_2)
= \int_{\R_{>0}} z^{\alpha_1}\left(\frac{x_1 x_2}{z}\right)^{\alpha_2} \exp\left\{-\frac{x_2}{z}-\frac{z}{x_1}\right\} \frac{\diff z}{z} \, .
\end{equation}

The representation of $\gl_n$-Whittaker functions in Definition~\ref{def:glWhittakerFn} 
has a recursive structure: setting $\Psi_{\emptyset}^{\gl_{0}}(\emptyset) :=1$, it turns out that for all $n\geq 1$,
 $\bm{\alpha}\in\C^n$ and $\bm{x}\in\R_{>0}^n$
\begin{equation}
\label{eq:glWhittakerRecurs}
\Psi_{\bm{\alpha}}^{\gl_{n}}(\bm{x})
= \int_{\R_{>0}^{n-1}}
Q^{\gl_{n}}_{\alpha_n}(\bm{x},\bm{u})
\Psi_{\tilde{\bm{\alpha}}}^{\gl_{n-1}}(\bm{u})
\prod_{i=1}^{n-1} \frac{\diff u_i}{u_i} \, ,
\end{equation}
where $\tilde{\bm{\alpha}}:=(\alpha_1,\dots,\alpha_{n-1})$ and the kernel is defined by
\[
Q^{\gl_{n}}_{\alpha_n}(\bm{x},\bm{u})
:= \left( \frac{\prod_{i=1}^n x_i}{\prod_{i=1}^{n-1} u_i} \right)^{\alpha_n}
\prod_{i=1}^{n-1} \exp\left\{-\frac{x_{i+1}}{u_i} - \frac{u_i}{x_i}\right\} \, .
\]

The following properties of $\gl_n$-Whittaker functions are straightforward consequences of the definition:
\begin{itemize}
\item
if $c\in\C$ and $\bm{\alpha}+c$ stands for $(\alpha_1+c,\dots,\alpha_n+c)$, then
\begin{equation}
\label{eq:glWhittakerFnTranslation}
\Psi^{\gl_n}_{\bm{\alpha}+c}(\bm{x})
= \bigg(\prod_{i=1}^n x_i\bigg)^c
\Psi^{\gl_n}_{\bm{\alpha}}(\bm{x}) \, ;
\end{equation}
\item
if $s>0$, then
\begin{equation}
\label{eq:glWhittakerFnScalarMultiplication}
\Psi^{\gl_n}_{\bm{\alpha}}(s \bm{x})
= s^{\sum_{i=1}^n \alpha_i} \Psi^{\gl_n}_{\bm{\alpha}}(\bm{x}) \, ;
\end{equation}
\item
if we set $y_i := x_{n-i+1}^{-1}$ for $1\leq i\leq n$, then
\begin{equation}
\label{eq:glWhittakerFnSignChange}
\Psi^{\gl_n}_{-\bm{\alpha}}(\bm{x})
= \Psi^{\gl_n}_{\bm{\alpha}}(\bm{y}) \, .
\end{equation}
\end{itemize}

Properly rescaled $\gl_n$-Whittaker functions yield continuous Schur functions, as the next theorem states.
Such a rescaling is strongly related to the tropicalization procedure explained in subsection~\ref{subsec:tropicalization&RSK} and to the zero temperature limit that will be the subject of appendix~\ref{appendix:zeroTempLimit}.
\begin{proposition}
\label{prop:glWhittakerRescaling}
For all $\bm{\alpha}\in\C^n$ and $\bm{x}\in\R^n$ we have that
\begin{equation}
\label{eq:glWhittakerRescaling}
\lim_{\epsilon\downarrow 0}
\epsilon^{n(n-1)/2} 
\Psi_{\epsilon\bm{\alpha}}^{\gl_n} \left(\e^{x_1/\epsilon},\dots,\e^{x_n/\epsilon}\right)
= \schur^{\cont}_{\bm{\alpha}}(\bm{x})
\1_{\{x_1 > \dots > x_n \}} \, .
\end{equation}
\end{proposition}

\begin{proof}
For any index set $\mathcal{I}$ and array $\bm{a}= \{ a_k\colon k\in\mathcal{I} \} \in \R^{\mathcal{I}}$ and for all $\epsilon>0$, let us set $g_{\epsilon}(\bm{a}) := \{ \e^{a_k / \epsilon}\colon k\in\mathcal{I} \} \in \R_{>0}^{\mathcal{I}}$ as in subsection~\ref{subsec:tropicalization&RSK}.
In definition~\eqref{eq:glWhittakerFn}, we change variables by setting $z_{i,j}\mapsto g_{\epsilon}(z_{i,j})$ for all $1\leq j\leq i< n$ and obtain
\[
\Psi^{\gl_n}_{\epsilon \bm{\alpha}}(g_{\epsilon}(\bm{x}))
= \int_{\T{n}{\R}(\bm{x})}
\prod_{k=1}^{n} \gtype(g_{\epsilon}(\bm{z}))_k^{\epsilon\alpha_k}
\exp\left\{-\En(g_{\epsilon}(\bm{z}))\right\}
\prod_{\substack{1\leq i<n \\ 1\leq j\leq i}} \frac{\diff z_{i,j}}{\epsilon} \, .
\]
Using definitions~\eqref{eq:glType} and~\eqref{eq:typeGTpattern}, one can verify that $\gtype(g_{\epsilon}(\bm{z}))_k = g_{\epsilon}(\type(\bm{z})_k)$.
Furthermore, since $\exp\{-g_{\epsilon}(a)\} = \exp\{-\e^{a/\epsilon}\} \xrightarrow{\epsilon\downarrow 0} \1_{\{a\leq 0\}}$ for all $a\neq 0$, we have that
\[
\begin{split}
\exp\left\{-\En(g_{\epsilon}(\bm{z}))\right\}
= \, &\prod_{i=1}^{n-1} \prod_{j=1}^i
\exp\{ -g_{\epsilon}(z_{i+1,j+1}-z_{i,j})\}
\exp\{ -g_{\epsilon}(z_{i,j}-z_{i+1,j})\} \\
\xrightarrow{\epsilon\downarrow 0} \, &\prod_{i=1}^{n-1} \prod_{j=1}^i \1_{\{z_{i+1,j+1}\leq z_{i,j} \leq z_{i+1,j} \}}
= \1_{\GT{n}{\R}(\bm{x})}(\bm{z})
\end{split}
\]
for a.e.\ $\bm{z}\in \T{n}{\R}(\bm{x})$.
Notice that if $x_1 > \dots > x_n$ is \emph{not} satisfied, then the Lebesgue measure of $\GT{n}{\R}(\bm{x})$ vanishes.
By dominated convergence, it follows that
\[
\epsilon^{n(n-1)/2} 
\Psi_{\epsilon\bm{\alpha}}^{\gl_n} (g_{\epsilon}(\bm{x}))
\xrightarrow{\epsilon\downarrow 0}
\1_{\{x_1 > \dots > x_n \}}
\int_{\GT{n}{\R}(\bm{x})}
\prod_{k=1}^n
\e^{ \alpha_k \type(\bm{z})_k}
\prod_{\substack{1\leq i<n \\ 1\leq j\leq i}}
\!\!
\diff z_{i,j} \, .
\]
The latter integral defines the continuous Schur function on the right-hand side of~\eqref{eq:glWhittakerRescaling}.
\end{proof}

A property that Whittaker functions share with Schur functions is that they are invariant under the action of the corresponding Weyl group on the (spectral) parameters.
For the case $\gl_n$, although not obvious from the definition, this means that $\Psi^{\gl_n}_{\bm{\alpha}}(\cdot)$ is symmetric w.r.t.\ $(\alpha_1,\dots,\alpha_n)$.
 
As mentioned at the beginning of this section, Whittaker functions are related to certain integrable systems called quantum Toda lattices.
Let us define the quantum Toda hamiltonian associated to a Lie algebra $\mathfrak{g}$ as
\begin{equation}
\label{eq:quantumToda}
\mathcal{H}^{\mathfrak{g}} :=
\sum_{i=1}^n \frac{\partial^2}{\partial x_i^2}
- 2 \sum_{\bm{a}} d_{\bm{a}} \, \e^{-\langle \bm{a} , \bm{x} \rangle} \, ,
\end{equation}
where the sum is over a set of simple roots $\bm{a}\in \R^n$ of $\mathfrak{g}$, $\langle \cdot , \cdot \rangle$ denotes the standard scalar product in $\R^n$ and $d_{\bm{a}}$'s are appropriate rational constants (see~\cite{gerasimovLebedevOblezin12} for details).
Then this operator is diagonalized by $\mathfrak{g}$-Whittaker functions.
Since the simple roots of $\gl_n$ are $\bm{e}_i - \bm{e}_{i+1}$ for $1\leq i\leq n-1$, where $(\bm{e}_1,\dots,\bm{e}_n)$ is the canonical basis of $\R^n$, the quantum Toda hamiltonian associated to $\gl_n$ (also called of type $A_{n-1}$) is given by
\[
\mathcal{H}^{\gl_n} :=
\sum_{i=1}^n \frac{\partial^2}{\partial x_i^2}
-2\sum_{i=1}^{n-1} \e^{x_{i+1}-x_i} \, .
\]
It then turns out that the $\gl_n$-Whittaker function $\Psi^{\gl_n}_{\bm{\lambda}}(\e^{x_1},\dots,\e^{x_n})$ in exponential coordinates is eigenfunction of $\mathcal{H}^{\gl_n}$ with eigenvalue $\langle\bm{\lambda},\bm{\lambda}\rangle = \sum_{i=1}^n\lambda_i^2$.
As eigenfunctions of a self-adjoint operator, $\gl_n$-Whittaker functions come with a harmonic analysis, which is summarized in the following
\begin{theorem}[\cite{semenov-Tian-Shansky94, kharchevLebedev01}]
\label{thm:plancherel}
The integral transform
\[
\hat f(\bm{\lambda}) := \int_{\R_{>0}^n} f(\bm{x}) \Psi^{\gl_n}_{\bm{\lambda}}(\bm{x})
\prod_{i=1}^n \frac{\diff x_i}{x_i}
\]
defines an isometry from $L^2(\R_{>0}^n, \prod_{i=1}^n \diff x_i/x_i)$ to 
$L^2_{\sym}(\i\R^n,s_n(\bm{\lambda}) \diff \bm{\lambda})$, where $\i=\sqrt{-1}$, $L^2_{\sym}$ denotes the space of square integrable functions that are symmetric in their variables, and
\begin{equation}
\label{eq:sklyaninMeasure}
s_n(\bm{\lambda}) \diff \bm{\lambda}
:= \frac{1}{(2\pi\i)^n n!} \prod_{i\neq j} \frac{1}{\Gamma(\lambda_i-\lambda_j)}
\prod_{k=1}^n \diff\lambda_k
\end{equation}
is the \emph{Sklyanin measure}.
Namely, for all $f,g\in L^2(\R_{>0}^n, \prod_{i=1}^n \diff x_i/x_i)$ it holds that
\[
\int_{\R_{>0}^n} f(\bm{x}) \overline{g(\bm{x})} \prod_{i=1}^n\frac{\diff x_i}{x_i} 
= \int_{\i\R^n} \hat{f}(\bm{\lambda}) \overline{\hat{g}(\bm{\lambda})} s_n(\bm{\lambda}) \diff \bm{\lambda} \, .
\]
\end{theorem}

Notice that, using the functional equation for the gamma function and Euler's reflection formula, one can show that the Sklyanin measure defined in~\eqref{eq:sklyaninMeasure} is a positive measure on $\i\R^n$.

As will be outlined in section~\ref{sec:MaassForms&WhittakerFns} of the appendix, certain integrals of Whittaker functions play an important role in the theory of automorphic $L$-functions.
One such integral formula pairs two $\gl_n$-Whittaker functions and is associated with $L$-factors of automorphic $L$-functions on $\GL_n(\R) \times \GL_n(\R)$.
It was conjectured by Bump~\cite{bump89} and proved for the general $n$ case by Stade~\cite{stade02}.
\begin{theorem}[Bump-Stade identity]
\label{thm:bumpStade}
Let $r>0$ and $\bm{\alpha},\bm{\beta}\in\C^n$ such that $\Re(\alpha_i+\beta_j)>0$ for all $i,j$. Then
\begin{equation}
\label{eq:bumpStade}
\int_{\R_{>0}^n } \e^{-r x_1}
\Psi^{\gl_n}_{\bm{\alpha}}(\bm{x})
\Psi^{\gl_n}_{\bm{\beta}}(\bm{x})
\prod_{i=1}^n \frac{\diff x_i}{x_i}
= r^{-\sum_{k=1}^n (\alpha_k+\beta_k)} \prod_{i,j=1}^n \Gamma(\alpha_i+\beta_j) \, .
\end{equation}
\end{theorem}

A  bijective proof of this identity was given in~\cite{corwinOConnellSeppalainenZygouras14, oConnellSeppalainenZygouras14} via the use of the geometric $\RSK$ correspondence introduced in subsection~\ref{subsec:gRSK} - see Remark~\ref{rem:bumpStadeEquiv}.

Since $\gl_n$-Whittaker functions scale to continuous Schur functions by Proposition~\ref{prop:glWhittakerRescaling}, an appropriate limit of~\eqref{eq:bumpStade} again yields the continuous Cauchy identity~\eqref{eq:cauchyIdentityCont}.
This can be verified by setting $r=1$ and using~\eqref{eq:glWhittakerFnSignChange} and~\eqref{eq:glWhittakerRescaling}; in this limiting procedure, the exponential $\e^{- x_1}$ in~\eqref{eq:bumpStade} converts into the condition $x_n >0$ which appears in the region of integration of~\eqref{eq:cauchyIdentityCont}.
The Bump-Stade identity is thus the analog of the Cauchy identity in the setting of Whittaker functions.

As we will recap in the introduction to chapter~\ref{ch:polymer}, Theorems~\ref{thm:plancherel} and~\ref{thm:bumpStade} already played an important role in the computation of the Laplace transform of the point-to-point partition function of the log-gamma polymer.
We will likewise use these tools to study the log-gamma polymer in the point-to-line geometries (see section~\ref{sec:contourIntegrals}).

\subsection{$\so_{2n+1}$-Whittaker functions}
\label{subsec:soWhittaker}

Similarly to the $\gl_n$ case, we will define $\so_{2n+1}$-Whittaker functions as integrals on half-triangular arrays.
Such a definition was first given in~\cite{gerasimovLebedevOblezin08, gerasimovLebedevOblezin12}, but also naturally emerged in~\cite{nteka18} in the study of a system of interacting particles via \emph{intertwining} Markovian dynamics.
Let us consider a half-triangular array
\begin{equation}
\label{eq:halfTriangle}
\bm{z} = \{ z_{i,j}\colon 1\leq i\leq 2n \, , \,\, 1\leq j\leq \ceil{i/2}\}
\end{equation}
of height $2n$ with positive entries.
If the entries of $\bm{z}$ are interlaced in the sense of~\eqref{eq:interlacingSp}, such an array is known as a symplectic Gelfand-Tsetlin pattern.
In this context our arrays are not required to satisfy the interlacing conditions but are still encouraged 
to do so through the potential 
\begin{equation}
\label{eq:soEnergy}
\spEn(\bm{z})
:= \sum_{i=1}^{2n-1} \sum_{j=1}^{\ceil{i/2}}
\left( \frac{z_{i+1,j+1}}{z_{i,j}} +\frac{z_{i,j}}{z_{i+1,j}} \right)
\, ,
\end{equation}
where by convention $z_{i,j}:=1$ if $j>\ceil{i/2}$.
See Figure~\ref{fig:spGTpattern} for an illustration of this potential.
 
We also make the analogous definitions as in the $\gl_n$ case.
We call $i$-th \emph{row} of $\bm{z}$ the vector $(z_{i,1},\dots,z_{i,\ceil{i/2}})$ of all entries with first index equal to $i$.
We define the \emph{geometric type} of $\bm{z}$ to be the vector $\type(\bm{z}) \in \R_{>0}^{2n}$ whose components are
\begin{equation}
\label{eq:soType}
\gtype(\bm{z})_i
:= \frac{\prod_{j = 1}^{\ceil{i/2}} z_{i,j}}{\prod_{j = 1}^{\ceil{(i-1)/2}} z_{i-1,j}} \qquad \text{for } i=1, \dots, 2n
\, .
\end{equation}
Denoting by $\spT{2n}{S}(\bm{x})$ the set of all half-triangular arrays $\bm{z}$ of height $2n$ with entries in $S\subseteq\R$ and $(2n)$-th
row equal to a given vector $\bm{x}\in \R^n$, we define the $\so_{2n+1}$-Whittaker functions via the following integral representation:
\begin{definition}
\label{def:soWhittakerFn}
The \emph{$\so_{2n+1}$-Whittaker function} with parameter $\bm{\alpha}=(\alpha_1,\dots,\alpha_n) \in \C^n$ and argument $\bm{x}=(x_1,\dots,x_n)\in\R_{>0}^n$ is given by
\begin{equation}
\label{eq:soWhittakerFn}
\Psi^{\so_{2n+1}}_{\bm{\alpha}}(\bm{x})
:= \int_{\spT{2n}{\R_{>0}}(\bm{x})}
\prod_{k=1}^n \left(\frac{\gtype(\bm{z})_{2k-1}}{\gtype(\bm{z})_{2k}}\right)^{\alpha_k}
\exp\left\{-\spEn(\bm{z})\right\}
\prod_{\substack{1\leq i<2n \\ 1\leq j\leq \ceil{i/2}}} \!\! \frac{\diff z_{i,j}}{z_{i,j}} \, .
\end{equation}
\end{definition}

As an example, the $\so_3$-Whittaker function is
\begin{equation}
\label{eq:so_3WhittakerFn}
\Psi^{\so_3}_{\alpha}(x)
= \int_{\R_{>0}} \left(\frac{z^2}{x}\right)^{\alpha} \exp\left\{-\frac{1}{z}-\frac{z}{x}\right\} \frac{\diff z}{z} \, .
\end{equation}

Setting $\Psi_{\emptyset}^{\so_{1}}(\emptyset) :=1$, for any $n\geq 1$ the $\so_{2n+1}$-Whittaker functions can be recursively defined by
\begin{equation}
\label{eq:soWhittakerRecurs}
\Psi_{\bm{\alpha}}^{\so_{2n+1}}(\bm{x})
= \int_{\R_{>0}^{n-1}}
Q^{\so_{2n+1}}_{\alpha_n}(\bm{x},\bm{u})
\Psi_{\tilde{\bm{\alpha}}}^{\so_{2n-1}}(\bm{u})
\prod_{i=1}^{n-1} \frac{\diff u_i}{u_i} \, ,
\end{equation}
where $\tilde{\bm{\alpha}}:=(\alpha_1,\dots,\alpha_{n-1})$ and the kernel $Q^{\so_{2n+1}}_{\alpha_n}$ is given by
\[
\begin{split}
Q^{\so_{2n+1}}_{\alpha_n}(\bm{x},\bm{u})
:= \int_{\R_{>0}^n}
&\left( \frac{\prod_{i=1}^n v_i^2}{\prod_{i=1}^n x_i \prod_{i=1}^{n-1} u_i} \right)^{\alpha_n}
\prod_{i=1}^{n-1} \exp\left\{-\frac{v_{i+1}}{u_i} - \frac{u_i}{v_i} -\frac{x_{i+1}}{v_i} - \frac{v_i}{x_i} \right\} \\
&\times \exp\left\{-\frac{1}{v_n} - \frac{v_n}{x_n}\right\}
\prod_{i=1}^{n} \frac{\diff v_i}{v_i} \, .
\end{split}
\]

\begin{proposition}
\label{prop:soWhittakerRescaling}
For all $\bm{\alpha}\in\C^n$ and $\bm{x}\in\R^n$ we have that
\begin{equation}
\label{eq:soWhittakerRescaling}
\lim_{\epsilon\downarrow 0}
\epsilon^{n^2} 
\Psi_{\epsilon\bm{\alpha}}^{\so_{2n+1}} \left(\e^{x_1/\epsilon},\dots,\e^{x_n/\epsilon}\right)
= \sp^{\cont}_{\bm{\alpha}}(\bm{x})
\1_{\{x_1 > \dots > x_n > 0 \}} \, .
\end{equation}
\end{proposition}

\begin{proof}
Let us again set $g_{\epsilon}(\bm{a}) := \{ \e^{a_k / \epsilon}\colon k\in\mathcal{I} \} \in \R_{>0}^{\mathcal{I}}$ for any index set $\mathcal{I}$ and array $\bm{a}= \{ a_k\colon k\in\mathcal{I} \} \in \R^{\mathcal{I}}$ and for all $\epsilon>0$.
In definition~\eqref{eq:soWhittakerFn}, we change variables by setting $z_{i,j}\mapsto g_{\epsilon}(z_{i,j})$ for all $1\leq i< 2n$, $1\leq j\leq \ceil{i/2}$ and obtain
\[
\Psi^{\so_{2n+1}}_{\epsilon \bm{\alpha}}(g_{\epsilon}(\bm{x}))
= \int_{\spT{2n}{\R}(\bm{x})}
\prod_{k=1}^{n} \left(\frac{\gtype(g_{\epsilon}(\bm{z}))_{2k-1}}{\gtype(g_{\epsilon}(\bm{z}))_{2k}}\right)^{\epsilon\alpha_k}
\exp\left\{-\spEn(g_{\epsilon}(\bm{z}))\right\}
\prod_{\substack{1\leq i<2n \\ 1\leq j\leq \ceil{i/2}}} \!\!\! \frac{\diff z_{i,j}}{\epsilon} \, .
\]
Similarly to the proof of Proposition~\ref{prop:glWhittakerRescaling}, one can verify that $\gtype(g_{\epsilon}(\bm{z}))_k
= g_{\epsilon}(\type(\bm{z})_k)$ for $1\leq k\leq 2n$ and
\[
\exp\left\{-\spEn(g_{\epsilon}(\bm{z}))\right\}
\xrightarrow{\epsilon\downarrow 0}
\1_{\spGT{2n}{\R}(\bm{x})}(\bm{z})
\]
for a.e.\ $\bm{z}\in \spT{2n}{\R}(\bm{x})$.
Notice in particular that, in the latter limit, the potential associated to an inhomogeneous addend of~\eqref{eq:soEnergy} of the form $1/z_{2k-1,k}$ converts into the ``boundary condition at the wall'' $z_{2k-1,k}\geq 0$.
Since the Lebesgue measure of $\spGT{2n}{\R}(\bm{x})$ vanishes whenever the condition $x_1 > \dots > x_n >0$ fails to hold, we conclude by dominated convergence that
\[
\epsilon^{n^2} 
\Psi_{\epsilon\bm{\alpha}}^{\gl_n} (g_{\epsilon}(\bm{x}))
\xrightarrow{\epsilon\downarrow 0}
\1_{\{x_1 > \dots > x_n >0 \}}
\int_{\spGT{2n}{\R}(\bm{x})}
\prod_{k=1}^n
\e^{ \alpha_k \left(\type(\bm{z})_{2k-1} - \type(\bm{z})_{2k}\right)}
\prod_{\substack{1\leq i<2n \\ 1\leq j\leq \ceil{i/2}}}
\!\! \diff z_{i,j} \, .
\]
The latter integral defines the continuous symplectic Schur function on the right-hand side of~\eqref{eq:soWhittakerRescaling}.
\end{proof}

We have seen that any symplectic Schur function $\sp_{\bm{\lambda}}(\bm{a})$ is invariant under permutations and multiplicative inversion of its variables $a_i$'s.
Analogously, even though not obvious from this definition, any $\so_{2n+1}$-Whittaker function $\Psi^{\so_{2n+1}}_{\bm{\alpha}}(\cdot)$ is invariant under permutations and change of sign
of the parameters $(\alpha_1,\dots,\alpha_n)$.

Let us now mention the interpretation of orthogonal Whittaker functions in terms of the quantum Toda lattice, see~\eqref{eq:quantumToda}.
Since the simple roots of $\so_{2n+1}$ are $\bm{e}_i - \bm{e}_{i+1}$ for $1\leq i\leq n-1$ and $\bm{e}_n$, the quantum Toda hamiltonian associated to the Lie algebra $\so_{2n+1}$ (also called of type $B_n$) takes the form
\[
\mathcal{H}^{\so_{2n+1}} :=
\sum_{i=1}^n \frac{\partial^2}{\partial x_i^2}
-2\sum_{i=1}^{n-1} \e^{x_{i+1} - x_i}
-\e^{-x_n} \, .
\]
Each $\so_{2n+1}$-Whittaker function $\Psi^{\so_{2n+1}}_{\bm{\lambda}}(\e^{x_1},\dots,\e^{x_n})$ is then an eigenfunction of $\mathcal{H}^{\so_{2n+1}}$ with eigenvalue $\langle\bm{\lambda},\bm{\lambda}\rangle = \sum_{i=1}^n\lambda_i^2$.

We finally introduce an integral identity which pairs one $\so_{2n+1}$-Whittaker function with a $\gl_n$-Whittaker function and is associated with $L$-factors of automorphic $L$-functions on $\SO_{2n+1}(\R) \times \GL_n(\R)$.
It is of similar nature as the Bump-Stade identity~\eqref{eq:bumpStade}, and was proven by Ishii and Stade~\cite{ishiiStade13}.
It will also play an important role in our polymer analysis, in particular in section~\ref{sec:contourIntegrals}.
\begin{theorem}[Ishii-Stade identity]
\label{thm:ishiiStade}
Let $\bm{\alpha},\bm{\beta}, \in\C^n$, where $\Re(\beta_i \pm \alpha_j) >0$ for all $i,j$. Then
\begin{equation}
\label{eq:ishiiStade}
\int_{\R_{>0}^n}
\Psi_{\bm{\alpha}}^{\so_{2n+1}}(\bm{x})
\Psi_{-\bm{\beta}}^{\gl_n}(\bm{x})
\prod_{i=1}^n \frac{\diff x_i}{x_i}
= \frac{\prod_{1\leq i,j\leq n}
\Gamma(\beta_i + \alpha_j)
\Gamma(\beta_i - \alpha_j)}
{\prod_{1\leq i<j\leq n} \Gamma(\beta_i+\beta_j)} \, .
\end{equation}
\end{theorem}

Let us point out that the parametrization used for Whittaker functions in \cite{ishiiStade13} is different from ours.
In section~\ref{sec:ishiiStadeWhittaker} of the appendix we will show the relation between such different parametrizations and the equivalence between~\eqref{eq:ishiiStade} and the corresponding integral formula in~\cite{ishiiStade13}.

Since $\so_{2n+1}$-Whittaker functions scale to continuous symplectic Schur functions by Proposition~\ref{prop:soWhittakerRescaling}, it is easy to verify that an appropriate limit of~\eqref{eq:ishiiStade} yields~\eqref{eq:cauchyIdentitySpCont}.
Therefore, we may consider the Ishii-Stade identity as the analog of the symplectic Cauchy identity in the setting of Whittaker functions.
\chapter{Log-gamma polymer models}
\label{ch:polymer}

Recall that a \emph{polymer partition function} is a random variable of the form 
\[
Z
:= \sum_{\pi \in \Pi} \prod_{(i,j)\in\pi} W_{i,j} \, ,
\]
where $\Pi$ is a given set of nearest neighbor directed paths on a finite lattice $\mathcal{I}\subset\Z_{>0}^2$, and  $\bm{W}=\{W_{i,j} \colon (i,j)\in\mathcal{I} \}$ is an array of positive random weights.
As explained in the Introduction, $Z$ is indeed the partition function of the random directed polymer model associated to the given path geometry.

In this chapter, we apply the tools introduced in chapter~\ref{ch:preliminary} (in particular, the geometric $\RSK$ correspondence and Whittaker functions) to study the log-gamma polymer partition function in a few different path geometries.
In section~\ref{sec:WhittakerFormulas} we express the Laplace transforms of the point-to-line, the point-to-half-line and the restricted point-to-half-line log-gamma polymer partition functions in terms of Whittaker functions on both $\so_{2N+1}$ and $\gl_N$.
In section~\ref{sec:contourIntegrals} we obtain contour integral formulas involving gamma functions, in the point-to-line and point-to-half-line cases.

The techniques we are going to use are inspired by the analogous computations of Corwin, O'Connell, Sepp{\"{a}}l{\"{a}}inen and Zygouras~\cite{corwinOConnellSeppalainenZygouras14, oConnellSeppalainenZygouras14} for the point-to-point log-gamma polymer model.
We will then recall them here for convenience of the reader.
Denote by $Z_{n,n}$ the point-to-point polymer partition function associated with the set of directed lattice paths from $(1,1)$ to $(n,n)$ on the lattice $\{1,\dots,n\}^2$.
Proposition~\ref{prop:gRSKproperties}-\ref{prop:gRSKproperties_partitionFn} implies that, if we denote by $\bm{T}=\{T_{i,j}\colon 1\leq i,j\leq n\}$ the image of the weight array $\bm{W}=\{W_{i,j}\colon 1\leq i,j\leq n\}$ under $\gRSK$, the point-to-point partition function $Z_{n,n}$ coincides with the bottom-right entry $T_{n,n}$ of $\bm{T}$.
We then have that
\begin{equation}
\label{eq:pointToPointPartitionFnLaplace}
\E\left[\e^{-r Z_{n,n}}\right] = \int_{\R_{>0}^{n^2}} \e^{-r t_{n,n}} \P(\bm{T}\in \diff\bm{t}) \, .
\end{equation}
Assume now that the weights are independent and inverse-gamma distributed; more precisely, given parameters $\bm{\alpha},\bm{\beta}\in\R_{>0}^n$ such that $\alpha_i - \beta_j >0$ for all $i,j$, assume that $W_{i,j}^{-1} \sim {\rm Gamma}(\alpha_i - \beta_j,1)$ for $1\leq i,j\leq n$.
The joint distribution of $\bm{W}$ is then given by
\[
\begin{split}
&\quad\, \P(\bm{W}\in \diff\bm{w})
= \prod_{i,j=1}^n
w_{i,j}^{-\alpha_i + \beta_j} \e^{-\frac{1}{w_{i,j}}}
\frac{\1_{\R_{>0}}(w_{i,j})}{\Gamma(\alpha_i-\beta_j)}
\frac{\diff w_{i,j}}{w_{i,j}} \\
&= \left[ \prod_{i=1}^n \left(\prod_{j=1}^n w_{i,j}\right)^{-\alpha_i} \right]
\left[ \prod_{j=1}^n \left(\prod_{i=1}^n w_{i,j}\right)^{\beta_j} \right]
\exp\left\{ -\sum_{i,j=1}^n \frac{1}{w_{i,j}} \right\}
\prod_{i,j=1}^n
\frac{\1_{\R_{>0}}(w_{i,j})}{\Gamma(\alpha_i-\beta_j)}
\frac{\diff w_{i,j}}{w_{i,j}}  \, .
\end{split}
\]
Thanks to properties~\ref{prop:gRSKproperties_type}, \ref{prop:gRSKproperties_invWeights}, and~\ref{prop:gRSKproperties_Jacobian} of Proposition~\ref{prop:gRSKproperties}, we deduce from the above a formula for the joint distribution of $\bm{T}$:
\[
\begin{split}
\P(\bm{T} \in \diff \bm{t})
= \, &\prod_{i,j=1}^n \frac{1}{\Gamma(\alpha_i-\beta_j)}
\prod_{k=1}^n \left(\frac{\pi_{n-k}(\bm{t})}{\pi_{n-k+1}(\bm{t})} \right)^{-\alpha_k}
\left(\frac{\pi_{k-n}(\bm{t})}{\pi_{k-1-n}(\bm{t})} \right)^{\beta_k} \\
&\times \exp\left\{ -\frac{1}{t_{1,1}} - \sum_{i,j=1}^n \frac{t_{i-1,j}+t_{i,j-1}}{t_{i,j}} \right\}
\1_{\R_{>0}^{n^2}}(\bm{t})
\prod_{i,j=1}^n \frac{\diff t_{i,j}}{t_{i,j}} \, ,
\end{split}
\]
where $\pi_k(\bm{t})$ denotes the product of the $k$-th diagonal of $\bm{t}$ as in~\eqref{eq:prodDiagonal}, and $t_{i,j}:=0$ by convention if $(i,j)\notin\{1,\dots,n\}^2$.
We now plug this expression into~\eqref{eq:pointToPointPartitionFnLaplace}, integrating over the strictly upper and lower triangular parts of $\bm{t}$ first, and then over the diagonal:
\[
\begin{split}
\E\left[\e^{-r Z_{n,n}}\right]
= \, &\prod_{i,j=1}^n \frac{1}{\Gamma(\alpha_i-\beta_j)}
\int_{\R_{>0}^n} \prod_{i=1}^n \frac{\diff t_{i,i}}{t_{i,i}}
\e^{-r t_{n,n} - 1/t_{1,1}} \\
&\times \int_{\R_{>0}^{n(n-1)/2}}
\prod_{i<j} \frac{\diff t_{i,j}}{t_{i,j}}
\prod_{k=1}^n \left(\frac{\pi_{n-k}(\bm{t})}{\pi_{n-k+1}(\bm{t})} \right)^{-\alpha_k}
\exp\left\{ -\sum_{1<i\leq j} \frac{t_{i-1,j}}{t_{i,j}}
- \sum_{i<j} \frac{t_{i,j-1}}{t_{i,j}} \right\} \\
&\times \int_{\R_{>0}^{n(n-1)/2}}
\prod_{j<i} \frac{\diff t_{i,j}}{t_{i,j}}
\prod_{k=1}^n \left(\frac{\pi_{k-n}(\bm{t})}{\pi_{k-1-n}(\bm{t})} \right)^{\beta_k}
\exp\left\{ -\sum_{j<i} \frac{t_{i-1,j}}{t_{i,j}}
- \sum_{1<j\leq i} \frac{t_{i,j-1}}{t_{i,j}} \right\} \, .
\end{split}
\]
In the latter formula, the second and the third integral turn out to be $\gl_n$-Whittaker functions (see Definition~\ref{def:glWhittakerFn}) with parameters $-\bm{\alpha}$ and $\bm{\beta}$ respectively and common argument $(t_{n,n},\dots,t_{1,1})$.
Setting the latter vector equal to $\bm{x}=(x_1,\dots,x_n)$, we see that the Laplace transform of $Z_{n,n}$ can be expressed in terms of $\gl_n$-Whittaker functions as
\begin{equation}
\label{eq:pointToPointWhittakerFormula}
\E\left[\e^{-r Z_{n,n}}\right]
= \prod_{i,j=1}^n \frac{1}{\Gamma(\alpha_i-\beta_j)}
\int_{\R_{>0}^n} \e^{-rx_1-1/x_n} \Psi^{\gl_n}_{-\bm{\alpha}}(\bm{x}) \Psi^{\gl_n}_{\bm{\beta}}(\bm{x}) 
\prod_{i=1}^n \frac{\diff x_i}{x_i} \, ,
\end{equation}
as established in~\cite{corwinOConnellSeppalainenZygouras14, oConnellSeppalainenZygouras14}.

\begin{remark}
\label{rem:bumpStadeEquiv}
Taking $r\to 0$ in~\eqref{eq:pointToPointWhittakerFormula}, the left-hand side converges to $1$, hence we obtain
\begin{equation}
\label{eq:bumpStadeEquiv}
\int_{\R_{>0}^n} \e^{-1/x_n} \Psi^{\gl_n}_{-\bm{\alpha}}(\bm{x}) \Psi^{\gl_n}_{\bm{\beta}}(\bm{x}) 
\prod_{i=1}^n \frac{\diff x_i}{x_i}
= \prod_{i,j=1}^n \Gamma(\alpha_i-\beta_j) \, .
\end{equation}
Applying the change of variables $x_i \mapsto (rx_{n-i+1})^{-1}$ for $1\leq i\leq n$ and using properties~\eqref{eq:glWhittakerFnScalarMultiplication} and~\eqref{eq:glWhittakerFnSignChange} of $\gl_n$-Whittaker functions, one can show that~\eqref{eq:bumpStadeEquiv} is equivalent to the Bump-Stade identity~\eqref{eq:bumpStade}.
\end{remark}

Formula~\eqref{eq:pointToPointWhittakerFormula} can be turned into a contour integral formula~\cite{oConnellSeppalainenZygouras14}.
To do this, one first observes that the integral on the right-hand side is the $L^2(\R_{>0}^n,\prod_{i=1}^n \diff x_i/x_i)$-inner product\footnote{It can be checked that $f$ and $g$ indeed belong to $L^2(\R_{>0}^n,\prod_{i=1}^n \diff x_i/x_i)$.} of the functions
\[
f(\bm{x}) := \e^{-r x_1} \Psi^{\gl_n}_{\bm{\beta}}(\bm{x}) \, , \qquad\quad
g(\bm{x}) := \e^{-1/x_n} \Psi^{\gl_n}_{-\bm{\alpha}}(\bm{x}) \, .
\]
The aim is then to apply the $\gl_n$-Whittaker-Plancherel theorem~\ref{thm:plancherel}.
The $\gl_n$-Whittaker transforms $\hat{f}(\bm{\lambda})$ and $\hat{g}(\bm{\lambda})$ can be computed explicitly in terms of gamma functions using the Bump-Stade identity (or its equivalent version~\eqref{eq:bumpStadeEquiv} replacing $\bm{\beta}$ with $\bm{\lambda}$, in the case of $\hat{g}(\bm{\lambda})$).
The contour integral formula for the Laplace transform of $Z_{n,n}$ follows~\cite{oConnellSeppalainenZygouras14}:
\begin{equation}
\label{eq:pointToPointContourInt}
\E\left[\e^{-r Z_{n,n}}\right]
= \int_{\i\R^n} r^{-\sum_{i=1}^n (\lambda_i + \beta_i)}
\prod_{i,j=1}^n \frac{\Gamma(\lambda_i+\alpha_j) \Gamma(\lambda_i+\beta_j)}{\Gamma(\alpha_i-\beta_j)}
s_n(\bm{\lambda}) \diff\bm{\lambda} \, ,
\end{equation}
being $s_n(\bm{\lambda}) \diff\bm{\lambda}$ the Skylanin measure~\eqref{eq:sklyaninMeasure}.

\section{Point-to-line polymers and Whittaker functions}
\label{sec:WhittakerFormulas}

In this section we consider the log-gamma polymer in three path geometries: point-to-line, point-to-half-line and point-to-half-line restricted to a half-plane.
We compute the Laplace transforms of the  corresponding partition functions at any even time $2N$ as integrals of $\gl_N$ and $\so_{2N+1}$-Whittaker functions.

In each case, we will first need to compute the joint law of the point-to-point partition functions along a ``fixed time'' line (or half-line). This can be done by using the geometric Robinson-Schensted-Knuth correspondence ($\gRSK$) for polygonal arrays and its properties, discussed in subsection~\ref{subsec:gRSK}.

\subsection{Point-to-line polymer}
\label{subsec:flatPolymerWhittaker}

We define that the \emph{point-to-line polymer partition function} at time $N\in\Z_{>0}$ by
\begin{equation}
\label{eq:flatPartitionFn}
\fZ_{N}
:= \sum_{\pi \in \fPi_{N}} \prod_{(i,j)\in\pi} W_{i,j} \, ,
\end{equation}
where $\fPi_{N}$ is the set of directed paths from $(1,1)$ to the line $\{(m,n)\colon m+n=N+1\}$ (see Figure~\ref{subfig:flatPath}), and  $\bm{W}=\{W_{i,j} \colon (i,j)\in\fI_{N} \}$ is an array of positive random weights on the triangular lattice
\begin{equation}
\label{eq:flatLattice}
\fI_{N}:= \{(i,j)\in\Z_{>0}^2 \colon i+j\leq N+1\} \, .
\end{equation}
Notice that the point-to-line polymer partition function at time $N$ can be written as the sum of all point-to-point partition functions with endpoint on the line $\{m+n = N+1\}$:
\begin{equation}
\label{eq:flatPartitionFn=sumPointToPoint}
\fZ_N
= \sum_{m+n=N+1} Z_{m,n} \, .
\end{equation}

We will show that, when the weights are inverse-gamma distributed with a certain parametrization, the Laplace transform of $\fZ_{2N}$ can be essentially written as an integral of two orthogonal Whittaker functions.
We first provide a sketchy and intuitive argument, which may also work as a quick recap of the proof.
In principle, the Laplace transform of $\fZ_{2N}$ is an integral w.r.t.\ the density of $\bm{W} = \{W_{i,j}\colon (i,j)\in\fI_{2N}\}$.
Now, we know from Proposition~\ref{prop:gRSKproperties}-\ref{prop:gRSKproperties_partitionFn} that each point-to-point partition function $Z_{m,n}$ with $m+n=2N+1$ coincides with $T_{m,n}$, where $\bm{T}$ is the image of $\bm{W}$ under the $\gRSK$ correspondence.
Thanks to~\eqref{eq:flatPartitionFn=sumPointToPoint}, we will then express the Laplace transform of $\fZ_{2N}$ as an integral w.r.t.\ the new variables $\bm{t}=\{t_{i,j} \colon (i,j)\in\fI_{2N}\}$ (i.e.\ w.r.t.\ the joint density of $\bm{T}$, which can be computed using the properties of $\gRSK$, see the proof  of Lemma~\ref{lemma:flatP2PjointLaw}).
The integrand will contain a rational part and an exponential one.
The latter will be a function of the potential $\mathcal{E}(\bm{t})$ defined in~\eqref{eq:energy} and illustrated by the arrows of Figure~\ref{subfig:triangularArray}.
The further change of variables $t_{i,j} \mapsto 1/t_{i,j}$ for all $(i,j)$ will \emph{reverse} all such arrows.
Consider then the two ``half-triangular'' arrays that the triangular array of Figure~\ref{subfig:triangularArray} is divided into by the main diagonal $\{t_{i,i}\colon 1\leq i\leq N\}$, picturing to place a ``wall'' just below the antidiagonal line of outer variables $\{t_{m,n}\colon m+n=2N+1\}$: after reversing the arrows, they will essentially have the same arrow diagram\footnote{Figure~\ref{fig:spGTpattern} also shows arrows connecting the vertical wall with the array: these are missing in Figure~\ref{subfig:triangularArray}, but they will be ``constructed'' \emph{ad hoc} using another exponential function present in the integral, i.e.\ the one that defines the Laplace transform as a functional of $\fZ_{2N}$.
On the other hand, the extra arrow at the top-left corner of Figure~\ref{subfig:triangularArray} can be ignored for the moment; it will give rise to the extra exponential in the final integral~\eqref{eq:flatWhittakerFormula}.} as Figure~\ref{fig:spGTpattern}.
Rearranging the rational part of the integral, we will then recognize two orthogonal Whittaker functions corresponding to these two half-triangular arrays (see Definition~\ref{def:soWhittakerFn}).
The main diagonal of the array $\bm{t}$ will play the role of the common bottom row of the two half-triangular arrays, or equivalently the common argument of the two Whittaker functions: integrating it out essentially yields an integral formula involving two orthogonal Whittaker functions, see~\eqref{eq:flatWhittakerFormula}.

We adopt the following exactly solvable parametrization of the inverse-gamma variables for the point-to-line geometry:
\begin{definition}
\label{def:flatLogGammaMeasure}
Let $N\in\Z_{>0}$, $\bm{\alpha},\bm{\beta}\in\R_{>0}^N$, and $\gamma\in\R_{\geq 0}$.
We define the \emph{$(\bm{\alpha}, \bm{\beta},\gamma)$-log-gamma measure} on the lattice $\fI_{2N}= \{(i,j)\in\Z_{>0}^2 \colon i+j\leq 2N+1\}$ to be the law of a family of independent random variables $\{W_{i,j} \colon (i,j)\in\fI_{2N}\}$ such that:
\begin{equation}
\label{eq:flatLogGammaMeasure}
\frac{1}{W_{i,j}} \sim
\begin{cases}
\mathrm{ Gamma}(\alpha_i + \beta_j + \gamma,1) &1\leq i,j\leq N \, , \\
\mathrm{ Gamma}(\alpha_i + \alpha_{2N-j+1},1) &1\leq i\leq N\, , \,\, N < j\leq 2N-i+1 \, , \\
\mathrm{ Gamma}(\beta_{2N-i+1} + \beta_j,1) &1\leq j\leq N \, , \,\, N < i \leq 2N-j+1 \, .
\end{cases}
\end{equation}
\end{definition}

\begin{remark}
The choice of the parameters in Definition~\ref{def:flatLogGammaMeasure} is tailored so that it fits the link between Whittaker functions and $\gRSK$.
More specifically, this is due to property~\ref{prop:gRSKproperties_type} in Proposition~\ref{prop:gRSKproperties} and the presence of the geometric type in the integral representation of Whittaker functions, cf.\ Definitions~\ref{def:glWhittakerFn} and~\ref{def:soWhittakerFn}.
This will become clear in the proofs of Lemmas~\ref{lemma:flatP2PjointLaw} and~\ref{lemma:hFlatP2PjointLaw} below.
We have also included an extra parameter $\gamma$ in the distribution of the weights $W_{i,j}$ for $1\leq i,j\leq N$.
This might seem rather unnatural, but it will turn out to be useful in the proof of Theorem~\ref{thm:flatContourInt} to obtain contour integral formulas.
More specifically, the Plancherel theorem for $\gl_n$-Whittaker functions can be applied in~\eqref{eq:plancherel1step} thanks to estimation~\eqref{eq:estimationL2}, which in turn relies on the presence of the parameter $\gamma$ in~\eqref{eq:flatWhittakerFormula}.
\end{remark}

In next lemma, which is a modification of ~\cite[Thm~3.5]{nguyenZygouras17} that accommodates the extra parameter $\gamma$ in Definition~\ref{def:flatLogGammaMeasure}, we compute the joint law of all the point-to-point partition functions at a fixed time horizon.
Recall from subsection~\ref{subsec:gRSK} that $\pi_k(\bm{t})$ denotes the product of the $k$-th diagonal of a polygonal array $\bm{t}$, as in~\eqref{eq:prodDiagonal}.
\begin{lemma}
\label{lemma:flatP2PjointLaw} 
For the $(\bm{\alpha},\bm{\beta},\gamma)$-log-gamma polymer, the joint distribution of the point-to-point partition functions at time $2N$ is
\begin{equation}
\label{eq:flatP2PjointLaw}
\P(Z_{m,2N-m+1} \in \diff y_m \colon m=1,\dots,2N)
= \frac{1}{\fG_{\bm{\alpha},\bm{\beta},\gamma}}
\fPhi_{\bm{\alpha},\bm{\beta},\gamma}(\bm{y})
\1_{\R_{>0}^{2N}}(\bm{y}) \prod_{m=1}^{2N} \frac{\diff y_m}{y_m} \, .
\end{equation}
Here, the normalization constant $\fG_{\bm{\alpha},\bm{\beta},\gamma}$ and the function $\fPhi_{\bm{\alpha},\bm{\beta},\gamma}$ are given by
\begin{align}
\label{eq:flatPolymerNormalization}
\fG_{\bm{\alpha},\bm{\beta},\gamma}
:= &\prod_{1\leq i,j\leq N} \Gamma(\alpha_i + \beta_j + \gamma)
\prod_{1\leq i\leq j\leq N} \Gamma(\alpha_i + \alpha_j) \Gamma(\beta_i + \beta_j) \, , \\
\label{eq:flatP2PjointDensity}
\begin{split}
\fPhi_{\bm{\alpha},\bm{\beta},\gamma}(\bm{y})
:= &\! \int_{\fT_{\R_{>0}}(\bm{y})}
\prod_{k=1}^N
\left(\frac{\pi_{2N-2k+1}(\bm{t})^2}{\pi_{2N-2k+2}(\bm{t}) \pi_{2N-2k}(\bm{t})}\right)^{\!-\alpha_k} \!
\left(\frac{\pi_{-2N+2k-1}(\bm{t})^2}{\pi_{-2N+2k-2}(\bm{t}) \pi_{-2N+2k}(\bm{t})}\right)^{\!-\beta_k}
\!\!\!\pi_0(\bm{t})^{-\gamma} \\
&\times
\exp\left\{ -\frac{1}{t_{1,1}} - \sum_{(i,j)\in\fI_{2N}} \frac{t_{i-1,j}+t_{i,j-1}}{t_{i,j}} \right\}
\prod_{i+j\leq 2N} \frac{\diff t_{i,j}}{t_{i,j}} \, ,
\end{split}
\end{align}
using the convention that $t_{i,j}:=0$ when $(i,j)\notin \fI_{2N}$, and denoting by $\fT_{\R_{>0}}(\bm{y})$ the set of all arrays $\bm{t} = \{t_{i,j}\colon (i,j)\in\fI_{2N}\}$ with entries in $\R_{>0}$ such that $(t_{1,2N},t_{2,2N-1},\dots,t_{2N,1}) = (y_1,\dots,y_{2N}) =: \bm{y}$.
\end{lemma}
\begin{proof}
The joint law of the triangular array $\bm{W} = \{ W_{i,j}\colon (i,j)\in\fI_{2N}\}$ for the $(\bm{\alpha},\bm{\beta},\gamma)$-log-gamma measure, according to Definition~\eqref{def:flatLogGammaMeasure}, is given by
\begin{equation}
\label{eq:flatLogGammaJointLaw}
\begin{split}
&\P(\bm{W}\in \diff\bm{w})
= \prod_{i,j = 1}^N \frac{w_{i,j}^{-\alpha_i-\beta_j-\gamma}}{\Gamma(\alpha_i + \beta_j+\gamma)}
\prod_{\substack{1\leq i\leq N \\ N < j \leq 2N-i+1}}
\!\!\!\!
\frac{w_{i,j}^{-\alpha_i - \alpha_{2N-j+1}}}{\Gamma(\alpha_i + \alpha_{2N-j+1})} \\
&\qquad\times
\prod_{\substack{1\leq j\leq N \\ N < i \leq 2N-j+1}}
\!\!\!\!
\frac{w_{i,j}^{-\beta_{2N-i+1} - \beta_j}}{\Gamma(\beta_{2N-i+1} + \beta_j)}
\exp\left\{-\sum_{(i,j)\in\fI_{2N}} \frac{1}{w_{i,j}}\right\}
\prod_{(i,j)\in\fI_{2N}} \!\!
\1_{\R_{>0}}(w_{i,j})
\frac{\diff w_{i,j}}{w_{i,j}} \, .
\end{split}
\end{equation}
We now rewrite the above density in terms of the image array $\bm{t}= \{ t_{i,j}\colon (i,j)\in\fI_{2N} \}$ of $\bm{w}$ under the $\gRSK$ bijection.
Property~\ref{prop:gRSKproperties_type} of Proposition~\ref{prop:gRSKproperties} yields
\[
\prod_{j=1}^{2N-k+1} \!\!\! w_{k,j} = \frac{\pi_{2N-2k+1}(\bm{t})}{\pi_{2N-2k+2}(\bm{t})} \, , \quad\qquad
\prod_{i=1}^{2N-k+1} \!\!\! w_{i,k} = \frac{\pi_{-2N+2k-1}(\bm{t})}{\pi_{-2N+2k-2}(\bm{t})}
\]
for $1\leq k\leq 2N$.
The product of all $w_{i,j}$'s that are raised to the power $-\alpha_k$ in~\eqref{eq:flatLogGammaJointLaw} can be written as
\[
 \prod_{j=1}^{2N-k+1} \!\!\! w_{k,j} \prod_{i=1}^k w_{i,2N-k+1}
= \frac{\pi_{2N-2k+1}(\bm{t})}{\pi_{2N-2k+2}(\bm{t})}
\frac{\pi_{2N-2k+1}(\bm{t})}{\pi_{2N-2k}(\bm{t})}
\, ,
\]
and similarly the product of all $w_{i,j}$'s that are raised to the power $-\beta_k$ can be written as
\[
 \prod_{i=1}^{2N-k+1} \!\!\! w_{i,k} \prod_{j=1}^k w_{2N-k+1,j}
= \frac{\pi_{-2N+2k-1}(\bm{t})}{\pi_{-2N+2k-2}(\bm{t})}
\frac{\pi_{-2N+2k-1}(\bm{t})}{\pi_{-2N+2k}(\bm{t})}
\, .
\]
By formula~\eqref{eq:gRSKprodSubarray}, the product of $w_{i,j}$'s that are raised to power $-\gamma$ is $\prod_{i,j=1}^N w_{i,j} = \pi_0(\bm{t})$.
Using property~\ref{prop:gRSKproperties_invWeights} of Proposition~\ref{prop:gRSKproperties} for dealing with the exponential term in~\eqref{eq:flatLogGammaJointLaw}, and property~\ref{prop:gRSKproperties_Jacobian} regarding the volume preserving property of the differential form, we obtain:
\[
\begin{split}
\P(\bm{T}\in\diff\bm{t})
= \, & \frac{1}{\fG_{\bm{\alpha},\bm{\beta},\gamma}}
\prod_{k=1}^N
\left(\frac{\pi_{2N-2k+1}(\bm{t})^2}{\pi_{2N-2k+2}(\bm{t}) \pi_{2N-2k}(\bm{t})}\right)^{-\alpha_k}
\left(\frac{\pi_{-2N+2k-1}(\bm{t})^2}{\pi_{-2N+2k-2}(\bm{t}) \pi_{-2N+2k}(\bm{t})}\right)^{-\beta_k} \\
&\times \pi_0^{-\gamma}(\bm{t})
\exp\left\{ -\frac{1}{t_{1,1}} - \sum_{(i,j)\in\fI_{2N}} \frac{t_{i-1,j}+t_{i,j-1}}{t_{i,j}} \right\}
\prod_{(i,j)\in\fI_{2N}} \!\!
\1_{\R_{>0}}(t_{i,j})
\frac{\diff t_{i,j}}{t_{i,j}} \, .
\end{split}
\]
Finally, by property~\ref{prop:gRSKproperties_partitionFn} of Proposition~\ref{prop:gRSKproperties}, the point-to-point partition functions at time $2N$ coincide with the $\gRSK$ output variables indexed by outer indices, i.e.\ $Z_{m,n} = T_{m,n}$ for $m+n=2N +1$. 
Therefore, the joint density of $(Z_{1,2N},Z_{2,2N-1},\dots,Z_{2N,1})$ at $\bm{y}\in\R_{>0}^{2N}$ is obtained by integrating the above over $\fT_{\R_{>0}}(\bm{y})$: this yields~\eqref{eq:flatP2PjointLaw}.
\end{proof}

We can now derive the Whittaker integral formula for the Laplace transform of $\fZ_{2N}$.
\begin{theorem}
\label{thm:flatWhittakerFormula}
The Laplace transform of the point-to-line partition function $\fZ_{2N}$ for the $(\bm{\alpha},\bm{\beta},\gamma)$-log-gamma polymer can be written in terms of orthogonal Whittaker functions as
\begin{equation}
\label{eq:flatWhittakerFormula}
\E\left[\e^{- r \fZ_{2N}}\right] 
= \frac{r^{\sum_{k=1}^N (\alpha_k+\beta_k +\gamma)}}
{\fG_{\bm{\alpha},\bm{\beta},\gamma}}
\int_{\R_{>0}^N} 
\e^{-r x_1}
\Psi_{\bm{\alpha}}^{\so_{2N+1}}(\bm{x})
\Psi_{\bm{\beta}}^{\so_{2N+1}}(\bm{x})
\left(\prod_{i=1}^N x_i\right)^{\gamma}
\prod_{i=1}^N \frac{\diff x_i}{x_i}
\end{equation}
for all $r>0$, where $\fG_{\bm{\alpha},\bm{\beta},\gamma}$ is defined by~\eqref{eq:flatPolymerNormalization}.
\end{theorem}

\begin{proof}
Since $\fZ_{2N}=\sum_{m=1}^{2N} Z_{m,2N-m+1}$, Lemma~\ref{lemma:flatP2PjointLaw} implies that our Laplace transform can be written as the following integral over the variables $t_{m,n}$'s for $m+n=2N+1$:
\[
\begin{split}
\E\left[\e^{- r \fZ_{2N}}\right] 
= \frac{1}{\fG_{\bm{\alpha},\bm{\beta},\gamma}}
\int_{\R_{>0}^{2N}} 
\e^{-r \sum_{m=1}^{2N} t_{m,2N-m+1}} \,
\fPhi_{\bm{\alpha},\bm{\beta},\gamma}\left(\left(t_{m,2N-m+1}\right)_{m=1}^{2N}\right)
\prod_{m=1}^{2N} \frac{\diff t_{m,2N-m+1}}{t_{m,2N-m+1}} \, .
\end{split}
\]
We next use~\eqref{eq:flatP2PjointDensity} to express the density function above as an integral over the variables $t_{i,j}$'s for $i+j\leq 2N$:
\[
\begin{split}
\E\left[\e^{- r \fZ_{2N}}\right]
= \, &\frac{1}{\fG_{\bm{\alpha},\bm{\beta},\gamma}} 
 \int_{\R_{>0}^{2N}}
\prod_{m=1}^{2N} \frac{\diff t_{m,2N-m+1}}{t_{m,2N-m+1}}
\exp\left\{-r \sum_{m = 1}^{2N} t_{m,2n-m+1}\right\} \\
&\times \int_{\R_{>0}^{N(2N-1)}}
\prod_{k=1}^n
\left(\frac{\pi_{2N-2k+1}(\bm{t})^2}{\pi_{2N-2k+2}(\bm{t}) \pi_{2N-2k}(\bm{t})}\right)^{-\alpha_k}
\left(\frac{\pi_{-2N+2k-1}(\bm{t})^2}{\pi_{-2N+2k-2}(\bm{t}) \pi_{-2N+2k}(\bm{t}) }\right)^{-\beta_k} \\
&\times \pi_0(\bm{t})^{-\gamma}
\exp\left\{-\frac{1}{t_{1,1}}
- \sum_{i>1,j} \frac{t_{i-1,j}}{t_{i,j}}
- \sum_{j>1,i} \frac{t_{i,j-1}}{t_{i,j}} \right\}
\prod_{i+j\leq 2N} \frac{\diff t_{i,j}}{t_{i,j}}
\, ,
\end{split}
\]
where the implicit range of indices is $(i,j)\in\fI_{2N}$.

We now define the new array of variables $\bm{s} = \{s_{i,j}\colon (i,j)\in\fI_{2N}\}$ by setting 
\begin{equation}
\label{eq:variablesInversion}
t_{i,j} = \frac{1}{r s_{i,j}}
\qquad \text{for all }  (i,j)\in\fI_{2N} \, .
\end{equation}
Visually, this change of variables reverses the arrows in Figure~\ref{subfig:triangularArray}. 
Recall that the summands $t_{i-1,j}/t_{i,j}$ and $t_{i,j-1} / t_{i,j}$ of the functional $\mathcal{E}(\bm{t})$ (as defined in~\eqref{eq:energy}) are represented as the arrows $t_{i-1,j}\to t_{i,j}$ and $t_{i,j-1}\to t_{i,j}$ respectively in the figure.
The change of variables~\eqref{eq:variablesInversion} then transforms these ratios into $s_{i,j}/s_{i-1,j}$ and $s_{i,j} / s_{i,j-1}$, which in turn can be represented, by the same convention, as the arrows $s_{i,j} \to s_{i-1,j}$ and $s_{i,j} \to s_{i,j-1}$.

Recalling~\eqref{eq:prodDiagonal} we obtain
\[
\begin{split}
&\left(\frac{\pi_{2N-2k+1}(\bm{t})^2}{\pi_{2N-2k+2}(\bm{t}) \pi_{2N-2k}(\bm{t}) }\right)^{-\alpha_k}
= \left(\frac{\prod_{j-i=2N-2k+1} t_{i,j}^2}
{\prod_{j-i=2N-2k+2} t_{i,j}
\prod_{j-i=2N-2k} t_{i,j}} \right)^{-\alpha_k} \\
= \,& r^{\alpha_k} \left(\frac{\prod_{j-i=2N-2k+1} s_{i,j}^2}
{\prod_{j-i=2N-2k+2} s_{i,j} \prod_{j-i=2N-2k} s_{i,j}} \right)^{\alpha_k}
= r^{\alpha_k} \left(\frac{\pi_{2N-2k+1}(\bm{s})^2}{\pi_{2N-2k+2}(\bm{s}) \pi_{2N-2k}(\bm{s})} \right)^{\alpha_k} \, .
\end{split}
\]
Similarly, we have that
\[
\left(\frac{\pi_{-2N+2k-1}(\bm{t})^2}{\pi_{-2N+2k-2}(\bm{t}) \pi_{-2N+2k}(\bm{t})}\right)^{-\beta_k}
= r^{\beta_k} \left(\frac{\pi_{-2N+2k-1}(\bm{s})^2}{\pi_{-2N+2k-2}(\bm{s}) \pi_{-2N+2k}(\bm{s})} \right)^{\beta_k}
\]
and $\pi_0(\bm{t})^{-\gamma}= r^{N\gamma} \pi_0(\bm{s})^{\gamma}$.
Moreover, the change of variables~\eqref{eq:variablesInversion} preserves the volume:
\[
\frac{\diff t_{i,j}}{t_{i,j}}
= \frac{\diff s_{i,j}}{s_{i,j}}
\qquad \text{for all }  (i,j)\in\fI_{2N} \, .
\]
We thus obtain
\[
\begin{split}
\E\left[e^{- r \fZ_{2N}}\right] 
= \, &\frac{r^{\sum_{k=1}^N(\alpha_k + \beta_k + \gamma)}  }{\fG_{\bm{\alpha},\bm{\beta},\gamma}}
\int_{\R_{>0}^{2N}}
\prod_{m=1}^{2N} \frac{\diff s_{m,2N-m+1}}{s_{m,2N-m+1}}
\exp\left\{- \sum_{m=1}^{2N} \frac{1}{s_{m,2N-m+1}}\right\} \\
&\times \int_{\R_{>0}^{N(2N-1)}}
\prod_{k=1}^N
\left(\frac{\pi_{2N-2k+1}(\bm{s})^2}{\pi_{2N-2k+2}(\bm{s}) \pi_{2N-2k}(\bm{s})}\right)^{\alpha_k}
\left(\frac{\pi_{-2N+2k-1}(\bm{s})^2}{\pi_{-2N+2k-2}(\bm{s}) \pi_{-2N+2k}(\bm{s})}\right)^{\beta_k} \\
&\times 
\pi_0(\bm{s})^{\gamma}
\exp\left\{ -r s_{1,1}
- \sum_{i>1,j} \frac{s_{i,j}}{s_{i-1,j}}
- \sum_{j>1,i} \frac{s_{i,j}}{s_{i,j-1}} \right\}
\prod_{i+j\leq 2N} \frac{\diff s_{i,j}}{s_{i,j}}
\, .
\end{split}
\]
We now change the order in which variables are integrated in the above expression: we first integrate over the two ``half-triangular'' arrays $\{s_{i,j}\colon i<j\}$ and $\{s_{i,j}\colon j<i\}$ into which the whole triangular shape (see Figure~\ref{subfig:triangularArray}) is divided by the main diagonal; next, we integrate over the diagonal variables $s_{1,1},\dots,s_{N,N}$.
This yields:
\[
\begin{split}
\E\left[\e^{- r \fZ_{2N}}\right]
= \,&\frac{r^{\sum_{k=1}^N(\alpha_k+\beta_k+\gamma)}}{\fG_{\bm{\alpha},\bm{\beta},\gamma}}
\int_{\R_{>0}^N}
\prod_{i=1}^N \frac{\diff s_{i,i}}{s_{i,i}}
\left(\prod_{i=1}^N s_{i,i}\right)^{\gamma}
\e^{-r s_{1,1} } \\
&\times \int_{\R_{>0}^{N^2}}
\prod_{i<j} \frac{\diff s_{i,j}}{s_{i,j}}
\prod_{k=1}^N
\left(\frac{\pi_{2N-2k+1}(\bm{s})^2}{\pi_{2N-2k+2}(\bm{s}) \pi_{2N-2k}(\bm{s})}\right)^{\alpha_k} \\
&\qquad\qquad\qquad \exp\left\{
- \sum_{m=1}^N \frac{1}{s_{m,2N-m+1}}
- \sum_{1<i\leq j} \frac{s_{i,j}}{s_{i-1,j}}
- \sum_{i<j} \frac{s_{i,j}}{s_{i,j-1}} \right\} \\
&\times \int_{\R_{>0}^{N^2}}
\prod_{j<i} \frac{\diff s_{i,j}}{s_{i,j}}
\prod_{k=1}^N
\left(\frac{\pi_{-2N+2k-1}(\bm{s})^2}{\pi_{-2N+2k-2}(\bm{s}) \pi_{-2N+2k}(\bm{s})}\right)^{\beta_k} \\
&\qquad\qquad\qquad \exp\left\{
- \sum_{m=N+1}^{2N} \frac{1}{s_{m,2N-m+1}}
- \sum_{j<i} \frac{s_{i,j}}{s_{i-1,j}}
- \sum_{1<j\leq i} \frac{s_{i,j}}{s_{i,j-1}} \right\} \, .
\end{split}
\]
Comparing with Definition~\ref{def:soWhittakerFn}, we identify the second and the third integral in the above formula as $\so_{2N+1}$-Whittaker functions with parameters $\bm{\alpha}$ and $\bm{\beta}$ respectively, both evaluated in $(s_{1,1},\dots,s_{N,N})$. Setting the latter vector equal to $\bm{x}$ concludes the proof of~\eqref{eq:flatWhittakerFormula}.
\end{proof}

\begin{remark}
Take the limit $r\to 0$ in~\eqref{eq:flatWhittakerFormula}: since $\E\left[\e^{-r \fZ_{2N}} \right] \to 1$ and $r^{\sum_{k=1}^N (\alpha_k+\beta_k+\gamma)}$ vanishes, we observe that the integral
\[
\int_{\R_{>0}^N}
\Psi_{\bm{\alpha}}^{\so_{2N+1}}(\bm{x})
\Psi_{\bm{\beta}}^{\so_{2N+1}}(\bm{x})
\left(\prod_{i=1}^N x_i\right)^{\gamma}
\prod_{i=1}^N \frac{\diff x_i}{x_i}
\]
diverges for $\bm{\alpha},\bm{\beta}\in\R_{>0}^N$ and $\gamma \geq 0$.
In particular, we cannot deduce any useful identity from~\eqref{eq:flatWhittakerFormula} in the $r\to 0$ limit, contrary to the point-to-point case (see Remark~\ref{rem:bumpStadeEquiv}).
\end{remark}

\subsection{Point-to-half-line polymer}
\label{subsec:hFlatPolymerWhittaker}

We define the \emph{point-to-half-line polymer partition function} at time $N\in\Z_{>0}$ by
\begin{equation}
\label{eq:hFlatPartitionFn}
\hZ_{N}
:= \sum_{\pi \in \hPi_{N}} \prod_{(i,j)\in\pi} W_{i,j} \, ,
\end{equation}
where $\hPi_{N}$ is the set of directed paths starting from $(1,1)$ and ending on the half-line $\{(m,n)\colon m+n=N+1, \,\, m\leq n \}$ (see Figure~\ref{subfig:hFlatPath}), and $\bm{W} = \{ W_{i,j} \colon (i,j)\in\hI_N \}$ is an array of positive random weights on the trapezoidal lattice
\begin{equation}
\label{eq:hFlatLattice}
\hI_{N}:= \left\{(i,j)\in\Z_{>0}^2 \colon i+j\leq N+1, \, i\leq \frac{N+1}{2} \right\} \, .
\end{equation}
The point-to-half-line polymer partition function at time $N$ turns out to be the sum of all point-to-point partition functions with endpoint on the half-line $\{m+n = N+1, \,\, m\leq n\}$:
\begin{equation}
\label{eq:hFlatPartitionFn=sumPointToPoint}
\hZ_N
= \sum_{\substack{m+n=N+1 \\ m\leq n}} Z_{m,n} \, .
\end{equation}

We will show that, when the weights are inverse-gamma distributed with a certain parametrization, the Laplace transform of $\hZ_{2N}$ can be essentially written as an integral of the product of two Whittaker functions associated to $\so_{2N+1}$ and $\gl_N$ respectively.
The idea behind the proof is the same as for the point-to-line partition function (see beginning of subsection~\ref{subsec:flatPolymerWhittaker}).
The main difference here is that the Laplace transform of $\hZ_{2N}$ is an integral over a \emph{trapezoidal} array of variables as in Figure~\ref{subfig:hFlatPath}, instead of a triangular array as in Figure~\ref{subfig:flatPath}.
Such a trapezoidal array will be decomposed into a ``half-triangular'' part (above the main diagonal) generating the $\so_{2N+1}$-Whittaker function, and a triangular part (below the diagonal) generating the $\gl_N$-Whittaker function.

Let us now fix the parametrization of the inverse-gamma weights.
It is clear from Definition~\eqref{eq:flatLogGammaMeasure} that the $(\bm{\alpha},\bm{\beta},\gamma)$-log-gamma measure, when restricted to the trapezoidal lattice $\hI_{2N}$, coincides with the $(\bm{\alpha},\bm{\beta}+\gamma,0)$-log-gamma measure.
For this reason, in this subsection, we will assume without loss of generality that $\gamma=0$, thus assigning the following measure to the weights:
\begin{definition}
\label{def:hFlatLogGammaMeasure}
Let $N\in\Z_{>0}$ and $\bm{\alpha},\bm{\beta}\in\R_{>0}^N$.
We define the \emph{$(\bm{\alpha}, \bm{\beta})$-log-gamma measure} on the lattice $\hI_{2N} = \{(i,j)\in\Z_{>0}^2 \colon i+j\leq 2N+1, \, i\leq N\}$ to be the law of a family of independent random variables $\{W_{i,j} \colon (i,j)\in\hI_{2N}\}$ such that:
\begin{equation}
\label{eq:hFlatLogGammaMeasure}
\frac{1}{W_{i,j}} \sim
\begin{cases}
\mathrm{ Gamma}(\alpha_i + \beta_j,1) &1\leq i,j\leq N \, , \\
\mathrm{ Gamma}(\alpha_i + \alpha_{2N-j+1},1) &1\leq i\leq N\, , \,\, N < j\leq 2N-i+1 \, .
\end{cases}
\end{equation}
\end{definition}

We first give an expression for the joint law of the point-to-point partition functions on a half-line.
Again, the proof will be based on Proposition~\ref{prop:gRSKproperties}.

\begin{lemma}
\label{lemma:hFlatP2PjointLaw} 
For the $(\bm{\alpha},\bm{\beta})$-log-gamma polymer, the joint distribution of the point-to-point partition functions on the half-line $\{i+j=2N+1,\, i\leq j\}$ is
\begin{equation}
\label{eq:hFlatP2PjointLaw}
\P(Z_{m,2N-m+1} \in \diff y_m \colon m=1,\dots,N)
= \frac{1}{\hG_{\bm{\alpha},\bm{\beta}}}
\hPhi_{\bm{\alpha},\bm{\beta}}(\bm{y})
\1_{\R_{>0}^{N}}(\bm{y}) \prod_{m=1}^{N} \frac{\diff y_m}{y_m} \, .
\end{equation}
Here, the normalization constant $\hG_{\bm{\alpha},\bm{\beta}}$ and the function $\hPhi_{\bm{\alpha},\bm{\beta}}$ are given by
\begin{align}
\label{eq:hFlatPolymerNormalization}
\hG_{\bm{\alpha},\bm{\beta}}
:= &\prod_{1\leq i,j\leq N} \!\! \Gamma(\alpha_i + \beta_j)
\prod_{1\leq i\leq j\leq N} \!\! \Gamma(\alpha_i + \alpha_j) \, , \\
\label{eq:hFlatP2PjointDensity}
\begin{split}
\hPhi_{\bm{\alpha},\bm{\beta}}(\bm{y})
:= & \int_{\hT_{\R_{>0}}(\bm{y})}
\prod_{k=1}^N
\left(\frac{\pi_{2N-2k+1}(\bm{t})^2}{\pi_{2N-2k+2}(\bm{t}) \pi_{2N-2k}(\bm{t})}\right)^{-\alpha_k}
\left(\frac{\pi_{k-N}(\bm{t})}{\pi_{k-1-N}(\bm{t})}\right)^{-\beta_k} \\
&\times \exp\left\{ -\frac{1}{t_{1,1}} - \sum_{(i,j)\in\hI_{2N}} \frac{t_{i-1,j}+t_{i,j-1}}{t_{i,j}} \right\}
\prod_{\substack{i+j\leq 2N \\ i\leq N}} \frac{\diff t_{i,j}}{t_{i,j}} \, ,
\end{split}
\end{align}
using the convention that $t_{i,j}:=0$ when $(i,j)\notin \hI_{2N}$, and denoting by $\hT_{\R_{>0}}(\bm{y})$ the set of all arrays $\bm{t} = \{t_{i,j}\colon (i,j)\in\hI_{2N}\}$ with entries in $\R_{>0}$ such that $(t_{1,2N},t_{2,2N-1},\dots,t_{N,N+1}) = (y_1,\dots,y_{N}) =: \bm{y}$.
\end{lemma}

\begin{proof}
The joint law of the trapezoidal array $\bm{W} = \{ W_{i,j}\colon (i,j)\in\hI_{2N}\}$ for the $(\bm{\alpha},\bm{\beta})$-log-gamma measure, according to Definition~\eqref{def:hFlatLogGammaMeasure}, is given by
\begin{equation}
\label{eq:hFlatLogGammaJointLaw}
\begin{split}
\P(\bm{W}\in \diff\bm{w})
= &\prod_{i,j = 1}^N \frac{w_{i,j}^{-\alpha_i-\beta_j}}{\Gamma(\alpha_i + \beta_j)}
\prod_{\substack{1\leq i\leq N \\ N < j \leq 2N-i+1}}
\!\!\!\!
\frac{w_{i,j}^{-\alpha_i - \alpha_{2N-j+1}}}{\Gamma(\alpha_i + \alpha_{2N-j+1})} \\
&\times \exp\left\{-\sum_{(i,j)\in\hI_{2N}} \frac{1}{w_{i,j}}\right\}
\prod_{(i,j)\in\hI_{2N}} \!\!\!
\1_{\R_{>0}}(w_{i,j})
\frac{\diff w_{i,j}}{w_{i,j}} \, .
\end{split}
\end{equation}
We now rewrite the above density in terms of the image array $\bm{t}= \{ t_{i,j} \colon (i,j)\in\hI_{2N} \}$ of $\bm{w}$ under the $\gRSK$ bijection.
The powers of $w_{i,j}$'s are sorted out by noting that
\[
\prod_{j=1}^{2N-k+1} \!\!\! w_{k,j} = \frac{\pi_{2N-2k+1}(\bm{t})}{\pi_{2N-2k+2}(\bm{t})} \, ,
\qquad
\prod_{i=1}^{N} w_{i,k} = \frac{\pi_{k-N}(\bm{t})}{\pi_{k-1-N}(\bm{t})} \, ,
\qquad
\prod_{i=1}^{k} w_{i,2N-k+1} = \frac{\pi_{2N-2k+1}(\bm{t})}{\pi_{2N-2k}(\bm{t})}
\]
for $1\leq k\leq N$, thanks to property~\ref{prop:gRSKproperties_type} of Proposition~\ref{prop:gRSKproperties}.
Using property~\ref{prop:gRSKproperties_invWeights} for dealing with the exponential term in~\eqref{eq:hFlatLogGammaJointLaw}, and property~\ref{prop:gRSKproperties_Jacobian} for the differential form, we obtain:
\[
\begin{split}
\P(\bm{T}\in\diff\bm{t})
= \, &\frac{1}{\hG_{\bm{\alpha},\bm{\beta}}}
\prod_{k=1}^N
\left(\frac{\pi_{2N-2k+1}(\bm{t})^2}{\pi_{2N-2k+2}(\bm{t}) \pi_{2N-2k}(\bm{t})}\right)^{-\alpha_k}
\left(\frac{\pi_{k-N}(\bm{t})}{\pi_{k-1-N}(\bm{t})}\right)^{-\beta_k} \\
&\times \exp\left\{ -\frac{1}{t_{1,1}} - \sum_{(i,j)\in\hI_{2N}} \frac{t_{i-1,j}+t_{i,j-1}}{t_{i,j}} \right\}
\prod_{(i,j)\in\hI_{2N}} \!\!\!
\1_{\R_{>0}}(t_{i,j})
\frac{\diff t_{i,j}}{t_{i,j}} \, .
\end{split}
\]
Finally, by property~\ref{prop:gRSKproperties_partitionFn} of Proposition~\ref{prop:gRSKproperties}, the joint density of $(Z_{1,2N},Z_{2,2N-1},\dots,Z_{N,N+1})$ at $\bm{y}\in\R_{>0}^{N}$ is obtained by integrating the above over $\hT_{\R_{>0}}(\bm{y})$: this yields~\eqref{eq:hFlatP2PjointLaw}.
\end{proof}

An adaptation of the proof of the analogous Theorem~\ref{thm:flatWhittakerFormula} for the point-to-line case leads to the announced Whittaker functions formula for the point-to-half-line partition function.
\begin{theorem}
\label{thm:hFlatWhittakerFormula}
The Laplace transform of the point-to-half-line partition function $\hZ_{2N}$ for the $(\bm{\alpha},\bm{\beta})$-log-gamma polymer can be written in terms of Whittaker functions as
\begin{equation}
\label{eq:hFlatWhittakerFormula}
\E\left[\e^{- r \hZ_{2N}}\right] 
= \frac{r^{\sum_{k=1}^{N} (\alpha_k+\beta_k)}}
{\hG_{\bm{\alpha},\bm{\beta}}}
\int_{\R_{>0}^N} \
\e^{-r x_1}
\Psi_{\bm{\alpha}}^{\so_{2N+1}}(\bm{x})
\Psi_{\bm{\beta}}^{\gl_N}(\bm{x})
\prod_{i=1}^N \frac{\diff x_i}{x_i}
\end{equation}
for all $r>0$, where $\hG_{\bm{\alpha},\bm{\beta}}$ is given by~\eqref{eq:hFlatPolymerNormalization}.
\end{theorem}

\begin{proof}
Given that $\hZ_{2N}=\sum_{m=1}^N Z_{m,2n-m+1}$, we have via Lemma~\ref{lemma:hFlatP2PjointLaw} that
\[
\begin{split}
\E\left[\e^{- r \hZ_{2N}}\right] 
= \frac{1}{\hG_{\bm{\alpha},\bm{\beta}}}
\int_{\R_{>0}^N} 
\e^{-r \sum_{m=1}^{N} t_{m,2N-m+1}} \,
\hPhi_{\bm{\alpha},\bm{\beta}}\left(\left(t_{m,2N-m+1}\right)_{m=1}^{N}\right)
\prod_{m=1}^{N} \frac{\diff t_{m,2N-m+1}}{t_{m,2N-m+1}} \, .
\end{split}
\]
Using definition~\eqref{eq:hFlatP2PjointDensity} of $\hPhi_{\bm{\alpha},\bm{\beta}}$ and performing the same change of variables $t_{i,j} = (r s_{i,j})^{-1}$ for all $(i,j)\in\hI_{2N}$ as in the proof of Theorem~\ref{thm:flatWhittakerFormula}, we obtain
\[
\begin{split}
\E\left[e^{- r \hZ_{2N}}\right] 
= \, &\frac{r^{\sum_{k=1}^N(\alpha_k + \beta_k)}  }{\hG_{\bm{\alpha},\bm{\beta}}}
\int_{\R_{>0}^N}
\prod_{m=1}^{N} \frac{\diff s_{m,2N-m+1}}{s_{m,2N-m+1}}
\exp\left\{- \sum_{m=1}^N \frac{1}{s_{m,2N-m+1}}\right\} \\
&\times \int_{\R_{>0}^{(3N^2-N)/2}}
\prod_{k=1}^N
\left(\frac{\pi_{2N-2k+1}(\bm{s})^2}{\pi_{2N-2k+2}(\bm{s}) \pi_{2N-2k}(\bm{s})}\right)^{\alpha_k}
\left(\frac{\pi_{k-N}(\bm{s})}{\pi_{k-1-N}(\bm{s})}\right)^{\beta_k} \\
&\times 
\exp\left\{ -r s_{1,1}
- \sum_{i>1,j} \frac{s_{i,j}}{s_{i-1,j}}
- \sum_{j>1,i} \frac{s_{i,j}}{s_{i,j-1}} \right\}
\prod_{\substack{i+j\leq 2N \\ i\leq N}} \frac{\diff s_{i,j}}{s_{i,j}}
\, .
\end{split}
\]
We now change the order in which variables are integrated in the above expression: we first integrate over the ``half-triangular'' array $\{s_{i,j}\colon i<j\}$ and the triangular array $\{s_{i,j}\colon j<i\}$ into which the whole trapezoidal shape of Figure~\ref{subfig:hFlatPath} is divided by the main diagonal; next, we integrate w.r.t.\ the diagonal variables $s_{1,1},\dots,s_{N,N}$.
This yields:
\[
\begin{split}
\E\left[\e^{- r \hZ_{2N}}\right]
= \,& \frac{r^{\sum_{k=1}^N(\alpha_k+\beta_k)}}{\hG_{\bm{\alpha},\bm{\beta}}}
\int_{\R_{>0}^N}
\prod_{i=1}^N \frac{\diff s_{i,i}}{s_{i,i}}
\, \e^{-r s_{1,1} } \\
&\times \int_{\R_{>0}^{N^2}}
\prod_{i<j} \frac{\diff s_{i,j}}{s_{i,j}}
\prod_{k=1}^N
\left(\frac{\pi_{2N-2k+1}(\bm{s})^2}{\pi_{2N-2k+2}(\bm{s}) \pi_{2N-2k}(\bm{s})}\right)^{\alpha_k} \\
&\qquad\qquad\qquad \exp\left\{
- \sum_{m=1}^N \frac{1}{s_{m,2N-m+1}}
- \sum_{1<i\leq j} \frac{s_{i,j}}{s_{i-1,j}}
- \sum_{i<j} \frac{s_{i,j}}{s_{i,j-1}} \right\} \\
&\times \int_{\R_{>0}^{N(N-1)/2}}
\prod_{j<i} \frac{\diff s_{i,j}}{s_{i,j}}
\prod_{k=1}^N
\left(\frac{\pi_{k-N}(\bm{s})}{\pi_{k-1-N}(\bm{s})}\right)^{\beta_k}
\exp\left\{
- \sum_{j<i} \frac{s_{i,j}}{s_{i-1,j}}
- \sum_{1<j\leq i} \frac{s_{i,j}}{s_{i,j-1}} \right\} \, .
\end{split}
\]
Comparing with Definition~\ref{def:soWhittakerFn} and~\ref{def:glWhittakerFn}, we identify the second integral as an $\so_{2N+1}$-Whittaker function with parameter $\bm{\alpha}$, and the third integral as a $\gl_N$-Whittaker function with parameter $\bm{\beta}$, both evaluated in $(s_{1,1},\dots,s_{N,N})$.
Setting the latter vector equal to $\bm{x}$ concludes the proof of~\eqref{eq:hFlatWhittakerFormula}.
\end{proof}

\begin{remark}
Take the limit $r\to 0$ in~\eqref{eq:hFlatWhittakerFormula}: since $\E\left[\e^{-r \hZ_{2N}} \right] \to 1$ and $r^{\sum_{k=1}^N (\alpha_k+\beta_k)} \to 0$, we observe that the integral
\[
\int_{\R_{>0}^N} 
\Psi_{\bm{\alpha}}^{\so_{2N+1}}(\bm{x})
\Psi_{\bm{\beta}}^{\gl_N}(\bm{x})
\prod_{i=1}^N \frac{\diff x_i}{x_i}
\]
diverges for $\bm{\alpha},\bm{\beta}\in\R_{>0}^N$.
This does not contradict Ishii-Stade identity, as in~\eqref{eq:ishiiStade} the parameters of the $\gl_N$-Whittaker function are required to have negative real part.
\end{remark}

\subsection{Restricted and symmetric point-to-line polymers}
\label{subsec:rFlatPolymerWhittaker}

We now study the point-to-line polymer restricted to stay in a half-plane (in short \emph{restricted point-to-line polymer}), i.e.\ not allowed to go below the main diagonal.
The \emph{restricted point-to-line polymer partition function} at time $N\in\Z_{>0}$ is defined as
\begin{equation}
\label{eq:rFlatPartitionFn}
\rZ_{N}
:= \sum_{\pi \in \rPi_{N}} \prod_{(i,j)\in\pi} W_{i,j} \, ,
\end{equation}
where $\rPi_{N}$ is the set of directed paths starting from $(1,1)$ and ending on the half-line $\{(m,n)\colon m+n=N+1, \,\, m\leq n \}$ such that $i\leq j$ for all site $(i,j)$ of the path (see Figure~\ref{subfig:rFlatPath}), and $\bm{W} = \{ W_{i,j} \colon (i,j)\in\rI_n \}$ is an array of positive random weights on the triangular lattice
\begin{equation}
\label{eq:rFlatLattice}
\rI_{N}:= \{(i,j)\in\Z_{>0}^2 \colon i+j\leq N+1, \, i\leq j \} \, .
\end{equation}
For convenience's sake, let us also define the \emph{restricted point-to-point polymer partition function}:
\begin{equation}
\label{eq:rP2PPartitionFn}
\rZ_{m,n}
:= \sum_{\pi \in \rPi_{m,n}} \prod_{(i,j)\in\pi} W_{i,j}
\qquad \text{for } m\leq n
\, ,
\end{equation}
where $\rPi_{m,n}$ is the set of directed paths from $(1,1)$ to $(m,n)$ restricted to stay in the half-plane $\{(i,j)\colon i\leq j\}$.
It then turns out that the restricted point-to-line partition function is the sum of the restricted point-to-point partition functions with endpoint on the half-line $\{m+n=N+1, \,\, m\leq n \}$:
\begin{equation}
\label{eq:rFlatPartitionFn=sumPointToPoint}
\rZ_N
= \sum_{\substack{m+n=N+1 \\ m\leq n}} \rZ_{m,n} \, .
\end{equation}

We will show that, when the weights are inverse-gamma distributed with a certain parametrization, the Laplace transform of $\rZ_{2N}$ can be essentially written as an integral of an $\so_{2N+1}$-Whittaker function.
Since the lattice $\rI_{2N}$ is \emph{not} the index set of a Young diagram, we cannot apply the $\gRSK$ correspondence, as defined in subsection~\ref{subsec:gRSK}, directly to such an array of weights.
We will be working instead with a symmetric array, i.e.\ $\bm{W}=\{W_{i,j} \colon (i,j)\in\fI_{2N} \}$ satisfying $W_{i,j}=W_{j,i}$ for all $(i,j)$, as we will see that the restricted and the symmetric polymers are closely connected.
Indeed, one can easily convince oneself that for each time a given restricted polymer path touches the main diagonal $\{i=j\}$ (including the starting point $(1,1)$), the symmetric point-to-line polymer partition function counts the weight of that path twice\footnote{This exact statement holds for the restricted polymer at \emph{even} time $2N$ only.
At an odd time $2N+1$, it does not apply to the last element $(N+1,N+1)$ of the main diagonal.}.
Therefore, the \emph{symmetric point-to-line polymer partition function} can be expressed as a sum over restricted paths as:
\begin{equation}
\label{eq:sym-resFlatPartitionFns}
\sZ_{2N}
= \sum_{\pi\in\rPi_{2N}} 2^{\# \{i\colon (i,i)\in\pi\} } \prod_{(i,j)\in \pi} W_{i,j}
= \sum_{\pi\in\rPi_{2N}} \prod_{(i,j)\in \pi} (1+\delta_{i,j}) W_{i,j}
\, .
\end{equation}
A similar reasoning leads to the conclusion that each \emph{symmetric point-to-point polymer partition function} can be expressed as a sum over restricted paths as:
\begin{equation}
\label{eq:sym-resP2PPartitionFns}
\sZ_{m,n}
= \frac{1}{2} \sum_{\pi\in\rPi_{m,n}} \prod_{(i,j)\in \pi} (1+\delta_{i,j}) W_{i,j}
\qquad \text{for } m\leq n
\, .
\end{equation}
It follows from~\eqref{eq:sym-resFlatPartitionFns} and \eqref{eq:sym-resP2PPartitionFns} that the partition functions of the symmetric and restricted polymers are essentially the same, provided that the weights of the restricted polymer are doubled on the diagonal.
Let us clarify this in the following remark.
\begin{remark}
\label{rem:sym-resPolymer}
Consider a restricted polymer model on the lattice $\rI_{2N}$, and a symmetric polymer with the same weights above the diagonal ($i< j$) and \emph{halved} weights on the diagonal ($i=j$). Then:
\begin{enumerate}
\item
the partition functions of the restricted and symmetric point-to-line polymers coincide, i.e.\ $\rZ_{2N} = \sZ_{2N}$;
\item
the partition functions of the restricted and symmetric point-to-point polymers coincide \emph{up to a factor $2$}, i.e.\ $\rZ_{m,n} = 2 \sZ_{m,n}$ for $m\leq n$, $m+n=2N+1$.
\end{enumerate}
\end{remark}

We now see the exactly solvable distributions on restricted/symmetric arrays (one can be deduced from the other via Remark~\eqref{rem:sym-resPolymer}) that link to $\so_{2N+1}$-Whittaker functions.
\begin{definition}
\label{def:rFlatLogGammaMeasure}
Let $N\in\Z_{>0}$, $\bm{\alpha}\in\R_{>0}^N$ and $\gamma\in\R_{\geq 0}$.
We define the \emph{restricted $(\bm{\alpha}, \gamma)$-log-gamma measure} on the lattice $\rI_{2N}:= \{(i,j)\in\Z_{>0}^2 \colon i+j\leq 2N+1, \, i\leq j\}$ to be the law of a family of independent random variables $\{W_{i,j} \colon (i,j)\in\rI_{2N}\}$ such that
\begin{equation}
\label{eq:rFlatLogGammaMeasure}
\frac{1}{W_{i,j}} \sim
\begin{cases}
\mathrm{ Gamma}(\alpha_i + \gamma,1) &1\leq i=j \leq N \, , \\
\mathrm{ Gamma}(\alpha_i + \alpha_j + 2\gamma,1) &1\leq i<j\leq N \, , \\
\mathrm{ Gamma}(\alpha_i + \alpha_{2N-j+1},1) &1\leq i\leq N \, , \,\, N < j\leq 2N-i+1 \, .
\end{cases}
\end{equation}
We analogously define the \emph{symmetric $(\bm{\alpha},\gamma)$-log-gamma measure} on the lattice $\fI_{2N}:= \{(i,j)\in\Z_{>0}^2 \colon i+j\leq 2N+1 \}$ to be the law of a symmetric array $\{W_{i,j} \colon (i,j)\in\fI_{2N}\}$ such that the entries on and above the diagonal are independent and distributed as in~\eqref{eq:rFlatLogGammaMeasure}, except for the fact that the diagonal weights are halved, i.e.\ $W_{i,i}^{-1} \sim \mathrm{ Gamma}(\alpha_i + \gamma,1/2)$.
\end{definition}

In next lemma we obtain the joint law of the point-to-point restricted polymer partition functions at a fixed time.
The proof is based on the analysis of the corresponding symmetric polymer, and relies on the properties of $\gRSK$ acting on symmetric arrays - see Proposition~\ref{prop:symmetricgRSK}.

\begin{lemma}
\label{lemma:rFlatP2PjointLaw}
For the restricted $(\bm{\alpha},\gamma)$-log-gamma polymer, the joint distribution of the point-to-point partition functions at time $2N$ is
\begin{equation}
\label{eq:rFlatP2PjointLaw}
\P(\rZ_{m,2N-m+1} \in \diff y_m \colon m=1,\dots,N)
= \frac{1}{\rG_{\bm{\alpha},\gamma}}
\rPhi_{\bm{\alpha},\bm{\beta}}(\bm{y})
\1_{\R_{>0}^{N}}(\bm{y})
\prod_{m=1}^{N} \frac{\diff y_m}{y_m} \, .
\end{equation}
Here, the normalization constant $\rG_{\bm{\alpha},\gamma}$ and the function $\rPhi_{\bm{\alpha},\gamma}$ are given by
\begin{align}
\label{eq:rFlatPolymerNormalization}
\rG_{\bm{\alpha},\gamma}
:= &\prod_{i=1}^N \Gamma(\alpha_i + \gamma)
\prod_{1\leq i < j\leq N} \!\! \Gamma(\alpha_i + \alpha_j +2\gamma)
\prod_{1\leq i \leq j\leq N} \!\!
\Gamma(\alpha_i + \alpha_j) \, , \\
\label{eq:rFlatP2PjointDensity}
\begin{split}
\rPhi_{\bm{\alpha},\gamma}(\bm{y})
:= &\int_{\rT_{\R_{>0}}(\bm{y})}
\prod_{k=1}^N
\left(\frac{\pi_{2N-2k+1}(\bm{t})^2}{\pi_{2N-2k+2}(\bm{t}) \pi_{2N-2k}(\bm{t})}\right)^{-\alpha_k}
\pi_0(\bm{t})^{-\gamma} \\
&\times \exp\left\{-\frac{1}{t_{1,1}} -\sum_{(i,j)\in\rI_{2N}} \frac{t_{i-1,j} + t_{i,j-1}}{t_{i,j}} \right\}
\prod_{\substack{i+j\leq 2N \\ i\leq j}} \frac{\diff t_{i,j}}{t_{i,j}}
\, ,
\end{split}
\end{align}
using the convention that $t_{i,j}:=0$ when $(i,j)\notin \rI_{2N}$, and denoting by $\rT_{\R_{>0}}(\bm{y})$ the set of all arrays $\{ t_{i,j} \colon (i,j)\in\rI_{2N}\}$ with entries in $\R_{>0}$ such that $(t_{1,2N},t_{2,2N-1},\dots,t_{N,N+1}) = (y_1,\dots,y_{N}) =: \bm{y}$.
\end{lemma}

\begin{proof}
We work with the symmetric polymer.
If $\bm{W}= \{W_{i,j} \colon (i,j)\in\fI_{2N}\}$ is distributed according to the symmetric $(\bm{\alpha},\gamma)$-log-gamma measure (see Definition~\ref{def:rFlatLogGammaMeasure}), the joint law of its upper entries is
\begin{equation}
\label{eq:sFlatLogGammaJointLaw}
\begin{split}
&\P(W_{i,j}\in \diff w_{i,j} \colon i\leq j)
= \prod_{i=1}^N \frac{w_{i,i}^{-\alpha_i - \gamma}}
{2^{\alpha_i+\gamma} \Gamma(\alpha_i+\gamma)}
\prod_{1\leq i< j\leq N} \frac{w_{i,j}^{-\alpha_i-\alpha_j-2\gamma}}{\Gamma(\alpha_i + \alpha_j+2\gamma)} \\
&\qquad \times \prod_{\substack{1\leq i\leq N \\ N < j \leq 2N-i+1}} \!\!
\frac{w_{i,j}^{-\alpha_i - \alpha_{2N-j+1}}}{\Gamma(\alpha_i + \alpha_{2N-j+1})} 
 \exp\left\{-\sum_{i=1}^N \frac{1}{2w_{i,i}}
-\sum_{i<j} \frac{1}{w_{i,j}}\right\}
\prod_{i\leq j}
\1_{\R_{>0}}(w_{i,j})
\frac{\diff w_{i,j}}{w_{i,j}} \, .
\end{split}
\end{equation}
Let $\bm{t}$ be the image of $\bm{w}$ under gRSK, which is also symmetric by Proposition~\ref{prop:symmetricgRSK}.
Property~\ref{prop:gRSKproperties_type} of Proposition~\ref{prop:gRSKproperties} yields
\[
\prod_{j=1}^{2N-k+1} \!\! w_{k,j} = \frac{\pi_{2N-2k+1}(\bm{t})}{\pi_{2N-2k+2}(\bm{t})} \, , \quad\qquad
\prod_{i=1}^{k} w_{i,2N-k+1} = \frac{\pi_{2N-2k+1}(\bm{t})}{\pi_{2N-2k}(\bm{t})}
\]
for $1\leq k\leq N$.
Using this and the symmetry of $\bm{w}$, the product of all $w_{i,j}$'s raised to the $-\alpha_k$ in~\eqref{eq:sFlatLogGammaJointLaw} can be written as
\[
\prod_{j=k}^{2N-k+1} \!\! w_{k,j}
\prod_{i=1}^{k-1} w_{i,k}
\prod_{i=1}^k w_{i,2N-k+1}
= \prod_{j=1}^{2N-k+1} \!\! w_{k,j} \prod_{i=1}^k w_{i,2N-k+1}
= \frac{\pi_{2N-2k+1}(\bm{t})}{\pi_{2N-2k+2}(\bm{t})}
\frac{\pi_{2N-2k+1}(\bm{t})}{\pi_{2N-2k}(\bm{t})}
\, .
\]
By formula~\eqref{eq:gRSKprodSubarray}, the product of $w_{i,j}$'s raised to the power of $-\gamma$ in~\eqref{eq:sFlatLogGammaJointLaw} is 
\[
\prod_{i=1}^N w_{i,i} \prod_{1\leq i<j\leq N} \!\! w_{i,j}^2
=\prod_{i,j=1}^N w_{i,j} = \pi_0(\bm{t}) \, .
\]
For dealing with the exponential term in~\eqref{eq:sFlatLogGammaJointLaw}, we use the symmetry of $\bm{w}$ and $\bm{t}$ and property~\ref{prop:gRSKproperties_invWeights} of Proposition~\ref{prop:gRSKproperties} to see that
\[
\sum_{i=1}^N \frac{1}{2w_{i,i}}
+\sum_{i<j} \frac{1}{w_{i,j}}
= \frac{1}{2} \sum_{(i,j)\in\fI_{2N}} \frac{1}{w_{i,j}}
= \frac{1}{2} \mathcal{E}(\bm{t})
= \frac{1}{2t_{1,1}} +\sum_{1<i\leq j} \frac{t_{i-1,j}}{t_{i,j}} + \sum_{i<j} \frac{t_{i,j-1}}{t_{i,j}} \, .
\]
Using the volume preserving property of the symmetric gRSK (see Proposition~\ref{prop:symmetricgRSK}), we thus obtain:
\[
\begin{split}
\P(T_{i,j}\in \diff t_{i,j} \colon i\leq j)
= \, &\frac{2^{-N\gamma-\sum_{k=1}^N \alpha_k}}{\rG_{\bm{\alpha},\gamma}}
\prod_{k=1}^N
\left(\frac{\pi_{2N-2k+1}(\bm{t})^2}{\pi_{2N-2k+2}(\bm{t}) \pi_{2N-2k}(\bm{t})}\right)^{-\alpha_k}
\pi_0(\bm{t})^{-\gamma} \\
&\times 
\exp\left\{-\frac{1}{2t_{1,1}} -\sum_{1<i\leq j} \frac{t_{i-1,j}}{t_{i,j}} - \sum_{i<j} \frac{t_{i,j-1}}{t_{i,j}} \right\}
\prod_{i\leq j}
\1_{\R_{>0}}(t_{i,j})
\frac{\diff t_{i,j}}{t_{i,j}} \, .
\end{split}
\]
The change of variables $t_{i,j} \mapsto t_{i,j}/2$ for all $i\leq j$ cancels out the power of $2$ and replaces the term $1/(2t_{1,1})$ with $1/t_{1,1}$.
Rewriting the sums inside the exponential above as a single sum over $\rI_{2N}$ (with the convention that $t_{i,j}:=0$ when $(i,j)\notin\rI_{2N}$), we obtain:
\[
\begin{split}
\P(2 T_{i,j}\in \diff t_{i,j} \colon i\leq j)
= \, &\frac{1}{\rG_{\bm{\alpha},\gamma}}
\prod_{k=1}^N
\left(\frac{\pi_{2N-2k+1}(\bm{t})^2}{\pi_{2N-2k+2}(\bm{t}) \pi_{2N-2k}(\bm{t})}\right)^{-\alpha_k}
\pi_0(\bm{t})^{-\gamma} \\
&\times 
\exp\left\{-\frac{1}{t_{1,1}} -\sum_{(i,j)\in\rI_{2N}} \frac{t_{i-1,j} + t_{i,j-1}}{t_{i,j}} \right\}
\prod_{i\leq j}
\1_{\R_{>0}}(t_{i,j})
\frac{\diff t_{i,j}}{t_{i,j}} \, .
\end{split}
\]
By property~\ref{prop:gRSKproperties_partitionFn} of Proposition~\ref{prop:gRSKproperties}, the joint law of $(2\sZ_{1,2N}, 2\sZ_{2,2N-1},\dots,2\sZ_{N,N+1})$ at $\bm{y}\in\R_{>0}^N$ is obtained by integrating the resulting expression over $\rT_{\R_{>0}}(\bm{y})$.
Since $2\sZ_{m,n} = \rZ_{m,n}$ for $m\leq n$ by Remark~\ref{rem:sym-resPolymer}, our claim~\eqref{eq:rFlatP2PjointLaw} follows.
\end{proof}

Lemma~\ref{lemma:rFlatP2PjointLaw} allows us to obtain our Whittaker integral formula for the Laplace transform of $\rZ_{2N}$ (or equivalently of $\sZ_{2N}$, as they coincide by Remark~\ref{rem:sym-resPolymer}), which is stated in next theorem.
The proof is omitted as it follows the same steps as in Theorems~\ref{thm:flatWhittakerFormula} and~\ref{thm:hFlatWhittakerFormula}.
\begin{theorem}
\label{thm:rFlatWhittakerFormula}
The Laplace transform of the point-to-line partition function for the restricted $(\bm{\alpha},\gamma)$-log-gamma polymer can be written in terms of orthogonal Whittaker functions as
\begin{equation}
\label{eq:rFlatWhittakerFormula}
\E\left[\e^{- r \rZ_{2N}}\right]
= \frac{r^{\sum_{k=1}^N (\alpha_k +\gamma)}}
{\rG_{\bm{\alpha},\gamma}}
\int_{\R_{>0}^N} 
\left(\prod_{i=1}^N x_i\right)^{\gamma}
\e^{-r x_1}
\Psi_{\bm{\alpha}}^{\so_{2N+1}}(\bm{x})
\prod_{i=1}^N \frac{\diff x_i}{x_i}
\end{equation}
for all $r>0$, where $\rG_{\bm{\alpha},\gamma}$ is defined by~\eqref{eq:rFlatPolymerNormalization}.
\end{theorem}

\section{Point-to-line polymers and contour integrals}
\label{sec:contourIntegrals}

Formulas~\eqref{eq:flatWhittakerFormula} and~\eqref{eq:hFlatWhittakerFormula} express the Laplace transforms of the point-to-line and the point-to-half-line log-gamma polymer partitions functions, respectively, as integrals of Whittaker functions.
In this section we rewrite such integrals as contour integrals of gamma functions.

The integrals appearing in formulas~\eqref{eq:hFlatWhittakerFormula} and~\eqref{eq:flatWhittakerFormula} are analogous to the Bump-Stade identity~\eqref{eq:bumpStade}, where either one or both $\gl_N$-Whittaker functions are replaced with the corresponding orthogonal ones.
However, a closed formula in terms of products and ratios of gamma functions for our integrals does not appear in the literature.
We can still turn our integrals into contour integrals of gamma functions using the $\gl_N$-Whittaker-Plancherel theorem~\ref{thm:plancherel}, combined with the Bump-Stade~\eqref{eq:bumpStade} and  the Ishii-Stade~\eqref{eq:ishiiStade} identities.
The key tool is the following lemma.

\begin{lemma}
\label{lemma:WhittakerTransforms}
The $\gl_n$-Whittaker isometry between the spaces $L^2(\R_{>0}^n, \prod_{i=1}^n \diff x_i / x_i)$ and $L^2_{\sym}(\i\R^n, s_n(\bm{\lambda}) \diff\bm{\lambda})$ defined in Theorem~\ref{thm:plancherel} maps
\begin{enumerate}
\Item
\label{lemma:WhittakerTransforms_f}
\[
f(\bm{x}) := \e^{-r x_1} \Psi^{\gl_n}_{\bm{\alpha}}
\quad\longmapsto\quad
\hat{f}(\bm{\lambda}) := r^{-\sum_{i=1}^n (\lambda_i + \alpha_i)}
\prod_{1\leq i,j\leq n} \!\! \Gamma(\lambda_i + \alpha_j)
\]
for all $r>0$ and $\bm{\alpha}\in\C^n$ such that $\Re(\alpha_j) >0$ for all $j$;
\Item
\label{lemma:WhittakerTransforms_g}
\[
g(\bm{x}) := \left(\prod_{i=1}^n x_i \right)^{-s} \Psi^{\so_{2n+1}}_{\bm{\alpha}}(\bm{x})
\quad\longmapsto\quad
\hat{g}(\bm{\lambda})
:= \frac{\prod_{1\leq i,j\leq n} \Gamma(s-\lambda_i+\alpha_j) \Gamma(s-\lambda_i-\alpha_j)}{\prod_{1\leq i<j\leq n} \Gamma(2s-\lambda_i - \lambda_j)}
\]
for all $s\in\C$ and $\bm{\alpha}\in\C^n$ such that $\Re(s \pm \alpha_j) > 0$ for all $j$.
\end{enumerate}
\end{lemma}

\begin{proof}
\begin{enumerate}
\item
Assuming that $f$ is square-integrable, the Bump-Stade identity~\eqref{eq:bumpStade} implies that $\hat{f}$ is indeed the $\gl_n$-Whittaker transform of $f$.
To prove that $f$ belongs to the space $L^2(\R_{>0}^n, \prod_{i=1}^n \diff x_i / x_i)$, we will show instead the equivalent statement that $\hat{f}$ is in $L^2_{\sym}(\i\R^n, s_n(\bm{\lambda}) \diff\bm{\lambda})$.
It is clear that $\hat{f}$ is a symmetric function, and has no poles, thanks to the assumption that each $\alpha_j$ has positive real part.
Recalling the Stirling approximation of the gamma function
\begin{equation}
\label{eq:StirlingApprox}
\abs{\Gamma(x+iy)}
\sim
\sqrt{2\pi} \abs{y}^{x-\frac{1}{2}} \e^{-\frac{\pi}{2} \abs{y}}
\qquad
\text{as } \abs{y} \to \infty \, ,
\end{equation}
we can compute the asymptotics of $\abs{\hat{f}(\bm{\lambda})}^2 s_n(\bm{\lambda})$ as $\abs{\lambda_i} \to\infty$ for $1\leq i\leq n$ and $\abs{\lambda_i - \lambda_j} \to\infty$ for all $i< j$ (which is when $s_n(\bm{\lambda})$ has the worst diverging behavior).
Denoting by the symbol $\sim$ asymptotic behavior up to multiplicative constants and powers, we have that
\[
\begin{split}
\abs{\hat{f}(\bm{\lambda})}^2 s_n(\bm{\lambda})
&= \frac{r^{-2\sum_{i} \Re(\alpha_i)}\prod_{i,j} \abs{\Gamma(\lambda_i + \alpha_j)}^2}{(2\pi)^n n! \prod_{i\neq j} \abs{\Gamma(\lambda_i-\lambda_j)}}
\sim \frac{\prod_{i,j}
\e^{-\pi\abs{\lambda_i}}}
{\prod_{i< j} \e^{-\pi\abs{\lambda_i-\lambda_j}}} \\
&= \exp\left\{-\pi n \sum_{i} \abs{\lambda_i} + \pi \sum_{i<j} \abs{\lambda_i - \lambda_j} \right\}
\leq \exp\left\{- \pi \sum_i \abs{\lambda_i} \right\} \, ,
\end{split}
\]
where for the last inequality we have used the following rough estimate:
\begin{equation}
\label{eq:roughEstimate}
\sum_{i<j} \abs{\lambda_i \pm \lambda_j}
\leq \sum_{i<j} \left(\abs{\lambda_i} + \abs{\lambda_j}\right)
= (n-1) \sum_i \abs{\lambda_i} \, .
\end{equation}
This proves that $\abs{\hat{f}(\bm{\lambda})}^2 s_n(\bm{\lambda})$ is integrable on $\i\R^n$.
\item
Assuming the integrability properties, the fact that $\hat{g}$ is indeed the $\gl_n$-Whittaker transform of $g$ follows from property~\eqref{eq:glWhittakerFnTranslation} and Ishii-Stade identity~\eqref{eq:ishiiStade}.
Therefore, we can just reduce ourselves to prove that $\hat{g}$ belongs to $L^2_{\sym}(\i\R^n, s_n(\bm{\lambda}) \diff\bm{\lambda})$.
Again, $\hat{g}$ is a symmetric function, and has no poles, thanks to the assumption that $\Re(s \pm \alpha_j) > 0$ for all $j$.
Using~\eqref{eq:StirlingApprox} and~\eqref{eq:roughEstimate}, we compute the asymptotics (up to multiplicative constants and powers) of $\abs{\hat{g}(\bm{\lambda})}^2 s_n(\bm{\lambda})$ as $\abs{\lambda_i} \to\infty$ for all $i$ and $\abs{\lambda_i \pm \lambda_j} \to\infty$ for all $i< j$:
\[
\begin{split}
&\quad \abs{\hat{g}(\bm{\lambda})}^2 s_n(\bm{\lambda}) \\
&= \frac{\prod_{i,j}
\abs{\Gamma(s-\lambda_i+\alpha_j)
\Gamma(s-\lambda_i-\alpha_j)}^2}
{(2\pi)^n n! \prod_{i<j}
\abs{\Gamma(2s-\lambda_i - \lambda_j)}^2
\prod_{i\neq j}
\abs{\Gamma(\lambda_i - \lambda_j)}}
\sim \frac{\prod_{i,j}
\e^{-2\pi\abs{\lambda_i}}}
{\prod_{i<j}
\e^{-\pi\abs{\lambda_i + \lambda_j}}
\e^{-\pi\abs{\lambda_i-\lambda_j} }} \\
&= \exp\left\{-2n\pi \sum_{i} \abs{\lambda_i} + \pi \sum_{i<j} \abs{\lambda_i + \lambda_j} + \pi \sum_{i<j} \abs{\lambda_i - \lambda_j} \right\}
\leq \exp\left\{-2 \pi \sum_i \abs{\lambda_i} \right\} \, ,
\end{split}
\]
which proves the integrability of $\abs{\hat{g}(\bm{\lambda})}^2 s_n(\bm{\lambda})$ on $\i \R^n$.
\qedhere
\end{enumerate}
\end{proof}

As a consequence of this lemma, we will prove the contour integral formulas for the point-to-line and point-to-half-line log-gamma polymer partition functions, in subsection~\ref{subsec:flatPolymerWhittaker} and~\ref{subsec:hFlatPolymerWhittaker} respectively.

\subsection{Point-to-line polymer}
\label{subsec:flatPolymerContour}

In the following theorem we derive a contour integral formula for the Laplace transform of the point-to-line log-gamma polymer partition function.
The proof consists of applying the Whittaker-Plancherel theorem~\ref{thm:plancherel}, in a two-step procedure, to the integral of Whittaker functions in~\eqref{eq:flatWhittakerFormula}.

\begin{theorem}
\label{thm:flatContourInt}
The Laplace transform of the point-to-line partition function $Z_{2N}$ for the
$(\bm{\alpha},\bm{\beta},\gamma)$-log-gamma polymer is given by
\begin{equation}
\label{eq:flatContourInt}
\begin{split}
\E\left[\e^{- r \fZ_{2N}}\right] 
= &\frac{1}
{\fG_{\bm{\alpha},\bm{\beta},\gamma}}
\int_{(\epsilon + \i\R)^N} 
s_N(\bm{\rho}) \diff \bm{\rho}
\int_{(\delta + \i\R)^N}
s_N(\bm{\lambda}) \diff \bm{\lambda} \,\,
r^{-\sum_{k=1}^N (\lambda_k +\rho_k -\alpha_k -\beta_k +\gamma)}  \\
&\times \frac{\prod_{1\leq i,j\leq N}
\Gamma(\lambda_i + \rho_j + \gamma)
\Gamma(\lambda_i + \alpha_j)
\Gamma(\lambda_i - \alpha_j)
\Gamma(\rho_i + \beta_j)
\Gamma(\rho_i - \beta_j)}
{\prod_{1\leq i<j\leq N}
\Gamma(\lambda_i+\lambda_j)
\Gamma(\rho_i + \rho_j) }
\end{split}
\end{equation}
for all $r>0$, where $\fG_{\bm{\alpha},\bm{\beta},\gamma}$ is the constant defined in~\eqref{eq:flatPolymerNormalization}, $s_N(\bm{\lambda}) \diff \bm{\lambda}$ is the Sklyanin measure as in~\eqref{eq:sklyaninMeasure}, and $\delta, \epsilon$ are chosen so that $\delta > \alpha_j$ and $\epsilon > \beta_j$ for all $j$.
Furthermore, the multiple contour integral in~\eqref{eq:flatContourInt} is absolutely convergent.
\end{theorem}

\begin{proof}
The integral appearing in formula~\eqref{eq:flatWhittakerFormula} can be written as
\begin{equation}
\label{eq:plancherel1step}
\int_{\R_{>0}^N}
\e^{-r x_1}
\Psi_{\bm{\alpha}}^{\so_{2N+1}}(\bm{x})
\Psi_{\bm{\beta}}^{\so_{2N+1}}(\bm{x})
\left(\prod_{i=1}^N x_i\right)^{\gamma}
\prod_{i=1}^N \frac{\diff x_i}{x_i}
= \int_{\R_{>0}^N}
f(\bm{x}) g(\bm{x})
\prod_{i=1}^N \frac{\diff x_i}{x_i} \, ,
\end{equation}
where
\[
f(\bm{x}) :=
\left(\prod_{i=1}^N x_i\right)^{\gamma+\epsilon}
\e^{-r x_1}
\Psi_{\bm{\alpha}}^{\so_{2N+1}}(\bm{x}) \, ,
\qquad\quad
g(\bm{x}) :=
\left(\prod_{i=1}^N x_i\right)^{-\epsilon}
\Psi_{\bm{\beta}}^{\so_{2N+1}}(\bm{x}) \, .
\]
By Lemma~\ref{lemma:WhittakerTransforms}-\ref{lemma:WhittakerTransforms_g}, since $\epsilon > \beta_j > 0$ for all $j$, $g$ belongs to $L^2(\R_{>0}^N, \prod_{i=1}^N \diff x_i/x_i)$ and satisfies
\[
\overline{\hat{g}(\bm{\rho})}
= \frac{\prod_{1\leq i,j\leq N} \Gamma(\epsilon+\rho_i+\beta_j) \Gamma(\epsilon+\rho_i-\beta_j)}{\prod_{1\leq i<j\leq N} \Gamma(2\epsilon +\rho_i + \rho_j)}
\]
for all $\bm{\rho}\in\i\R^N$ (recall that $\overline{\Gamma(z)} = \Gamma(\overline{z})$).
On the other hand, applying Theorem~\ref{thm:flatWhittakerFormula} in the case where $\bm{\alpha}=\bm{\beta}$, $\gamma$ is replaced with $2(\gamma+\epsilon)$ and $r$ is replaced with $2r$, we obtain that 
\begin{equation}
\label{eq:estimationL2}
\int_{\R_{>0}^N} \abs{f(\bm{x})}^2 \prod_{i=1}^N \frac{\diff x_i}{x_i}
= \frac{\fG_{\bm{\alpha},\bm{\alpha},2(\gamma+\epsilon)}}
{(2r)^{2\sum_{k=1}^N (\alpha_k +\gamma + \epsilon)}}
\E\left[\e^{-2r \tilde{Z}_{2N}}\right]
< \infty \, ,
\end{equation}
where $\tilde{Z}_{2N}$ is the point-to-line partition function of the $(\bm{\alpha},\bm{\alpha},2(\gamma+\epsilon))$-log-gamma polymer.
This proves that $f$ also belongs to $L^2(\R_{>0}^N, \prod_{i=1}^N \diff x_i/x_i)$, so we can apply the $\gl_N$-Whittaker-Plancherel theorem to~\eqref{eq:plancherel1step} and obtain:
\begin{equation}
\begin{split}
\label{eq:plancherel2step}
&\int_{\R_{>0}^N} 
\left(\prod_{i=1}^N x_i\right)^{\gamma}
\e^{-r x_1}
\Psi_{\bm{\alpha}}^{\so_{2N+1}}(\bm{x})
\Psi_{\bm{\beta}}^{\so_{2N+1}}(\bm{x})
\prod_{i=1}^N \frac{\diff x_i}{x_i} \\
= &\int_{(\epsilon +\i\R)^N}
\hat{f}(\bm{\rho}-\epsilon)
\frac{\prod_{1\leq i,j\leq N}
\Gamma(\rho_i+\beta_j)
\Gamma(\rho_i-\beta_j)}
{\prod_{1\leq i<j\leq N}
\Gamma(\rho_i + \rho_j)}
s_N(\bm{\rho}) \diff \bm{\rho}
\, ,
\end{split}
\end{equation}
after the change of variables $\bm{\rho} \mapsto \bm{\rho}-\epsilon$ (notice that $s_N(\bm{\rho} - \epsilon) = s_N(\bm{\rho})$). To compute $\hat{f}(\bm{\rho}-\epsilon)$, we first notice that by property~\eqref{eq:glWhittakerFnTranslation}
\[
\begin{split}
\hat{f}(\bm{\rho}-\epsilon)
&= \int_{\R_{>0}^N}
\left(\prod_{i=1}^N x_i\right)^{\gamma+\epsilon}
\e^{-r x_1}
\Psi_{\bm{\alpha}}^{\so_{2N+1}}(\bm{x})
\Psi_{\bm{\rho}-\epsilon}^{\gl_N}(\bm{x})
\prod_{i=1}^N \frac{\diff x_i}{x_i} \\
&= \int_{\R_{>0}^N}
\left[ \e^{-r x_1}
\Psi_{\bm{\rho}+\gamma+\delta}^{\gl_N}(\bm{x}) \right]
\left[
\left(\prod_{i=1}^N x_i\right)^{-\delta}
\Psi_{\bm{\alpha}}^{\so_{2N+1}}(\bm{x}) \right]
\prod_{i=1}^N \frac{\diff x_i}{x_i} \, .
\end{split}
\]
By Lemma~\ref{lemma:WhittakerTransforms}, since $\gamma\geq 0$, $\delta > \alpha_j>0$ and $\Re(\rho_j)=\epsilon$ for all $j$, the two functions in the square brackets belong to $L^2(\R_{>0}^N, \prod_{i=1}^N \diff x_i/x_i)$, with $\gl_N$-Whittaker transforms given by the same lemma.
Applying the Whittaker-Plancherel theorem again, we then obtain
\[
\hat{f}(\bm{\rho}-\epsilon)
= \int_{(\delta+\i\R)^N}
r^{-\sum_{i=1}^N(\lambda_i + \rho_i + \gamma)} \frac{
\prod_{1\leq i,j\leq N}
\Gamma(\lambda_i + \rho_j + \gamma) \Gamma(\lambda_i + \alpha_j)
\Gamma(\lambda_i - \alpha_j)}
{\prod_{1\leq i<j\leq N}
\Gamma(\lambda_i + \lambda_j)}
s_N(\bm{\lambda}) \diff \bm{\lambda} \, ,
\]
after the change of variables $\bm{\lambda} \mapsto \bm{\lambda}-\delta$.
Plugging the latter formula into~\eqref{eq:plancherel2step} and combining with~\eqref{eq:flatWhittakerFormula} concludes the proof of~\eqref{eq:flatContourInt}.

Finally, we are going to show that the integral in \eqref{eq:flatContourInt} 
is absolutely convergent, so that the order of integration with respect to $\bm{\lambda}$ and $\bm{\rho}$ does not matter.
Note first that the integrand has no poles thanks to the choice of $\delta$ and $\epsilon$.
It is then sufficient to check the integrability as $\abs{L_i},\abs{R_i}\to\infty$ for all $i$ and $\abs{L_i \pm L_j}, \abs{R_i \pm R_j} \to\infty$ for all $i< j$, where $\bm{L}:=\Im(\bm{\lambda})$ and $\bm{R}:=\Im(\bm{\rho})$.
Using the asymptotics of the gamma function~\eqref{eq:StirlingApprox}, we may reduce ourselves to check that
\[
\frac{\prod_{i,j}
\e^{-\frac{\pi}{2} \abs{L_i +R_j}}
\e^{-\pi \abs{L_i}}
\e^{-\pi \abs{R_i}}}
{\prod_{i<j}
\e^{-\frac{\pi}{2} \abs{L_i+L_j}}
\e^{-\frac{\pi}{2} \abs{R_i+R_j}} }
\prod_{i \neq j}
\e^{\frac{\pi}{2} \abs{L_i - L_j}}
\e^{\frac{\pi}{2} \abs{R_i-R_j}}
\]
is integrable for $(\bm{L},\bm{R})\in\R^{2N}$.
The latter function equals $\e^{\pi/2}$ raised to
\begin{equation}
\label{eq:absConvergenceFlatContourIntegral}
-\sum_{i,j}
\left(\abs{L_i + R_j} + 2\abs{L_i} + 2\abs{R_i} \right)
+ \sum_{i<j}
\left(
\abs{L_i + L_j}
+ \abs{R_i + R_j} \right)
+\sum_{i \neq j}
\left(\abs{L_i - L_j}
+ \abs{R_i - R_j} \right) \, .
\end{equation} 
At this stage, since the above expression is symmetric both w.r.t.\ the variables $L_i$'s and w.r.t.\ the variables $R_j$'s, we may assume that
\[
L_1\geq L_2 \geq \dots \geq L_N 
\quad\qquad\text{and}\quad\qquad
R_1 \leq R_2 \leq \dots \leq R_N \, .
\]
This will then allow the bound
\[
\begin{split}
&\sum_{i \neq j} \left(\abs{L_i - L_j}
+ \abs{R_i - R_j} \right)
= 2 \sum_{i<j} (L_i - L_j + R_j - R_i )
= 2 \sum_{i<j} \abs{(L_i+R_j)
- (R_i+L_j)} \\
= &\sum_{i,j} \abs{(L_i+R_j)
- (R_i+L_j)}
\leq \sum_{i,j} \abs{L_i + R_j}
+ \sum_{i,j} \abs{R_i + L_j}
= 2 \sum_{i,j} \abs{L_i + R_j} \, .
\end{split}
\]
Using the latter estimate and the one given in~\eqref{eq:roughEstimate}, we obtain
\[
\begin{split}
\eqref{eq:absConvergenceFlatContourIntegral}
&\leq
-\sum_{i,j} \abs{L_i + R_j}
-2N \sum_i \left(\abs{L_i}
+\abs{R_i} \right)
+(N-1)\sum_i \left(\abs{L_i} + \abs{R_i} \right)
+2\sum_{i,j} \abs{L_i + R_j} \\
&= (-2N+N-1)\sum_i \left(\abs{L_i} + \abs{R_i} \right)
+ \sum_{i,j} \abs{L_i+R_j} \\
&\leq (-N-1)\sum_i \left(\abs{L_i} + \abs{R_i} \right)
+ N\sum_{i} \abs{L_i}
+ N\sum_{j} \abs{R_j} \\
&= -
\sum_i \left(\abs{L_i} + \abs{R_i}\right) \, ,
\end{split}
\]
hence $\e^{\pi/2}$ raised to~\eqref{eq:absConvergenceFlatContourIntegral} is integrable for $(\bm{L},\bm{R})\in\R^{2N}$ as desired.
\end{proof}

\subsection{Point-to-half-line polymer}
\label{subsec:hFlatPolymerContour}

The proof of the contour integral formula for the Laplace transform of the point-to-half-line partition function is simpler, as it requires to apply the Whittaker-Plancherel theorem only once.

\begin{theorem}
\label{thm:hFlatContourInt}
The Laplace transform of the point-to-half-line partition function $\hZ_{2N}$ for the $(\bm{\alpha},\bm{\beta})$-log-gamma polymer is given by
\begin{equation}
\label{eq:hFlatContourInt}
\begin{split}
\E\left[\e^{- r \hZ_{2N}}\right] 
= \, &\frac{1}
{\hG_{\bm{\alpha},\bm{\beta}}}
\int_{(\delta + \i\R)^N}
s_N(\bm{\lambda}) \diff \bm{\lambda}
\,\,
r^{-\sum_{k=1}^N (\lambda_k-\alpha_k)} \\
&\times \frac{\prod_{1\leq i,j\leq N}
\Gamma(\lambda_i + \alpha_j)
\Gamma(\lambda_i - \alpha_j)
\Gamma(\lambda_i + \beta_j)}
{\prod_{1\leq i<j\leq N}
\Gamma(\lambda_i+\lambda_j) }
\end{split}
\end{equation}
for all $r>0$, where $\hG_{\bm{\alpha},\bm{\beta}}$ is the constant defined in~\eqref{eq:hFlatPolymerNormalization}, $s_N(\bm{\lambda}) \diff \bm{\lambda}$ is the Sklyanin measure as in~\eqref{eq:sklyaninMeasure}, and $\delta$ is chosen so that $\delta > \alpha_j$ for all $j$.
\end{theorem}

\begin{proof}
Multiplying and dividing the integrand by $(\prod_{i=1}^N x_i)^{\delta}$ and using property~\eqref{eq:glWhittakerFnTranslation}, we can write the integral appearing in formula~\eqref{eq:hFlatWhittakerFormula} as
\[
\int_{\R_{>0}^N} \e^{-r x_1}
\Psi_{\bm{\alpha}}^{\so_{2N+1}}(\bm{x})
\Psi_{\bm{\beta}}^{\gl_N}(\bm{x})
\prod_{i=1}^N \frac{\diff x_i}{x_i}
= \int_{\R_{>0}^N}
f(\bm{x}) g(\bm{x})
\prod_{i=1}^N \frac{\diff x_i}{x_i} \, ,
\]
where
\[
f(\bm{x}):= \e^{-r x_1}
\Psi_{\bm{\beta}+\delta}^{\gl_N}(\bm{x}) \, ,
\qquad\quad
g(\bm{x}):=
\left(\prod_{i=1}^N x_i\right)^{-\delta}
\Psi_{\bm{\alpha}}^{\so_{2N+1}}(\bm{x}) \, .
\]
Applying now the isometry of Theorem~\ref{thm:plancherel} to the $L^2$-inner product of $f$ and $g$, whose $\gl_N$-Whittaker transforms have been computed in Lemma~\ref{lemma:WhittakerTransforms}, we obtain:
\[
\begin{split}
&\quad\, \int_{\R_{>0}^N} \e^{-r x_1}
\Psi_{\bm{\alpha}}^{\so_{2N+1}}(\bm{x})
\Psi_{\bm{\beta}}^{\gl_N}(\bm{x})
\prod_{i=1}^N \frac{\diff x_i}{x_i} \\
&= \int_{\i\R^N}
r^{-\sum_{i=1}^N (\lambda_i+\beta_i+\delta)} 
\frac{\prod_{1\leq i,j\leq N}
\Gamma(\delta + \lambda_i + \alpha_j)
\Gamma(\delta + \lambda_i - \alpha_j)
\Gamma(\lambda_i + \beta_j + \delta)}
{\prod_{1\leq i<j\leq N}
\Gamma(2\delta + \lambda_i+\lambda_j) }
s_N(\bm{\lambda}) \diff \bm{\lambda} \, .
\end{split}
\]
Changing variables $\bm{\lambda} \mapsto \bm{\lambda}-\delta$ and combining with~\eqref{eq:hFlatWhittakerFormula}, we obtain~\eqref{eq:hFlatContourInt}.
\end{proof}

\chapter{Last passage percolation models}
\label{ch:LPP}

Recall from the Introduction that the \emph{last passage percolation} (LPP) time is defined as
\[
\tau
:= \max_{\pi \in \Pi} \sum_{(i,j)\in\pi} W_{i,j} \, ,
\]
where $\Pi$ is a given set of nearest neighbor directed paths on a finite lattice $\mathcal{I}\subset\Z_{>0}^2$, and  $\bm{W}=\{W_{i,j} \colon (i,j)\in\mathcal{I} \}$ is an array of random \emph{waiting times}.

In this chapter we study the LPP in the same point-to-line path geometries considered in chapter~\ref{ch:polymer} for the corresponding polymer models.
We consider two different distributions on the waiting times: geometric in section~\ref{sec:geomLPP} and exponential in section~\ref{sec:expLPP}.
Via the $\RSK$ correspondence we show that these models are exactly solvable, by expressing the distribution function of the LPP time in terms of (discrete or continuous, standard or symplectic) Schur functions.
Finally, in section~\ref{sec:LPPasymptotics} we study the scaling limit of the point-to-line and point-to-half-line exponential models.

The connection between geometric point-to-point LPP, $\RSK$ correspondence and Schur functions goes back to Johansson~\cite{johansson00} and Baik-Rains~\cite{baikRains01a}.
In the cited article, Johansson also proved that the point-to-point LPP with i.i.d.\ geometric waiting times converges, under appropriate rescaling, to the GUE Tracy-Widom distribution.
We then start by illustrating the point-to-point case, which will be a useful source of inspiration and term of comparison for our point-to-line models (as it was for the polymer models, see beginning of chapter~\ref{ch:polymer}).
Let $\tau_{n,n}$ be the point-to-point LPP associated with the set of directed paths starting at $(1,1)$ and ending at $(n,n)$.
Assume that the waiting times $\bm{W}=\{W_{i,j}\colon 1\leq i,j\leq n\}$ are non-negative integer-valued random variables.
Then, by Proposition~\ref{prop:RSKGelfandTsetlin}, the $\RSK$ correspondence will map $\bm{W}$ to a pair $(\bm{Z},\bm{Z}')$ of $\Z_{\geq 0}$-Gelfand-Tsetlin patterns with the same shape $\bm{\lambda}$.
Since $\tau_{n,n}$ coincides with the first part $\lambda_1$ of the partition $\bm{\lambda}$ by property~\ref{prop:RSKGelfandTsetlin_LPP} of the same proposition, we have that
\begin{equation}
\label{eq:geomPointToPointLPP}
\P(\tau_{n,n} \leq u)
= \sum_{\lambda_1 \leq u}
\P(\shape(\bm{Z}) = \shape(\bm{Z}') = \bm{\lambda})
\end{equation}
for $u\in\Z_{\geq 0}$.
Assume now that the waiting times are independent and $W_{i,j} \sim \Geom(q_i p_j)$ for all $i,j$, where $\bm{q}=(q_1,\dots,q_n)$ and $\bm{p}=(p_1,\dots,p_n)$ are sets of parameters in $(0,1)^n$. 
Namely, the joint distribution of the matrix $\bm{W}$ is given by
\[
\P(\bm{W} = \bm{w})
= \prod_{i,j=1}^n (1 - q_i p_j) (q_i p_j  )^{w_{i,j}}
= \prod_{i,j=1}^n (1 - q_i p_j)
\left(\prod_{i=1}^n q_i^{\sum_{j=1}^n w_{i,j}}\right)
\left(\prod_{j=1}^n p_j^{\sum_{i=1}^n w_{i,j}}\right)
\]
for all $\bm{w}\in\Z_{\geq 0}^{n\times n}$.
By property~\ref{prop:RSKGelfandTsetlin_type} of Proposition~\ref{prop:RSKGelfandTsetlin}, the distribution that $\bm{W}$ induces on the common shape of $\bm{Z}'$ and $\bm{Z}$ writes as
\[
\begin{split}
\P(\shape(\bm{Z}') = \shape(\bm{Z}) = \bm{\lambda})
&= \sum_{\bm{z},\bm{z}' \in\GT{n}{\Z}(\bm{\lambda})}
\P(\bm{Z}'=\bm{z}' \, , \,\, \bm{Z}=\bm{z}) \\
&= \prod_{i,j=1}^n (1 - q_i p_j)
\left(\sum_{\bm{z}' \in\GT{n}{\Z}(\bm{\lambda})} \prod_{i=1}^n q_i^{\type(\bm{z}')_i}\right)
\left(\sum_{\bm{z} \in\GT{n}{\Z}(\bm{\lambda})} \prod_{j=1}^n p_j^{\type(\bm{z})_j}\right) \\
&= \prod_{i,j=1}^n (1 - q_i p_j) \, 
\schur_{\bm{\lambda}}(\bm{q}) \,
\schur_{\bm{\lambda}}(\bm{p}) \, ,
\end{split}
\]
where $\schur_{\bm{\lambda}}$ is the Schur polynomial defined in~\eqref{def:schur}.
It follows from~\eqref{eq:geomPointToPointLPP} that
\begin{equation}
\label{eq:geomPointToPointLPPSchur}
\P(\tau_{n,n} \leq u)
= \prod_{i,j=1}^n (1 - q_i p_j) \sum_{\lambda_1 \leq u} 
\schur_{\bm{\lambda}}(\bm{q}) \,
\schur_{\bm{\lambda}}(\bm{p}) \, ,
\end{equation}
where the sum is over all integer partition $\bm{\lambda}$ of length at most $n$ and first part at most $u$.
Notice that~\eqref{eq:geomPointToPointLPPSchur} provides a (probabilistic) proof of the Cauchy-Littlewood identity~\eqref{eq:cauchyIdentity}: it suffices to take the limit $u\to\infty$ and observe that the left-hand side tends to $1$. 
A formula such as~\eqref{eq:geomPointToPointLPPSchur} is useful because, due to the determinantal structure of Schur functions, it can be turned into a Fredholm determinant amenable to asymptotic analysis.

\section{Geometric last passage percolation}
\label{sec:geomLPP}

In section~\ref{sec:WhittakerFormulas} we used the $\gRSK$ correspondence to express the distribution of polymer partition functions in terms of Whittaker functions.
Here we carry out an analogous task at a \emph{zero temperature} level: namely, we use the $\RSK$ correspondence to express the distribution of LPP times in terms of Schur functions.
As in this section we work with \emph{geometric} waiting times, the resulting formulas are \emph{discrete}, in the sense that they involve sums of discrete Schur functions.
Conversely, in the next section we will work with exponential waiting times, hence our formulas will involve integrals of continuous Schur functions.

\subsection{Point-to-line geometric model}
\label{subsec:flatGeomLPP}

The \emph{point-to-line LPP} is defined as
\begin{equation}
\label{eq:flatLPP}
\fTau_{N}
:= \max_{\pi \in \fPi_{N}} \sum_{(i,j)\in\pi} W_{i,j} \, ,
\end{equation}
where $\fPi_{N}$ is the usual set of point-to-line directed paths, and  $\bm{W}$ is a random array indexed by the lattice $\fI_{N}$ defined in~\eqref{eq:flatLattice} - see Figure~\ref{subfig:flatPath}.
Notice that the point-to-line LPP at time $N$ can be written as a maximum of point-to-point LPPs with endpoint on the line $\{m+n = N+1\}$:
\begin{equation}
\label{eq:flatLPP=maxPointToPointLPP}
\tau_N = \max_{m+n = N+1} \tau_{m,n} \, .
\end{equation}

We will show in Theorem~\ref{thm:flatGeomLPP} that, when the waiting times are geometrically distributed with the special parametrization of Definition~\ref{def:flatGeomMeasure}, the CDF of $\tau_{2N}$ can be expressed as a sum of two symplectic Schur functions.
\emph{Mutatis mutandis}, the argument is similar to the one used in positive temperature, see Lemma~\ref{lemma:flatP2PjointLaw} and Theorem~\ref{thm:flatWhittakerFormula}.
Namely, we use the $\RSK$ bijection and its properties, next we apply a further change of variables that reverses the arrows in Figure~\ref{subfig:triangularArray} (in this context, it means that the inequalities that the $\RSK$ output satisfies are reversed).
The resulting array can be seen as a gluing of a pair of symplectic Gelfand-Tsetlin patterns with common shape (being the main diagonal of the array), which in turn generate the symplectic Schur functions.

\begin{definition}
\label{def:flatGeomMeasure}
Let $N\in\Z_{>0}$, $\bm{q},\bm{p}\in(0,1)^N$.
We define the \emph{$(\bm{q}, \bm{p})$-geometric measure} on the lattice $\fI_{2N}$ to be the law of an array $\{W_{i,j} \colon (i,j)\in\fI_{2N}\}$ of independent random variables such that:
\begin{equation}
\label{eq:flatGeomMeasure}
W_{i,j} \sim
\begin{cases}
\Geom(q_i p_j) &1\leq i,j\leq N \, , \\
\Geom(q_i q_{2N-j+1}) &1\leq i\leq N\, , \,\, N < j\leq 2N-i+1 \, , \\
\Geom(p_{2N-i+1} p_j) &1\leq j\leq N \, , \,\, N < i \leq 2N-j+1 \, .
\end{cases}
\end{equation}
\end{definition}

\begin{theorem}
\label{thm:flatGeomLPP}
The distribution of the point-to-line $(\bm{q},\bm{p})$-geometric LPP is given by
\begin{equation}
\label{eq:flatGeomLPP}
\P(\fTau_{2N} \leq u)
= \frac{\left( \prod_{k=1}^N q_k p_k\right)^u}{\fc_{\bm{q},\bm{p}}}  \sum_{\lambda_1 \leq u} \sp_{\bm{\lambda}}(\bm{q}) \sp_{\bm{\lambda}}(\bm{p})  \, ,
\end{equation}
where $u\in\Z_{\geq 0}$, the sum is over all integer partition $\bm{\lambda}$ of length at most $N$ and first part at most $u$, and the normalization constant is
\begin{equation}
\label{eq:flatGeomLPPnormalization}
\fc_{\bm{q},\bm{p}} := \prod_{1\leq i,j\leq N} \frac{1}{1-q_i p_j}\prod_{1\leq i\leq j\leq N} \frac{1}{(1-q_i q_j)(1-p_i p_j)} \, .
\end{equation}
\end{theorem}

\begin{proof}
If $\bm{W}=\{W_{i,j}\colon (i,j)\in\fI_{2N}\}$ is distributed according to the $(\bm{q},\bm{p})$-geometric measure, then
\[
\P(\bm{W} = \bm{w})
= \frac{1}{\fc_{\bm{q},\bm{p}}}
\prod_{i=1}^N q_i^{\sum_{j=1}^{2N-i+1} w_{i,j}}
p_i^{\sum_{j=1}^i w_{2N-i+1,j}}
\prod_{j=1}^N p_j^{\sum_{i=1}^{2N-j+1} w_{i,j}}
q_j^{\sum_{i=1}^j w_{i,2N-j+1}}
\]
for $\bm{w} \in \Z_{\geq 0}^{\fI_{2N}}$.
By Proposition~\ref{prop:RSKproperties}, $\RSK$ can be seen as a bijection between arrays indexed by $\fI_{2N}$ with non-negative integer entries, such that the output array satisfies the ordering~\eqref{eq:RSKordering}.
By property~\ref{prop:RSKproperties_type} of that proposition, the distribution that $\bm{W}$ induces on its $\RSK$ image $\bm{T}$ is then given by
\[
\begin{split}
\P(\bm{T} = \bm{t})
= \, &\frac{1}{\fc_{\bm{q},\bm{p}}}\prod_{i=1}^N
q_i^{\sigma_{2N-2i+1}(\bm{t}) - \sigma_{2N-2i+2}(\bm{t})}
p_i^{\sigma_{-(2N-2i+1)}(\bm{t})-\sigma_{-(2N-2i)}(\bm{t})} \\
&\times \prod_{j=1}^N
p_j^{\sigma_{-(2N-2j+1)}(\bm{t}) - \sigma_{-(2N-2j+2)}(\bm{t})}
q_j^{\sigma_{2N-2j+1}(\bm{t}) - \sigma_{2N-2j}(\bm{t})} \, ,
\end{split}
\]
where $\sigma_k(\bm{t})$ denotes the sum of the $k$-th diagonal of $\bm{t}$ as in~\eqref{eq:sumDiagonal}.
By equation~\eqref{eq:flatLPP=maxPointToPointLPP} and properties~\ref{prop:RSKproperties_LPP} and~\ref{prop:RSKproperties_interlacing} of Proposition~\ref{prop:RSKproperties}, our CDF is given by
\[
\P\left(\fTau_{2N} \leq u\right)
= \P\left(\fTau_{m,n} \leq u \colon m+n=2N+1 \right)
= \P(T_{i,j}\leq u \colon (i,j)\in\fI_{2N})
\]
and can be obtained by summing up $\P(\bm{T} = \bm{t})$ over all $\bm{t}\in\Z_{\geq 0}^{\fI_{2N}}$ satisfying the ordering~\eqref{eq:RSKordering} and the additional condition that $t_{i,j}\leq u$ for all $(i,j)\in\fI_{2N}$.

Let us now change variables, by setting $z_{i,j} := u - t_{2N+j-i,j}$ and $z'_{i,j} := u - t_{j,2N+j-i}$ for all $1\leq i\leq 2N$ and $1\leq j\leq \ceil{i/2}$.
Notice that the variables $z_{i,j}$'s and $z'_{i,j}$'s are bounded between $0$ and $u$, as all $t_{i,j}$'s are.
The arrays $\bm{z}$ and $\bm{z}'$ just defined turn thus out to be sympletic Gelfand-Tsetlin patterns of height $2N$ with a certain common shape $\bm{\lambda} = u-(t_{1,1},\dots,t_{N,N})$, such that $\lambda_1$ (and therefore all their entries) are not greater than $u$.
One may better visualize this by looking at Figure~\ref{subfig:triangularArray}: all the arrows are reversed, because the minus sign in the change of variables amounts to reversing all the inequalities; the lower triangular part and the upper triangular part of the array correspond to $\bm{z}$ and $\bm{z}'$ respectively, and the main diagonal turns into the common shape $\bm{\lambda}$.
Denoting by $\abs{\bm{z}_i} :=\sum_{j=1}^{\ceil{i/2}} z_{i,j}$ the sum of the $i$-th row of $\bm{z}$ (and analogously for $\bm{z}'$), we obtain that
\[
\begin{split}
& \fc_{\bm{q},\bm{p}} \cdot \P(\fTau_{2N} \leq u) \\
= \, &\sum_{\lambda_1\leq u}
\sum_{\bm{z}',\bm{z} \in \spGT{2N}{\Z}(\bm{\lambda})}
\prod_{i=1}^N
q_i^{u -\abs{\bm{z}'_{2i-1}} +\abs{\bm{z}'_{2i-2}}}
p_i^{-\abs{\bm{z}_{2i-1}} + \abs{\bm{z}_{2i}}}
\prod_{j=1}^N
p_j^{u -\abs{\bm{z}_{2j-1}} +\abs{\bm{z}_{2j-2}}}
q_j^{-\abs{\bm{z}'_{2j-1}} + \abs{\bm{z}'_{2j}}} \\
= \, & \left( \prod_{k=1}^N q_k p_k\right)^u
\sum_{\lambda_1 \leq u}
\sum_{\bm{z}'\in \spGT{2N}{\Z}(\bm{\lambda})}
\prod_{k=1}^N q_k^{-\type(\bm{z}')_{2k-1} + \type(\bm{z}')_{2k}}
\!\!\!\!
\sum_{\bm{z}\in \spGT{2N}{\Z}(\bm{\lambda})}
\prod_{k=1}^N p_k^{-\type(\bm{z})_{2k-1} + \type(\bm{z})_{2k}} \\
= \, & \left( \prod_{k=1}^N q_k p_k\right)^u
\sum_{\lambda_1 \leq u}
\sp_{\bm{\lambda}}(q_1^{-1},\dots,q_N^{-1}) 
\sp_{\bm{\lambda}}(p_1^{-1},\dots,p_N^{-1})\, .
\end{split}
\]
The latter two equalities follow from the definitions of type~\eqref{eq:typeSpGTpattern} and symplectic Schur function~\eqref{eq:spSchur}.
Recalling from subsection~\ref{subsec:symplecticSchur} that symplectic Schur functions are invariant under multiplicative inversion of their arguments, we obtain~\eqref{eq:flatGeomLPP}.
\end{proof}

We now make a comparison with the point-to-point LPP model with symmetry about the antidiagonal, as it turns out to be intimately connected to the point-to-line model.
Let us denote by $\antisymTau_{n,n}$ the point-to-point LPP from $(1,1)$ to $(n,n)$ with waiting times symmetric about the antidiagonal, i.e.\ $W_{i,j}=W_{n-j+1,n-i+1}$ for $1\leq i,j\leq n$.
Because of the symmetry constraint, at least one of the maximal paths\footnote{By maximal path we mean any of the allowed paths that maximizes the passage time. Notice that such a path does not need to be unique.} from $(1,1)$ to $(n,n)$ is symmetric about the antidiagonal; the waiting times collected along such a path will be all counted twice (once above and once below the antidiagonal), except the one on the antidiagonal itself.
It easily follows that the $\antisymTriangle[0.18]$-symmetric LPP  \emph{with doubled weights on the antidiagonal} coincides with twice the point-to-line LPP, i.e.\ $\antisymTau_{n,n} = 2\fTau_{n}$.

Baik-Rains~\cite{baikRains01a} and Forrester-Rains~\cite{forresterRains07} studied the distribution of the geometric point-to-point LPP where the matrix of waiting times is subject to certain symmetries, using the classical $\RSK$ acting on nonnegative integer matrices.
They proved in particular that, if the waiting times $\{W_{i,j}\colon i+j\leq n+1\}$ on and above the antidiagonal are independent and
\[
W_{i,j} \sim
\begin{cases}
\Geom(r_i r_{n-j+1}) &\text{if } i+j<n+1 \, , \\
2 \cdot \Geom(r_i^2) &\text{if } i+j=n+1
\end{cases}
\]
for a set of parameters $\bm{r}:=(r_1,\dots,r_n)\in (0,1)^n$, then
\begin{equation}
\label{eq:baikRainsAntisymLPP}
\P\left(\antisymTau_{n,n} \leq 2u\right)
= \prod_{1\leq i\leq j\leq n} (1-r_i r_j)
\sum_{\lambda_1 \leq u} \schur_{2\bm{\lambda}}(\bm{r})
\end{equation}
for all nonnegative integer $u$, being $2\bm{\lambda}=(2\lambda_1,\dots,2\lambda_n)$ any partition with even parts.
Thanks to the connection between $\antisymTau_{n,n}$ and $\fTau_n$ outlined above, the above also equals the distribution function $\P\left(\fTau_n \leq u\right)$ of the point-to-line LPP $\fTau_n$ with independent waiting times distributed as $W_{i,j} \sim \Geom(r_i r_{n-j+1})$ for $i+j\leq n+1$.
If we then take $n=2N$ to be even and replace $(r_1,\dots,r_{2N})$ with $(q_1,\dots,q_N,p_N,\dots,p_1)$, the distribution on the waiting times coincides with the $(\bm{q},\bm{p})$-geometric measure.
We thus find another formula for the point-to-line $(\bm{q},\bm{p})$-geometric LPP:
\begin{equation}
\label{eq:baikRainsFlatLPP}
\P\left(\fTau_{2N} \leq u\right)
= \frac{1}{\fc_{\bm{q},\bm{p}}}
\sum_{\lambda_1 \leq u} \schur_{2\bm{\lambda}}(q_1,\dots,q_N,p_N,\dots,p_1) \, ,
\end{equation}
where $\fc_{\bm{q},\bm{p}}$ is defined in~\eqref{eq:flatGeomLPPnormalization}.
Our~\eqref{eq:flatGeomLPP} involves a sum of the product between \emph{two symplectic} Schur functions, while Baik-Rains's~\eqref{eq:baikRainsFlatLPP} shows a sum of \emph{one standard} Schur function parametrized by even partitions: these formulas are essentially different, but still equivalent.
In particular, comparing them and recalling that Schur functions are symmetric in their variables, we deduce that
\begin{equation}
\label{eq:flatGeomLLP_comparison}
\sum_{\lambda_1 \leq u} \schur_{2\bm{\lambda}}(q_1,\dots,q_N,p_1,\dots,p_N)
= \left( \prod_{k=1}^N q_k p_k\right)^u
\sum_{\lambda_1 \leq u} \sp_{\bm{\lambda}}(q_1,\dots,q_N) \sp_{\bm{\lambda}}(p_1,\dots,p_N) \, .
\end{equation}
There are direct combinatorial/algebraic proofs of this nontrivial identity, which will be the subject of future work.

\subsection{Point-to-half-line geometric model}
\label{subsec:hFlatGeomLPP}

The \emph{point-to-half-line LPP} is defined as
\begin{equation}
\label{eq:hFlatLPP}
\hTau_{N}
:= \max_{\pi \in \hPi_{N}} \sum_{(i,j)\in\pi} W_{i,j} \, ,
\end{equation}
where $\hPi_{N}$ is the usual set of point-to-half-line directed paths, and $\bm{W}$ is a random array indexed by the lattice $\hI_{N}$ defined in~\eqref{eq:hFlatLattice} - see Figure~\ref{subfig:hFlatPath}.
Notice that the point-to-half-line LPP at time $N$ can be written as a maximum of point-to-point LPPs with endpoint on the half-line $\{m+n = N+1, \,\, m\leq n\}$:
\begin{equation}
\label{eq:hFlatLPP=maxPointToPointLPP}
\hTau_N = \max_{\substack{m+n = N+1 \\ m\leq n}} \tau_{m,n} \, .
\end{equation}

If the waiting times are distributed according to the $(\bm{q},\bm{p})$-geometric measure as in Definition~\ref{def:flatGeomMeasure} (restricted to the lattice $\hI_{2N}$), the CDF of $\hTau_{2N}$ is given by a sum of the product between a symplectic and a standard Schur function:

\begin{theorem}
\label{thm:hFlatGeomLPP}
The distribution of the point-to-half-line $(\bm{q},\bm{p})$-geometric LPP is given by
\begin{equation}
\label{eq:hFlatGeomLPP}
\P\left(\hTau_{2N} \leq u\right)
= \frac{\left( \prod_{k=1}^N q_k p_k\right)^u}{\hc_{\bm{q},\bm{p}}}  \sum_{\lambda_1 \leq u} \sp_{\bm{\lambda}}(q_1,\dots,q_N) \schur_{\bm{\lambda}}(p_1^{-1},\dots,p_N^{-1})  \, ,
\end{equation}
where $u\in\Z_{\geq 0}$ and the normalization constant is
\begin{equation}
\label{eq:hFlatGeomLPPnormalization}
\hc_{\bm{q},\bm{p}} := \prod_{1\leq i,j\leq N} \frac{1}{1-q_i p_j}\prod_{1\leq i\leq j\leq N} \frac{1}{1-q_i q_j} \, .
\end{equation}
\end{theorem}

\begin{proof}
The proof is again based on Proposition~\ref{prop:RSKproperties} and follows the same steps as in Theorem~\ref{thm:flatGeomLPP}, so we only sketch it.
Let $\bm{W}=\{W_{i,j}\colon (i,j)\in\hI_{2N}\}$ be distributed according to the $(\bm{q},\bm{p})$-geometric measure, i.e.
\[
\P(\bm{W} = \bm{w})
= \frac{1}{\hc_{\bm{q},\bm{p}}}
\prod_{i=1}^N q_i^{\sum_{j=1}^{2N-i+1} w_{i,j}}
\prod_{j=1}^N p_j^{\sum_{i=1}^{N} w_{i,j}}
q_j^{\sum_{i=1}^j w_{i,2N-j+1}}
\]
for $\bm{w} \in \Z_{\geq 0}^{\hI_{2N}}$.
The distribution that $\bm{W}$ induces on its $\RSK$ image $\bm{T}$ is then given by
\[
\P(\bm{T} = \bm{t})
= \frac{1}{\hc_{\bm{q},\bm{p}}}\prod_{i=1}^N
q_i^{\sigma_{2N-2i+1}(\bm{t}) - \sigma_{2N-2i+2}(\bm{t})}
\prod_{j=1}^N
p_j^{\sigma_{-(N-j)}(\bm{t}) - \sigma_{-(N-j+1)}(\bm{t})}
q_j^{\sigma_{2N-2j+1}(\bm{t}) - \sigma_{2N-2j}(\bm{t})} \, .
\]

We can obtain $\P(\hTau_{2N} \leq u)$ by summing up the probabilities computed above over all $\bm{t}\in\Z_{\geq 0}^{\hI_{2N}}$ satisfying the ordering~\eqref{eq:RSKordering} and the additional condition that $t_{i,j}\leq u$ for all $(i,j)\in\hI_{2N}$.
Let us now change variables, by setting $z'_{i,j} := u - t_{j,2N+j-i}$ for all $1\leq i\leq 2N$, $1\leq j\leq \ceil{i/2}$, and $z_{i,j} := u - t_{N+j-i,j}$ for all $1\leq j\leq i\leq N$.
Thanks to the inequalities that the $t_{i,j}$'s satisfy, the arrays $\bm{z}'$ and $\bm{z}$ just defined are Gelfand-Tsetlin patterns (sympletic and standard respectively) with common shape $\bm{\lambda} = u-(t_{1,1},\dots,t_{N,N})$, such that $\lambda_1\leq u$.
We then obtain:
\[
\begin{split}
& \hc_{\bm{q},\bm{p}} \cdot \P(\hTau_{2N} \leq u) \\
= \, & \left( \prod_{k=1}^N q_k p_k\right)^u
\sum_{\lambda_1 \leq u}
\sum_{\bm{z}'\in \spGT{2N}{\Z}(\bm{\lambda})}
\prod_{k=1}^N q_k^{-\type(\bm{z}')_{2k-1}+ \type(\bm{z}')_{2k}}
\!\!\!\!
\sum_{\bm{z}\in \GT{N}{\Z}(\bm{\lambda})}
\prod_{k=1}^N p_k^{-\type(\bm{z})_{k}} \\
= \, & \left( \prod_{k=1}^N q_k p_k\right)^u
\sum_{\lambda_1 \leq u}
\sp_{\bm{\lambda}}(q_1^{-1},\dots,q_N^{-1}) 
\schur_{\bm{\lambda}}(p_1^{-1},\dots,p_N^{-1})\, .
\end{split}
\]
The claim follows from the fact that symplectic Schur functions are invariant under multiplicative inversion of their arguments.
\end{proof}

\subsection{Restricted point-to-line geometric model}
\label{subsec:rFlatGeomLPP}

The \emph{restricted point-to-line LPP} is defined as
\begin{equation}
\label{eq:rFlatLPP}
\rTau_{N}
:= \max_{\pi \in \rPi_{N}} \sum_{(i,j)\in\pi} W_{i,j} \, ,
\end{equation}
where $\rPi_{N}$ is the usual set of point-to-line directed paths restricted to stay in a half-plane, and $\bm{W}$ is a random array indexed by the lattice $\rI_{N}$ defined in~\eqref{eq:rFlatLattice} - see Figure~\ref{subfig:rFlatPath}.
Notice that the restricted point-to-line LPP at time $N$ can be written as a maximum of restricted point-to-point LPPs with endpoint on the half-line $\{m+n = N+1, \,\, m\leq n\}$:
\begin{equation}
\label{eq:rFlatLPP=maxPointToPointLPP}
\rTau_N = \max_{\substack{m+n = N+1 \\ m\leq n}} \rTau_{m,n} \, .
\end{equation}

The point-to-line restricted LPP model is perfectly equivalent to the corresponding symmetric model.
Namely, $\rTau_N = \sTau_N$, where $\sTau_N$ is the (usual, i.e.\ non-restricted) point-to-line LPP on the lattice $\fI_N$ with the same waiting times as $\rTau_N$ on and above the main diagonal and symmetric waiting times below the diagonal (so that $W_{i,j}=W_{j,i}$ for all $(i,j)\in\fI_N$).
Similarly, the point-to-point restricted LPP model is perfectly equivalent\footnote{Notice that, in zero temperature as opposed to the positive temperature setting (see subsection~\ref{subsec:rFlatPolymerWhittaker}), no modification of the weights on the diagonal nor multiplicative constants are required to pass from the restricted to the symmetric model.} to the corresponding symmetric model, i.e.\ $\rTau_{m,n} = \sTau_{m,n}$.
Via this symmetrization argument, we will show that, for a certain parametrization of the geometrically distributed environment, $\rTau_{2N}$ can be essentially expressed as a sum over partitions of one symplectic Schur function indexed by that partition.

\begin{definition}
\label{def:rFlatGeomMeasure}
Let $N\in\Z_{>0}$ and $\bm{q}\in(0,1)^N$.
We define the \emph{restricted $\bm{q}$-geometric measure} on the lattice $\rI_{2N}$ to be the law of an array $\{W_{i,j} \colon (i,j)\in\rI_{2N}\}$ of independent random variables such that:
\begin{equation}
\label{eq:rFlatGeomMeasure}
W_{i,j} \sim
\begin{cases}
\Geom(q_i) &1\leq i=j\leq N \, , \\
\Geom(q_i q_j) & 1 \leq i<j\leq N \, , \\
\Geom(q_i q_{2N-j+1}) &1\leq i\leq N\, , \,\, N < j\leq 2N-i+1 \, .
\end{cases}
\end{equation}
\end{definition}

\begin{theorem}
\label{thm:rFlatGeomLPP}
The distribution of the restricted point-to-line $\bm{q}$-geometric LPP is given by
\begin{equation}
\label{eq:rFlatGeomLPP}
\P\left(\rTau_{2N} \leq u\right)
= \frac{\left( \prod_{k=1}^N q_k \right)^u}{\rc_{\bm{q}}} 
\sum_{\lambda_1 \leq u} \sp_{\bm{\lambda}}(q_1,\dots,q_N) \, ,
\end{equation}
where $u\in\Z_{\geq 0}$ and the normalization constant is
\begin{equation}
\label{eq:rFlatGeomLPPnormalization}
\rc_{\bm{q}} := \prod_{1\leq i\leq N} \frac{1}{1-q_i}
\prod_{1\leq i<j\leq N} \frac{1}{1-q_i q_j}
\prod_{1\leq i\leq j\leq N} \frac{1}{1-q_i q_j} \, .
\end{equation}
\end{theorem}

\begin{proof}
Since the restricted and symmetric LPP models are perfectly equivalent, we can work with the symmetric model and compute the CDF of $\sTau_{2N}$ instead of $\rTau_{2N}$.
The distribution of the symmetric array of waiting times $\bm{W}$ is given by
\[
\begin{split}
\P(\bm{W} = \bm{w})
&= \frac{1}{\rc_{\bm{q}}}
\prod_{i=1}^N q_i^{\sum_{j=i}^{2N-i+1} w_{i,j}}
\prod_{j=1}^N q_j^{\sum_{i=1}^{j-1} w_{i,j}}
q_j^{\sum_{i=1}^j w_{i,2N-j+1}} \\
&= \frac{1}{\rc_{\bm{q}}}
\prod_{i=1}^N q_i^{\sum_{j=1}^{2N-i+1} w_{i,j}}
\prod_{j=1}^N q_j^{\sum_{i=1}^j w_{i,2N-j+1}}
\end{split}
\]
for all symmetric $\bm{w}\in\Z_{\geq 0}^{\fI_{2N}}$.
By Proposition~\ref{prop:RSKproperties} and \ref{prop:symmetricRSK}, $\bm{W}$ induces on its $\RSK$ image $\bm{T}$ the distribution
\[
\P(\bm{T} = \bm{t})
= \frac{1}{\rc_{\bm{q}}}
\prod_{i=1}^N
q_i^{\sigma_{2N-2i+1}(\bm{t}) - \sigma_{2N-2i+2}(\bm{t})}
\prod_{j=1}^N
q_j^{\sigma_{2N-2j+1}(\bm{t}) - \sigma_{2N-2j}(\bm{t})}
\]
for all symmetric $\bm{t}\in\Z_{\geq 0}^{\fI_{2N}}$.
The rest of the proof proceeds as in Theorem~\ref{thm:flatGeomLPP}: the only difference is that, after the change of variables, we can identify one symplectic Gelfand-Tsetlin pattern (instead of two glued together), which generates a single symplectic Schur function.
\end{proof}

In the final discussion of subsection~\ref{subsec:flatGeomLPP} we compared Baik-Rains's formula for a $\antisymTriangle[0.18]$-symmetrized point-to-point LPP model with our point-to-line formula.
We address a similar task here, comparing the point-to-point LPP model with symmetry about both the antidiagonal and the diagonal with the intimately connected restricted point-to-line model.
Let us denote by $\doublesymTau_{2N,2N}$ the point-to-point LPP from $(1,1)$ to $(2N,2N)$ with waiting times symmetric about both the antidiagonal $\{i+j=2N+1\}$ and the diagonal $\{i=j\}$.
Reasoning similarly as in subsection~\ref{subsec:flatGeomLPP}, one can see that the $\doublesymTriangle[0.18]$-symmetric LPP \emph{with doubled weights on the antidiagonal} coincides with twice the restricted/symmetric point-to-line LPP, i.e.\ $\doublesymTau_{2N,2N} = 2\rTau_{2N}$.

Using the classical $\RSK$ acting on a doubly symmetric matrix, Baik-Rains~\cite{baikRains01a} and Forrester-Rains~\cite{forresterRains07} studied the $\doublesymTriangle[0.18]$-symmetric LPP with geometric waiting times.
The connection between $\doublesymTau_{2N,2N}$ and $\rTau_{2N}$ stated above permits rephrasing their result in terms of the restricted point-to-line LPP in the $\bm{q}$-geometric environment (see Definition~\ref{def:rFlatGeomMeasure}):
\begin{equation}
\label{eq:baikRainsrFlatLPP}
\P\left(\rTau_{2N} \leq u\right)
= \frac{1}{\rc_{\bm{q}}}
\sum_{\lambda_1 \leq u} \schur_{\bm{\lambda}}(q_1,\dots,q_N)
\schur_{\bm{\lambda}}(q_1,\dots,q_N,1)
\end{equation}
for all nonnegative integer $u$, where $\rc_{\bm{q}}$ is the same constant as in~\eqref{eq:rFlatGeomLPPnormalization}.
Comparing our formula~\eqref{eq:rFlatGeomLPP}, which involves one symplectic Schur function, with~\eqref{eq:baikRainsrFlatLPP}, which involves two standard Schur functions, we deduce the identity
\begin{equation}
\label{eq:rFlatGeomLLP_comparison}
\sum_{\lambda_1 \leq u} \schur_{\bm{\lambda}}(q_1,\dots,q_N)
\schur_{\bm{\lambda}}(q_1,\dots,q_N,1)
= \left( \prod_{k=1}^N q_k \right)^u
\sum_{\lambda_1 \leq u} \sp_{\bm{\lambda}}(q_1,\dots,q_N) \, .
\end{equation}
Again, it is possible to prove the latter identity via combinatorial/algebraic methods, without resorting to indirect arguments based on lattice path probabilistic models.
We plan to address this aspect in future work.

\section{Exponential last passage percolation}
\label{sec:expLPP}

In this section we analyze the same three LPP models as in section~\ref{sec:geomLPP}, but considering exponential distribution on the waiting times instead of geometric.
As mentioned in the Introduction, the exponential LPP models are the most closely related to the totally asymmetric simple exclusion process (TASEP); we now see explicitly how.
Recall that TASEP is a continuous time Markov process with state space $\{0,1\}^{\Z}$, interpreted as an interacting particle system.
Each $\eta\in\{0,1\}^{\Z}$ can be viewed as a configuration of particles and holes on the integer line: site $i\in\Z$ is either occupied or empty according to whether $\eta(i)=1$ or $\eta(i)=0$.
The dynamics works as follows: each particle independently, after a mean one exponential time, jumps to the site immediately to the right, provided that the latter site is vacant.
Given an enumeration $\{p_i\}_{i\in I}$ of the particles, we denote by $\mathcal{W}_{i,j}$ the time that particle $p_i$ needs to wait to perform its $j$-th jump, once the site to its right becomes vacant.
Let also $\mathcal{T}_{i,j}$ be the time when particle $p_i$ performs its $j$-th jump starting from the initial configuration, and set for convenience $\mathcal{T}_{i,j}=0$ if $i\notin I$ or $j<1$.
The $\{\mathcal{W}_{i,j}\}_{i\in I,j\geq 1}$ are i.i.d.\ mean one exponentially distributed variables, whereas the $\{\mathcal{T}_{i,j}\}_{i\in I,j\geq 1}$ can be expressed as (deterministic) functionals of the $\mathcal{W}_{i,j}$'s that depend on the initial conditions.
Such functionals can be defined via lattice paths and link the TASEP with a given initial condition to the LPP model with a corresponding path geometry.
More precisely, certain $\mathcal{T}_{i,j}$'s turn out to coincide with last passage times in an environment $\{W_{i,j}\}$ of mean one exponential waiting times, defined by just relabeling the $\mathcal{W}_{i,j}$'s.
In the following we will state these analogies between the two models for the basic point-to-point case and for the three path geometries we are concerned with.
We will give recurrence relations for the $\mathcal{T}_{i,j}$'s: solving these leads to closed expressions that, after relabeling the environment, coincide with the LPP definitions.
We will not go into the details of the relabeling; one may convince oneself of the equivalence between the models by solving the first few recurrence relations.
We consider the following (deterministic) initial configurations for the TASEP:
\begin{enumerate}
\item
\emph{Step initial configuration}: $\eta=\1_{\Z_{<0}}$.
Let us enumerate the particles by saying that $p_i$ is the particle that starts from $-i$, for all $i>0$.
Before particle $p_i$ is ``ready'' to perform its $j$-th jump, with the site to its right empty, one needs to wait for $p_i$ to perform the previous $j-1$ jumps and for the next particle to the right $p_{i-1}$ to perform $j$ jumps.
This implies the recurrence relation
\begin{equation}
\label{eq:stepTASEPrecurrence}
\mathcal{T}_{i,j} = \max(\mathcal{T}_{i,j-1}, \mathcal{T}_{i-1,j}) + \mathcal{W}_{i,j} \, ,
\end{equation}
where the base cases are given by the convention that $\mathcal{T}_{i,j}=0$ for $i=0$ or $j=0$.
For example, $\mathcal{T}_{2,2} = \mathcal{W}_{1,1} + \max(\mathcal{W}_{2,1}, \mathcal{W}_{1,2}) + \mathcal{W}_{2,2}$.
The variable $\mathcal{T}_{i,j}$ coincides with a \emph{point-to-point} LPP time from $(1,1)$ to $(i,j)$.
\item
\emph{Alternating initial configuration}: $\eta=\1_{2\Z}$.
Let $p_i$ be the particle that starts at site $-2i$ for all $i\in\Z$.
In this situation, in order for a particle to be able to perform $j$ jumps, the next one to the right needs to perform $j-1$ jumps only.
Recursive relation~\eqref{eq:stepTASEPrecurrence} is thus modified accordingly:
\begin{equation}
\label{eq:flatTASEPrecurrence}
{\mathcal{T}}_{i,j} = \max(\mathcal{T}_{i,j-1},\mathcal{T}_{i-1,j-1}) + \mathcal{W}_{i,j} \, ,
\end{equation}
with base cases given by the convention that $\mathcal{T}_{i,0}=0$ for all $i\in\Z$.
For example, $\mathcal{T}_{1,2} = \max(\mathcal{W}_{1,1}, \mathcal{W}_{0,1}) + \mathcal{W}_{1,2}$.
The time $\mathcal{T}_{N,2N}$ that particle $p_N$ takes to reach the origin\footnote{Equivalently, this is the time that particle $p_N$ (or, by translation invariance of the model, any other particle) takes to perform $2N$ jumps.} coincides with a \emph{point-to-line} LPP from $(1,1)$ to the line $\{m+n=2N+1\}$.
\item
\emph{Half-alternating initial configuration}: $\eta=\1_{2\Z_{<0}}$.
Let $p_i$ be the particle that starts at site $-2i$ for all $i>0$.
The recurrence relation is still given by~\eqref{eq:flatTASEPrecurrence}, but only for $i>1$, with the convention that $\mathcal{T}_{i,j}=0$ for $i=0$ or $j=0$.
For example, $\mathcal{T}_{2,3} = \max(\mathcal{W}_{2,2} + \mathcal{W}_{2,1}, \mathcal{W}_{2,2} + \mathcal{W}_{1,1}, \mathcal{W}_{1,2} + \mathcal{W}_{1,1}) + \mathcal{W}_{2,3}$.
The time $\mathcal{T}_{N,2N}$ that particle $p_N$ takes to reach the origin is now a \emph{point-to-half-line} LPP from $(1,1)$ to the half-line $\{m+n = 2N+1, \,\, m\leq n\}$.
\item
\emph{Half-alternating initial configuration with absorbing site}.
We consider the same initial configuration $\eta=\1_{2\Z_{<0}}$ (and enumeration of particles) as in the previous case, but with a slightly modified dynamics: we assume there is a ``black hole'' or ``absorbing site'' at the origin, so that, when a particle jumps from $-1$ to $0$, it simply disappears.
The recurrence relation is again given by~\eqref{eq:flatTASEPrecurrence} for $i>1$, now with the convention that $\mathcal{T}_{i,j}=0$ for $i=0$ or $j=2i+1$.
For example, $\mathcal{T}_{2,4} = \max(\mathcal{W}_{2,2} + \mathcal{W}_{2,1}, \mathcal{W}_{2,2} + \mathcal{W}_{1,1}, \mathcal{W}_{1,2} + \mathcal{W}_{1,1}) + \mathcal{W}_{2,3} + \mathcal{W}_{2,4}$.
The time $\mathcal{T}_{N,2N}$ that particle $p_N$ takes to be absorbed at the origin is now a \emph{restricted point-to-half-line} LPP from $(1,1)$ to the half-line $\{m+n = 2N+1, \,\, m\leq n\}$.
\end{enumerate}

After this digression about TASEP, let us now come to our analysis of the exponential LPP models.
We will see that, replacing the geometric distribution of section~\eqref{sec:geomLPP} with the exponential distribution, our formulas will involve integrals of continuous Schur functions instead of sums of (discrete) Schur functions.
At this stage, we have three alternative ways to study the exponential models:
\begin{enumerate}
\item
\emph{Via zero temperature limit} from the results obtained for the log-gamma polymer.
Under this scaling, the inverse-gamma distribution converges to the exponential distribution and the integrals of Whittaker functions that appear in the log-gamma polymer case (see section~\ref{sec:WhittakerFormulas}) converge to integrals of continuous Schur functions.
\item
\emph{Directly via $\RSK$}, as we did for geometric LPP in section~\ref{sec:geomLPP}.
Just notice that the continuous setting requires to also use the volume preserving property of $\RSK$ (property~\ref{prop:RSKproperties_Jacobian} of Proposition~\ref{prop:RSKproperties}), which is not necessary in the discrete setting.
\item
\emph{Via exponential limit} from the results obtained in geometric environment.
Under this scaling, the geometric distribution converges to the exponential distribution and the sums of discrete Schur functions that appear in the geometric case (see section~\ref{sec:geomLPP}) converge, via Riemann sum approximation, to integrals of continuous Schur functions.
\end{enumerate}
Any of these methods can be used for any path geometry, so we will use a different method for each of our three models.

In this section we will also go one step further by turning each Schur functions' formula into a determinantal or Pfaffian formula.
This is motivated by the fact that in section~\ref{sec:LPPasymptotics} we will carry out asymptotics of the point-to-line and point-to-half-line exponential LPP models.
To obtain such determinantal and Pfaffian formulas we will use the determinantal structure of (continuous) Schur functions along with the following two identities:
\begin{itemize}
\item
\emph{Cauchy-Binet identity} (whose integral formulation is due to Andr\'{e}ief~\cite{andreief1886}): if $\nu$ is a Borel measure on $\R$ and $f_j,g_j\in L^2(\R,\nu)$ for all $1\leq j\leq N$, then
\begin{equation}
\label{eq:cauchyBinet}
\frac{1}{N!}
\int_{\R^N}
\det\left(f_j(x_i)\right)_{1\leq i,j\leq N}
\det\left(g_j(x_i)\right)_{1\leq i,j\leq N}
\prod_{i=1}^N \nu(\diff x_i)
= \det\left( A^{(N)} \right) \, ,
\end{equation}
where $A^{(N)}$ is the $N\times N$ matrix given by
\[
A^{(N)}(i,j)
= \int_{\R} f_i(x) g_j(x) \nu(\diff x) \, .
\]
\item
\emph{de Bruijn identity}~\cite{deBruijn55}: if $\nu$ is a Borel measure on $\R$ and $\phi_j\in L^2(\R,\nu)$ for all $1\leq j\leq N$, then
\begin{equation}
\label{eq:deBruijn}
\int_{\{x_1\leq\dots\leq x_N\}}
\det\left(\phi_j(x_i)\right)_{1\leq i,j\leq N} \prod_{i=1}^N \nu(\diff x_i)
= \Pf\left(\Phi^{(N)}\right) \, ,
\end{equation}
where $\Phi^{(N)}$ is a skew-symmetric matrix of order $N$ or $N+1$, according to whether $N$ is even or odd respectively, and is defined by
\[
\Phi^{(N)}(i,j) :=
\begin{dcases}
\int_{\R^2} \sgn(y-x) \phi_i(x) \phi_j(y) \nu(\diff x) \nu(\diff y)
&  1\leq i,j\leq N , \\
\int_{\R} \phi_i(x) \nu(\diff x)
& 1\leq i\leq N, \,\, j=N+1; \,\, \text{$N$ odd.}
\end{dcases}
\]
\end{itemize}

\subsection{Point-to-line exponential model}
\label{subsec:flatExpLPP}

In this section we will express the law of the point-to-line exponential LPP $\tau_{2N}$ via zero temperature limit from the corresponding log-gamma polymer model.
Under this limit, since orthogonal Whittaker functions scale to continuous symplectic Schur functions (see Proposition~\ref{prop:soWhittakerRescaling}), we will obtain from~\eqref{eq:flatWhittakerFormula} an integral of two such Schur functions.
Before deriving the CDF of $\tau_{2N}$ via this rather indirect method, we show how to compute it directly in the simplest example.
\begin{example}
Let us compute the CDF of $\fTau_2 = W_{1,1} + \max\left(W_{1,2}, W_{2,1}\right)$, where $W_{i,j}$'s are all independent and $\Exp(2\gamma)$.
It can be easily checked that $
\max\left(W_{1,2},W_{2,1}\right) \stackrel{\diff}{=} W_{1,2} + \frac{W_{2,1}}{2}$, so that
\[
\fTau_2 \stackrel{\diff}{=}
W_{1,1} + W_{1,2} + \frac{W_{2,1}}{2} \, .
\]
Since $W_{1,1}+W_{1,2}\sim \Gamma(2,2\gamma)$ and $\frac{W_{2,1}}{2}\sim \Exp(4\gamma)$ and they are independent, the density $f_{\fTau_2}$ of $\fTau_2$ is given by the convolution of these two distributions:
\[
\begin{split}
f_{\fTau_2}(x)
&= \int_{\R} f_{W_{2,1}/2}(x-y) f_{W_{1,1}+W_{1,2}}(y) \diff y \\
&= \int_{\R} 4\gamma \e^{-4\gamma (x-y)} \1_{\{x-y\geq 0\}} \frac{(2\gamma)^2}{\Gamma(2)} y \e^{-2\gamma y} \1_{\{y\geq 0\}} \diff y \\
&= \1_{\{x\geq 0\}} \left(8\gamma^2 x \e^{-2\gamma x} -4\gamma \e^{-2\gamma x} + 4 \gamma \e^{-4\gamma x} \right) \, .
\end{split}
\]
Integrating $f_{\fTau_2}$ from $-\infty$ to $u>0$, we obtain
\begin{equation}
\label{eq:flatExpLPPN=2}
\P(\fTau_2\leq u)
= 1 -4\gamma u \e^{-2\gamma u} - \e^{-4 \gamma u} \, .
\end{equation}
\end{example}

Let us pass to the general case now.
We first outline how the zero temperature limit works, i.e.\ how to pass from the polymer model to the corresponding LPP model; the details of this argument are given in Appendix~\ref{appendix:zeroTempLimit}.
We start from the observation that, if $W^{(\epsilon)}$ is inverse-gamma distributed with parameter $\epsilon \gamma$ and $W$ is exponentially distributed with rate $\gamma>0$, then
\begin{equation}
\label{eq:invGammaToExp}
\epsilon \log W^{(\epsilon)} \xrightarrow{\epsilon \downarrow 0} W
\end{equation}
in distribution.
For all $\epsilon>0$, let now $Z_{N}^{(\epsilon)}$ be the point-to-line polymer partition function with disorder given by independent weights $\bm{W}^{(\epsilon)} = \left\{W^{(\epsilon)}_{i,j}\right\}$ such that each $W^{(\epsilon)}_{i,j}$ is inverse-gamma distributed with parameter $\epsilon\gamma_{i,j}$.
We then have that
\[
\epsilon \log Z_{N}^{(\epsilon)}
\xrightarrow{\epsilon\downarrow 0}
\fTau_{N}
\]
in distribution, where $\fTau_{N}$ is the corresponding LPP with independent waiting times $\bm{W} = \left\{W_{i,j}\right\}$ such that $W_{i,j}\sim \Exp(\gamma_{i,j})$ for all $(i,j)$.
Up to some technicalities concerning the convergence of the expectation, we can conclude that
\begin{equation}
\label{eq:zeroTempLimitLaplaceTransf}
\E\left[\exp\left\{-\e^{{-u/\epsilon}} Z^{(\epsilon)}_{N} \right\}\right]
\xrightarrow{\epsilon \downarrow 0} \P(\fTau_{N} \leq u)
\end{equation}
for all $u\in\R$.
In other words, properly rescaling the Laplace transform of the log-gamma polymer partition function yields the CDF of the corresponding exponential LPP.

From the argument above it is clear that, taking the zero temperature limit of our log-gamma\footnote{For the sake of simplicity we will take $\gamma:=0$ in Definition~\ref{def:flatLogGammaMeasure} of log-gamma measure.} polymer formula~\eqref{eq:flatWhittakerFormula}, we will be able to obtain the distribution of $\fTau_{2N}$ for the following exponential environment.
\begin{definition}
\label{def:flatExpMeasure}
Let $N\in\Z_{>0}$, $\bm{\alpha},\bm{\beta}\in\R_{>0}^N$.
We define the \emph{$(\bm{\alpha}, \bm{\beta})$-exponential measure} on the lattice $\fI_{2N}= \{(i,j)\in\Z_{>0}^2 \colon i+j\leq 2N+1\}$ to be the law of an array $\{W_{i,j} \colon (i,j)\in\fI_{2N}\}$ of independent random variables such that:
\begin{equation}
\label{eq:flatExpMeasure}
W_{i,j} \sim
\begin{cases}
\Exp(\alpha_i + \beta_j) &1\leq i,j\leq N \, , \\
\Exp(\alpha_i + \alpha_{2N-j+1}) &1\leq i\leq N \, , \,\, N < j\leq 2N-i+1 \, , \\
\Exp(\beta_{2N-i+1} + \beta_j) &1\leq j\leq N \, , \,\, N < i \leq 2N-j+1 \, .
\end{cases}
\end{equation}
\end{definition}

\begin{theorem}
\label{thm:flatExpLPP}
The distribution of the point-to-line $(\bm{\alpha}, \bm{\beta})$-exponential LPP $\fTau_{2N}$ can be expressed in terms of continuous symplectic Schur functions as
\begin{equation}
\label{eq:flatExpLPPSchur}
\P(\fTau_{2N}\leq u)
= \frac{\e^{-u \sum_{k=1}^N ( \alpha_k+\beta_k)}}{\fk_{\bm{\alpha},\bm{\beta}}}
\int_{\{0< x_N< \cdots < x_1< u\}}  
\sp^{\cont}_{\bm{\alpha}}(\bm{x})
\sp^{\cont}_{\bm{\beta}}(\bm{x})
\prod_{i=1}^N \diff x_i \, ,
\end{equation}
for all $u>0$, where
\begin{equation}
\label{eq:flatExpLPPnormalization}
\fk_{\bm{\alpha},\bm{\beta}}
:= \prod_{1\leq i,j\leq N} \frac{1}{\alpha_i + \beta_j}
\prod_{1\leq i\leq j\leq N} \frac{1}{(\alpha_i + \alpha_j)(\beta_i + \beta_j)} \, .
\end{equation}
We can further write~\eqref{eq:flatExpLPPSchur} as a ratio of $N\times N$ determinants:
\begin{equation}
\label{eq:flatExpLPPdet}
\P\left(\fTau_{2N}\leq u\right) = 
\frac{\det\left(\fH_u(\alpha_i,\beta_j)\right)_{1\leq i,j\leq N}}{\det(C(\alpha_i,\beta_j))_{1\leq i,j\leq N}} \, ,
\end{equation}
where $C(z,w) := (z+w)^{-1}$ and
\[
\fH_u(z,w) := \e^{-u(z + w)} \int_0^u (\e^{z x}-\e^{-z x})
(\e^{w x}-\e^{-w x}) \diff x \, .
\]
\end{theorem}

The denominator in~\eqref{eq:flatExpLPPdet} is the so-called \emph{Cauchy determinant}, which has a simple closed form:
\begin{equation}
\label{eq:cauchyDet}
\det(C(\alpha_i,\beta_j))_{1\leq i,j\leq N}
= \det\left( \frac{1}{\alpha_i + \beta_j} \right)_{1\leq i,j\leq N}
= \frac{\prod_{1\leq i<j\leq N} (\alpha_i -\alpha_j)(\beta_i - \beta_j)}
{\prod_{1\leq i,j\leq N} (\alpha_i + \beta_j)} \, .
\end{equation}

Notice that the right-hand side of~\eqref{eq:flatExpLPPdet} takes the form $0/0$ whenever $\alpha_i=\alpha_j$ or $\beta_i=\beta_j$ for some $i\neq j$.
In such cases, the formula should be understood in the limit as $\alpha_i - \alpha_j \to 0$ or $\beta_i - \beta_j \to 0$, analogously to the determinantal formulas for Schur functions (see Remarks~\ref{rem:schurEqualParameters} and~\ref{rem:spSchurEqualParameters}).
For instance, it can be easily checked that \eqref{eq:flatExpLPPdet} for $N=1$ and $\alpha_1=\beta_1=\gamma$ reduces to~\eqref{eq:flatExpLPPN=2}.

\begin{proof}
Since the inverse-gamma distribution scales to the exponential distribution in the sense of~\eqref{eq:invGammaToExp}, we can use~\eqref{eq:zeroTempLimitLaplaceTransf} (see also Proposition~\ref{prop:zeroTempLimit}-\ref{prop:zeroTempLimitLaplaceTransf} for details) to compute $\P(\fTau_{2N}\leq u)$.
Namely, we just need to compute $\lim_{\epsilon \downarrow 0} \E\left[\exp\left(\e^{-u/\epsilon}Z_{2N}^{(\epsilon)}\right)\right]$, where $Z_{2N}^{(\epsilon)}$ is the point-to-line 
$(\epsilon \bm{\alpha},\epsilon \bm{\beta},0)$-log-gamma polymer partition function.
Formula~\eqref{eq:flatWhittakerFormula} with $\gamma=0$ yields
\[
\E\left[\exp\left(-\e^{-u/\epsilon}Z_{2N}^{(\epsilon)}\right)\right] 
= \frac{\e^{-u \sum_{k=1}^N (\alpha_k + \beta_k) }}
{\fG_{\epsilon\bm{\alpha},\epsilon\bm{\beta},0}}
\int_{\R_{>0}^N}
\e^{-\e^{-u/\epsilon} x_1}
\Psi_{\epsilon\bm{\alpha}}^{\so_{2N+1}} (\bm{x})
\Psi_{\epsilon\bm{\beta}}^{\so_{2N+1}} (\bm{x})
\prod_{i=1}^N \frac{\diff x_i}{x_i} \, .
\]
The integral can be rewritten, changing variables $x_i \mapsto \e^{x_i/\epsilon}$ for $1\leq i\leq N$, as follows:
\[
\begin{split}
&\quad\, \int_{\R_{>0}^N}
\e^{-\e^{-u/\epsilon} x_1}
\Psi_{\epsilon\bm{\alpha}}^{\so_{2N+1}} (\bm{x})
\Psi_{\epsilon\bm{\beta}}^{\so_{2N+1}} (\bm{x})
\prod_{i=1}^N \frac{\diff x_i}{x_i} \\
&= \epsilon^{-2N^2}
\int_{\R^N} \e^{-\e^{(x_1-u)/\epsilon}}
\epsilon^{N^2} \Psi_{\epsilon\bm{\alpha}}^{\so_{2N+1}} \left(\e^{x_1/\epsilon},\dots,\e^{x_N/\epsilon}\right)
\epsilon^{N^2} \Psi_{\epsilon\bm{\beta}}^{\so_{2N+1}} \left(\e^{x_1/\epsilon},\dots,\e^{x_N/\epsilon}\right)
\prod_{i=1}^N \frac{\diff x_i}{\epsilon} \\
&\asymptotic{\epsilon\downarrow 0}
\epsilon^{-2N^2-N}
\int_{\{0< x_N< \cdots < x_1< u\}}  
\sp^{\cont}_{\bm{\alpha}}(\bm{x})  \sp^{\cont}_{\bm{\beta}}(\bm{x})
\prod_{i=1}^N \diff x_i \, ,
\end{split}
\]
where the asymptotics follow from the fact that
$\e^{-\e^{(x_1-u)/\epsilon}} \xrightarrow{\epsilon \downarrow 0} \1_{\{x_1< u\}}$ for all $x_1\neq u$ and Proposition~\ref{prop:soWhittakerRescaling}.
On the other hand, using the definition of $\fG_{\epsilon\bm{\alpha},\epsilon\bm{\beta},0}$ given in~\eqref{eq:flatPolymerNormalization} and the asymptotics of the Gamma function near $0$, we have:
\[
\fG_{\epsilon\bm{\alpha},\epsilon\bm{\beta},0}
= \prod_{1\leq i,j\leq N}
\Gamma(\epsilon(\alpha_i + \beta_j))
\prod_{1\leq i\leq j\leq N}
\Gamma (\epsilon(\alpha_i + \alpha_j))
\Gamma (\epsilon(\beta_i + \beta_j))
\asymptotic{\epsilon\downarrow 0}
\epsilon^{-2N^2-N}
\fk_{\bm{\alpha},\bm{\beta}} \, .
\]
Therefore, \eqref{eq:flatExpLPPSchur} easily follows from the combination of the foregoing formulas.

We now further elaborate~\eqref{eq:flatExpLPPSchur} by making use of the determinantal formula~\eqref{eq:spSchurContDet}:
\[
\begin{split}
&\quad\, \e^{u \sum_{k=1}^N (\alpha_k + \beta_k)} \P(\fTau_{2N} \leq u) \\
&= \int_{\{0< x_N< \cdots < x_1< u\}}
\frac{\det\left(\e^{\alpha_j x_i} - \e^{-\alpha_j x_i}\right)_{i,j}
\det\left(\e^{\beta_j x_i} - \e^{-\beta_j x_i}\right)_{i,j}}
{k_{\bm{\alpha},\bm{\beta}} \, \prod_{i<j}(\alpha_i-\alpha_j) (\beta_i - \beta_j)
\prod_{i\leq j}(\alpha_i + \alpha_j) (\beta_i + \beta_j)}
\prod_{i=1}^N \diff x_i \\
&= \frac{1}{\det(C(\alpha_i,\beta_j))_{i,j}} \frac{1}{N!}
\int_{(0,u)^N}
\det\left(\e^{\alpha_j x_i} - \e^{-\alpha_j x_i}\right)_{i,j}
\det\left(\e^{\beta_j x_i} - \e^{-\beta_j x_i}\right)_{i,j}
\prod_{i=1}^N \diff x_i \\
&= \frac{1}{\det(C(\alpha_i,\beta_j))_{i,j}}
\det\left(\int_0^u
\left(\e^{\alpha_i x} - \e^{-\alpha_i x}\right)
\left(\e^{\beta_j x} - \e^{-\beta_j x}\right) \diff x
\right)_{i,j} \, ,
\end{split}
\]
where the implicit range for $(i,j)$ is $1\leq i,j\leq N$.
In the latter computation, we have used: the fact that, by the alternating property of the determinant, the integral over $\{0< x_N< \cdots < x_1< u\}$ is invariant by applying any permutation to the variables $x_i$'s; the definition of $\fk_{\bm{\alpha},\bm{\beta}}$; the expression~\eqref{eq:cauchyDet} for the Cauchy determinant; and the Cauchy-Binet identity~\eqref{eq:cauchyBinet} for the Lebesgue measure on the interval $(0,u)$.
We now use the multilinearity of the determinant to finally obtain~\eqref{eq:flatExpLPPdet}.
\end{proof}

\begin{remark}
\label{rem:detPointProcesses}
We have seen in the latter proof that, thanks to the determinantal structure of symplectic Schur functions and the Cauchy-Binet identity, \eqref{eq:flatExpLPPSchur} can be turned into the determinantal formula~\eqref{eq:flatExpLPPdet}.
On the other hand, there is a standard way to associate the density given by the product of two determinantal functions of the form $\det(f_j(x_i))_{1\leq i,j\leq N}$ and $\det(g_j(x_i))_{1\leq i,j\leq N}$ to a determinantal point process (see e.g.~\cite{johannson06}).
We might then wonder if the product of two continuous symplectic Schur functions on the right-hand side of~\eqref{eq:flatExpLPPSchur} would be suitable to define a determinantal point process.
However, notice that such a product is not integrable on the domain $\{0<x_N <\dots < x_1\}$, i.e.\ the associated measure is infinite and cannot be renormalized.
This is also reflected by the presence, in front of the integral in~\eqref{eq:flatExpLPPSchur}, of an exponential factor depending on $u$, which makes the formula converge as $u\to \infty$.
Consequently, there is no natural way to define a determinantal point process in our case.

Of course, a similar remark holds for our symplectic Schur functions' formula~\eqref{eq:flatGeomLPP} in the context of geometric LPP.
Notice that, on the other hand, Baik-Rains' formula~\eqref{eq:baikRainsFlatLPP} for the same model defines a measure, whose density is given by one standard Schur function, which can be renormalized and hence associated to a Pfaffian point process.

Analogous remarks also hold in the other two path geometries, both for geometric and exponential environment.
\end{remark}

\subsection{Point-to-half-line exponential model}
\label{subsec:hFlatExpLPP}

Here we express the law of the point-to-half-line exponential LPP $\hTau_{2N}$ in terms of an integral of (the continuum version of) a symplectic Schur function and a standard Schur function; this time we give a direct proof, using the $\RSK$ correspondence and its properties.
We still work with the $(\bm{\alpha},\bm{\beta})$-exponential measure of Definition~\ref{def:flatExpMeasure}, but we restrict it to the lattice $\hI_{2N} = \{(i,j)\in\Z_{>0}^2 \colon i+j\leq 2N+1, \, i\leq N\}$ which the waiting times of the point-to-half-line LPP $\tau_{2N}$ lie on.

\begin{theorem}
\label{thm:hFlatExpLPP}
The distribution of the point-to-half-line $(\bm{\alpha}, \bm{\beta})$-exponential LPP $\hTau_{2N}$ can be expressed in terms of continuous Schur functions as
\begin{equation}
\label{eq:hFlatExpLPPSchur}
\P(\hTau_{2N}\leq u)
= \frac{\e^{-u \sum_{k=1}^N ( \alpha_k+\beta_k)}}{\hk_{\bm{\alpha},\bm{\beta}}}
\int_{\{0< x_N< \cdots < x_1< u\}}  
\sp^{\cont}_{\bm{\alpha}}(\bm{x})
\schur^{\cont}_{\bm{\beta}}(\bm{x})
\prod_{i=1}^N \diff x_i
\end{equation}
for all $u>0$, where
\begin{equation}
\label{eq:hFlatExpLPPnormalization}
\hk_{\bm{\alpha},\bm{\beta}}
:= \prod_{1\leq i,j\leq N} \frac{1}{\alpha_i + \beta_j}
\prod_{1\leq i\leq j\leq N} \frac{1}{\alpha_i + \alpha_j} \, .
\end{equation}
We can further write~\eqref{eq:hFlatExpLPPSchur} as a ratio of $N\times N$ determinants:
\begin{equation}
\label{eq:hFlatExpLPPdet}
\P\left(\hTau_{2N}\leq u\right) = 
\frac{\det\left(\hH_u(\alpha_i,\beta_j)\right)_{1\leq i,j\leq N}}{\det(C(\alpha_i,\beta_j))_{1\leq i,j\leq N}} \, ,
\end{equation}
where $C(z,w) := (z+w)^{-1}$ and
\[
\hH_u(z,w) := \e^{-u(z + w)} \int_0^u (\e^{z x}-\e^{-z x})
\e^{w x} \diff x \, .
\]
\end{theorem}

Notice again that, in case $\alpha_i=\alpha_j$ or $\beta_i=\beta_j$ for some $i\neq j$, \eqref{eq:hFlatExpLPPdet} should be understood in the limit as $\alpha_i - \alpha_j \to 0$ or $\beta_i - \beta_j \to 0$ respectively.

\begin{proof}
Let $\bm{W}=\{W_{i,j}\colon (i,j)\in\hI_{2N}\}$ be distributed according to the $(\bm{\alpha},\bm{\beta})$-exponential measure, i.e.
\[
\P(\bm{W} \in\diff \bm{w})
= \frac{1}{\hk_{\bm{\alpha},\bm{\beta}}}
\prod_{i=1}^N \e^{-\alpha_i \sum_{j=1}^{2N-i+1} w_{i,j}}
\prod_{j=1}^N
\e^{-\beta_j \sum_{i=1}^{N} w_{i,j}}
\e^{-\alpha_j \sum_{i=1}^j w_{i,2N-j+1}}
\prod_{(i,j)\in\hI_{2N}} \diff w_{i,j}
\]
for $\bm{w} \in \R_{\geq 0}^{\hI_{2N}}$.
By Proposition~\ref{prop:RSKproperties} (notice that the volume preserving property~\ref{prop:RSKproperties_Jacobian} also needs to be used here), the distribution that $\bm{W}$ induces on its $\RSK$ image $\bm{T}$ is then given by
\[
\begin{split}
\P(\bm{T} \in \diff\bm{t})
= \, &\frac{1}{\hk_{\bm{\alpha},\bm{\beta}}}\prod_{i=1}^N
\e^{-\alpha_i [\sigma_{2N-2i+1}(\bm{t}) - \sigma_{2N-2i+2}(\bm{t})]} \\
&\times \prod_{j=1}^N
\e^{-\beta_j [\sigma_{-(N-j)}(\bm{t}) - \sigma_{-(N-j+1)}(\bm{t})]}
\e^{-\alpha_j [\sigma_{2N-2j+1}(\bm{t}) - \sigma_{2N-2j}(\bm{t})]}
\prod_{(i,j)\in\hI_{2N}} \diff t_{i,j}
\end{split}
\]
for all $\bm{t}\in\R_{\geq 0}^{\hI_{2N}}$ satisfying the ordering~\eqref{eq:RSKordering}.

By~\eqref{eq:hFlatLPP=maxPointToPointLPP} and Proposition~\ref{prop:RSKproperties}-\ref{prop:RSKproperties_LPP}, we can obtain $\P(\hTau_{2N} \leq u)$ by integrating the density above over all $\bm{t}\in\R_{\geq 0}^{\hI_{2N}}$ satisfying the ordering~\eqref{eq:RSKordering} and the additional condition that $t_{i,j}\leq u$ for all $(i,j)\in\hI_{2N}$.
Let us now change variables, by setting $z'_{i,j} := u - t_{j,2N+j-i}$ for all $1\leq i\leq 2N$, $1\leq j\leq \ceil{i/2}$, and $z_{i,j} := u - t_{N+j-i,j}$ for all $1\leq j\leq i\leq N$.
Thanks to the inequalities that the $t_{i,j}$'s satisfy, the arrays $\bm{z}'$ and $\bm{z}$ just defined are \emph{real} Gelfand-Tsetlin patterns (sympletic and standard respectively) with common shape $\bm{x} = u-(t_{1,1},\dots,t_{N,N})$, such that $x_1\leq u$.
We then obtain:
\[
\begin{split}
\P(\hTau_{2N} \leq u)
= \frac{\e^{-u \sum_{k=1}^N ( \alpha_k+\beta_k)}}{\hk_{\bm{\alpha},\bm{\beta}}}
&\int_{\{0\leq x_N \leq \dots \leq x_1 \leq u\}}
\prod_{i=1}^N \diff x_i \\
&\times \int_{\spGT{2N}{\R}(\bm{x})}
\prod_{\substack{1\leq i< 2N \\ 1\leq j\leq \ceil{i/2}}} \!\! \diff z'_{i,j}
\prod_{k=1}^N \e^{\alpha_k [\type(\bm{z}')_{2k-1}- \type(\bm{z}')_{2k}]} \\
&\times \int_{\GT{N}{\R}(\bm{x})}
\prod_{1\leq j\leq i< N} \!\! \diff z_{i,j}
\prod_{k=1}^N \e^{\beta_k \type(\bm{z})_{k}} \, .
\end{split}
\]
The claim then follows from the fact that the second and the third integrals in the expression above are the desired continuous Schur functions, symplectic and standard respectively, according to Definitions~\ref{def:schurCont} and~\ref{def:spSchurCont}.

We omit the proof of~\eqref{eq:hFlatExpLPPdet}, which uses the Cauchy-Binet identity in the same fashion as in the proof of~\eqref{eq:flatExpLPPdet}.
\end{proof}

\subsection{Restricted point-to-line exponential model}
\label{subsec:rFlatExpLPP}

We finally study the restricted point-to-line exponential LPP model.
We will prove that the CDF of $\rTau_{2N}$ can be expressed in terms of an integral of one continuous symplectic Schur function, this time by computing the scaling limit of the corresponding geometric model of subsection~\ref{subsec:rFlatGeomLPP}.

The exactly solvable exponential distribution on the waiting times is as follows.
\begin{definition}
\label{def:rFlatExpMeasure}
Let $N\in\Z_{>0}$ and $\bm{\alpha}\in\R_{>0}^N$.
We define the \emph{$\bm{\alpha}$-exponential measure} on the lattice $\rI_{2N}:= \{(i,j)\in\Z_{>0}^2 \colon i+j\leq 2N+1, \, i\leq j\}$ to be the law of an array $\{W_{i,j} \colon (i,j)\in\rI_{2N}\}$ of independent random variables such that:
\begin{equation}
\label{eq:rFlatExpMeasure}
W_{i,j} \sim
\begin{cases}
\Exp(\alpha_i) &1\leq i=j \leq N \, , \\
\Exp(\alpha_i + \alpha_j) &1\leq i<j\leq N \, , \\
\Exp(\alpha_i + \alpha_{2N-j+1}) &1\leq i\leq N \, , \,\, N < j\leq 2N-i+1 \, .
\end{cases}
\end{equation}
\end{definition}

\begin{theorem}
The law of the restricted point-to-line $\bm{\alpha}$-exponential LPP $\rTau_{2N}$ is given by
\begin{equation}
\label{eq:rFlatExpLPPSchur}
\P(\rTau_{2N}\leq u)
= \frac{\e^{-u \sum_{k=1}^N \alpha_k}}{\rk_{\bm{\alpha},\bm{\beta}}}
\int_{\{0< x_N< \cdots < x_1< u\}}  
\sp^{\cont}_{\bm{\alpha}}(\bm{x})
\prod_{i=1}^N \diff x_i \, ,
\end{equation}
for all $u>0$, where
\begin{equation}
\label{eq:rFlatExpLPPnormalization}
\rk_{\bm{\alpha}}
:= \prod_{1\leq i\leq N} \frac{1}{\alpha_i}
\prod_{1\leq i<j\leq N} \frac{1}{\alpha_i + \alpha_j}
\prod_{1\leq i\leq j\leq N} \frac{1}{\alpha_i + \alpha_j} \, .
\end{equation}
We can further write~\eqref{eq:rFlatExpLPPSchur} in a Pfaffian form as
\begin{equation}
\label{eq:rFlatExpLPPpfaffian}
\P(\rTau_{2N}\leq u)
= \frac{\Pf\left(\Phi^{(N)}_u\right)}{\Pf\left(S^{(N)}\right)} \, ,
\end{equation}
where matrices $\Phi^{(N)}_u$ and $S^{(N)}$ are skew-symmetric of order $N$ or $N+1$, according to whether $N$ is even or odd respectively, and are defined by
\[
\Phi^{(N)}_u(i,j) :=
\begin{dcases}
\int_0^u \int_0^u \sgn(y-x) \phi_i(x) \phi_j(y) \diff x \diff y
& 1\leq i,j\leq N , \\
\int_0^u \phi_i(x) \diff x
& 1\leq i\leq N, \,\, j=N+1; \,\, \text{$N$ odd}
\end{dcases}
\]
having set $\phi_j(x) := \alpha_j \e^{-u \alpha_j} (\e^{\alpha_j x} - \e^{-\alpha_j x})$ for $1\leq j\leq N$, and
\[
S^{(N)}(i,j) :=
\begin{dcases}
\frac{\alpha_j - \alpha_i}{\alpha_j + \alpha_i}
& 1\leq i,j\leq N , \\
1
&1\leq i\leq N, \,\, j=N+1; \,\, \text{$N$ odd.}
\end{dcases}
\]
\end{theorem}
The denominator $\Pf\left(S^{(N)}\right)$ in~\eqref{eq:rFlatExpLPPpfaffian} is known as Schur Pfaffian and, no matter the parity of $N$, it has the following explicit evaluation:
\begin{equation}
\label{eq:SchurPf}
\Pf\left(S^{(N)}\right)
= \prod_{1\leq i<j\leq N} \frac{\alpha_j -\alpha_i}{\alpha_j + \alpha_i} \, .
\end{equation}

Notice that the ratio on the right-hand side of~\eqref{eq:rFlatExpLPPpfaffian} takes the form $0/0$ whenever $\alpha_i = \alpha_j$ for some $i\neq j$; in this case, the formula still makes sense in the limit as $\alpha_i -\alpha_j \to 0$.

\begin{proof}
We will deduce~\eqref{eq:rFlatExpLPPSchur} from~\eqref{eq:rFlatGeomLPP} via exponential limit.
Recall that, if $X\sim \Exp(\gamma)$ and $X^{(\delta)} \sim \Geom(\e^{-\delta \gamma})$ for some fixed $\gamma>0$ and for all $\delta >0$, then $\delta X^{(\delta)} \xrightarrow{\delta \downarrow 0} X$ in law.
Therefore, if the law of $\bm{W}$ is $\bm{\alpha}$-exponential (see Definition~\ref{def:rFlatExpMeasure}) and the law of $\bm{W}^{(\delta)}$ is $\e^{-\delta \bm{\alpha}}$-geometric (see Definition~\ref{def:rFlatGeomMeasure}), where $\e^{-\delta \bm{\alpha}} := (\e^{-\delta\alpha_1},\dots,\e^{-\delta\alpha_N})$, then $\delta \bm{W}^{(\delta)} \xrightarrow{\delta \downarrow 0} \bm{W}$ in distribution.
It follows from this observation and from formula~\eqref{eq:rFlatGeomLPP} for the restricted point-to-line geometric LPP that
\[
\begin{split}
\P(\rTau_{2N}\leq u)
&= \lim_{\delta\downarrow 0} \P\left( \max_{\pi \in \rPi_{2N}} \sum_{(i,j)\in\pi} W^{(\delta)}_{i,j} \leq \frac{u}{\delta} \right) \\
&= \lim_{\delta\downarrow 0}
\frac{\e^{-u\sum_{k=1}^N \alpha_k}}{\rc_{\e^{-\delta \bm{\alpha}}}}
\sum_{\lambda_1 \leq u/\delta} \sp_{\bm{\lambda}}\left(\e^{-\delta\alpha_1},\dots,\e^{-\delta\alpha_N}\right) \, ,
\end{split}
\]
where the sum is over integer partitions $\bm{\lambda}$ of length not greater than $N$ and with first part not greater than $u/\delta$.
Recalling the definitions of the normalization constants~\eqref{eq:rFlatGeomLPPnormalization} and~\eqref{eq:rFlatExpLPPnormalization}, we see that
\[
\delta^{N^2+N} \rc_{\e^{-\delta \bm{\alpha}}} \xrightarrow{\delta\downarrow 0} \rk_{\bm{\alpha}} \, .
\]
On the other hand, using Definitions~\ref{def:spSchur} and~\ref{def:spSchurCont} of discrete and continuous symplectic Schur functions, one can easily prove via Riemann sum approximation that
\[
\delta^{N^2 + N} \sum_{\lambda_1 \leq u/\delta} \sp_{\bm{\lambda}}\left(\e^{-\delta\alpha_1},\dots,\e^{-\delta\alpha_N}\right)
\xrightarrow{\delta\downarrow 0}
\int_{\{0< x_N< \cdots < x_1< u\}}  
\sp^{\cont}_{-\bm{\alpha}}(\bm{x})
\prod_{i=1}^N \diff x_i \, .
\]
Formula~\eqref{eq:rFlatExpLPPSchur} follows from these observations and the fact that $\sp^{\cont}_{-\bm{\alpha}} = \sp^{\cont}_{\bm{\alpha}}$.

For the proof of~\eqref{eq:rFlatExpLPPpfaffian}, we first rewrite~\eqref{eq:rFlatExpLPPSchur} as
\[
\P(\rTau_{2N}\leq u)
= \frac{1}{\Pf\left(S^{(n)}_{\bm{\alpha}}\right)}
\int_{\{0\leq x_1\leq \dots\leq x_N \leq u\}} \det(\phi_j(x_i))_{1\leq i,j\leq N} \prod_{i=1}^N \diff x_i \, ,
\]
where $\phi_j(x) := \alpha_j \e^{-u \alpha_j} (\e^{\alpha_j x} - \e^{-\alpha_j x})$ for $1\leq j\leq N$.
Here we have used the determinantal form~\eqref{eq:spSchurContDet} of $\sp^{\cont}_{\bm{\alpha}}$, the change of variables $x_i \mapsto x_{N-i+1}$ for $1\leq i\leq N$, the multilinear and alternating properties of determinants, and the Schur Pfaffian expression~\eqref{eq:SchurPf}.
Applying de Bruijn identity~\eqref{eq:deBruijn} for the Lebesgue measure on $(0,u)$ to the integral above, we finally obtain the numerator Pfaffian in~\eqref{eq:rFlatExpLPPpfaffian}.
\end{proof}

\section{Scaling limits}
\label{sec:LPPasymptotics}

The goal of this section is to carry out the asymptotic analysis of the point-to-line and point-to-half-line exponential LPP models studied in subsections~\ref{subsec:flatExpLPP} and~\ref{subsec:hFlatExpLPP}.
Both models are characterized by KPZ fluctuations of order $N^{1/3}$.
In the scaling limit we obtain Sasamoto's formula~\cite{sasamoto05} for the GOE Tracy-Widom distribution and the one-point marginal of the $\rm{Airy}_{2\to 1}$ process~\cite{borodinFerrariSasamoto08}, respectively.

The procedure we follow is typical of KPZ and random matrix models.
We start from formulas~\eqref{eq:flatExpLPPdet} and~\eqref{eq:hFlatExpLPPdet}, which express the CDF of the corresponding LPP as a ratio of $N\times N$ determinants, the denominator being of Cauchy type.
The space $\R^N$ on which these determinants are taken is growing with $N$, so these formulas are not directly amenable to asymptotic analysis in the large $N$ limit.
We then use the so-called Sylvester's identity and the closed form for the inverse of a Cauchy-type matrix to turn such formulas into Fredholm determinants on $L^2(\R_{>0})$, with kernel expressed in the form of a contour integral.
This $L^2$ space is infinite-dimensional but fixed, in the sense that it does not depend on $N$.
One is thus reduced to compute the asymptotics of the kernel, which is a contour integral whose integrand depends on $N$, via the steepest descent method.
Finally, some standard estimates permit deducing the convergence of the Fredholm determinants from the convergence of their kernels.

In subsection~\ref{subsec:fredholm} we present a general scheme to turn a ratio of determinants into a Fredholm determinant.
In subsection~\ref{subsec:steepestDescent} we perform the steepest descent analysis of a central integral.
In subsections~\ref{subsec:flatLPPasymptotics} and~\ref{subsec:hFlatLPPasymptotics} we use these general tools to derive the scaling limits of the point-to-line and point-to-half-line LPP models respectively.

\subsection{From determinants to Fredholm determinants}
\label{subsec:fredholm}

In this subsection we present a general scheme to turn ratios of determinants, the denominator being of Cauchy type as~\eqref{eq:cauchyDet}, into a Fredholm determinant.
Such a scheme has been already used in a similar fashion for example in~\cite{johansson01, okounkov01, borodinGorin16}, but we adapt it here to our framework.
Let us start by briefly recalling the notion of a Fredholm determinant.
Given a measure space $(\mathcal{X},\nu)$, any linear operator $K\colon L^2(\mathcal{X},\nu) \to L^2(\mathcal{X},\nu)$ can be defined in terms of its integral kernel $K(x,y)$ by
\[
(Kh)(x) := \int_{\mathcal{X}} K(x,y)h(y) \nu(\diff y) \, ,\qquad h\in L^2(\mathcal{X},\nu) \, .
\]
The \emph{Fredholm determinant} of $K$ can then be defined through its series expansion:
\begin{equation}
\label{eq:fredholmDet}
\det(I + K)_{L^2(\mathcal{X})}
:= 1+\sum_{n=1}^{\infty} \frac{1}{n!} \int_{\mathcal{X}^n}
\det(K(x_i,x_j))_{i,j=1}^n \,
\nu(\diff x_1)\dots \nu(\diff x_n) \, ,
\end{equation}
whenever the series converges.
Denoting from now on by $\C_{>0} := \{z\in\C\colon \Re(z)>0\}$ the complex right half-plane, we can state:
\begin{theorem}
\label{thm:detToFredDet}
Let
\begin{equation}
\label{eq:Hbar}
H(z,w) := C(z,w) - \underline{H}(z,w) \, ,
\end{equation}
where $C$ is the function
\[
C(z,w)=\frac{1}{z+w}
\]
and $\underline{H}$ is a holomorphic function in the region $\C_{>0} \times \C_{>0}$.
For any choice of positive parameters $\alpha_1,\dots,\alpha_N,\beta_1,\dots,\beta_N$, define the operator $K_N$ on $L^2(\R_{>0})$ through the kernel
\begin{equation}
\label{eq:kernel}
K_N(\lambda,\xi) := \frac{1}{(2\pi\i)^2} \int_{\Gamma_1} \diff z \int_{\Gamma_2} \diff w \,
\e^{-\lambda z-\xi w}
\underline{H}(z,w) \,
\prod_{m=1}^N \frac{(z+\beta_m)(w+\alpha_m)}{(z-\alpha_m)(w-\beta_m)} \, ,
\end{equation}
where $\Gamma_1,\Gamma_2 \subset \C_{>0}$ are any positively oriented simple closed contours such that $\Gamma_1$ encloses $\alpha_1,\dots,\alpha_N$ and $\Gamma_2$ encloses $\beta_1,\dots,\beta_N$.
Then
\begin{equation}
\label{eq:detToFredDet}
\frac{\det(H(\alpha_i,\beta_j))_{1\leq i,j\leq N}}
{\det(C(\alpha_i,\beta_j))_{1\leq i,j\leq N}}
= \det(I - K_N)_{L^2(\R_{>0})} \, .
\end{equation}
\end{theorem}

\begin{proof}
For convenience, let us denote by $\mathcal{C}$, $\mathcal{H}$ and $\mathcal{\underline{H}}$ the $N\times N$ matrices $(C(\alpha_i,\beta_j))_{1\leq i,j\leq N}$, $(H(\alpha_i,\beta_j))_{1\leq i,j\leq N}$ and $(\underline{H}(\alpha_i,\beta_j))_{1\leq i,j\leq N}$ respectively. We then have:
\begin{equation}
\label{eq:detToFredDet1}
\frac{\det(\mathcal{H})}{\det(\mathcal{C})}
= \det\left( \mathcal{C}^{-1}\left( \mathcal{C} - \mathcal{\underline{H}} \right) \right)
= \det\left( I - \mathcal{C}^{-1} \mathcal{\underline{H}} \right) \, ,
\end{equation}
where $I$ is the identity matrix of order $N$.
To invert $\mathcal{C}$, we use Cramer's formula:
\[
\mathcal{C}^{-1}(i,k) = (-1)^{i+k} \frac{\det\left(\mathcal{C}^{(k,i)}\right)}{\det(\mathcal{C})} \, ,
\]
where $\mathcal{C}^{(k,i)}$ is the matrix of order $N-1$ obtained from $\mathcal{C}$ by removing its $k$-th row and $i$-th column.
In our case, both determinants in the above formula are of Cauchy type:
\begin{align*}
\det(\mathcal{C})
&= \prod_{l<m} (\alpha_l -\alpha_m)\prod_{l<m} (\beta_l - \beta_m)
\prod_{l,m} (\alpha_l + \beta_m)^{-1} \\
\det\left(\mathcal{C}^{(k,i)}\right)
&= \prod_{\substack{l<m \\ l,m\neq k}} (\alpha_l -\alpha_m)
\prod_{\substack{l<m \\ l,m\neq i}}(\beta_l - \beta_m)
\prod_{\substack{l\neq k \\ m\neq i}} (\alpha_l + \beta_m)^{-1} \, ,
\end{align*}
where indices $l$ and $m$ run in $\{1,\dots,N\}$.
The inverse of $\mathcal{C}$ is thus readily computed: 
\[
\mathcal{C}^{-1}(i,k)
= \frac{\prod_{m=1}^N
(\alpha_k + \beta_m)
(\beta_i + \alpha_m)}
{(\alpha_k + \beta_i)
\prod_{m\neq k}(\alpha_k-\alpha_m)
\prod_{m\neq i} (\beta_i - \beta_m)} \, .
\]
Writing $\frac{1}{\alpha_k + \beta_i} = \int_0^{\infty} \e^{-(\alpha_k + \beta_i)\lambda} \diff \lambda $, we obtain:
\[
\left(\mathcal{C}^{-1} \mathcal{\underline{H}} \right)(i,j)
= \sum_{k=1}^N \mathcal{C}^{-1}(i,k) \mathcal{\underline{H}}(k,j)
= \int_0^{\infty} f(i,\lambda ) g(\lambda ,j) \diff \lambda  \, ,
\]
where for all $\lambda >0$
\[
f(i,\lambda ) := \e^{-\beta_i \lambda } \frac{\prod_{m=1}^N(\beta_i + \alpha_m)}
{\prod_{m\neq i}(\beta_i -\beta_m)} \, ,
\qquad
g(\lambda ,j) := \sum_{k=1}^N \e^{-\alpha_k \lambda } \frac{\prod_{m=1}^N (\alpha_k + \beta_m)}
{\prod_{m\neq k} (\alpha_k - \alpha_m)}
\underline{H}(\alpha_k,\beta_j) \, .
\]
This proves that matrix $\mathcal{C}^{-1} \mathcal{\underline{H}}$, viewed as a linear operator on $\R^N$, equals the composition $FG$, where $F$ and $G$ are the linear operators
\begin{align*}
&F \colon L^2(\R_{>0}) \to \R^N \, ,
&&\phi \mapsto \left[ \int_0^{\infty} f(i,\lambda ) \phi(\lambda ) \diff \lambda  \right]_{i=1}^N \, , \\
&G \colon \R^N \to L^2(\R_{>0}) \, ,
&&(a(j))_{j=1}^N \mapsto \sum_{j=1}^N g(\lambda ,j) a(j) \, .
\end{align*}
We note that these are well-defined operators, as $f(i, \cdot)$ and $g(\cdot,j)$ are square integrable functions on $\R_{>0}$, for all $i$ and $j$. 
We will next use Sylvester's identity:
\begin{equation}
\label{eq:sylvester}
\det(I + K_1 K_2)_{L^2(\mathcal{X}_2)}
= \det(I + K_2 K_1)_{L^2(\mathcal{X}_1)}
\end{equation}
for any operators $K_1\colon L^2(\mathcal{X}_1) \to L^2(\mathcal{X}_2)$ and $K_2\colon L^2(\mathcal{X}_2) \to L^2(\mathcal{X}_1)$ such that both sides converge.
By applying this identity, we obtain
\begin{equation}
\label{eq:detToFredDet2}
\det\left( I - \mathcal{C}^{-1} \mathcal{\underline{H}} \right)_{\R^N}
= \det(I - F G)_{\R^N}
= \det(I - K_N)_{L^2(\R_{>0})} \, ,
\end{equation}
where $K_N:= G F$ is the operator on $L^2(\R_{>0})$ defined through the kernel
\begin{equation}
\label{eq:kernelDoubleSum}
K_N(\lambda ,\xi ) :=
\sum_{i,k=1}^N \e^{-\alpha_k \lambda  - \beta_i \xi }
\underline{H}(\alpha_k, \beta_i)
\frac{\prod_{m=1}^N
(\alpha_k + \beta_m)
(\beta_i + \alpha_m)}
{\prod_{m\neq k}(\alpha_k-\alpha_m)
\prod_{m\neq i} (\beta_i - \beta_m)} \, .
\end{equation}
From the latter formula, it is clear that $|K_N(\lambda,\xi)| \leq c_1 \e^{- c_2 \lambda}$ for all $\lambda\in [0,\infty)$, where the positive constants $c_1$ and $c_2$ depend on $N$ and on the parameters.
Hadamard's bound then implies that
\[
\abs{ \det(K_N(\lambda_i, \lambda_j))_{i,j=1}^n }
\leq n^{n/2} \prod_{i=1}^n c_1 \e^{-c_2 \lambda_i } \, .
\]
It then follows from the series expansion~\eqref{eq:fredholmDet} that
\[
\begin{split}
\abs{ \det(I - K_N)_{L^2(\R_{>0})} }
\leq 1 + \sum_{n=1}^{\infty} \frac{n^{n/2}}{n!} \left( \int_{0}^{\infty} c_1 \e^{-c_2 \lambda} \diff \lambda \right)^n
< \infty \, ,
\end{split}
\]
hence the right-hand side of~\eqref{eq:detToFredDet2} is indeed a converging Fredholm determinant.
By applying the residue theorem (recalling the assumption that $\underline{H}(z,w)$ is holomorphic in $\C_{>0} \times \C_{>0}$), the double sum in~\eqref{eq:kernelDoubleSum} can be turned into a double contour integral, yielding representation~\eqref{eq:kernel} for the kernel.
By combining~\eqref{eq:detToFredDet1} and~\eqref{eq:detToFredDet2} we obtain~\eqref{eq:detToFredDet}.
\end{proof}

Thanks to Theorem~\ref{thm:detToFredDet}, the determinantal formulas~\eqref{eq:flatExpLPPdet} and~\eqref{eq:hFlatExpLPPdet} for the point-to-line and point-to-half-line exponential LPP models will be turned into Fredholm determinants.
We will see this in details in subsections~\ref{subsec:flatLPPasymptotics} and~\ref{subsec:hFlatLPPasymptotics}.

\subsection{Steepest descent analysis}
\label{subsec:steepestDescent}

As announced in the previous subsection, the CDF of the point-to-line and point-to-half-line LPP models will be expressed as Fredholm determinants on $L^2(\R_{>0})$ with kernels of type~\eqref{eq:kernel}.
In the limit $N\to\infty$ these kernels will converge, after rescaling, to expressions involving Airy functions.
In order to see this, one needs to perform the asymptotic analysis of a few contour integrals via steepest descent.
This procedure is very similar in all cases, as it always involves the same functions.
Therefore, we will carry it out in detail only for one of such contour integrals, arguably the most archetypal one as it just approximates the Airy function.
Other very similar steepest descent analyses are sketched where needed, specifically in the proof of Theorem~\ref{thm:hFlatLPPasymptotics}.

Let us first recall that the Airy function $\Ai$ has the following contour integral representation:
\begin{equation}
\label{eq:Airy}
\Ai(x) := \frac{1}{2\pi \i}
\int_{\e^{-\i\pi/3} \infty}^{\e^{\i\pi/3} \infty}
\exp\left\{\frac{z^3}{3} -xz \right\} \diff z \, ,
\end{equation}
where the integration path starts at infinity with argument $-\pi/3$ and ends at infinity with argument $\pi/3$ (see for example the red contour in figure~\ref{fig:steepestDescentPath}).

\begin{proposition}
\label{prop:steepestDescent}
For any fixed $\gamma>0$ and $f:=2/\gamma$, let us define
\begin{equation}
\label{eq:steepestDescentInt}
J_N(x) := -\frac{1}{2\pi\i} \int_{\Gamma}
\e^{-z(fN + x)}
\left[\frac{\gamma + z}{\gamma - z}\right]^N \diff z \, ,
\end{equation}
where $\Gamma\subset \C$ is a positively oriented contour enclosing $\gamma$.
Then, for all $x\in\R$,
\begin{equation}
\label{eq:convergence2Airy}
\tilde{J}_N(x)
:= \frac{\sqrt[3]{2N}}{\gamma} J_N\left( \frac{\sqrt[3]{2N}}{\gamma} x \right)
\xrightarrow{N\to\infty} \Ai(x) \, .
\end{equation}
\end{proposition}

\begin{proof}
The proof is based on the steepest descent analysis of integral~\eqref{eq:steepestDescentInt}, which we rewrite as
\[
J_N(x)
= - \frac{1}{2\pi\i} \int_{\Gamma}
\exp\left\{ F(z) N - x z \right\} \diff z \, ,
\]
being $F(z):= \log(\gamma+z) - \log(\gamma - z) - fz$. 
We need to compute the critical points of the function $F$, whose first three derivatives are given by:
\begin{align*}
F'(z) &= \frac{1}{\gamma+z} + \frac{1}{\gamma-z} - f \, , \\
F''(z) &= -\frac{1}{(\gamma+z)^2} + \frac{1}{(\gamma-z)^2} \, , \\
F'''(z) &= \frac{2}{(\gamma+z)^3} + \frac{2}{(\gamma-z)^3} \, .
\end{align*}
The second derivative vanishes if and only if $z=0$.
As in the statement of the theorem, we then set $f:=2/\gamma$, which is the only value of $f$ such that the first derivative also vanishes at $z=0$.
The first non-vanishing derivative of $F$ at the critical point $z=0$ is then the third one.
In particular, we have that
\[
F(0)=F'(0)=F''(0)=0 \, , \qquad
F'''(0)=\frac{4}{\gamma^3} \, ,
\]
hence the Taylor expansion of $F$ near the critical point is
\begin{equation}
\label{eq:expansionCriticalPoint}
F(z) = \frac{2}{3\gamma^3} z^3 + R(z) \, ,
\end{equation}
where $R(z) = o(z^3)$ as $z\to 0$.
\begin{figure}[h!]
\begin{center}
\begin{tikzpicture}[scale=1.43]
\draw (0,-3.2) -- (0,3.2);
\draw (0,0) -- (2.2,0);

\node[label={[label distance=1pt]180:$0$},draw,circle,inner sep=1pt,fill] (origin) at (0,0) {};

\node[label={[label distance=1pt]0:$2a\e^{\i\pi/3}$},draw,circle,inner sep=1pt,fill] (V+) at (1,1.73) {};

\node[label={[label distance=1pt]0:$2a\e^{-\i\pi/3}$},draw,circle,inner sep=1pt,fill] (V-) at (1,-1.73) {};

\node[label={[label distance=1pt]315:$a$},draw,circle,inner sep=1pt,fill] (a) at (1,0) {};

\node[label={[label distance=1pt]270:$\gamma$},draw,circle,inner sep=1pt,fill] (gamma) at (0.7,0) {};

\node[label={[label distance=5pt]0:$T_a$}] (gamma) at (1,0.8) {};

\node[inner sep=0pt, label={[label distance=5pt]0:\red{$C$}}] (W+) at (1.515,2.595) {};
\node[inner sep=0pt] (W++) at (1.715,2.941) {};
\node[inner sep=0pt] (W-) at (1.515,-2.595) {};
\node[inner sep=0pt] (W--) at (1.715,-2.941) {};

\begin{scope}[thick,decoration={
    markings,
    mark=at position 0.5 with {\arrow[scale=2.5]{stealth}}}
    ] 
\draw[postaction={decorate}] (origin) -- (V+);
\draw[postaction={decorate}] (V-) -- (origin);
\end{scope}

\begin{scope}[thick,decoration={
    markings,
    mark=at position 0.35 with {\arrow[scale=2.5]{stealth}}}
    ] 
\draw[postaction={decorate}] (V+) -- (V-);
\end{scope}
    
\begin{scope}[thick,red]
\draw[postaction={decorate}, decoration={markings, mark=at position 0.15 with {\arrow{>}}}]
(W-) -- (0.015,0);
\draw[postaction={decorate}, decoration={markings, mark=at position 0.85 with {\arrow{>}}}]
(0.015,0) -- (W+);
\draw[dashed] (W+) -- (W++);
\draw[dashed] (W-) -- (W--);
\end{scope}

\end{tikzpicture}
\end{center}
\caption[Contours in the steepest descent analysis of an integral that converges to the Airy function]{The red path $C$ is involved in the integral representation of the Airy function. The black contour $T_a$ refers to the steepest descent analysis in the proof of Proposition~\ref{prop:steepestDescent}.}
\label{fig:steepestDescentPath}
\end{figure}
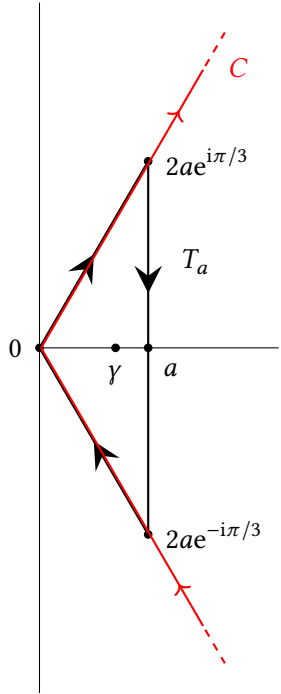
Since the directions of steepest descent of $F$ from $z=0$ correspond to the angles $\pm \pi/3$, we deform the \emph{positively} oriented contour $\Gamma$ into the \emph{negatively} oriented triangular path $T_a$ with vertices $0$, $2a\e^{\i\pi/3}$ and $2a\e^{-\i\pi/3}$ for some $a>\gamma$ (so that the pole $z=\gamma$ is still enclosed, see figure~\ref{fig:steepestDescentPath}).
This only implies a change of sign in the integral, corresponding to the change of orientation of the contour.
In fact, in order to obtain the right estimates in the proof of Corollary~\ref{coro:steepestDescentEstimate}, it is convenient to consider an infinitesimal shift of $T_a$, by setting the contour to be $T_a + \epsilon\gamma/\sqrt[3]{2N}$, where $\epsilon > 0$ is an arbitrary constant.
Moreover, we split the integral into two regions, i.e.\ a neighborhood of the critical point, where the main contribution of the integral is expected to come from, and its exterior (we choose the neighborhood to be a ball centered at $\epsilon\gamma / \sqrt[3]{2N}$ with radius $\gamma N^{-\alpha}$, where $\alpha>0$ will be suitably specified later on):
\[
J_N(x)
= J_N^{\rm in}(x) + J_N^{\rm ex}(x) \, ,
\]
where
\begin{align*}
J_N^{\rm in}(x)
&:= \frac{1}{2\pi\i} \int_{T_a + \frac{\epsilon\gamma}{\sqrt[3]{2N}}}
\exp\left\{ F(z) N - x z \right\}
\1_{\left\{\abs{z - \epsilon\gamma/\sqrt[3]{2N}} \leq \gamma N^{-\alpha}\right\}} \diff z \, , \\
J_N^{\rm ex}(x)
&:= \frac{1}{2\pi\i} \int_{T_a + \frac{\epsilon\gamma}{\sqrt[3]{2N}}}
\exp\left\{ F(z) N - x z \right\}
\1_{\left\{\abs{z - \epsilon\gamma/\sqrt[3]{2N}} \geq \gamma N^{-\alpha}\right\}} \diff z \, .
\end{align*}

Let us first focus on the former integral and denote by $C$ the piecewise linear path going from the point at infinity with argument $-\pi/3$ to the origin to the point at infinity with argument $\pi/3$ (see figure~\ref{fig:steepestDescentPath}).
We then have that
\[
J_N^{\rm in}(x)
= \frac{1}{2\pi\i} \int_{C + \frac{\epsilon\gamma}{\sqrt[3]{2N}}}
\exp\left\{ \frac{2N}{3\gamma^3} z^3 - x z + R(z) N \right\}
\1_{\left\{\abs{z - \epsilon\gamma/\sqrt[3]{2N}} \leq \gamma N^{-\alpha}\right\}} \diff z \, ,
\]
where $R(z)$ is defined by~\eqref{eq:expansionCriticalPoint}.
If we now rescale both the integration variable and the function  $J_N^{\rm in}$ by the factor $\sqrt[3]{2N}/\gamma$, by setting $\tilde{z} := z \sqrt[3]{2N}/\gamma$ and defining $\tilde{J}_N^{\rm in}$ as in~\eqref{eq:convergence2Airy}, we obtain:
\begin{equation}
\label{eq:steepestDescentIn}
\tilde{J}_N^{\rm in}(x)
= \frac{1}{2\pi\i} \int_{C + \epsilon}
\exp\left\{ \frac{\tilde{z}^3}{3} - x \tilde{z}
+  R\left( \frac{\gamma}{\sqrt[3]{2N}} \tilde{z} \right) N \right\}
\1_{\left\{\abs{\tilde{z} - \epsilon} \leq \sqrt[3]{2} N^{1/3-\alpha}\right\}}
\diff \tilde{z} \, .
\end{equation}
A standard estimate of the remainder in the Taylor expansion~\eqref{eq:expansionCriticalPoint} yields
\[
\abs{R\left( \frac{\gamma}{\sqrt[3]{2N}} \tilde{z} \right) N} 
\leq \frac{m}{4!} \abs{\frac{\gamma}{\sqrt[3]{2N}} \tilde{z}}^4 N
\leq \frac{m}{4!} \left(\gamma N^{-\alpha} + \frac{\epsilon\gamma}{\sqrt[3]{2N}} \right)^4 N
\]
for $\abs{\tilde{z} - \epsilon} \leq \sqrt[3]{2} N^{1/3-\alpha}$, where the constant $m$ is the maximum modulus of $F^{(4)}$ in some fixed neighborhood of the origin.
If we take $\alpha > 1/4$, the above expression vanishes as $N\to\infty$.
If we further choose $\alpha < 1/3$, the indicator function in~\eqref{eq:steepestDescentIn} converges to $1$, yielding
\begin{equation}
\label{eq:steepestDescentInConv}
\exp\left\{ R_N\left( \frac{\gamma}{\sqrt[3]{2N}} \tilde{z} \right) N \right\} \,
\1_{\left\{\abs{\tilde{z} - \epsilon} \leq \sqrt[3]{2} N^{1/3-\alpha}\right\}}
\xrightarrow{N\to\infty} 1 \, .
\end{equation}
Since the argument of the points of $C$ is $\pm \pi/3$, we have that
\[
\int_{C + \epsilon}
\abs{\exp\left\{\frac{\tilde{z}^3}{3} -x\tilde{z} \right\} } \abs{\diff \tilde{z}}
< \infty \, ,
\]
hence by dominated convergence
\[
\tilde{J}_N^{\rm in}(x)
\xrightarrow{N\to\infty} \int_{C + \epsilon}
\exp\left\{\frac{\tilde{z}^3}{3} -x\tilde{z} \right\} \diff \tilde{z}
= \Ai(x) \, .
\]
Observe that, varying $\epsilon$, we have different integral representations of the Airy function, which are all equivalent thanks to~\eqref{eq:Airy}.

To conclude the proof, it remains to show that
\[
\tilde{J}_N^{\rm ex}(x)
:= \frac{\sqrt[3]{2N}}{2\pi\i\gamma} \int_{T_a \cap \{\abs{z} \geq \gamma N^{-\alpha}\}}
\exp\left\{ F\left(z + \frac{\epsilon\gamma}{\sqrt[3]{2N}} \right) N - x \left(z + \frac{\epsilon\gamma}{\sqrt[3]{2N}} \right) \frac{\sqrt[3]{2N}}{\gamma} \right\}
\diff z
\xrightarrow{N\to\infty} 0 \, .
\]
We may decompose the integration domain as
\[
T_a \cap \{\abs{z} \geq \gamma N^{-\alpha}\}
= \mathcal{V}
\cup \mathcal{O}_N
\cup \overline{\mathcal{V}}
\cup \overline{\mathcal{O}}_N \, ,
\]
where $\mathcal{V}$ and $\mathcal{O}_N$ are the vertical and oblique segments respectively given by
\[
\mathcal{V} := \left\{\Re(z)=a \, , \,\, 0 \leq \arg(z) \leq \frac{\pi}{3} \right\} \, ,
\quad
\mathcal{O}_N := \left\{ \gamma N^{-\alpha} \leq \abs{z} \leq 2a \, , \,\, \arg(z) = \frac{\pi}{3} \right\} \, ,
\]
and $\overline{\mathcal{V}}$ and $\overline{\mathcal{O}}_N$ are their complex conjugates.
We thus estimate
\begin{equation}
\label{eq:steepestDescentExt}
\abs{\tilde{J}_N^{\rm ex}(x)}
\leq \frac{\mathcal{L}(T_a) \sqrt[3]{2N}}{2\pi\gamma} \exp\left\{ \max \left[ G_N\left(z + \frac{\epsilon\gamma}{\sqrt[3]{2N}} \right) \colon z\in \mathcal{V} \cup \mathcal{O}_N \cup \overline{\mathcal{V}} \cup \overline{\mathcal{O}}_N \right] \right\} \, ,
\end{equation}
where $\mathcal{L}(\cdot)$ denotes the length of a contour and
\[
G_N(z) := \Re[F(z)] N - x \, \Re(z) \frac{\sqrt[3]{2N}}{\gamma} \, .
\]
Since
\[
\Re[F(z)] = \log\abs{\frac{\gamma+z}{\gamma-z}} - \frac{2}{\gamma} \Re(z) \, ,
\]
it is clear that $G_N(\overline{z}) = G_N(z)$.
Therefore, it suffices to bound the maximum of $G_N$ over $\mathcal{V}$ and over $\mathcal{O}_N$.
Since $\Re(z)=a$ and $a\leq \abs{z}\leq 2a$ for $z\in\mathcal{V}$, we have that
\begin{equation}
\label{eq:steepestDescentExtVert}
\max_{z\in\mathcal{V}} G_N\left(z + \frac{\epsilon\gamma}{\sqrt[3]{2N}} \right)
\leq - c N - x a \frac{\sqrt[3]{2N}}{\gamma} \, ,
\end{equation}
where
\begin{equation}
\label{eq:steepestDescentExtVertConst}
c := \frac{2}{\gamma} a - \log\left(\frac{2a + \epsilon\gamma +\gamma}{a -\gamma}\right) \, .
\end{equation}
If we fix a large enough $a$ such that $c$ is positive, the maximum in~\eqref{eq:steepestDescentExtVert} is asymptotically bounded above by $-cN$ and diverges to $-\infty$.
On the other hand, for all $z$ such that $\arg(z)=\pi/3$, we have that
\[
G_N(z)
= \left[ \frac{1}{2} \log\frac{\gamma^2 + \gamma\abs{z} + \abs{z}^2}
{\gamma^2 - \gamma\abs{z} + \abs{z}^2}
- \frac{\abs{z}}{\gamma} \right] N
- x \frac{\abs{z}}{2} \frac{\sqrt[3]{2N}}{\gamma} \, ,
\]
whose derivative w.r.t.\ the modulus is
\[
\frac{\diff}{\diff \abs{z}} G_N(z)
= -\frac{\abs{z}^2(2\gamma^2+\abs{z}^2)}
{\gamma(\gamma^4+\gamma^2\abs{z}^2+\abs{z}^4)} N
- \frac{x \sqrt[3]{2N}}{2\gamma} \, .
\]
A trivial estimate then gives
\[
\frac{\diff}{\diff \abs{z}} G_N(z)
\leq - \frac{\gamma N^{1-2\alpha} (2\gamma^2 + \gamma^2 N^{-2\alpha})}
{(\gamma^4 + \gamma^2 (2a)^2 + (2a)^4)}
- \frac{x \sqrt[3]{2N}}{2\gamma}
\qquad
\text{for } z\in\mathcal{O}_N \, .
\]
Since $\alpha < 1/3$, no matter the sign of $x$, the above derivative is negative for $N$ large enough, so $G_N(z)$ is decreasing w.r.t.\ $\abs{z}$ in $\mathcal{O}_N$.
By continuity, for $N$ large, $G_N(z + \epsilon\gamma/\sqrt[3]{2N})$ is also decreasing w.r.t.\ $\abs{z}$ 
in $\mathcal{O}_N$, hence
\[
\begin{split}
&\max_{z\in \mathcal{O}_N} G_N\left(z + \frac{\epsilon\gamma}{\sqrt[3]{2N}} \right)
= G_N\left(\gamma N^{-\alpha} \e^{\i\pi/3} + \frac{\epsilon\gamma}{\sqrt[3]{2N}} \right) \\
= \, &\left[ \log\Bigg| \frac{1 + N^{-\alpha} \e^{\i\pi/3} + \epsilon(2N)^{-1/3}}
{1 - N^{-\alpha} \e^{\i\pi/3} - \epsilon(2N)^{-1/3}} \Bigg|
- N^{-\alpha} -\frac{2\epsilon}{\sqrt[3]{2N}} \right] N
- x\left(\frac{N^{-\alpha}}{2} + \frac{\epsilon}{\sqrt[3]{2N}} \right) \sqrt[3]{2N} \, .
\end{split}
\]
After a tedious computation, which uses the third order Taylor expansion of $\log(1+\delta)$ as $\delta \to 0$, we obtain that
\begin{equation}
\label{eq:steepestDescentExtObl}
\max_{z\in \mathcal{O}_N} G_N\left(z + \frac{\epsilon\gamma}{\sqrt[3]{2N}} \right)
= -\frac{2}{3} N^{1-3\alpha} + o(N^{1-3\alpha})
- x (2^{-2/3} N^{1/3-\alpha}
+ \epsilon) \, .
\end{equation}
Since $\alpha < 1/3$, the latter expression is asymptotic to $-(2/3)N^{1-3\alpha}$ and diverges to $-\infty$.
Thanks to estimates~\eqref{eq:steepestDescentExt}, \eqref{eq:steepestDescentExtVert} and~\eqref{eq:steepestDescentExtObl}, we thus conclude that $\tilde{J}_N^{\rm ex}(x)$ vanishes (at least choosing a large enough $a$).
\end{proof}

The proof of Proposition~\ref{prop:steepestDescent} directly provides a uniform bound on $\tilde{J}_N$, which will turn out to be useful in the next subsections.

\begin{corollary}
\label{coro:steepestDescentEstimate}
Let $\tilde{J}_N(x)$ be defined as in~\eqref{eq:convergence2Airy} and $s\in\R$.
Then, there exist two positive constants $c_1$ and $c_2$ such that
\[
\sup_{N\in\N} \big| \tilde{J}_N(x) \big| \leq c_1 \e^{-c_2 x} \quad\qquad
\forall x\in [s,\infty) \, .
\]
\end{corollary}
\begin{proof}
Since by continuity the converging sequence $\tilde{J}_N(x)$ is bounded uniformly in $N$ on any compact set, it suffices to prove the claim for $s=0$.
The proof is then a straightforward consequence of the estimates obtained in the proof of Proposition~\ref{prop:steepestDescent}.
Using the notation adopted there, we will show that the uniform exponential bound is valid for both $\tilde{J}^{\rm in}_N$ and $\tilde{J}^{\rm ex}_N$, i.e.\ the contributions near and away from the critical point respectively.
From~\eqref{eq:steepestDescentIn} and~\eqref{eq:steepestDescentInConv}, it follows that for all $x \in [0,\infty)$
\[
\sup_{N\in\N} \big| \tilde{J}^{\rm in}_N(x) \big|
\lesssim \e^{-\epsilon x}
\int_{C+\epsilon} \e^{\Re(\tilde{z}^3)/3} \abs{\diff \tilde{z}}
\, ,
\]
with $\epsilon$ chosen to be strictly positive.
By definition of the contour $C$, the above integral converges, providing the desired exponential bound for $\tilde{J}^{\rm in}_N$.
On the other hand, estimates~\eqref{eq:steepestDescentExt}, \eqref{eq:steepestDescentExtVert} and~\eqref{eq:steepestDescentExtObl} show that for all $N\in\N$ and $x\in [0,\infty)$
\[
\begin{split}
\big| \tilde{J}^{\rm ex}_N(x) \big|
&\leq
\left[ \frac{\mathcal{L}(T_a) \sqrt[3]{2N}}{2\pi\gamma} \e^{ - \min\{ cN, (2/3) N^{1-3\alpha} + o(N^{1-3\alpha}) \} } \right]
\e^{-x \min\{ a\sqrt[3]{2N}/\gamma, 2^{-2/3} N^{1/3-\alpha} + \epsilon \}} \\
&\leq c' \e^{-x \min\{ a\sqrt[3]{2}/\gamma, 2^{-2/3} + \epsilon \}} \, ,
\end{split}
\]
where $c$ is the constant (positive if $a$ is chosen large enough) defined in~\eqref{eq:steepestDescentExtVertConst}, and $c'$ is an upper bound for the vanishing sequence inside the square bracket above.
This provides the desired exponential bound for $\tilde{J}^{\rm ex}_N$.
\end{proof}

\subsection{Point-to-line and GOE Tracy-Widom}
\label{subsec:flatLPPasymptotics}

We will now specialize the results of sections~\ref{subsec:fredholm} and~\ref{subsec:steepestDescent} to study the scaling limits of the exponential LPP models.
We first analyze the point-to-line model, writing its CDF as a Fredholm determinant.
\begin{theorem}
\label{thm:flatExpLPPfredholm}
The distribution of the point-to-line $(\bm{\alpha}, \bm{\beta})$-exponential LPP $\fTau_{2N}$ can be given in terms of a Fredholm determinant as
\begin{equation}
\label{eq:flatExpLPPfredholm}
\P\left(\fTau_{2N}\leq u\right) = \det(I - \fK_{N,u})_{L^2(\R_{>0})} \, ,
\end{equation}
where $\fK_{N,u} \colon L^2(\R_{>0})\to L^2(\R_{>0})$ is the operator defined through the kernel
\begin{equation}
\label{eq:flatExpLPPkernel}
\fK_{N,u}(\lambda,\xi)
= \frac{1}{(2\pi\i)^2}
\int_{\Gamma_1} \diff z
\int_{\Gamma_2} \diff w \,
\e^{-\lambda z - \xi w} \,
\fHbar_u(z,w)
\prod_{m=1}^N \frac{(z+\beta_m)(w+\alpha_m)}{(z-\alpha_m)(w-\beta_m)} \, .
\end{equation}
Here, $\Gamma_1,\Gamma_2 \subset \C_{>0}$ are any positively oriented simple closed contours such that $\Gamma_1$ encloses $\alpha_1,\dots,\alpha_N$, $\Gamma_2$ encloses $\beta_1,\dots,\beta_N$ as well as the whole $\Gamma_1$, and
\begin{equation}
\label{eq:flatExpLPPHbar}
\fHbar_u(z,w)
= \frac{\e^{-2uz}}{w-z}
+ \frac{\e^{-2uw}}{z-w}
+ \frac{\e^{-2u(z+w)}}{z+w} \, .
\end{equation}
\end{theorem}

\begin{proof}
The claim is an immediate consequence of the determinantal formula~\eqref{eq:flatExpLPPdet} and Theorem~\ref{thm:detToFredDet}.
According to~\eqref{eq:Hbar}, function $\fHbar_u$ in~\eqref{eq:flatExpLPPkernel} is defined through the relation $\fH_u = C - \fHbar_u$ (using the notation of Theorem~\ref{thm:flatExpLPP} for $\fH_u$),  i.e.
\[
\fHbar_u(z,w)
= \frac{1}{z+w} - \e^{-u(z + w)} \int_0^u (\e^{z x}-\e^{-z x})
(\e^{w x}-\e^{-w x}) \diff x \, .
\]
If we assume\footnote{Notice that, by symmetry, the reverse inclusion would lead to the same result.} that $\Gamma_2$ encloses $\Gamma_1$ (so that $z\neq w$ for all $z,w$), integrating the above expression yields~\eqref{eq:flatExpLPPHbar}.
\end{proof}

\begin{remark}
As discussed in Remark~\ref{rem:detPointProcesses}, our formula~\eqref{eq:flatExpLPPSchur} does not naturally lead to a determinantal point process.
Therefore, the kernel in~\eqref{eq:flatExpLPPfredholm} is not (naturally) associated to a determinantal point process, and it might not even be positive.
In fact, we will see in next theorem that the limiting kernel is given by the Airy function, which is not always positive.
\end{remark}

We are now ready to derive the scaling limit of the point-to-line LPP with exponential i.i.d.\ waiting times.
We will prove that the fluctuations of $\fTau_{2N}$ are of order $N^{1/3}$ and its the scaling limit is given by the GOE Tracy-Widom distribution.
In particular, we will find Sasamoto's Fredholm determinant formula~\cite{sasamoto05} for the CDF of such a distribution:
\begin{equation}
\label{eq:GOE}
F_1(s)
= \det(I - \mathcal{K}_1)_{L^2([s,\infty))}
\end{equation}
for $s\in\R$, where $\mathcal{K}_1$ is the operator on $L^2([s,\infty))$ defined through the kernel
\begin{equation}
\label{eq:GOEkernel}
\mathcal{K}_1(\lambda,\xi) := \frac{1}{2}\Ai\bigg(\frac{\lambda + \xi}{2}\bigg) \, .
\end{equation}

\begin{theorem}
\label{thm:flatLPPasymptotics}
If the waiting times are independent and exponentially distributed with rate $2\gamma$, the limiting distribution of the point-to-line LPP $\fTau_{2N}$ is given, for $r\in\R$, by
\begin{equation}
\label{eq:fAsympt}
\lim_{N\to\infty}\P\bigg(\fTau_{2N} \leq \frac{2N}{\gamma} + rN^{1/3} \bigg) 
= F_1\big(2^{1/3}\gamma r\big) \, .
\end{equation}
\end{theorem}

\begin{proof}
The starting point is the Fredholm determinant formula of Theorem~\ref{thm:flatExpLPPfredholm}. 
We will first show the pointwise convergence of the kernel after suitable rescaling, and next sketch the (standard) argument for the convergence of the Fredholm determinant.
Setting $\alpha_m=\beta_m=\gamma$ for all $m$, so that the waiting times are all exponential with rate $2\gamma$ (see Definition~\ref{def:flatExpMeasure}), kernel~\eqref{eq:flatExpLPPkernel} reads as
\[
\fK_{N,u}(\lambda,\xi)
= \frac{1}{(2\pi\i)^2}
\int_{\Gamma_1} \diff z
\int_{\Gamma_2} \diff w \,
\e^{-\lambda z - \xi w} \,
\fHbar_u(z,w)
\left[ \frac{(z+\gamma)(w+\gamma)}{(z-\gamma)(w-\gamma)} \right]^N \, ,
\]
where $\fHbar_u$ is given by formula~\eqref{eq:flatExpLPPHbar}.
Our kernel is thus a sum of three double contour integrals, each corresponding to one of the addends in~\eqref{eq:flatExpLPPHbar}.
In the second one \emph{only}, we swap the two contours taking into account the residue at the pole $w=z$.
We can then readily write the kernel as the sum of four terms:
\[
\fK_{N,u}
= \fKone_{N,u} + \fKtwo_{N,u} + \fKthree_{N,u} + \fKfour_{N,u} \, ,
\]
where the first one corresponds to the above mentioned residue:
\begin{align}
\label{eq:fKer1}
\fKone_{N,u}(\lambda,\xi)
&:= - \frac{1}{2\pi\i} \int_{\Gamma_1} \diff z \,
\e^{-(2u + \lambda+\xi) z}
\left[\frac{\gamma+z}{\gamma-z}\right]^{2N} \, ,
\end{align}
and the other three terms are
\begin{align}
\label{eq:fKer2}
\fKtwo_{N,u}(\lambda,\xi)
&:= \frac{1}{(2\pi\i)^2}
\int_{\Gamma_1} \diff z \,
\int_{\Gamma_2} \diff w \, \e^{-\lambda z - \xi w}
\frac{\e^{-2uz}}{w-z}
\left[ \frac{(z+\gamma)(w+\gamma)}{(z-\gamma)(w-\gamma)} \right]^N \, , \\
\label{eq:fKer3}
\fKthree_{N,u}(\lambda,\xi)
&:= \fKtwo_{N,u}(\xi,\lambda) \, , \\
\label{eq:fKer4}
\fKfour_{N,u}(\lambda,\xi)
&:= \frac{1}{(2\pi\i)^2}
\int_{\Gamma_1} \diff z \,
\int_{\Gamma_2} \diff w \, \e^{-\lambda z - \xi w}
\frac{\e^{-2u(z+w)}}{z+w}
\left[ \frac{(z+\gamma)(w+\gamma)}{(z-\gamma)(w-\gamma)} \right]^N \, .
\end{align}

{ \bfseries Step 1: Main contribution in the kernel.}
The Airy kernel emerges from a rescaling of $\fKone_{N,u}$ through Proposition~\ref{prop:steepestDescent}, whereas the other terms turn out to be negligible under the same rescaling, as we will see.
From now on, fixing $r\in\R$ once for all, we will take $u$ to be $u_N:=2N/\gamma+rN^{1/3}$, as in~\eqref{eq:fAsympt}.
Moreover, we denote by $\tilde{\Psi}$ the rescaling of any function $\Psi(\lambda,\xi)$ by the factor $\sqrt[3]{2N}/\gamma$:
\begin{equation}
\label{eq:rescaling}
\tilde{\Psi}(\lambda,\xi)
:= \frac{\sqrt[3]{2N}}{\gamma} \Psi\left( \frac{\sqrt[3]{2N}}{\gamma} \lambda, \frac{\sqrt[3]{2N}}{\gamma} \xi \right) \, .
\end{equation}
By Proposition~\ref{prop:steepestDescent}, $\fKonetilde_{N,u_N}$ has a non-trivial limit:
\[
\begin{split}
\fKonetilde_{N,u_N}(\lambda,\xi)
= &- \frac{\sqrt[3]{4N}}{\sqrt[3]{2}\gamma}
\frac{1}{2\pi\i} \int_{\Gamma_1} \diff z \,
\exp\left\{-z\left[\frac{2}{\gamma} 2N +\left(\frac{\lambda + \xi}{\sqrt[3]{2}} + 2^{1/3}\gamma r \right) \frac{\sqrt[3]{4N}}{\gamma} \right]\right\}
\left[\frac{\gamma+z}{\gamma-z}\right]^{2N} \\
\xrightarrow{N\to\infty} 
& \, 2^{-1/3} \Ai\left(2^{-1/3}(\lambda+\xi)+2^{1/3}\gamma r \right) \, .
\end{split}
\]
We thus need to replace our whole kernel with its rescaling by the factor $\sqrt[3]{2N}/\gamma$:
\[
\fKtilde_{N,u_N}
= \fKonetilde_{N,u_N} +\fKtwotilde_{N,u_N} +\fKthreetilde_{N,u_N} + \fKfourtilde_{N,u_N}
\, .
\]
This does not affect formula~\eqref{eq:flatExpLPPfredholm}, as it just amounts to a change of variables in the multiple integrals defining the Fredholm determinant expansion (see~\eqref{eq:fredholmDet}), so that:
\[
\P\left(\fTau_{2N}\leq u_N\right)
= \det(I - \fKtilde_{N,u_N})_{L^2(\R_{>0})}
\, .
\]

{ \bfseries Step 2: Vanishing terms in the kernel.}
We will now show that all the remaining terms $\tilde{K}^{(i)}_{N,u_N}(\lambda,\xi)$ for $i=2,3,4$ vanish, starting from $\fKtwotilde_{N,u_N}(\lambda,\xi)$.
For this purpose, we specify the contours appropriately.
We choose $\Gamma_1$ to be a circle of radius $\rho_1$ around $\gamma$, where $0<\rho_1<\gamma$.
Next, we choose $\Gamma_2$ to be a semicircle of radius $\rho_2$ centered at $\delta$, where $0<\delta < \gamma -\rho_1$, composed by concatenating the segment $\delta + \i [-\rho_2,\rho_2]$ and the arc parametrized by $\delta + \rho_2 \e^{\i \theta}$ for $\theta \in [-\pi/2,\pi/2]$.
It is clear that both contours lie in the right half-plane and, for $\rho_2$ large enough, $\Gamma_2$ encloses $\Gamma_1$.
Rescaling~\eqref{eq:fKer2}, setting $u:=u_N$, and using the fact that $\lambda,\xi> 0$ and $\delta \leq \Re(z),\Re(w) \leq \delta + \rho_2$ for $z\in \Gamma_1$ and $w\in \Gamma_2$, we estimate
\begin{equation}
\begin{split}
\label{eq:fKer2Estimate}
\abs{\fKtwotilde_{N,u_N}(\lambda,\xi)}
&\leq \frac{(2N)^{1/3} \mathcal{L}(\Gamma_1) \mathcal{L}(\Gamma_2)}{\gamma (2\pi)^2 \rm{dist}(\Gamma_1,\Gamma_2)}
\e^{- (\lambda +\xi) \delta \sqrt[3]{2N} / \gamma }
\e^{(m_1 + m_2) N + 2 \abs{r} (\delta + \rho_2) N^{1/3} } \\
&\leq c \, \e^{- (\lambda +\xi) \delta \sqrt[3]{2} / \gamma }
\exp\left\{(m_1 + m_2) N + 2 \abs{r} (\delta + \rho_2) N^{1/3} + \frac{1}{3} \log N \right\} \, .
\end{split}
\end{equation}
In the first inequality, we have denoted by $\mathcal{L}( \cdot)$ the length of a curve and by $\rm{dist}( \cdot, \cdot)$ the Euclidean distance in $\C$.
In the second inequality, $c$ is a constant depending on the parameters $\gamma$, $\delta$, $\rho_1$ and $\rho_2$ only, whereas $m_1$ and $m_2$ are defined by
\begin{align*}
m_1 := \max_{z \in \Gamma_1} \left\{ -\frac{4}{\gamma}\Re(z) + 
\log \abs{\frac{z + \gamma}{z - \gamma}} \right\}  \, , \qquad\quad
m_2 := \max_{w \in \Gamma_2} \log \abs{\frac{w + \gamma}{w - \gamma}} \, .
\end{align*}
A trivial estimate yields
\[
m_1 \leq -4\left(1-\frac{\rho_1}{\gamma} \right)
+ \log\left(1+ 2 \frac{\gamma}{\rho_1}\right) \, .
\]
Now, the function
\[
g(t) := -4(1-t) + \log\left(1 + \frac{2}{t}\right)
\]
attains its minimum for $t\in (0,1)$ at $t_0 :=\sqrt{3/2}-1$, with $g(t_0) <0$; hence, choosing $\rho_1 := t_0 \gamma$, we have that $m_1<0$.
On the other hand, we estimate
\[
\begin{split}
m_2
\leq \max_{\Re(w)=\delta} \log \abs{\frac{w + \gamma}{w - \gamma}}
+ \max_{\abs{w-\delta}=\rho_2} \log \abs{\frac{w + \gamma}{w - \gamma}}
\leq \log\frac{\gamma+\delta}{\gamma-\delta}
+ \log \frac{\rho_2 +\delta + \gamma}{\rho_2 +\delta - \gamma}  \, .
\end{split}
\]
We can now choose $\delta>0$ small enough and $\rho_2$ big enough such that  $m_2 < -m_1$.
It thus follows that, for certain values of $\rho_1$, $\rho_2$ and $\delta$, the quantity after the last inequality in~\eqref{eq:fKer2Estimate} decays exponentially with rate $N$, allowing us to conclude that $\fKtwotilde_{N,u_N}(\lambda,\xi)$ vanishes as $N\to\infty$, for all $\lambda,\xi\in\R_{>0}$.
We remark that, in~\eqref{eq:fKer2Estimate}, the  exponential containing variables $\lambda$ and $\xi$ has not played any role so far, but will provide a useful estimate in the next step.

Because of~\eqref{eq:fKer3}, we have that estimate~\eqref{eq:fKer2Estimate} is exactly valid for $\fKthreetilde_{N,u_N}(\lambda,\xi)$ too, so that this term also vanishes.

Finally, an estimate similar to~\eqref{eq:fKer2Estimate} holds for $\fKfourtilde_{N,u_N}(\lambda,\xi)$: to see this, we make the same contour choice as we made when showing that $\fKtwotilde_{N,u_N}(\lambda,\xi)$ vanishes.
Rescaling~\eqref{eq:fKer4} and setting $u:=u_N$, a generous estimate yields
\begin{equation}
\begin{split}
\label{eq:fKer4Estimate}
\abs{\fKfourtilde_{N,u_N}(\lambda,\xi)}
&\leq \frac{(2N)^{1/3} \mathcal{L}(\Gamma_1) \mathcal{L}(\Gamma_2)}{\gamma (2\pi)^2 \rm{dist}(\Gamma_1,-\Gamma_2)}
\e^{- (\lambda +\xi) \delta \sqrt[3]{2N} / \gamma }
\e^{(m_1 + m_2) N + 4 \abs{r} (\delta + \rho_2) N^{1/3} } \\
&\leq c' \e^{- (\lambda +\xi) \delta \sqrt[3]{2} / \gamma }
\exp\left\{(m_1 + m_2) N + 4 \abs{r} (\delta + \rho_2) N^{1/3} + \frac{1}{3} \log N \right\} \, ,
\end{split}
\end{equation}
where the constant $c'$ depends on $\gamma$, $\delta$, $\rho_1$ and $\rho_2$ only.
We have already proved that $m_1 + m_2 <0$ for a certain choice of $\rho_1$, $\rho_2$ and $\delta$, hence $\fKfourtilde_{N,u_N}(\lambda,\xi)$ also vanishes.

{ \bfseries Step 3: Convergence of Fredholm determinants.}
In the first two steps, we have proven that
\begin{equation}
\label{eq:fKerConvergence}
\lim_{N\to\infty} \fKtilde_{N,u_N}(\lambda,\xi)
= 2^{-1/3} \Ai\left(2^{-1/3}(\lambda+\xi)+ 2^{1/3}\gamma r \right)
\end{equation}
for all $\lambda,\xi\in\R_{>0}$.
We now need to show the convergence of the corresponding Fredholm determinants on $L^2(\R_{>0})$.
The argument is standard, and based on the series expansion~\eqref{eq:fredholmDet} of the Fredholm determinant.
Notice first that there exist two positive constant $c_1$ and $c_2$ such that
\[
\sup_{N\in\N} \abs{\fKtilde_{N,u_N}(\lambda,\xi)}
\leq c_1 \e^{-c_2 \lambda}
\]
for all $\lambda,\xi\in\R_{>0}$.
The exponential bound for $\fKonetilde_{N,u_N}$ comes from Corollary~\ref{coro:steepestDescentEstimate}, whereas the estimates for the remaining terms directly follow from~\eqref{eq:fKer2Estimate} and~\eqref{eq:fKer4Estimate}. Hadamard's bound then implies that
\[
\abs{ \det(\fKtilde_{N,u_N}(\lambda_i, \lambda_j))_{i,j=1}^n }
\leq n^{n/2} \prod_{i=1}^n c_1 \e^{-c_2 \lambda_i } \, .
\]
It follows that
\[
\begin{split}
\abs{ \det(I - \fKtilde_{N,u_N})_{L^2(\R_{>0})} }
&\leq 1 + \sum_{n=1}^{\infty} \frac{1}{n!} \int_0^{\infty} \dots \int_0^{\infty} \abs{ \det(K_N(\lambda_i, \lambda_j))_{i,j=1}^n } \diff \lambda_1 \cdots \lambda_n \\
& \leq 1 + \sum_{n=1}^{\infty} \frac{n^{n/2}}{n!} \left( \int_{0}^{\infty} c_1 \e^{-c_2 \lambda} \diff \lambda \right)^n
< \infty \, .
\end{split}
\]
These inequalities, apart from providing a further proof that the Fredholm determinants of our kernels are well-defined, allow us to conclude, by dominated convergence, that limit~\eqref{eq:fKerConvergence} still holds when passing to the corresponding Fredholm determinants on the space $L^2(\R_{>0})$. 
After a rescaling of the limiting kernel by a factor $2^{-2/3}$, one can see that the operator on $L^2(\R_{>0})$ defined through the kernel $(\lambda,\xi) \mapsto 2^{-1/3} \Ai\left(2^{-1/3}(\lambda+\xi)+s\right)$ has the same Fredholm determinant as the operator $\mathcal{K}_1$ on $L^2([s,\infty))$ defining the GOE Tracy-Widom distribution $F_1(s)$ as in~\eqref{eq:GOE}.
This concludes the proof.
\end{proof}

\subsection{Point-to-half-line and ${\rm Airy}_{2\to 1}$}
\label{subsec:hFlatLPPasymptotics}

We now pass to the point-to-half-line exponential LPP model, again starting with the Fredholm determinant formula.
\begin{theorem}
\label{thm:hFlatExpLPPfredholm}
The distribution of the point-to-half-line $(\bm{\alpha}, \bm{\beta})$-exponential LPP $\hTau_{2N}$ can be given in terms of a Fredholm determinant as
\begin{equation}
\label{eq:hFlatExpLPPfredholm}
\P\left(\hTau_{2N}\leq u\right) = \det\left(I - \hK_{N,u}\right)_{L^2(\R_{>0})} \, ,
\end{equation}
where $\hK_{N,u} \colon L^2(\R_{>0})\to L^2(\R_{>0})$ is the operator defined through the kernel
\begin{equation}
\label{eq:hFlatExpLPPkernel}
\hK_{N,u}(\lambda,\xi)
= \frac{1}{(2\pi\i)^2}
\int_{\Gamma_1} \diff z
\int_{\Gamma_2} \diff w \,
\e^{-\lambda z - \xi w} \,
\hHbar_u(z,w)
\prod_{m=1}^N \frac{(z+\beta_m)(w+\alpha_m)}{(z-\alpha_m)(w-\beta_m)} \, .
\end{equation}
Here, $\Gamma_1,\Gamma_2 \subset \C_{>0}$ are any positively oriented simple closed contours such that $\Gamma_1$ encloses $\alpha_1,\dots,\alpha_N$, $\Gamma_2$ encloses $\beta_1,\dots,\beta_N$ as well as the whole $\Gamma_1$, and
\begin{equation}
\label{eq:hHbar}
\hHbar_u(z,w)
= \frac{\e^{-u(z+w)}}{z+w}
+ \frac{\e^{-u(z+w)}}{z-w}
+ \frac{\e^{-2uz}}{w-z} \, .
\end{equation}
\end{theorem}

\begin{proof}
The claim follows from the determinantal formula~\eqref{eq:hFlatExpLPPdet} and Theorem~\ref{thm:detToFredDet}.
Function $\hHbar_u$ in~\eqref{eq:hFlatExpLPPkernel} is defined through the relation $\hH_u = C - \hHbar_u$ (using the notation of Theorem~\ref{thm:hFlatExpLPP} for $\hH_u$), i.e.
\[
\hHbar_u(z,w)
= \frac{1}{z+w} - \e^{-u(z + w)} \int_0^u (\e^{z x}-\e^{-z x})
\e^{w x} \diff x \, .
\]
If we assume that $\Gamma_2$ encloses $\Gamma_1$ (so that $z\neq w$ for all $z,w$), integrating the above expression yields~\eqref{eq:hHbar}.
\end{proof}

Via the above Fredholm determinant formula, we can now derive the scaling limit of the point-to-half-line LPP with exponential i.i.d.\ waiting times.
We will prove that the fluctuations of $\hTau_{2N}$ are of order $N^{1/3}$ and its the scaling limit is given by the one-point marginal distribution $F_{2\to 1}$ of the $\rm{Airy}_{2\to 1}$ process.
The expression we will arrive at is the following~\cite{borodinFerrariSasamoto08}:
\begin{equation}
\label{eq:Airy21}
F_{2 \to 1}(s)
= \det(I - \mathcal{K}_{2\to 1})_{L^2([s,\infty))}
\end{equation}
for $s\in\R$, where $\mathcal{K}_{2\to 1}$ is the operator on $L^2([s,\infty))$ defined through the kernel
\begin{equation}
\label{eq:Airy21kernel}
\mathcal{K}_{2\to 1}(\lambda,\xi)
:= \int_0^{\infty} \Ai(\lambda+x) \Ai(\xi+x) \diff x
+ \int_0^{\infty} \Ai(\lambda+x) \Ai(\xi-x) \diff x
\, .
\end{equation}
Our result is consistent with the one obtained by Borodin-Ferrari-Sasamoto~\cite{borodinFerrariSasamoto08} in the analysis of the TASEP with half-alternating initial configuration, which is equivalent to the point-to-half-line exponential LPP, as seen in the introduction to this section.

\begin{theorem}
\label{thm:hFlatLPPasymptotics}
If the waiting times are independent and exponentially distributed with rate $2\gamma$, the limiting distribution of the point-to-half-line LPP $\hTau_{2N}$ is given, for $r\in\R$, by
\[
\lim_{N\to\infty}\P\bigg(\hTau_{2N} \leq \frac{2N}{\gamma} + rN^{1/3} \bigg) 
= F_{2\to 1}\big(2^{-1/3} \gamma r \big) \, .
\]
\end{theorem}

\begin{proof}
In order to perform the asymptotics of formula~\eqref{eq:hFlatExpLPPfredholm} in the i.i.d.\ case, we set $\alpha_m=\beta_m=\gamma$ for all $m$ in the definition of $(\bm{\alpha},\bm{\beta})$-exponential measure.
Our kernel~\eqref{eq:hFlatExpLPPkernel} thus becomes
\[
\begin{split}
\hK_{N,u}
= \hKone_{N,u} + \hKtwo_{N,u} + \hKthree_{N,u} \, ,
\end{split}
\]
where
\begin{align*}
\hKone_{N,u}(\lambda,\xi)
&= \frac{1}{(2\pi\i)^2}
\int_{\Gamma_1} \diff z
\int_{\Gamma_2} \diff w \,
\e^{-\lambda z - \xi w}
\frac{\e^{-u(z+w)}}{z+w}
\left[ \frac{(z+\gamma)(w+\gamma)}{(z-\gamma)(w-\gamma)} \right]^N
\, , \\
\hKtwo_{N,u}(\lambda,\xi)
&= \frac{1}{(2\pi\i)^2}
\int_{\Gamma_1} \diff z
\int_{\Gamma_2} \diff w \,
\e^{-\lambda z - \xi w}
\frac{\e^{-u(z+w)}}{z-w}
\left[ \frac{(z+\gamma)(w+\gamma)}{(z-\gamma)(w-\gamma)} \right]^N \, , \\
\hKthree_{N,u}(\lambda,\xi)
&= \frac{1}{(2\pi\i)^2}
\int_{\Gamma_1} \diff z
\int_{\Gamma_2} \diff w \,
\e^{-\lambda z - \xi w}
\frac{\e^{-2uz}}{w-z}
\left[ \frac{(z+\gamma)(w+\gamma)}{(z-\gamma)(w-\gamma)} \right]^N \, .
\end{align*}

For the steepest descent analysis of the first two terms, we are going to adapt the proof of Proposition~\ref{prop:steepestDescent}, taking into account that we now have double contour integrals instead of single ones.
Noticing that $\hKone_{N,u}$ and $\hKtwo_{N,u}$ only differ for the sign in $(z\pm w)^{-1}$, we study both at the same time, denoting by $K^{\pm}_{N,u}$ either of them:
\[
K^{\pm}_{N,u}(\lambda,\xi)
:= \frac{1}{(2\pi\i)^2}
\int_{\Gamma_1} \diff z
\int_{\Gamma_2} \diff w \,
\e^{-\lambda z - \xi w}
\frac{\e^{-u(z+w)}}{z \pm w}
\left[ \frac{(z+\gamma)(w+\gamma)}{(z-\gamma)(w-\gamma)} \right]^N
\, .
\]
We replace the contour $\Gamma_1$ with $T_{R_1} + 2\epsilon\gamma/\sqrt[3]{2N}$ and the contour $\Gamma_2$ with $T_{R_2} + \epsilon\gamma/\sqrt[3]{2N}$, for some $\gamma < R_1 < R_2$ and $\epsilon>0$; here, as in the proof of Proposition~\ref{prop:steepestDescent}, $T_R$ is the negatively oriented triangular path with vertices $0$, $2R\e^{\i\pi/3}$ and $2R\e^{-\i\pi/3}$.
Notice that changing the orientation of both paths does not yield any change of sign in the double contour integral; moreover, the first contour is still enclosed by the second one, and the singularities at $(z\pm w)^{-1}$ are not crossed by the deformed contours (the infinitesimal shifts of $T_{R_1}$ and $T_{R_2}$ are also done for this sake). 
Set now $u=u_N:= 2N/\gamma + r N^{1/3}$ and denote by $\tilde{\Psi}$ the rescaling of any function $\Psi(\lambda,\xi)$ by the factor $\sqrt[3]{2N}/\gamma$, as in~\eqref{eq:rescaling}.
We can thus write
\[
\tilde{K}^{\pm}_{N,u_N}(\lambda, \xi)
= \frac{\sqrt[3]{2N}}{\gamma (2\pi\i)^2}
\int_{T_{R_1} + \frac{2\epsilon\gamma}{\sqrt[3]{2N}}} \diff z
\int_{T_{R_2} + \frac{\epsilon\gamma}{\sqrt[3]{2N}}} \diff w
\frac{ \e^{N F(z) - \lambda_r z \sqrt[3]{2N}/\gamma }
 \e^{ N F(w) - \xi_r w \sqrt[3]{2N}/\gamma} }{z \pm w} \, ,
\]
where $\lambda_r := \lambda + 2^{-1/3}\gamma r$, $\xi_r := \xi + 2^{-1/3}\gamma r$, and $F(z):= \log(\gamma+z) - \log(\gamma - z) - 2z/\gamma$.
Since the main contribution in the integral is expected to come from $z=w=0$, which is the critical point of $F$, we split the above integral into the following sum:
\[
\tilde{K}^{\pm}_{N,u_N}
= \tilde{K}^{\pm, \rm{in, in}}_{N,u_N}
+ \tilde{K}^{\pm, \rm{in, ex}}_{N,u_N}
+ \tilde{K}^{\pm, \rm{ex, in}}_{N,u_N}
+ \tilde{K}^{\pm, \rm{ex, ex}}_{N,u_N} \, .
\]
Here, the first superscript ``$\rm{in}$'' (``$\rm{ex}$'') indicates that the integration w.r.t.\ $z$ is performed only in the interior (exterior, respectively) of the ball $\{ |z - 2\epsilon\gamma/\sqrt[3]{2N}| \leq \gamma N^{-\alpha}\}$ for some exponent $\alpha>0$ to be specified later on, while the second superscript ``$\rm{in}$'' (``$\rm{ex}$'') indicates that the integration w.r.t.\ $w$ is performed only in the interior (exterior, respectively) of the ball $\{ |w - \epsilon\gamma/\sqrt[3]{2N}| \leq \gamma N^{-\alpha}\}$.
In~$\tilde{K}^{\pm, \rm{in, in}}_{N,u_N}$, after the changes of variables $\tilde{z} := z \sqrt[3]{2N}/\gamma$ and $\tilde{w} := w \sqrt[3]{2N}/\gamma$, we obtain
\[
\begin{split}
\tilde{K}^{\pm, \rm{in, in}}_{N,u_N}(\lambda,\xi)
= &\frac{1}{(2\pi\i)^2}
\int_{C + 2\epsilon} \diff \tilde{z}
\exp\left\{ \frac{\tilde{z}^3}{3} - \lambda_r \tilde{z} +  R\left( \frac{\gamma}{\sqrt[3]{2N}} \tilde{z}\right) N  \right\}
\1_{\left\{ \abs{ \tilde{z} - 2\epsilon} \leq \sqrt[3]{2} N^{1/3-\alpha} \right\}} \\
&\times \int_{C + \epsilon} \diff \tilde{w}
\exp\left\{ \frac{\tilde{w}^3}{3} - \xi_r \tilde{w} +  R\left( \frac{\gamma}{\sqrt[3]{2N}} \tilde{w}\right) N \right\}
\1_{\{\abs{\tilde{w} - \epsilon} \leq \sqrt[3]{2} N^{1/3-\alpha} \}}
\frac{1}{\tilde{z} \pm \tilde{w}} \, ,
\end{split}
\]
where $C$ is the piecewise linear path going from $\e^{-\i\pi/3}\infty$ to the origin to $\e^{\i\pi/3}\infty$, and $R$ is the remainder defined by~\eqref{eq:expansionCriticalPoint}.
The indicator functions clearly converge to $1$ if $\alpha < 1/3$.
As in the proof of Proposition~\ref{prop:steepestDescent}, one can also show that the remainders, even when multiplied by $N$, vanish uniformly for $\tilde{z},\tilde{w}$ in the support of the integrand, if we choose $1/4< \alpha < 1/3$.
Applying dominated convergence, one can see that
\[
\lim_{N\to\infty} \tilde{K}^{\pm, \rm{in, in}}_{N,u_N}(\lambda,\xi)
= \frac{1}{(2\pi\i)^2}
\int_{C + 2\epsilon} \diff \tilde{z}
\int_{C + \epsilon} \diff \tilde{w}
\frac{1}{\tilde{z} \pm \tilde{w}}
\exp\left\{ \frac{\tilde{z}^3}{3} - \lambda_r \tilde{z}
+ \frac{\tilde{w}^3}{3} - \xi_r \tilde{w} \right\} \, .
\]
Using similar arguments as in the proof of Proposition~\ref{prop:steepestDescent}, together with the bound $\abs{z \pm w}^{-1} \leq 2\sqrt[3]{2N}/(\sqrt{3}\epsilon\gamma)$, one can see that the other terms $\tilde{K}^{\pm, \rm{in, ex}}_{N,u_N}$, $\tilde{K}^{\pm, \rm{ex, in}}_{N,u_N}$, and $\tilde{K}^{\pm, \rm{ex, ex}}_{N,u_N}$ vanish exponentially fast in the limit $N\to\infty$.
We thus have:
\begin{align*}
\lim_{N\to\infty} \hKonetilde_{N,u_N}(\lambda,\xi)
&= \frac{1}{(2\pi\i)^2}
\int_{C + 2\epsilon} \diff \tilde{z}
\int_{C + \epsilon} \diff \tilde{w}
\frac{1}{\tilde{z} + \tilde{w}}
\e^{ \tilde{z}^3/3 - \lambda_r \tilde{z}}
\e^{\tilde{w}^3/3 - \xi_r \tilde{w}} \, , \\
\lim_{N\to\infty} \hKtwotilde_{N,u_N}(\lambda,\xi)
&= \frac{1}{(2\pi\i)^2}
\int_{C + 2\epsilon} \diff \tilde{z}
\int_{C + \epsilon} \diff \tilde{w}
\frac{1}{\tilde{z} - \tilde{w}}
\e^{ \tilde{z}^3/3 - \lambda_r \tilde{z}}
\e^{\tilde{w}^3/3 - \xi_r \tilde{w}} \, .
\end{align*}
We will now rewrite these expressions as integrals of Airy functions.
In the first one, since $\Re(\tilde{z}+\tilde{w}) >0$ for all $\tilde{z}$ and $\tilde{w}$, we can make the substitution $(\tilde{z}+\tilde{w})^{-1} = \int_0^{\infty} \e^{-(\tilde{z}+\tilde{w})x} \diff x$.
The resulting $\diff \tilde{z} \diff \tilde{w} \diff x$  integral is absolutely convergent, hence Fubini's Theorem can be applied to obtain:
\[
\begin{split}
\lim_{N\to\infty} \hKonetilde_{N,u_N}(\lambda,\xi)
&= \int_0^{\infty}
\left[ \frac{1}{2\pi\i}
\int_{C + 2\epsilon}
\e^{ \tilde{z}^3/3 - (\lambda_r +x) \tilde{z}} \diff \tilde{z} \right]
\left[ \frac{1}{2\pi\i}
\int_{C + \epsilon}
\e^{\tilde{w}^3/3 - (\xi_r +x) \tilde{w}} \diff \tilde{w} \right] \diff x \\
&= \int_0^{\infty} \Ai(\lambda_r + x) \Ai(\xi_r + x) \diff x \, ,
\end{split}
\]
according to definition~\ref{eq:Airy} of Airy function.
In $\hKtwotilde_{N,u_N}$, we deform\footnote{A standard argument justifies this contour deformation.} the contour $C + \epsilon$ into the straight line $l_{\epsilon}$ going from $\epsilon-\i\infty$ to $\epsilon+\i\infty$; since now $\Re(\tilde{z}-\tilde{w})\geq \epsilon >0$ for all $\tilde{z}$ and $\tilde{w}$, we can make the substitution $(\tilde{z}-\tilde{w})^{-1} = \int_0^{\infty} \e^{-(\tilde{z}-\tilde{w})x} \diff x$.
The resulting $\diff \tilde{z} \diff \tilde{w} \diff x$  integral is absolutely convergent, hence Fubini's Theorem can be applied to obtain:
\[
\begin{split}
\lim_{N\to\infty} \hKtwotilde_{N,u_N}(\lambda,\xi)
&= \int_0^{\infty}
\left[ \frac{1}{2\pi\i}
\int_{C + 2\epsilon}
\e^{ \tilde{z}^3/3 - (\lambda_r +x) \tilde{z}} \diff \tilde{z} \right]
\left[ \frac{1}{2\pi\i}
\int_{l_{\epsilon}}
\e^{\tilde{w}^3/3 - (\xi_r - x) \tilde{w}} \diff \tilde{w} \right] \diff x \\
&= \int_0^{\infty} \Ai(\lambda_r + x) \Ai(\xi_r - x) \diff x \, .
\end{split}
\]
We remark that the second square bracket above is an Airy function as well, since the path $l_{\epsilon}$ can be deformed back to a contour, like $C+\epsilon$, whose arguments at $\infty$ are $\pm\pi/3$.

We finally notice that $\hKthree_{N,u_N}(\lambda,\xi)$ equals exactly the term $\fKtwo_{N,u_N}(\lambda,\xi)$ defined in the proof of Theorem~\ref{thm:flatLPPasymptotics}.
Therefore, as we have already proved there, the rescaled version $\hKthreetilde_{N,u_N}(\lambda,\xi)$ vanishes as $N\to\infty$.

We conclude that, as a whole, our rescaled kernel has the following limit:
\[
\lim_{N\to\infty} \hKtilde_{N,u_N}(\lambda,\xi)
= \mathcal{K}_{2\to 1}(\lambda_r, \xi_r )
= \mathcal{K}_{2\to 1}(\lambda + 2^{-1/3}\gamma r, \xi + 2^{-1/3}\gamma r ) \, ,
\]
where $\mathcal{K}_{2\to 1}$ is defined in~\eqref{eq:Airy21kernel}.
Using the key fact that all contours are chosen to have positive distance from the imaginary axis (as in the analogous estimates obtained in Corollary~\ref{coro:steepestDescentEstimate} and in the proof of Theorem~\ref{thm:flatLPPasymptotics}), one can show that there exist two positive constants $c_1$ and $c_2$ such that, for all $\lambda,\xi\in\R_{>0}$,
\[
\sup_{N\in\N} \big| \hKtilde_{N,u_N}(\lambda,\xi) \big|
\leq c_1 \e^{-c_2 \lambda} \, .
\]
The latter bound provides, as in the third step of the proof of Theorem~\ref{thm:flatLPPasymptotics}, the right estimates for the series expansion of $\det\left(I - \hKtilde_{N,u_N}\right)_{L^2(\R_{>0})}$ in order to justify its convergence. 
The claim thus follows from the Fredholm determinant representation~\eqref{eq:Airy21} of $F_{2\to 1}$.
\end{proof}

\chapter*{Future plans and open problems}
\markboth{Future plans and open problems}{Future plans and open problems}
\addcontentsline{toc}{chapter}{Future plans and open problems}

The following is a nonexhaustive list of possible future working plans on open problems that can be inferred from this thesis:

\begin{enumerate}

\item
The \emph{asymptotic analysis of the log-gamma polymer partition functions in the point-to-line and point-to-half-line geometries}, starting from the contour integral formulas~\eqref{eq:flatContourInt} and~\eqref{eq:hFlatContourInt}.
By KPZ universality, it is expected that the limiting distributions are GOE Tracy-Widom and $\rm{Airy}_{2 \to 1}$ one-point marginal in the two geometries respectively, as already proved in the zero temperature case (see Theorems~\ref{thm:flatLPPasymptotics} and~\ref{thm:hFlatLPPasymptotics}).
However, in that context the task was facilitated by the fact that the formulas we obtained for the corresponding last passage percolation problems are integrals of (symplectic) Schur functions, which have a natural determinantal structure; thanks to this, it was possible to derive a Fredholm determinant formula, amenable to asymptotic analysis.
On the other hand, via a Bethe ansatz non-rigorous approach, Grange~\cite{grange17} obtained a conjectural Fredholm Pfaffian formula for the Laplace transform of the point-to-line log-gamma polymer, leading to the GOE Tracy-Widom asymptotics.
Our working plan is methodologically different from this and more similar to the one adopted in~\cite{borodinCorwinRemenik13} to study the point-to-point log-gamma polymer free energy limiting distribution.
Namely, comparing with the Fredholm determinant formulas~\eqref{eq:flatExpLPPfredholm} and~\eqref{eq:hFlatExpLPPfredholm} obtained in the zero temperature case, we could make an educated guess on Fredholm determinant formulas in positive temperature and show, by means of algebraic manipulations, that they are equivalent to our contour integral formulas.

\item
A more extensive study of \emph{polymer and last passage percolation models for point-to-half-line paths restricted to a half-plane}, generalizing the results obtained in sections~\ref{subsec:rFlatPolymerWhittaker}, \ref{subsec:rFlatGeomLPP} and \ref{subsec:rFlatExpLPP}.
In particular, we may allow a pinning parameter in the distribution of the weights lying on the diagonal ``hard wall'' that divides the two half-planes.
For the point-to-point case, this was already done in~\cite{oConnellSeppalainenZygouras14}.
Comparing with the analogous results that were obtained in~\cite{baikRains01b} in the setting of increasing subsequences of random involutions constructed via Poisson points, different fluctuations and limiting distributions are expected, according to the intensity of the pinning parameter: either diffusive fluctuations and Gaussian limiting distribution when the paths are almost forced to ``stick'' to the diagonal, or KPZ fluctuations with Tracy-Widom type limiting distributions (GOE squared or GUE) otherwise.
We already have some partial results in this direction, leading to apparently new determinantal formulas for the GOE squared and GUE Tracy-Widom distributions.

\item
The construction of a \emph{symplectic Schur process}.
The positivity of Schur functions and the classical Cauchy identity allow defining a Schur measure on partitions~\cite{okounkov01}, which weights a partition $\bm{\lambda}$ proportionally to the product of two Schur functions parametrized by $\bm{\lambda}$.
It is possible to generalize such a measure to a determinantal process on partitions, called Schur process~\cite{okounkovReshetikhin03}.
Inspired by our formula~\eqref{eq:flatGeomLPP} for the point-to-line geometric LPP in terms of symplectic Schur functions, one could aim at constructing a symplectic Schur process on partitions.
The first difficulty is that letting $u$ to infinity in~\eqref{eq:flatGeomLPP} does not provide a Cauchy-type identity, as the sum diverges and the prefactor vanishes.
In this direction, an alternative way to construct such a process could be to start from a Cauchy identity for symplectic Schur functions that can be found for example in~\cite{sepehri16, wheelerZinn-Justin16}.
If via both routes it is possible to extract a determinantal point process, then an interesting question would be if these are the same or not.
Furthermore, by analogy with Okounkov's Schur measure~\cite{okounkov01}, Betea~\cite{betea18} studied measures on partitions that are proportional to the product between one symplectic (or orthogonal) Schur function and one standard Schur function.
It would be interesting to compare our point-to-half-line geometric LPP formula~\eqref{eq:hFlatGeomLPP} with Betea's measures and find out if the latter have any interpretation in terms of LPP or related models.
\end{enumerate}

\appendix

\chapter{Zero temperature limit}
\label{appendix:zeroTempLimit}

We have briefly explained in the Introduction why the directed last passage percolation (LPP) model is considered to be the zero temperature version of the directed polymer partition function.
Moreover, in section~\ref{sec:expLPP} we have derived the CDF of certain LPPs in various point-to-line path geometries as a an appropriate scaling limit of the Laplace transform of the corresponding polymer partition functions studied in section~\ref{sec:WhittakerFormulas}.
In this appendix we give the complete argument necessary to justify this procedure, called \emph{zero temperature limit} and already outlined at the beginning of section~\ref{sec:expLPP}.

We first need the following technical lemma:
\begin{lemma}
\label{lemma:convEpsilon}
Let $\left\{A_{\delta}\right\}_{\delta>0}$ be a collection of closed subsets of a Borel space $\mathcal{X}$.
Let $\left\{X^{(\epsilon)}\right\}_{\epsilon>0}$ and $X$ be $\mathcal{X}$-valued random variables, and let $\{f_{\epsilon}\}_{\epsilon>0}$ and $f$ be measurable functions $\mathcal{X}\to \R$.
Assume that:
\begin{enumerate}
\item
\label{item:convEpsilon_weakConv}
$X^{(\epsilon)} \xrightarrow{\epsilon \downarrow 0} X$ in distribution;
\item
\label{item:convEpsilon_unifConv}
$f_{\epsilon} \xrightarrow{\epsilon\downarrow 0} f$ uniformly in $\mathcal{X}\setminus A_{\delta}$, for all $\delta>0$;
\item
\label{item:convEpsilon_probVanishing}
$\P\left(X\in A_{\delta}\right) \xrightarrow{\delta \downarrow 0} 0$;
\item
\label{item:convEpsilon_probDiscontinuitySet}
$\P(X \in D_f)=0$, where $D_f\subseteq \mathcal{X}$ is the set of discontinuity points of $f$.
\end{enumerate}
Then $f_{\epsilon}\left(X^{(\epsilon)}\right) \xrightarrow{\epsilon \downarrow 0} f(X)$ in distribution.
\end{lemma}
\begin{proof}
By the Portmanteau theorem, it suffices to show that
\begin{equation}
\label{eq:convDistrEpsilon}
\E\left[g \circ f_{\epsilon}\left(X^{(\epsilon)}\right)\right]
\xrightarrow{\epsilon \downarrow 0}
\E\left[g \circ f(X)\right]
\end{equation}
for all Lipschitz bounded function $g\colon \R \to \R$.
Using the triangle inequality, 
\[
\begin{split}
&\abs{\E\left[ g\circ f_{\epsilon}\left(X^{(\epsilon)}\right)\right] - \E\left[ g\circ f\left(X\right) \right]} \\
\leq\, & \E\left[ \abs{ g\circ f_{\epsilon} \left(X^{(\epsilon)}\right) - g\circ f\left(X^{(\epsilon)}\right) } \1_{\mathcal{X}\setminus A_{\delta}}\left(X^{(\epsilon)}\right) \right] \\
&+ \E\left[ \abs{ g\circ f_{\epsilon} \left(X^{(\epsilon)}\right) - g\circ f\left(X^{(\epsilon)}\right) } \1_{A_{\delta}}\left(X^{(\epsilon)}\right) \right]
+ \abs{\E\left[ g\circ f\left(X^{(\epsilon)}\right)\right] - \E\left[ g\circ f\left(X\right) \right]} \\
\leq\, & l_g \sup_{x\in \mathcal{X}\setminus A_{\delta}} \abs{f_{\epsilon}(x) - f(x)}
+ 2\, \norm{g}_{\infty}\, \P\left(X^{(\epsilon)} \in A_{\delta}\right)
+ \abs{\E\left[ g\circ f\left(X^{(\epsilon)}\right)\right] - \E\left[ g\circ f\left(X\right) \right]} \, ,
\end{split}
\]
where $l_g$ is the Lipschitz constant of $g$.
In the latter expression, we will now take the limits as $\epsilon\downarrow 0$ first, and as $\delta\downarrow 0$ next.
By assumption~\ref{item:convEpsilon_unifConv}, the first addend goes to $0$ as $\epsilon\downarrow 0$, for any fixed $\delta>0$.
By hypothesis~\ref{item:convEpsilon_weakConv} and the fact that $A_{\delta}$ is closed, the Portmanteau theorem implies that
\[
\limsup_{\epsilon\downarrow 0} \P\left(X^{(\epsilon)} \in A_{\delta}\right)
\leq \P\left(X \in A_{\delta}\right)
\]
for all $\delta>0$; taking now the limit as $\delta \downarrow 0$ and using~\ref{item:convEpsilon_probVanishing}, we see that the second addend goes to $0$ too.
Finally, the continuous mapping theorem, along with assumptions~\ref{item:convEpsilon_weakConv} and~\ref{item:convEpsilon_probDiscontinuitySet}, implies that $f\left(X^{(\epsilon)}\right) \xrightharpoonup{\epsilon\downarrow 0} f(X)$; by the Portmanteau theorem, the third addend goes to $0$ as $\epsilon\downarrow 0$.
Claim~\eqref{eq:convDistrEpsilon} follows.
\end{proof}

Consider now a polymer partition function $Z$ and a LPP $\tau$ on the same set of allowed paths (see definitions~\eqref{eq:polymerPartitionFn} and~\eqref{eq:LPP}).
Roughly speaking, next proposition states that, if the environment of $Z$ converges to the environment of $\tau$ under an appropriate scaling as the ``temperature'' parameter $\epsilon$ vanishes, then in the same limit: (i) a rescaled $Z$ converges to $\tau$ in distribution, and (ii) a rescaled Laplace transform of $Z$ converges to the CDF of $\tau$.

\begin{proposition}
\label{prop:zeroTempLimit}
Let $Z^{(\epsilon)}$ be any polymer partition function on a lattice $\mathcal{I}$ with disorder given by independent positive weights $\bm{W}^{(\epsilon)}=\{W_{i,j}^{(\epsilon)}\colon (i,j)\in\mathcal{I}\}$, whose distributions depend on a parameter $\epsilon>0$.
Let $\tau$ be the LPP in the same path geometry with independent waiting times $\bm{W}=\{W_{i,j}\colon (i,j)\in\mathcal{I} \}$.
Assume that
\begin{equation}
\label{eq:environmentConvergence}
\epsilon \log W_{i,j}^{(\epsilon)} \xrightarrow{\epsilon \downarrow 0} W_{i,j}
\end{equation}
in distribution for all $(i,j)$.
Then:
\begin{enumerate}
\item
\label{prop:zeroTempLimitWeakConv}
$\epsilon \log Z^{(\epsilon)}
\xrightarrow{\epsilon \downarrow 0}
\tau$ in distribution;
\item
\label{prop:zeroTempLimitLaplaceTransf}
$\E\left[\exp\left\{-\e^{{-u/\epsilon}} Z^{(\epsilon)} \right\}\right]
\xrightarrow{\epsilon \downarrow 0} \P(\tau \leq u)$ for all $u\in\R$ such that $\P(\tau=u)=0$.
\end{enumerate}
\end{proposition}

Notice that hypothesis~\eqref{eq:environmentConvergence} is valid in particular if $W^{(\epsilon)}_{i,j} = \e^{W_{i,j}/\epsilon}$ for all $(i,j)$.
This trivial case of the zero temperature limit was already explained in the Introduction, see~\eqref{eq:zeroTempLimit}.

Notice also that, if the waiting times $W_{i,j}$'s are continuous random variables, besides being independent, then $\tau$ is continuous as well and~\ref{prop:zeroTempLimitLaplaceTransf} holds for all $u\in\R$.

\begin{proof}
Both claims follow from Lemma~\ref{lemma:convEpsilon}.
\begin{enumerate}
\item
Let $\mathcal{X}:=\R^{\mathcal{I}}$, $A_{\delta}:=\emptyset$ for all $\delta>0$, $X^{(\epsilon)}:=\epsilon\log \bm{W}^{(\epsilon)} = \{\epsilon \log W_{i,j}^{(\epsilon)}\colon (i,j)\in\mathcal{I}\}$, $X:=\bm{W}$ and $f_{\epsilon},f\colon \R^{\mathcal{I}} \to \R$ defined by
\[
f_{\epsilon}(\bm{x})
:= \epsilon \log\left( \sum_{\pi \in \Pi} \exp\left\{\frac{1}{\epsilon} \sum_{(i,j)\in\pi} x_{i,j} \right\}\right) \, , \qquad
f(\bm{x})
:= \max_{\pi\in\Pi} \sum_{(i,j) \in \pi} x_{i,j}
\]
for all $\bm{x} = \{x_{i,j}\colon (i,j)\in \mathcal{I} \} \in\R^{\mathcal{I}}$.
The whole array $\epsilon \log \bm{W}^{(\epsilon)}$, because of the hypotheses of convergence and independence of its entries, converges in distribution to $\bm{W}$; therefore, condition~\ref{item:convEpsilon_weakConv} of the lemma holds.
It is easy to prove that $f_{\epsilon}$ converges to $f$ uniformly in $\R^{\mathcal{I}}$ as $\epsilon\downarrow 0$, so condition~\ref{item:convEpsilon_unifConv} is also satisfied.
Condition~\ref{item:convEpsilon_probVanishing} trivially holds because $A_{\delta}=\emptyset$, and~\ref{item:convEpsilon_probDiscontinuitySet} follows from the continuity of $f$.
Lemma~\ref{lemma:convEpsilon} thus implies that
\[
\epsilon \log Z^{(\epsilon)} =
f_{\epsilon}\left(\epsilon \log \bm{W}^{(\epsilon)}\right)
\xrightarrow{\epsilon \downarrow 0}
f\left(\bm{W}\right)
= \tau
\]
in distribution.
\item
Let now $\mathcal{X}:=\R$, $A_{\delta}:=[u-\delta,u+\delta]$ for all $\delta>0$, $X^{(\epsilon)} := \epsilon\log Z^{(\epsilon)}$, $X:=\tau$ and $f_{\epsilon},f\colon \R \to \R$ defined by
\[
f_{\epsilon}(x) := \exp\left\{-\e^{(x-u)/\epsilon}\right\}
\, , \qquad
f(x) := \1_{(-\infty,u]}(x)
\]
for all $x\in\R$.
Condition~\ref{item:convEpsilon_weakConv} of the lemma follows from what we have just proved.
It is easy to show that, for all $\delta>0$, $f_{\epsilon}$ converges to $f$ uniformly in $\mathcal{X}\setminus A_{\delta} = \R\setminus [u-\delta,u+\delta]$, so that condition~\ref{item:convEpsilon_unifConv} holds.
Since the waiting times are independent and continuous, $\tau$ is continuous; it follows that
\[
\P\left(\tau \in [u-\delta,u+\delta]\right)\xrightarrow{\delta \downarrow 0}
\P(\tau = u) =0 \, ,
\]
showing that condition~\ref{item:convEpsilon_probVanishing} holds. Finally, $f$ is discontinuous in $u$, but $\P(\tau =u)=0$ again because $\tau$ is a continuous random variable, so condition~\ref{item:convEpsilon_probDiscontinuitySet} applies.
Lemma~\ref{lemma:convEpsilon} thus implies that
\[
\exp\left\{-\e^{{-u/\epsilon}} Z^{(\epsilon)} \right\}
= f_{\epsilon}\left(\epsilon\log Z^{(\epsilon)}\right)
\xrightarrow{\epsilon \downarrow 0} 
f(\tau)
= \1_{\{\tau \leq u\}}
\]
in distribution.
The collection of random variables $\left\{f_{\epsilon}\left(\epsilon\log Z^{(\epsilon)}\right)\right\}_{\epsilon>0}$ is uniformly bounded by $1$, hence uniformly integrable.
Convergence in distribution along with uniform integrability implies convergence of expectations, thus proving the claim.
\qedhere
\end{enumerate}
\end{proof}


\chapter{Whittaker functions in number theory}
\label{appendix:Whittaker}

Whittaker functions on the groups $\GL_n(\R)$ and $\SO_{2n+1}(\R)$ and certain integral identities involving them, such as the Bump-Stade identity~\eqref{eq:bumpStade} and the Ishii-Stade identity~\eqref{eq:ishiiStade}, play a major role in this thesis.
We have introduced these functions in section~\ref{sec:Whittaker} and used them in our polymer analysis of chapter~\ref{ch:polymer}.
In this appendix, we wish to study a few aspects of Whittaker functions from a number theoretic standpoint.
In section~\ref{sec:MaassForms&WhittakerFns} we motivate their appearance in number theory and the contextual study of the integral identities mentioned above, focusing on the case $\GL_2(\R)$.
In section~\ref{sec:ishiiStadeWhittaker} we show the connection between our parametrization for Whittaker functions, given in section~\ref{sec:Whittaker}, and the one usually adopted by number theorists (in particular by Ishii and Stade in~\cite{ishiiStade13}).

\section{Maass forms and Whittaker functions on $\GL_2(\R)$}
\label{sec:MaassForms&WhittakerFns}

In number theory, Whittaker functions appear in the series expansion of certain automorphic forms (more specifically, Maass forms).
For the sake of conciseness and clarity, we focus on $\GL_2(\R)$ and review the theory of automorphic forms for this group, following the exposition of~\cite{goldfeld06}.

An \emph{automorphic form} is a function that satisfies certain differential equations and invariance/periodicity conditions.
A familiar example is given by the complex exponential $f(x):=\e^{2\pi\i x}$, $x\in\R$.
It is an eigenfunction of the one-dimensional Laplacian, as it satisfies the differential equation $\frac{\diff^2 f}{\diff x^2} = -4\pi^2 f$.
It is also a periodic function, in the sense that  $f(x+n)=f(x)$ for all $x\in\R$ and $n\in\Z$.
Such a periodicity condition, in particular, may be rephrased as the invariance under the additive action of $\Z$ on $\R$.
The notion of automorphic form essentially arises when we instead consider the action of a general group on a general topological space.

In the case of automorphic forms for $\GL_2(\R)$, the topological space one works with is the complex upper-half plane $\H:= \{ z=x+\i y\colon x\in\R \, , \,\, y>0 \}$, which can also be realized as the set of $\GL_2(\R)$-matrices $z = \left[\begin{smallmatrix} y &x \\ 0 &1 \end{smallmatrix}\right]$ such that $y>0$ and $x\in\R$.
Indeed, the Iwasawa decomposition states that every $g\in\GL_2(\R)$ can be written as $g=z \cdot k \cdot d$, where $z$ is a (uniquely determined) matrix as before, $k$ is a $2\times 2$ orthogonal matrix, and $d$ is a nonzero multiple of the $2\times 2$ identity matrix (i.e.\ belongs to the center $\rm{Z}(\GL_2(\R))$ of $\GL_2(\R)$).
Therefore, we can identify
\[
\H \equiv \GL_2(\R) / \left[ \O_2(\R) \cdot \rm{Z}(\GL_2(\R)) \right] \, .
\]
Consider now the continuous action\footnote{One can check that the right-hand side of~\eqref{eq:actionOnUpperHalfPlane} is indeed an element of $\H$.} of the group $\SL_2(\Z)$ on $\H$ given by
\begin{equation}
\label{eq:actionOnUpperHalfPlane}
\begin{bmatrix}
a &b \\ c &d
\end{bmatrix}
z := \frac{az +b}{cz+d} \, .
\end{equation}
It can be verified that $\SL_2(\Z)$ acts equivalently by matrix multiplication on the elements of $\H$, viewed in their matrix representation as left cosets of $\O_2(\R) \cdot \rm{Z}(\GL_2(\R))$ in $\GL_2(\R)$.

A \emph{Maass form} for $\SL_2(\Z)$ is a smooth function $f\colon \H \to \C$ that:
\begin{enumerate}
\item
is invariant under the action of $\SL_2(\Z)$ on $\H$, i.e.\ $f(g\cdot z) = f(z)$ for all $g\in\SL_2(\Z)$ and $z\in\H$;
\item
is an eigenfunction of the hyperbolic Laplacian
\[
\Delta := -y^2 \left(\frac{\partial^2}{\partial x^2} + \frac{\partial^2}{\partial y^2}\right) \, ,
\]
in particular it satisfies $\Delta f = \nu(1-\nu)f$ for some $\nu\in\C$;
\item
has exponential decay as $y \to \infty$.
\end{enumerate}

The invariance by the action of $\left[\begin{smallmatrix} 1 &n \\ 0 &1 \end{smallmatrix}\right] \in \SL_2(\Z)$ guarantees that $f(z+n)=f(z)$ for all $n\in\Z$.
Namely, $f(x+\i y)$ is a periodic function of $x$ and can thus be expanded in the Fourier series with coefficients depending on $y$.
Such coefficients turn out to be uniquely determined up to a multiplicative constant (\emph{multiplicity one} property) and can be evaluated explicitly in terms of Bessel functions, yielding the series representation:
\begin{equation}
\label{eq:fourierWhittaker}
f(z) = \sum_{n \in \Z\setminus\{0\}} a(n) \cdot W(z,\nu,n) \, , \qquad
W(z,\nu,n) := \sqrt{2\pi y} \cdot K_{\nu - \frac{1}{2}}(2 \pi \abs{n} y) \cdot \e^{2\pi\i n x} \, .
\end{equation}
Here, $K_{\nu}$ is the \emph{Macdonald function} (or modified Bessel function of the second kind), defined by
\begin{equation}
\label{eq:macdonaldFn}
K_{\nu}(x)
:= \frac{1}{2} \int_0^{\infty} t^{\nu}
\exp\left\{ -\frac{x}{2}\left(t+\frac{1}{t}\right)\right\}
\frac{\diff t}{t} \, .
\end{equation}
The function $W(z,\nu,n)$ is called, in this context, \emph{$\GL_2(\R)$-Whittaker function} of type $\nu$ associated to the $n$-th additive character $\e^{2\pi\i n x}$.
Notice that the $\gl_2$-Whittaker function defined in subsection~\ref{subsec:glWhittaker} (see~\eqref{eq:gl_2WhittakerFn}) can also be written in terms of the Macdonald function as
\begin{equation}
\label{eq:gl_2WhittakerMacdonald}
\Psi^{\mathfrak{gl}_2}_{(\alpha_1,\alpha_2)} (x_1,x_2)
= 2 (x_1 x_2)^{(\alpha_1 + \alpha_2)/2}
K_{\alpha_2 - \alpha_1}\bigg(2\sqrt{\frac{x_2}{x_1}}\bigg) \, .
\end{equation}
It is then easy to check that the following equality links the two differently defined Whittaker functions:
\begin{equation}
\label{eq:equivalenceGoldfeldWhittaker}
W(\i y,\nu,1) =
\frac{\Psi^{\mathfrak{gl}_2}_{(\alpha_1,\alpha_2)} (x_1,x_2)}
{\sqrt{2 x_1^{\alpha_1 + \alpha_2 +\frac{1}{2}}
x_2^{\alpha_1 + \alpha_2 -\frac{1}{2}}}} \, ,
\end{equation}
if we set $y=\frac{1}{\pi}\sqrt{\frac{x_2}{x_1}}$ and $\nu-\frac{1}{2}= \alpha_2 -\alpha_1$.

The complex coefficients $a(n)$'s occurring in the Fourier-Whittaker expansion~\eqref{eq:fourierWhittaker} are conjectured to satisfy the growth condition $a(n) = O(d(n))$, where $d(n)$ denotes the number of divisors of $n$ (Ramanujan-Petersson conjecture).
Notice that the coefficient $a(0)$ must be zero because a Maass form is required to have exponential decay as $y\to\infty$.

Let us now define the $L$-function associated with a Maass form $f$ by
\begin{equation}
\label{eq:maassFormLfun}
L_f(s) := \sum_{n=1}^{\infty} \frac{a(n)}{n^{s}} \, .
\end{equation}
From now on, as in the latter definition, we will suppose that $\Re(s)$ is large enough so that all series and integrals converge absolutely.
Number theorists have been interested in studying this function especially because, under certain conditions, it has an Euler product representation, similar to the one of the Riemann zeta function:
\[
L_f(s)
= \prod_{p \text{ prime}} \left(1 - \frac{a(p)}{p^s} + \frac{1}{p^{2s}} \right)^{-1} \, .
\]
For simplicity's sake, let us suppose from now on that $f$ is an \emph{even} Maass form, i.e.\ the coefficients of its Fourier-Whittaker expansion satisfy $a(-n)=a(n)$ for all $n\in\Z$, so that
\begin{equation}
\label{eq:evenMaassForm}
f(\i y) = 2\sum_{n=1}^{\infty} a(n) \cdot W(\i y, \nu,n) 
= 2\sum_{n=1}^{\infty} a(n) \cdot \sqrt{2\pi y} \cdot K_{\nu - \frac{1}{2}}(2 \pi n y) \, .
\end{equation}
In this case, it turns out that $L_f$ satisfies the functional equation
\begin{equation}
\label{eq:funEqMaassFormLfn}
\begin{split}
\Lambda_f(s)
:= &\, \pi^{-s}
\Gamma\left( \frac{s+\frac{1}{2} - \nu}{2} \right)
\Gamma\left( \frac{s-\frac{1}{2} + \nu}{2} \right)
L_f(s) \\
= & \, \Lambda_f(1-s) \, .
\end{split}
\end{equation}
The proof of~\eqref{eq:funEqMaassFormLfn}, similar to Riemann's original proof of the functional equation for the zeta function, relies on the computation of the Mellin transform in $y$ of $f(\i y)$:
\[
\begin{split}
\int_0^{\infty} y^s f(\i y) \frac{\diff y}{y}
&= 2 \sum_{n=1}^{\infty} \frac{a(n)}{n^{s+\frac{1}{2}}}
\int_0^{\infty} y^s \cdot W(\i y,\nu,1) \frac{\diff y}{y} \\
&= 2^{-\frac{1}{2}} \pi^{-s}
L_f\left(s+\frac{1}{2}\right)
\Gamma\left( \frac{1+s-\nu}{2} \right)
\Gamma\left( \frac{s+\nu}{2} \right) \, .
\end{split}
\]
The first equality follows from the Fourier-Whittaker expansion~\eqref{eq:evenMaassForm}, the change of variables $y \mapsto y/n$ in the integral, and the fact that $W(\i y/n,\nu,n) = n^{-1/2} W(\i y,\nu,1)$; the second equality follows from definition~\eqref{eq:maassFormLfun} and an explicit formula for the Mellin transform of the Macdonald function in terms of Gamma functions.
On the other hand, by definition of Maass form, $f$ is invariant by the action of $\left[\begin{smallmatrix} 0 &-1 \\ 1 &0 \end{smallmatrix}\right] \in \SL_2(\Z)$, i.e.\ $f(\i y) = f(\i y^{-1})$.
This clearly implies that the Mellin transform of $f$ is invariant under the transformation $s \mapsto -s$, from which the functional equation~\eqref{eq:funEqMaassFormLfn} follows.
Notice that the function $\Lambda_f(s)$ defined in~\eqref{eq:funEqMaassFormLfn} is given by the $L$-function associated to $f$ multiplied by a prefactor, usually called $L$-factor.
As our computation shows, such an $L$-factor essentially corresponds to the Mellin transform of $W(\i y, \nu, 1)$, which can be computed explicitly in terms of Gamma functions.
This observation points out the importance of the existence of explicit integral identities involving Whittaker functions in the theory of $L$-functions associated to Maass forms.

The integral identities~\eqref{eq:bumpStade} and~\eqref{eq:ishiiStade} that we have used within this thesis involve \emph{two} distinct Whittaker functions instead of one.
In the number theoretic setting, these have arisen in the study of convolution $L$-functions of two Maass forms.
The convolution $L$-function associated with two Maass forms $f$ and $g$ for $\SL_2(\Z)$, of type $\nu$ and $\mu$ respectively and with Fourier-Whittaker coefficients $a(n)$'s and $b(n)$'s respectively, is defined by
\[
L_{f \times g}(s):= \zeta(2s) \sum_{n=1}^{\infty} \frac{a(n) \overline{b(n)}}{n^{s}} \, ,
\]
where $\zeta$ is the Riemann zeta function.
By integrating $f \overline{g}$ against a particular Maass form called Eisenstein series (whose functional equation was already known by other means), Rankin~\cite{rankin39a,rankin39b} and Selberg~\cite{selberg40} found a functional equation for $L_{f \times g}$:
\begin{equation}
\label{eq:funEqConvolutionLFn}
\begin{split}
\Lambda_{f\times g} (s)
:= & \, G_{\nu,\mu}(s) L_{f\times g}(s) \\
= & \, \Lambda_{f\times g} (1-s) \, .
\end{split}
\end{equation}
The $L$-factor $G_{\nu,\mu}(s)$ is essentially given by the Mellin transform of the product of two $\GL_2(\R)$-Whittaker function of type $\mu$ and $\nu$ respectively:
\[
G_{\nu,\mu}(s)
:= \pi^{-s} \Gamma(s)
\int_0^{\infty} y^{s-1} W(\i y,\nu,1) W(\i y,\mu,1) \frac{\diff y}{y} \, ,
\]
and can be evaluated explicitly in terms of Gamma functions.
Such a computation amounts to the Bump-Stade identity~\eqref{eq:bumpStade} in the case $n=2$.
To see this, it suffices to write down the definition of Gamma function as an integral w.r.t.\ $x_1$, change variable in the $y$-integral by setting $y=\frac{1}{\pi}\sqrt{\frac{x_2}{x_1}}$, and use~\eqref{eq:equivalenceGoldfeldWhittaker}:
\[
\begin{split}
G_{\nu,\mu}(s)
&= \pi^{-s}\int_0^{\infty} \frac{\diff x_1}{x_1} \e^{-x_1} x_1^s
\int_0^{\infty} \frac{1}{2} \frac{\diff x_2}{x_2} \left(\frac{1}{\pi} \sqrt{\frac{x_2}{x_1}} \right)^{s-1}
\frac{\Psi^{\mathfrak{gl}_2}_{(\alpha_1,\alpha_2)} (x_1,x_2)}
{\sqrt{2 x_1^{\frac{s+1}{2}}
x_2^{\frac{s-1}{2}}}}
\frac{\Psi^{\mathfrak{gl}_2}_{(\beta_1,\beta_2)} (x_1,x_2)}
{\sqrt{2 x_1^{\frac{s+1}{2}}
x_2^{\frac{s-1}{2}}}} \\
&= \frac{\pi^{1-2s}}{4}
\int_0^{\infty} \int_0^{\infty} \e^{-x_1}
\Psi^{\mathfrak{gl}_2}_{(\alpha_1,\alpha_2)} (x_1,x_2)
\Psi^{\mathfrak{gl}_2}_{(\beta_1,\beta_2)} (x_1,x_2)
\frac{\diff x_1}{x_1}
\frac{\diff x_2}{x_2} \\
&= \frac{\pi^{1-2s}}{4}
\Gamma(\alpha_1 + \beta_1)
\Gamma(\alpha_2 + \beta_1)
\Gamma(\alpha_1 + \beta_2)
\Gamma(\alpha_2 + \beta_2) \\
&= \frac{\pi^{1-2s}}{4}
\Gamma\left(\frac{s+1-\nu-\mu}{2}\right)
\Gamma\left(\frac{s+\nu-\mu}{2}\right)
\Gamma\left(\frac{s-\nu+\mu}{2}\right)
\Gamma\left(\frac{s-1+\nu+\mu}{2}\right) \, .
\end{split}
\]
Here, the choice of parameters $\alpha_i$ and $\beta_i$ ($i=1,2$) is such that $\alpha_1 + \alpha_2 = \beta_1 + \beta_2 = s/2$, $\alpha_2 - \alpha_1 = \nu - \frac{1}{2}$, and $\beta_1 - \beta_2 = \mu - \frac{1}{2}$.
The third equality indeed corresponds to the Bump-Stade identity for $n=2$.

Maass forms for $\SL_n(\Z)$ with $n>2$ are defined on a generalized upper half-plane, which only has a matrix realization but not a complex one: this makes their study considerably more difficult.
However, the $L$-factor of the convolution $L$-function always turns out to be essentially the Mellin transform of the product of two $\GL_n(\R)$-Whittaker functions.
This is essentially the number theoretic motivation for obtaining explicit integral identities involving Whittaker functions, such as~ \eqref{eq:bumpStade} and~\eqref{eq:ishiiStade}.

\section{Ishii-Stade parametrization}
\label{sec:ishiiStadeWhittaker}

The integral parametrization for Whittaker functions used in number theory and in particular by Ishhi and Stade~\cite{ishiiStade13} is different from ours: in this subsection, we explain their connection\footnote{In this setting, we strictly stick to the notation of~\cite{ishiiStade13}.
Accordingly, we remark that the hat in $\hat{W}^A_{n,\bm{a}}$ and $\hat{W}^B_{n,\bm{b}}$ does \emph{not} denote any transform here.}, and then show the equivalence between~\eqref{eq:ishiiStade} and the corresponding integral formula in~\cite{ishiiStade13}.
As it will become clear, for both $\mathfrak{gl}_n$ and $\mathfrak{so}_{2n+1}$ the two parametrizations are linked via the change of variables
\begin{equation}
\label{eq:ishiiStadeChangeOfVars}
y_1 := \frac{1}{\pi} \sqrt{\frac{x_2}{x_1}} 
\, , \quad \dots , \quad
y_{n-1} := \frac{1}{\pi} \sqrt{\frac{x_n}{x_{n-1}}} \, , \quad
y_n := \frac{1}{\pi\sqrt{{x_n}}} \, .
\end{equation}

For $\bm{a}\in\C^n$, set $\abs{a}:= a_1+\dots +a_n$.
The $\mathfrak{gl}_n$-Whittaker function $\hat{W}^A_{n,\bm{a}}$ indexed by $\bm{a}\in\C^n$, according to the integral representation~\cite[Prop.~1.2]{ishiiStade13}, is defined as follows.
For $n=2$,
\begin{equation}
\label{eq:gl_2WhittakerFnIshiiStade}
\hat{W}^A_{2,(a_1,a_2)}(y_1,y_2)
:= 2 y_1^{\abs{a}/2} y_2^{\abs{a}} K_{\frac{a_1 - a_2}{2}}(2\pi y_1) \, ,
\end{equation}
where $K$ is the Macdonald function defined in~\ref{eq:macdonaldFn}.
Recursively, for all $n\geq 3$,
\begin{equation}
\label{eq:glWhittakerFnIshiiStade}
\begin{split}
\hat{W}^A_{n,\bm{a}}(\bm{y})
:=\, &\pi^{-\abs{a}/2} \int_{\R_{+}^{n-1}} \hat{W}^A_{n-1,\tilde{\bm{a}}} \bigg(y_2\sqrt{\frac{t_2}{t_1}},\dots, y_{n-1}\sqrt{\frac{t_{n-1}}{t_{n-2}}}, y_n \frac{1}{\sqrt{t_{n-1}}} \bigg) \\
&\times \prod_{j=1}^{n-1} \exp\Big(- (\pi y_j)^2 t_j - \frac{1}{t_j} \Big) (\pi y_j)^{\frac{(n-j)a_1}{n-1}} t_j^{\frac{n a_1}{2(n-1)}} \frac{\diff t_j}{t_j} \, ,
\end{split}
\end{equation}
where $\tilde{\bm{a}}=(\tilde{a}_1,\dots,\tilde{a}_{n-1})$ is defined by $\tilde{a}_i := a_{i+1} + \frac{a_1}{n-1}$.

\begin{proposition}
\label{prop:glWhittakerFnIshiiStade}
If $a_i = 2 \alpha_{n-i+1}$ for $1\leq i\leq n$, and $\bm{x}$ and $\bm{y}$ satisfy~\eqref{eq:ishiiStadeChangeOfVars}, then
\[
\hat{W}^A_{n,\bm{a}}(\bm{y})
= \pi^{-(n+1)\abs{\bm{\alpha}}}
\Psi^{\mathfrak{gl}_n}_{-\bm{\alpha}} (\bm{x}) \, .
\]
\end{proposition}
\begin{proof}
For $n=2$, eq.~\eqref{eq:gl_2WhittakerFnIshiiStade} and the relations defining $\bm{y}$ and $\bm{a}$ in terms of $\bm{x}$ and $\bm{\alpha}$ yield
\[
\hat{W}^A_{2,(a_1,a_2)}(y_1,y_2)
= 2 \pi^{-3\abs{\bm{\alpha}}}
(x_1 x_2)^{-\abs{\bm{\alpha}}/2}
K_{\alpha_2 - \alpha_1}\bigg(2\sqrt{\frac{x_2}{x_1}}\bigg) \, .
\]
The desired identity for $n=2$ then follows from~\eqref{eq:gl_2WhittakerMacdonald}.

Assume now that the result holds for $n-1$.
In light of~\eqref{eq:ishiiStadeChangeOfVars} and after the change of variables $u_j = t_j x_{j+1}$ for $1\leq j\leq n-1$ in the integral~\eqref{eq:glWhittakerFnIshiiStade}, we obtain
\[
\begin{split}
\hat{W}^A_{n,\bm{a}}(\bm{y})
= \,& \pi^{-\abs{\bm{a}}/2} \int_{\R_{+}^{n-1}} \hat{W}^A_{n-1,\tilde{\bm{a}}} \bigg(\frac{1}{\pi}\sqrt{\frac{u_2}{u_1}},\dots, \frac{1}{\pi}\sqrt{\frac{u_{n-1}}{u_{n-2}}}, \frac{1}{\pi\sqrt{{u_{n-1}}}} \bigg) \\
&\times \prod_{i=1}^n x_i^{-\frac{a_1}{2}}
\prod_{j=1}^{n-1}
\exp\Big( -\frac{u_j}{x_j} - \frac{x_{j+1}}{u_j}\Big)
u_j^{\frac{n a_1}{2(n-1)}}
\frac{\diff u_j}{u_j} \, .
\end{split}
\]
Using the induction hypothesis and the property stated in~\eqref{eq:glWhittakerFnTranslation}, it is easy to see that
\[
\hat{W}^A_{n-1,\tilde{\bm{a}}} \bigg(\frac{1}{\pi}\sqrt{\frac{u_2}{u_1}},\dots, \frac{1}{\pi}\sqrt{\frac{u_{n-1}}{u_{n-2}}}, \frac{1}{\pi\sqrt{{u_{n-1}}}} \bigg)
= \pi^{-n \abs{\bm{\alpha}}}
\bigg(\prod_{j=1}^{n-1} u_j\bigg)^{-\frac{\alpha_n}{n-1}}
\Psi^{\mathfrak{gl}_{n-1}}_{-(\alpha_1,\dots,\alpha_{n-1})}(\bm{u}) \, .
\]
It follows that
\[
\begin{split}
\hat{W}^A_{n,\bm{a}}(y_1,\dots,y_n)
= \pi^{-(n+1)\abs{\bm{\alpha}}}
\int_{\R_{+}^{n-1}} &
\Psi^{\mathfrak{gl}_{n-1}}_{-(\alpha_1,\dots,\alpha_{n-1})}(\bm{u}) \\
&\times\bigg(\frac{\prod_{i=1}^n x_i}{\prod_{i=1}^{n-1}u_i}\bigg)^{-\alpha_n}
\prod_{j=1}^{n-1}
\exp\Big( -\frac{u_j}{x_j} - \frac{x_{j+1}}{u_j}\Big)
\frac{\diff u_j}{u_j} \, .
\end{split}
\]
Using the recursive relation~\eqref{eq:glWhittakerRecurs}, we get the desired identity for $n$.
\end{proof}

The $\mathfrak{so}_{2n+1}$-Whittaker function $\hat{W}^B_{n,\bm{b}}$ indexed by $\bm{b}\in\C^n$, according to the integral representation~\cite[Prop.~1.3]{ishiiStade13}, is defined as follows.
For $n=1$,
\begin{equation}
\label{eq:so_3WhittakerFnIshiiStade}
\hat{W}^B_{1,b_1}(y_1)
:= 2 K_{b_1}(2\pi y_1) \, .
\end{equation}
Recursively, for all $n\geq 2$,
\begin{equation}
\label{eq:soWhittakerFnIshiiStade}
\begin{split}
\hat{W}^B_{n,\bm{b}}(\bm{y})
:=\, &\int_{\R_{+}^n} \int_{\R_{+}^{n-1}}
\hat{W}^B_{n-1,\tilde{\bm{b}}} \bigg(y_2\sqrt{\frac{t_2 s_2}{t_3 s_1}},\dots, y_{n-1}\sqrt{\frac{t_{n-1} s_{n-1}}{t_n s_{n-2}}}, y_n \sqrt{\frac{t_n}{s_{n-1}}} \bigg) \\
&\times \prod_{j=1}^{n-1} \bigg[\exp\Big(- (\pi y_j)^2 \frac{t_j}{t_{j+1}} s_j - \frac{1}{s_j} \Big)
(t_{j+1} s_j)^{\frac{b_n}{2}} \frac{\diff s_j}{s_j} \bigg] \\
&\times t_1^{b_n} \prod_{j=1}^n \bigg[
\exp\Big(-(\pi y_j)^2 t_j - \frac{1}{t_j}\Big)
(\pi y_j)^{b_n} \frac{\diff t_j}{t_j} \bigg] \, ,
\end{split}
\end{equation}
where $\tilde{\bm{b}}=(b_1,\dots,b_{n-1})$.

\begin{proposition}
\label{prop:soWhittakerFnIshiiStade}
If $b_i = 2 \beta_i$ for $1\leq i\leq n$, and $\bm{x}, \bm{y}$ satisfy~\eqref{eq:ishiiStadeChangeOfVars}, then
\[
\hat{W}^B_{n,\bm{b}}(\bm{y})
= \Psi^{\mathfrak{so}_{2n+1}}_{\bm{\beta}} (\bm{x}) \, .
\]
\end{proposition}
\begin{proof}
For $n=1$, we have indeed
\[
\hat{W}^B_{1,b_1}(y_1)
= 2 K_{2\beta_1}\bigg(\frac{2}{\sqrt{x_1}}\bigg)
= \Psi^{\mathfrak{so}_3}_{\beta_1} (x_1) \, .
\]
Here, the first equality follows from~\eqref{eq:so_3WhittakerFnIshiiStade} and the relations defining $\bm{y}$ and $\bm{b}$ in terms of $\bm{x}$ and $\bm{\beta}$, whereas the second equality is deduced by combining \eqref{eq:so_3WhittakerFn} and~\eqref{eq:macdonaldFn}.

Assume now that the result holds for $n-1$.
In light of~\eqref{eq:ishiiStadeChangeOfVars} and after the changes of variables $v_j = x_j / t_j$ for $1\leq j\leq n$ and $u_j = v_{j+1} s_j$ for $1\leq j\leq n-1$ in the integral~\eqref{eq:soWhittakerFnIshiiStade}, we obtain
\[
\begin{split}
\hat{W}^B_{n,\bm{b}}(\bm{y})
= \int_{\R_{+}^n} \int_{\R_{+}^{n-1}}
&\hat{W}^B_{n-1,\tilde{\bm{b}}} \bigg(\frac{1}{\pi}\sqrt{\frac{u_2}{u_1}},\dots, \frac{1}{\pi}\sqrt{\frac{u_{n-1}}{u_{n-2}}}, \frac{1}{\pi\sqrt{{u_{n-1}}}} \bigg)  \,\bigg( \frac{\prod_{j=1}^n v_j^2}{\prod_{j=1}^n x_j \prod_{j=1}^{n-1} u_j} \bigg)^{-\frac{b_n}{2}}\\
&\times
\prod_{j=1}^{n-1} \bigg[\exp\Big(- \frac{u_j}{v_j} -\frac{v_{j+1}}{u_j} \Big) \frac{\diff u_j}{u_j} \bigg]
\prod_{j=1}^n \bigg[ \exp\Big(-\frac{x_{j+1}}{v_j} - \frac{v_j}{x_j}\Big)
\frac{\diff v_j}{v_j} \bigg]
 \, .
\end{split}
\]
Using the induction hypothesis and the fact that $\bm{b}=2\bm{\beta}$, we see that the latter expression coincides with $\Psi^{\mathfrak{so}_{2n+1}}_{(\beta_1,\dots,\beta_{n-1},-\beta_n)}(\bm{x})$ (see recursive formula~\eqref{eq:soWhittakerRecurs}), which in turn equals $\Psi^{\mathfrak{so}_{2n+1}}_{\bm{\beta}} (\bm{x})$ due to the invariance of $\so_{2n+1}$-Whittaker functions under change of sign of the parameters.
\end{proof}

Now, the integral formula we are interested in is stated in~\cite[Thm.~3.2]{ishiiStade13}:
\[
2^n \int_{\R_{+}^n} \bigg(\prod_{j=1}^n y_j^j\bigg)^s
\hat{W}^A_{n,\bm{a}}(\bm{y}) \hat{W}^B_{n,\bm{b}}(\bm{y})
\prod_{j=1}^n \frac{\diff y_j}{y_j}
= \frac{\prod_{1\leq i,j\leq n}
\Gamma_{\mathbf{R}}(s+a_i + b_j)
\Gamma_{\mathbf{R}}(s+a_i-b_j)}
{\prod_{1\leq i<j\leq n} \Gamma_{\mathbf{R}}(2s+a_i+a_j)} \, ,
\]
where $\Gamma_{\mathbf{R}}(z) := \pi^{-z/2} \Gamma(z/2)$.
Using the change of variables~\eqref{eq:ishiiStadeChangeOfVars} and Propositions~\ref{prop:glWhittakerFnIshiiStade} and~\ref{prop:soWhittakerFnIshiiStade}, the above formula can be easily rewritten as
\[
\int_{\R_{+}^n}
\bigg(\prod_{i=1}^n x_i\bigg)^{-s/2}
\Psi_{-\bm{\alpha}}^{\mathfrak{gl}_n}(\bm{x})
\Psi_{\bm{\beta}}^{\mathfrak{so}_{2n+1}}(\bm{x})
\prod_{i=1}^n \frac{\diff x_i}{x_i}
= \frac{\prod_{1\leq i,j\leq n}
\Gamma(s/2+\alpha_i + \beta_j)
\Gamma(s/2 + \alpha_i - \beta_j)}
{\prod_{1\leq i<j\leq n} \Gamma(s + \alpha_i+\alpha_j)} \, .
\]
Theorem~\ref{thm:ishiiStade} now follows by taking $s=0$.
Note that, in turn, the latter identity can be deduced by Theorem~\ref{thm:ishiiStade}: indeed, the term $(\prod_{i=1}^n x_i)^{-s/2}
\Psi_{-\bm{\alpha}}^{\mathfrak{gl}_n}(\bm{x})$
is itself a $\mathfrak{gl}_n$-Whittaker function as a whole, because of~\eqref{eq:glWhittakerFnTranslation}.

\cleardoublepage
\printbibliography[heading=bibintoc]

\end{document}